\documentclass[12pt,leqno]{amsart}
\usepackage{amssymb}
\usepackage{amsmath,amssymb}
\newtheorem{theorem}{Theorem}[section]
\newtheorem{corollary}{Corollary}[section]

\newtheorem{proposition}{Proposition}[section]

\begin{document}
\theoremstyle{plain}
\newtheorem{MainThm}{Theorem}
\newtheorem{thm}{Theorem}[section]
\newtheorem{clry}[thm]{Corollary}
\newtheorem{prop}[thm]{Proposition}
\newtheorem{lem}[thm]{Lemma}
\newtheorem{deft}[thm]{Definition}
\newtheorem{hyp}{Assumption}
\newtheorem*{ThmLeU}{Theorem (J.~Lee, G.~Uhlmann)}

\theoremstyle{definition}
\newtheorem{rem}[thm]{Remark}
\newtheorem*{acknow}{Acknowledgements}
\numberwithin{equation}{section}
\renewcommand{\d}{\partial}
\newcommand{\re}{\mathop{\rm Re} }
\newcommand{\im}{\mathop{\rm Im}}
\newcommand{\R}{\mathbf{R}}
\newcommand{\C}{\mathbf{C}}
\newcommand{\N}{\mathbf{N}}
\newcommand{\D}{C^{\infty}_0}
\renewcommand{\O}{\mathcal{O}}
\newcommand{\dbar}{\overline{\d}}
\newcommand{\supp}{\mathop{\rm supp}}
\newcommand{\abs}[1]{\lvert #1 \rvert}
\newcommand{\csubset}{\Subset}
\newcommand{\detg}{\lvert g \rvert}
\title[partial Cauchy data for
general operator]
{Partial Cauchy data for general second-order
elliptic operators
in two dimensions}

\author[O. Imanuvilov]{Oleg Yu. Imanuvilov}
\address{Department of Mathematics, Colorado State
University, 101 Weber Building, Fort Collins
CO, 80523 USA\\
e-mail: oleg@math.colostate.edu}
\thanks{First author partly supported
by NSF grant DMS 0808130}

\author[G. Uhlmann]{Gunther Uhlmann}
\address{Department  mathematics, UC Irvine, Irvine CA 92697\\
Department of Mathematics,
University of Washington, Seattle,
WA 98195 USA\\
e-mail: gunther@math.washington.edu}
\thanks{Second author partly supported by
NSF and a Walker Family Endowed
Professorship}

\author[M. Yamamoto]{Masahiro Yamamoto}
\address{Department of Mathematical Sciences,
University of Tokyo, Komaba, Meguro,
Tokyo 153, Japan \\e-mail:  myama@ms.u-tokyo.ac.jp}

\begin{abstract}
We consider the inverse problem of determining the
coefficients of a general second-order elliptic operator
in two dimensions by
measuring the corresponding Cauchy data on an
arbitrary open subset of
the boundary. We show
that one can determine the coefficients of the operator up to
natural obstructions such as  conformal invariance,  gauge
transformations and  diffeomorphism invariance.
We use the main result to prove that
the  $curl$ of  the magnetic field and the electric potential are
uniquely determined by measuring partial Cauchy data
associated to the magnetic Schr\"odinger equation
measured on an arbitrary open subset of the boundary.
We also show that any two of the three coefficients
of a  second order elliptic operator
whose principal part is the Laplacian, are uniquely
determined by their partial Cauchy data.
\end{abstract}
\maketitle \setcounter{tocdepth}{1} \setcounter{secnumdepth}{2}

\section{\bf Introduction}
Let $\Omega\subset \R^2$ be a bounded domain with
smooth boundary $\partial\Omega=\cup_{k=1}^{\mathcal N}
\gamma_k$, where $\gamma_k$, $1 \le k \le
{\mathcal N}$, are smooth closed contours, and
$\gamma_{\mathcal N}$ is the external contour.

Let $\widetilde \Gamma \subset \partial\Omega$ be an
arbitrarily fixed non-empty relatively open subset
of $\partial\Omega$ and let
$\Gamma_0=\partial\Omega\setminus
\overline{\widetilde \Gamma}$.
Let $\nu$ be the unit outward normal vector to
$\partial\Omega$ and let $\frac{\partial u}
{\partial\nu} = \nabla u \cdot\nu$.

Henceforth we set $i=\sqrt{-1}$, $x_1, x_2 \in \R$,
$z=x_1+ix_2$,
$\overline{z}$ denotes the complex conjugate
of $z \in \C$, and we identify $x = (x_1,x_2) \in \R^2$ with $z = x_1
+ix_2 \in \C$. We also denote $\frac{\partial}{\partial\overline z} =
\frac 12(\frac{\partial}{\partial x_1}+i\frac{\partial}{\partial x_2})$,
$\frac{\partial}{\partial z}= \frac12(\frac{\partial}{\partial x_1}
-i\frac{\partial}{\partial x_2}).$

Let $u\in H^1(\Omega)$ be a solution to the following boundary value
problem
\begin{equation}\label{OMX}
L(x,D)u = \Delta_g u + 2A\frac{\partial u}{\partial z}
+ 2B\frac{\partial u}{\partial \overline z} + qu = 0,
\quad u\vert_{ \Gamma_0}=0,\quad u\vert_{\widetilde\Gamma}=f.
\end{equation}
Here $\Delta_g$ denotes the Laplace-Beltrami operator associated to
the Riemannian metric $g.$ We assume that $g$ is a positive definite
symmetric matrix in $\Omega$ and
\begin{equation}
\Delta_g=  \frac{1}{\sqrt{\mbox{det} g}}
\sum_{j,k=1}^2\frac{\partial}{\partial x_k}
(\sqrt{\mbox{det} g}\,g^{jk}
\frac{\partial }{\partial x_j}),
\end{equation}
where $\{g^{jk}\}$ denotes the inverse of $g=\{g_{jk}\}.$
From now on we assume that
$g\in C^{7+\alpha}(\overline{\Omega})$,
$(A,B,q),$ $ (A_j,B_j,q_j) \in
C^{5+\alpha}(\overline\Omega)\times C^{5+\alpha}
(\overline\Omega)\times C^{4+\alpha}(\overline\Omega)$,
$j=1,2$ for
some $\alpha\in (0,1)$ are complex-valued functions.
Henceforth $\alpha$ denotes a constant such that
$0 < \alpha < 1$.

We set
$$
L_j(x,D)=\Delta_{g_j} + 2A_j\frac{\partial}{\partial z}
+ 2B_j\frac{\partial}{\partial \overline z}
+ q_j, \quad j=1,2.
$$

We define the partial Cauchy data by
\begin{equation}
{\mathcal C_{g,A,B,q}} = \left\{\left(u|_{\widetilde\Gamma},
\frac{\partial u}{\partial \nu_g}
\Big|_{\widetilde\Gamma}\right)\mid
(\Delta_g+2A\frac{\partial }{\partial z}+2B\frac{\partial }
{\partial \overline z}+q) u= 0\hbox{ in
}\Omega,\,\, \ u\in H^1(\Omega),u\vert_{\Gamma_0}
= 0\right\},
\end{equation}
where $\frac{\partial}{\partial\nu_g}
=\root\of{\mbox{det} g}\sum_{j,k=1}^2 g^{jk}
\nu_k\frac{\partial}{\partial x_j}.$

The goal of this paper is to determine  the coefficients of the
operator $L.$ In the general case this is impossible.
As for the invariance of the Cauchy data, there are of
the following three types.\\
(i) The partial Cauchy data for the operators
$e^{-\eta} L(x,D) e^{\eta}$ and $L(x,D)$ are the
same provided that $\eta\in C^{6+\alpha}
(\overline{\Omega})$ is a complex-valued
function and $\eta\vert_{\widetilde \Gamma}
=\frac{\partial\eta}{\partial\nu}\vert
_{\widetilde \Gamma}=0$.
\\
(ii) Let $\beta \in C^{7+\alpha}(\overline{\Omega})$
be a positive function on $\overline \Omega$.
The partial Cauchy data for the operators $L(x,D)$ and
$\frac 1\beta L(x,D)
= \Delta_{\beta g}+\frac 1\beta(2 A\frac{\partial }{\partial z}
+ 2B\frac{\partial }{\partial \overline z} +  q)$
are exactly the same.
\\
(iii) Let $F \in C^{8+\alpha}(\overline{\Omega}):
\overline \Omega\rightarrow \overline
\Omega$ be a diffeomorphism such that
$F\vert_{\widetilde\Gamma}=Id$.
For any metric $g$ and complex valued functions $A$, $B$, $q$,
we introduce a metric $F^*g$ and functions $A_F$, $B_F$,
$q_F$ by
\begin{equation}\label{zz}
F^*g = ((DF)\circ g\circ (DF)^T )\circ F^{-1},
\end{equation}
\begin{eqnarray*}
&&A_F= \{(A+B)(\frac{\partial F_1}{\partial x_1}
-i\frac{\partial F_2}{\partial x_1})
+ i(B-A)(\frac{\partial F_1}{\partial x_2}
- i\frac{\partial F_2}{\partial x_2})\}\circ F^{-1}
\vert det \,D F^{-1}\vert, \\
&& B_F= \{(A+B)(\frac{\partial F_1}{\partial x_1}
+i\frac{\partial F_2}{\partial x_1})
+i(B-A)(\frac{\partial F_1}{\partial x_2}
+i\frac{\partial F_2}{\partial x_2})\}\circ F^{-1}
\vert det \,D F^{-1}\vert, \\
&& q= \vert det \,D F^{-1}\vert (q\circ F^{-1}),
\end{eqnarray*}
where $DF$ denotes the differential of $F$, $(DF)^T$ its
transpose and $\circ$ denotes matrix composition.

Then the operator
$$
K(x,D) = \Delta_{F^* g_1}
+ 2A_F\frac{\partial}{\partial z}
+ 2B_F\frac{\partial}{\partial \overline z} + q_F
$$
and the operator $L(x,D)$ have the same partial Cauchy data.
\vspace{0.5cm}

We show the converse
and state our main result below.
\\
\vspace{0.3cm}
Assume that for some $\alpha \in (0,1)$ and
$\alpha' > 0$
$$
g_{jk}\in C^{7+\alpha}(\overline \Omega),\quad g_{jk}=g_{kj}\quad
\forall k,j\in\{1,2\}, \quad \sum_{j,k=1}^2g_{jk}\xi_k\xi_j\ge
\alpha'\vert \xi\vert^2, \xi\in\R^2.
$$
Consider the following set of functions
\begin{equation}\label{victoryy}
\eta\in C^{6+\alpha}(\overline \Omega),
\quad \frac{\partial\eta}{\partial\nu}
\vert_{\widetilde\Gamma}=0,\quad \eta\vert_{\widetilde\Gamma}=0.
\end{equation}

We have
\begin{theorem}\label{general}
Suppose that for some $\alpha \in (0,1)$, there exists a positive
function $\widetilde \beta\in C^{7+\alpha}(\overline \Omega)$
such that $(g_1-\widetilde\beta g_2)\vert_{\widetilde \Gamma}
=\frac{\partial(g_1-\widetilde\beta g_2)}
{\partial\nu}\vert_{\widetilde\Gamma}
=0.$
Then
{\it ${\mathcal{C}}_{g_1,A_1,B_1,q_1}
= {\mathcal{C}}_{g_2,A_2,B_2,q_2}$
if and only if
there exist a diffeomorphism
$F\in C^{8+\alpha}(\overline\Omega), F:
\overline \Omega\rightarrow \overline \Omega$ satisfying
$F\vert_{\widetilde\Gamma}=Id$, a
positive function $\beta\in C^{7+\alpha}
(\overline \Omega)$ and a complex valued  function
$\eta$ satisfying (\ref{victoryy}) such that
$$
L_2(x,D) = e^{-\eta}K(x,D)e^{\eta},
$$
where
$$
K(x,D) = \Delta_{\beta F^* g_1}
+\frac{2}{\beta} A_{1,F}\frac{\partial}{\partial z}
+ \frac{2}{\beta}B_{1,F}\frac{\partial}{\partial \overline z}
+\frac 1\beta q_{1,F}.
$$
} and the functions $ F^*g_1, A_{1,F}, B_{1,F},
q_{1,F}$ are defined for $g_1, A_1, B_1, q_1$ by (\ref{zz}).
\end{theorem}



We point out that we can prove that that we can assume
$(g_1-\widetilde\beta g_2)\vert_{\widetilde \Gamma}=0.$ However we
can not determine the normal derivatives  as  pointed out in
\cite{J-G} for the case of the operator $\Delta_g.$ Next we discuss
the case of an anisotropic conductivity problem which is an
independent interest. In this case the conductivity depends on
direction and is represented by a positive definite symmetric matrix
$\sigma^{-1}=\{\sigma^{jk}\}$. The conductivity equation with
voltage potential $f$ on $\partial\Omega$ is given by
$$
\sum_{j,k=1}^2\frac{\partial}{\partial x_j}
(\sigma^{jk}\frac{\partial u}{\partial x_k}) = 0
\quad\mbox{in}\,\,\Omega,
$$
$$
u\vert_{\partial\Omega}=f.
$$

We define the partial Cauchy data by
\begin{equation}
{\mathcal V_{\sigma}} = \Biggl\{\left(f|_{\widetilde\Gamma},
\sum_{j,k=1}^2\sigma^{jk}\nu_j\frac{\partial u}{\partial x_k}
\Big|_{\widetilde\Gamma}\right) \Big\vert
\sum_{j,k=1}^2\frac{\partial}{\partial x_j}
(\sigma^{jk}\frac{\partial u}{\partial x_k})= 0
\end{equation}
$$
\qquad \mbox{in $\Omega,\,\, \ u\in H^1(\Omega)$},\,\,
u\vert_{\partial\Omega}=f, \thinspace \mbox{supp $f \subset
\widetilde{\Gamma}$} \Biggr\}.
$$
It has been known for a long time that even in the case
of $\widetilde{\Gamma} = \partial\Omega$, that
the full Cauchy data $\mathcal{V}_{\sigma}$
does not determine $\sigma$ uniquely in the anisotropic
case \cite{K-V}.
Let $F:\overline\Omega\rightarrow \overline\Omega$ be
a diffeomorphism
such that $F(x)=x$ for $x$ on $\widetilde\Gamma.$ Then
$$
\mathcal{V}_{\vert det\, D F^{-1}\vert F^*\sigma}
= \mathcal{V}_\sigma.
$$

In the case of full Cauchy data (i.e.,
$\widetilde{\Gamma} = \partial\Omega$),
the question whether one can determine the conductivity
up to the above obstruction
has been solved in two dimensions for $C^2$ conductivities in \cite{N},
Lipschitz conductivities in \cite{SuU} and merely
$L^\infty$ conductivities in \cite{ALP}.
The method of proof in all these papers
is the reduction to the isotropic case
using isothermal coordinates \cite{Ah}.
We have
%
\begin{theorem}\label{magia}
Let $\sigma_1, \sigma_2 \in C^{7+\alpha} (\overline\Omega)$ with
some $\alpha\in (0,1)$ be positive definite symmetric matrices on
$\overline\Omega.$ If ${\mathcal{V}}_{\sigma_1} =
{\mathcal{V}}_{\sigma_2}$, then there exists a diffeomorphism
$F:\overline{\Omega} \rightarrow \overline{\Omega}$ satisfying
$F\vert_{\widetilde\Gamma}=Id$ and $F\in
C^{8+\alpha}(\overline\Omega)$  such that
$$
\vert det\, D F^{-1}\vert F^*\sigma_1 = \sigma_2.
$$
\end{theorem}
For the isotropic case, the corresponding result is proved in
\cite{IUY}.
The proof of Theorem \ref{magia} is given in section 6.

Now we take the matrix $g$ to be the identity matrix.
We consider the problem of determining a complex-valued potential
$q$ and complex-valued  coefficients $A$ and $B$  in a bounded
two dimensional domain from the Cauchy data measured on an
arbitrary open subset of the boundary for the associated
second-order elliptic
operator $\Delta +2A\frac{\partial}{\partial z}+2B\frac{\partial}
{\partial \overline z}+ q$.
Specific cases of interest are the magnetic Schr\"odinger operator
and the Laplacian with convection terms. We remark that general
second order elliptic operators can be reduced to this form by using
isothermal coordinates (e.g., \cite{Ah}). The case of the conductivity
equation  and the Schr\"odinger  have been considered in \cite{IUY}. For global uniqueness
results in the two dimensional case
for the conductivity equation with full data measurements under
different regularity assumptions see \cite{AP}, \cite{BU}, \cite{N}.
Such a problem originates in \cite{C}.

Next we will consider the case where the principal part of
$L_j$ is the Laplacian (i.e., $g = I$; the identity
matrix).
Then our next result is the following:
\begin{theorem}\label{main}
Assume that
${\mathcal{C}}_{I,A_1,B_1,q_1} = {\mathcal{C}}
_{I,A_2,B_2,q_2}. $ Then
\begin{equation}\label{zona} A_1=A_2, \quad B_1=B_2
\quad \mbox{on}\quad \widetilde\Gamma,
\end{equation}
\begin{eqnarray}
&-2\frac{\partial}{\partial z}({A}_1-{A}_2)
-({B}_1-{B}_2)A_1
-({A}_1-{A}_2)B_2
+(q_1-q_2)=0  \quad \mbox{in $\Omega$}, \label{o1}\\
&-2\frac{\partial}{\partial \overline{z}}
({B}_1-{B}_2)
-({A}_1-{A}_2)B_1
-({B}_1-{B}_2)A_2
+(q_1-q_2)=0 \quad \mbox{in $\Omega$}.    \label{o2}
\end{eqnarray}
\end{theorem}
{\bf Remark.}  In the case  that $A_1 = A_2$ and $B_1 = B_2$ in
$\Omega$,  Theorem \ref{main} yields that $q_1 = q_2$,
which is the main result
in \cite{IUY}. The latter result was extended to Riemann
surfaces in \cite{GT}.
The case of full data in two dimensions was
settled in \cite{Bu}. This case is closely related to
the inverse conductivity problem, or
 Calder\'on's problem. See  the articles \cite{N}, \cite{BU},
\cite{AP} in two dimensions.

Theorem \ref{main} yields
\begin{corollary}\label{simplyconnected}

The relation $\mathcal{C}_{I,A_1,B_1,q_1}
= \mathcal{C}_{I,A_2,B_2,q_2}$ holds true
if and only if there exists a function $\eta\in C^{6+\alpha}
(\overline \Omega),$ $\eta\vert_{\widetilde \Gamma}
=\frac{\partial\eta}{\partial\nu}\vert_{\widetilde \Gamma}=0$ such that
\begin{equation}\label{opana}
L_1(x,D)=e^{-\eta}L_2(x,D)e^{\eta}.
\end{equation}
\end{corollary}

{\it Proof of Corollary \ref{simplyconnected}.}
We only prove the necessity  since the sufficiency of the condition
is easy to check.
By (\ref{o1}) and (\ref{o2}), we have
$\frac{\partial}{\partial z}({A}_1-{A}_2)
=\frac{\partial}{\partial \overline{z}}({B}_1-{B}_2)$.
This equality is equivalent to
$$
\frac{\partial (\widehat{A}-\widehat{B})}{\partial x_1}
= i\frac{\partial (\widehat{B}+ \widehat{A})}
{\partial x_2}\,\quad\mbox{where} \quad
(\widehat{A}, \widehat{B})=(A_1-A_2,B_1-B_2).
$$
Applying Lemma 1.1 (p.313) of \cite{Tem}, we obtain
that there exists a function $\widetilde\eta$
with domain $\Omega^0$  which satisfies
\begin{equation}\label{victoryy1}
\widetilde\eta=\eta_0+ h, \nabla \widetilde\eta\in C^{5+\alpha}
(\overline\Omega), \,\,
\Delta h=0\quad\mbox{in $\Omega^0$},
\end{equation}
$$
\mbox{$[h]\vert_{\Sigma_k}$ are constants,
$\left[\frac{\partial h}{\partial\nu_k}\right]
\vert_{\Sigma_k} = \frac{\partial h}{\partial\nu}
\vert_{\gamma_{\mathcal N}}=0 \quad
\forall k\in\{1,\dots, \mathcal N\}$}
$$
and
$$
(i(\widehat{B}+\widehat{A}),(\widehat{A}-\widehat{B}))
=\nabla \widetilde\eta.
$$
Here $\Omega^0=\Omega\setminus\Sigma$ is simply
connected where $\Sigma=\cup_{k=1}^{{\mathcal N}-1}
\Sigma_k$, $\Sigma_j\cap \Sigma_k=\emptyset$
for $j \ne k$, $\Sigma_k$ are smooth curves which do not
self-intersect and are orthogonal to $\partial\Omega.$
We choose a normal vector $\nu_k=\nu_k(x)$,
$1 \le k \le {\mathcal N}-1$ to $\Sigma_k$ at $x$
contained in the interior $\Sigma^0_k$ of the closed
curve $\Sigma_k$.  Then, for $x \in \Sigma_k^0$, we set
$[h](x) = \lim_{y\to x, (\vec{xy},\nu_k)>0} h(y)
- \lim_{y\to x, (\vec{xy},\nu_k)<0} h(y)$ where
$(\cdot,\cdot)$ denotes the scalar product in $\R^2$.
Setting $2\eta=-i\widetilde \eta$,
we have
$$
( (\widehat{B}+\widehat{A}),i(\widehat{B}-\widehat{A}))
=2\nabla \eta.$$
Therefore by (\ref{o1})
\begin{equation}\label{finish}
q_1=q_2+\Delta\eta +4\frac{\partial\eta}{\partial z}
\frac{\partial\eta}{\partial\overline z}
+2\frac{\partial\eta}{\partial z}A_2
+2\frac{\partial\eta}{\partial \overline z}B_2.
\end{equation}
The operator $L_1(x,D)$ given by $(\ref{opana})$ has the
Laplace operator as the principal part, the coefficients of
$\frac{\partial}{\partial x_1}$
is $A_2+B_2+2\frac{\partial\eta}{\partial x_2}$, the coefficient of
$\frac{\partial}{\partial x_2}$ is
$i(B_2-A_2)+2\frac{\partial\eta}{\partial  x_1}$,
and the coefficient of
the zero order term is given by the right-hand side of (\ref{finish}).
By (\ref{zona}) we have that
$\frac{\partial\eta}{\partial \nu}\vert_{\widetilde \Gamma}=0$ and
$\eta\vert_{\widetilde \Gamma}=\mathcal C$  where the function
$\mathcal C(x)$ is equal to constant on each connected component
 of $\widetilde \Gamma.$

Let us show that the function $\eta$ is continuous.
Our proof is by contradiction. Suppose that $\eta$ is
discontinuous say along the curve $\Sigma_j.$
Let the function $u_2\in H^1(\Omega)$ be a solution to
the following boundary value problem
\begin{equation}\label{zopa*}
L_2(x,D)u_2=0\quad\mbox{in}\,\,\Omega, \quad u_2\vert_{\Gamma_0}=0.
\end{equation}
Assume in addition that $u_2$ is not identically equal to zero
on $\Sigma_j.$
Let $\widetilde \Gamma_1$ be one connected component of the set
$\widetilde \Gamma$ and $\mathcal C\vert_{\widetilde \Gamma_1}=\widehat C.$
Without loss of generality  we may assume that $\widehat C=0.$
Indeed if $\widehat C\ne 0$, then we replace $\eta$ by the function
$\eta-\widehat C.$
Since the partial Cauchy data of the operators
$L_1(x,D)$ and $L_2(x,D)$ are the same, there exists a
solution $u_1$ to the following boundary value problem
\begin{equation}\label{zopa1}
L_1(x,D)u_1=0 \quad\mbox{in}\,\,\Omega, \quad u_1=u_2\quad
\mbox{on }\partial\Omega, \quad
\frac{\partial u_1}{\partial\nu}=\frac{\partial u_2}{\partial\nu}
\quad \mbox{on}\,\,\widetilde \Gamma.
\end{equation}
Then the function $v=e^{-\eta}u_2$ verifies $$
L_1(x,D)v=0 \quad\mbox{in}\,\,\Omega^0,\quad v\vert_{\Gamma_0}=0.
$$
Since on $\eta=\frac{\partial \eta}{\partial\nu}=0$ on
$\widetilde \Gamma_1$ we have that $v\equiv u_1.$
However
$u_1\in H^1(\Omega)$ and $v$ is discontinuous along one part of
$\Sigma_j$. Thus we arrive at a contradiction.

Let us show that $\mathcal C\equiv 0.$ Suppose that
there exists  another connected component of $\widetilde\Gamma_2$ of
the set $\widetilde \Gamma$ such that
$\mathcal C\vert_{\widetilde\Gamma_2}\ne 0.$
Suppose that the functions $u_1,u_2$ satisfy (\ref{zopa*}) and
(\ref{zopa1}) and $u_1\vert_{\tilde \Gamma_2}$ not identically zero.

Then the function $v=e^{-\eta}u_2$ verifies
$$
L_1(x,D)v=0 \quad\mbox{in}\,\,\Omega,\quad v\vert_{\Gamma_0}=0.
$$
Moreover, since on $\eta=\frac{\partial \eta}{\partial\nu}=0$ on
$\widetilde \Gamma_1$, we have that $$v=u_1, \quad\frac{\partial v}
{\partial\nu}=\frac{\partial u_1}{\partial\nu}\quad \mbox{on}
\quad\widetilde \Gamma_1.$$
By the uniqueness of the Cauchy problem for  second-order
elliptic equations, we have $v\equiv u_1.$ In particular $v=u_1$ on
$\widetilde \Gamma_2$.  Since $u_1=u_2$ on $\partial \Omega$,
this implies that $e^{-\eta}\vert_{\widetilde\Gamma_2}=1.$
We arrived at a contradiction.
The proof of the corollary is completed.
$\quad\qquad \quad \square$

We now apply our result to the case of the magnetic Schr\"odinger
operator. Denote $\widetilde A=(\widetilde A_1,\widetilde A_2)$, where
$\widetilde A_j$ are real-valued,
$\widetilde {\mathcal A}=\widetilde A_1-i\widetilde A_2$,
rot $\widetilde A=\frac{\partial \widetilde
A_2}{\partial x_1}-\frac{\partial \widetilde A_1}{\partial x_2},$
$D_k=\frac 1i\frac{\partial}{\partial x_k}.$  Consider the magnetic
Schr\"odinger operator
\begin{equation}\label{mol11}
\mathcal
L_{\widetilde A,\widetilde q}(x,D)
=\sum_{k=1}^2(D_k+\widetilde A_k)^2+\widetilde q.
\end{equation}
Let us define the following set of partial Cauchy data
$$
\widetilde C_{\widetilde A,\widetilde q}
= \left\{(u\vert_{\widetilde \Gamma},
\frac{\partial u}{\partial \nu}\vert_{\widetilde \Gamma})
\vert \mathcal L_{\widetilde A,\widetilde q}(x,D)u=0\,
\,\mbox{in}\,\,\Omega,
\, u\vert_{\Gamma_0}=0, u\in H^1(\Omega)\right\}.
$$

For the case of full data in two dimensions, it is known that there
is a gauge invariance in this problem and we can recover at best the
$curl$ of the magnetic field \cite{S}. The same is valid for the
three dimensional case with partial Cauchy data \cite{DKSU}. We
prove here that the converse holds in two dimensions.

\begin{corollary}\label{coro}
Let real-valued vector fields
$\widetilde A^{(1)}, \widetilde A^{(2)} \in
C^{5+\alpha}(\overline \Omega)$ and
complex-valued potentials $\widetilde q^{(1)}, \widetilde q^{(2)}
 \in C^{4+\alpha}(\overline\Omega)$ with
some $\alpha \in (0,1)$ satisfy
$\widetilde C_{\widetilde A^{(1)},\widetilde q^{(1)}}
=\widetilde C_{\widetilde A^{(2)},\widetilde q^{(2)}}$.
Then $\widetilde q^{(1)}=\widetilde q^{(2)},$  $\mbox{rot}\,
\widetilde A^{(1)}=\mbox{rot}\, \widetilde A^{(2)}$ in $\Omega$
and $\widetilde A^{(1)}=\widetilde A^{(2)}$ on $\widetilde \Gamma.$
\end{corollary}

\begin{proof} A straightforward calculation gives
\begin{eqnarray}\label{mol1}
\mathcal L_{\widetilde A,\widetilde q}(x,D)
= -\Delta +\frac{2}{i}\widetilde A_1\frac{\partial}{\partial x_1}
+ \frac{2}{i}\widetilde A_2\frac{\partial}{\partial x_2}
+\vert\widetilde A\vert^2+\frac{1}{i}\frac{\partial \widetilde A_1}
{\partial x_1}+\frac{1}{i}\frac{\partial \widetilde A_2}{\partial x_2}
+\widetilde q  \nonumber\\
= -\Delta+\frac{2}{i}{\overline{\widetilde{\mathcal A}}}
\frac{\partial}{\partial z}
+ \frac{2}{i}{\widetilde{\mathcal A}}\frac{\partial}
{\partial\overline z}
+ \frac{2}{i}\frac{\partial \overline{\widetilde{\mathcal A}}}
{\partial z}-\mbox{rot}\,\widetilde A +\vert \widetilde A\vert^2
+\widetilde q.
\end{eqnarray}
Then the operator $\mathcal L_{\widetilde A,\widetilde q}(x,D)$ is a
particular case of (\ref{OMX}) with the metric $g=\{\delta_{ij}\}$
$A=-\frac{1}{i}\overline{\widetilde{\mathcal  A}}$,
$B=-\frac{1}{i}\widetilde{\mathcal  A},$ $q=-(\frac{2}{i}
\frac{\partial \overline{\widetilde{\mathcal A}}} {\partial
z}-\mbox{rot}\,\widetilde A + \vert \widetilde A\vert^2+\widetilde
q).$ Suppose that Schr\"odinger operators with the vector fields
$\widetilde A^{(1)}$, $\widetilde A^{(2)}$ and the potentials
$\widetilde q^{(1)},\widetilde q^{(2)}$ have the same partial Cauchy
data.  Then (\ref{o1}) gives
$$
\mbox{rot}\,\widetilde A^{(1)}-\mbox{rot}\,\widetilde A^{(2)}
+ \widetilde q^{(2)}-\widetilde q^{(1)}\equiv 0
$$
and (\ref{o2}) gives
\begin{equation}\label{NONA}
\frac{2}{i}\frac{\partial{\widetilde{\mathcal A}^{(1)}}}
{\partial \overline z}
- \frac{2}{i}\frac{\partial{\widetilde{\mathcal A}^{(2)}}}
{\partial \overline z}
- \frac{2}{i}\frac{\partial \overline{\widetilde{\mathcal
A}^{(1)}}}{\partial z}
+ \frac{2}{i}\frac{\partial\overline{\widetilde{\mathcal
A}^{(2)}}}{\partial z}+\mbox{rot}\,\widetilde A^{(1)}-\mbox{rot}
\,\widetilde A^{(2)}
+ \widetilde q^{(2)}-\widetilde q^{(1)}\equiv 0.
\end{equation}
Using the identity $\frac 2i \frac{\partial\mathcal A}
{\partial \overline z}
- \frac 2i\frac{\partial\overline{\mathcal A}}{\partial z}
= -2\mbox{rot}\, \widetilde A$, we
transform (\ref{NONA}) to the form
$$
-(\mbox{rot}\,\widetilde A^{(1)}-\mbox{rot}\,\widetilde A^{(2)})
+\widetilde
q^{(2)}-\widetilde q^{(1)}\equiv 0.
$$
The proof of the corollary is completed.
\end{proof}

There is another way to define
partial Cauchy data for the Schr\"odinger operator.
$$
\widehat C_{\widetilde A,\widetilde q}
= \left\{\left(u\vert_{\widetilde \Gamma},
\left(\frac{\partial u}{\partial \nu}+i(\widetilde A,\nu)
u\right)\vert_{\widetilde \Gamma}\right)
\vert \mathcal L_{\widetilde A,\widetilde q}(x,D)u=0\,
\,\mbox{in}\,\,\Omega,
\, u\vert_{\Gamma_0}=0, u\in H^1(\Omega)\right\}.
$$

\begin{corollary}\label{coro1}
Let real-valued vector fields
$\widetilde A^{(1)}, \widetilde A^{(2)} \in
C^{5+\alpha}(\overline \Omega)$ and
complex-valued potentials $\widetilde q^{(1)}, \widetilde q^{(2)}
 \in C^{4+\alpha}(\overline\Omega)$ with
some $\alpha \in (0,1)$ satisfy
$\widehat C_{\widetilde A^{(1)},\widetilde q^{(1)}}
=\widehat  C_{\widetilde A^{(2)},\widetilde q^{(2)}}$.
Then $\widetilde q^{(1)}=\widetilde q^{(2)},$ and  $\mbox{rot}\,
\widetilde A^{(1)}=\mbox{rot}\, \widetilde A^{(2)}$
in $\Omega$.
\end{corollary}
\begin{proof} Suppose that there exist two vector fields
and potentials $(\widetilde A^{(j)},\widetilde q^{(j)})$
such that $\widehat C_{\widetilde A^{(1)},\widetilde q^{(1)}}
=\widehat  C_{\widetilde A^{(2)},\widetilde q^{(2)}}$.
Consider a complex valued function
$\eta\in C^{6+\alpha}(\overline\Omega),
\eta\vert_{\widetilde \Gamma}=0$
such that $i(\nu,\widetilde A^{1}-\widetilde A^{2})
=-\frac 1i\frac{\partial\eta}{\partial\nu}$ on
$\widetilde{\Gamma}$.
Then $\widetilde C_{\widetilde A^{(1)},\widetilde q^{(1)}}
=\widetilde C_{\widetilde A^{(2)}+i\nabla\eta,\widetilde q^{(2)}}$.
Applying Corollary \ref{coro}, we finish the proof.
\end{proof}


Corollaries \ref{coro} and \ref{coro1} were new in Nov. 2009 when the second
author posed a preliminary version of this manuscript with the proof of Theorem 1.3, i.e. the metric $g$ is the Euclidean metric.  Since then proofs for the magnetic Schr\"odinger equation and full data in the Euclidean setting was given in \cite{La} and in the Riemann surface case in \cite{GZ2}.
In two dimensions, Sun proved in \cite{S} that for measurements on
the whole boundary, the  uniqueness holds assuming that
both the magnetic
potential and the electric potential are small. Kang and Uhlmann
proved  global uniqueness for the case of measurements on
the whole boundary for a special case of the magnetic
Schr\"odinger equation,
namely the Pauli Hamiltonian \cite{KU}. In dimensions $n\ge 3$,
 global uniqueness was shown in \cite{nsu} for
the case of full data. The regularity assumptions in the result were
improved in \cite{s} and \cite{Sal}. The case of partial data was
considered in \cite{DKSU}, based on the methods of \cite{KSU} and
\cite{BuU}, with an improvement on the regularity of the
coefficients in \cite{KS}.

Our main theorem implies that the partial Cauchy data can uniquely
determine any two of $(A,B,q)$. First we can prove that $A$ and $B$
are uniquely determined if $q$ is known.  Consider the operator
\begin{equation}\label{OM}
L(x,D)u=\Delta u+a(x)\frac{\partial u}{\partial x_1}
+b(x)\frac{\partial u}{\partial x_2} + q(x)u.
\end{equation}
Here $a$, $b$, $q$ are complex-valued functions.
Let us define the following set
of partial Cauchy data
$$
\widetilde C_{a,b}=\left\{(u\vert_{\widetilde \Gamma},
\frac{\partial u}{\partial \nu}\vert_{\widetilde \Gamma}) \vert
\Delta u+a(x)\frac{\partial u}{\partial x_1}+b(x)\frac{\partial
u}{\partial x_2} + q(x)u=0\,\,\mbox{in}\,\,\Omega, \,
u\vert_{\Gamma_0}=0, u\in H^1(\Omega)\right\}.
$$

We have
\begin{corollary}\label{coroA}
Let $\alpha\in (0,1)$ and two pairs of complex-valued coefficients
$(a^{(1)},b^{(1)}) \in
C^{5+\alpha}(\overline \Omega)\times C^{5+\alpha}(\overline \Omega)$
and $(a^{(2)},b^{(2)}) \in C^{5+\alpha}(\overline \Omega)\times
C^{5+\alpha}(\overline \Omega)$ satisfy
$\widetilde C_{a^{(1)},b^{(1)}}=\widetilde C_{a^{(2)},b^{(2)}}$.
Then $(a^{(1)},b^{(1)})\equiv (a^{(2)},b^{(2)}).$
\end{corollary}

\begin{proof} Taking into account that $\frac{\partial}{\partial x_1}
= (\frac{\partial}{\partial z}+\frac{\partial}{\partial\overline z})$
and $\frac{\partial}{\partial x_2}=i(\frac{\partial}{\partial  z}
-\frac{\partial}{\partial\overline  z})$, we can rewrite the operator
(\ref{OM}) in the form
$$
L(x,D)u=\Delta u+(a(x)+ib(x))\frac{\partial u}{\partial z}
+(a(x)-ib(x))\frac{\partial u}{\partial \overline z} + q(x)u.
$$
For the pairs $(a^{(1)},b^{(1)})$ and $(a^{(2)},b^{(2)})$,
let the corresponding operators defined by (\ref{OM}) have the same
partial Cauchy data. Denote $2A_k(x)=a^{(k)}(x)+ib^{(k)}(x)$ and
$2B_k(x)=a^{(k)}(x)-ib^{(k)}(x).$
By (\ref{o1}), we have
\begin{eqnarray}
&-2\frac{\partial}{\partial z}({A}_1-{A}_2)
-({B}_1-{B}_2)A_1
-({A}_1-{A}_2)
B_2 = 0  \quad \mbox{in $\Omega$},    \label{o1'}\\
&-2\frac{\partial}{\partial \overline{z}}
({B}_1-{B}_2)
-({A}_1-{A}_2)
B_1
-({B}_1-{B}_2)A_2
=0 \quad \mbox{in $\Omega$}.    \label{o2'}
\end{eqnarray}
Applying to equation (\ref{o1'}) the operator $2\frac{\partial}
{\partial \overline{z}}$ and to equation (\ref{o2'}) the operator
$2\frac{\partial}{\partial z}$, we have
\begin{eqnarray}
&-\Delta ({A}_1-{A}_2)
-2\frac{\partial}{\partial \overline z}(({B}_1-{B}_2)A_1
+({A}_1-{A}_2)
B_2) = 0  \quad \mbox{in $\Omega$}, \label{o11'}\\
&-\Delta
({B}_1-{B}_2)
-2\frac{\partial}{\partial z}(({A}_1-{A}_2)
B_1
+({B}_1-{B}_2)A_2)
=0 \quad \mbox{in $\Omega$}.    \label{o22'}
\end{eqnarray}
By (\ref{zona})
$$
(A_1-A_2)\vert_{\widetilde \Gamma}=(B_1-B_2)\vert_{\widetilde \Gamma}=0.
$$
Using these identities and equations (\ref{o1'})
and (\ref{o2'}), we obtain
$$
\frac{\partial (A_1-A_2)}{\partial\nu}\vert_{\widetilde \Gamma}
=\frac{\partial (B_1-B_2)}{\partial\nu}\vert_{\widetilde \Gamma}=0.
$$
The uniqueness of the Cauchy problem for the system
(\ref{o11'})-(\ref{o22'}) can be proved in the standard way by using a
Carleman estimate (e.g., \cite{Ho}).  Therefore we have
$A_1=A_2$ and $B_1=B_2$ in $\Omega.$
\end{proof}

We remark that Corollary \ref{coroA} generalizes
the result of \cite{ChengYama} uniqueness is
obtained assuming that the measurements are made on the
whole boundary. In dimensions $n\ge 3$,  global
uniqueness was shown in
\cite{cns} for the case of full data.

Similarly to Corollary \ref{coroA}, we can prove that
the partial Cauchy data  can uniquely determine
a potential $q$ and one of the coefficients
$A$ and $B$ in (\ref{OMX}).
\begin{corollary}\label{coroB}
For $j=1,2$, let $(A_j,B_j,q_j) \in C^{5+\alpha}(\overline\Omega)
\times C^{5+\alpha} (\overline\Omega)\times C^{4+\alpha}(\overline
\Omega)$ for some $\alpha\in (0,1)$ and be complex-valued.
We assume either $A_1 = A_2$ or $B_1 = B_2$ in $\Omega$.
Then $ \mathcal{C}_{I,A_1,B_1,q_1}= \mathcal{C}_{I,A_2,B_2,q_2}$
implies $(A_1,B_1,q_1) = (A_2,B_2,q_2)$.
\end{corollary}

Corollaries \ref{coroA} and \ref{coroB} mean that the partial Cauchy
data on $\widetilde{\Gamma}$ uniquely determine any two coefficients
of the three coefficients of a second-order elliptic operator whose
principal part is the Laplacian.

The proof of Theorem \ref{general} uses  isothermal coordinates, the
Carleman estimate obtained in section 2,   and Theorem \ref{main}.
In this case we need to prove a new Carleman estimate with
degenerate harmonic weights to construct appropriate complex
geometrical optics solutions.  These solutions are different than
the case of zero magnetic potential. The new form of these solutions
considerably complicate the arguments, especially the asymptotic
expansions needed to analyze the behavior of the solutions. In
Section 2 we prove the Carleman estimate which we need. In Section 3
we state the estimates and asymptotics which we will use in the
construction of the complex geometrical optics solutions. This
construction is done in Section 4. The proof of Theorem \ref{main}
is completed in Section 5. In section 7 and 8 we discuss some
technical lemmas needed in the previous sections.

\section{Carleman estimate}
\noindent {\bf Notations.}
We use throughout the paper the following notations.
$i=\sqrt{-1}$, $x_1, x_2, \xi_1, \xi_2 \in \R$, $z=x_1+ix_2$,
$\zeta=\xi_1+i\xi_2$, $\overline{z}$ denotes the complex conjugate
of $z \in \C$, $D_k=\frac{1}{i}\frac{\partial}{\partial x_k},$
$\beta=(\beta_1,\beta_2)$ where $\beta_j\in {\Bbb N}_+$.
We identify $x = (x_1,x_2) \in \R^2$ with $z = x_1
+ix_2 \in \C$.  We set
$\partial_z = \frac{\partial}{\partial z} =
\frac 12(\frac{\partial}{\partial x_1}-i\frac{\partial}
{\partial x_2})$,
$\partial_{\overline{z}} = \frac{\partial}{\partial\overline{z}} =
\frac12(\frac{\partial}{\partial x_1}+i\frac{\partial}{\partial x_2})$,
$\mathcal O_\epsilon=\{x\in \Omega\vert
\mbox{dist}(x,\partial\Omega)\le \epsilon\}.$
Throughout the paper we use both notations $\partial_z$ and
$\frac{\partial}{\partial z}$, etc. and for example
we denote $\partial_{\overline{z}}^2 = \frac{\partial^2}{\partial \overline{z}^2}$.
We say that a function $a(x)$ is antiholomorphic in $\Omega$ if
$\partial_z a(x)\vert_\Omega \equiv 0$.
The tangential derivative on the boundary is given by
$\frac{\partial}{\partial \vec\tau}=\nu_2\frac{\partial}{\partial x_1}
-\nu_1\frac{\partial}{\partial x_2},$ with $\nu=(\nu_1, \nu_2)$ the
unit outer normal to $\partial\Omega$. The ball of radius $\delta$
centered at
$\widehat x$ is denoted by $B(\widehat x,\delta)=\{x\in
\R^2\vert \vert x-\widehat x\vert< \delta\}$. The corresponding sphere is
denoted by
$S(\widehat x,\delta)=\{x\in \R^2\vert \vert x-\widehat x\vert
= \delta\}. $  If $f:\R^2\rightarrow
\R^1$ is a function, then $f''$ is the Hessian matrix  with
entries $\frac{\partial^2
f}{\partial x_i\partial x_j}.$
Let $\Vert \cdot\Vert_{H^{k,\tau}(\Omega)}^2=\Vert\cdot\Vert^2
_{H^k(\Omega)}+\vert \tau\vert^{2k}\Vert\cdot\Vert^2_{L^2(\Omega)}$
be the norm in
the standard semiclassical Sobolev space with inner product
given by $(\cdot,\cdot)_{H^{k,\tau}(\Omega)}=(\cdot,\cdot)
_{H^{k}(\Omega)}+\vert \tau\vert^{2k}(\cdot,\cdot)_{L^2(\Omega)}.$ For any positive function $d$ we introduce the space $L^2_d(\Omega)=\{v(x)\vert \Vert v\Vert_{L^2_d(\Omega)}=(\int_\Omega d\vert v\vert^2 dx)^\frac 12<\infty\}.$
$\mathcal L(X,Y)$ denotes the Banach
space of all bounded linear operators from a Banach space $X$ to
another Banach space $Y$. By $o_X(\frac{1}{\tau^\kappa})$ we
denote a function $f(\tau,\cdot)$ such that
$
\Vert f(\tau,\cdot)\Vert_X=o(\frac{1}{\tau^\kappa})\quad \mbox{as}
\,\,\vert \tau\vert\rightarrow +\infty.
$
Finally  for any $\widetilde x\in \partial\Omega$ we introduce the left and
right tangential derivatives as follows:
$$
{\mathbf D}_+(\widetilde x) f=\lim_{s\rightarrow +0}
\frac{f(\ell(s))-f(\widetilde x)}{s}
$$
where $\ell(0)=\widetilde x,$ $\ell(s)$ is a parametrization of
$\partial\Omega$ near $\widetilde x$ , $s$ is the length of the curve,
and we are moving clockwise as $s$ increases;
$$
{\mathbf D}_-(\widetilde x) f=\lim_{s\rightarrow -0}
\frac{f(\widetilde \ell(s))-f(\widetilde x)}{s}
$$
where $\widetilde\ell(0)=\widetilde x,$ $\widetilde \ell(s)$ is the
parametrization
of $\partial\Omega$ near $\widetilde x$ , $s$ is the length of
the curve, and we are moving counterclockwise as $s$ increases.

For some $\alpha\in (0,1)$ we consider  a   function $\Phi(z)=\varphi(x_1,x_2)+i\psi(x_1,x_2) \in
C^{6+\alpha}(\overline{\Omega})$ with real-valued $\varphi$ and $\psi$
such that
\begin{equation}\label{zzz}
\frac{\partial\Phi}{\partial \overline z}(z) = 0 \quad \mbox{in} \,\,\Omega,
\quad\mbox{Im}\,\Phi\vert_{\Gamma_0^*}=0
\end{equation}
where $\Gamma_0^*$ is an open set on $\partial\Omega$ such that $\Gamma_0\subset\subset \Gamma_0^*.$
Denote by $\mathcal H$ the set of all the critical points of
the function $\Phi$
$$
\mathcal H = \{z\in\overline\Omega\vert \frac{\partial\Phi}{\partial z}
(z)=0\}.
$$
Assume that $\Phi$ has no critical points on
$\overline{\widetilde\Gamma}$, and that all critical points  are
nondegenerate:
\begin{equation}\label{mika}
\mathcal H\cap \partial\Omega\subset\Gamma_0,\quad
\frac{\partial^2\Phi}{\partial z^2}(z)\ne 0, \quad \forall z\in \mathcal H.
\end{equation}
Then $\Phi$  has only a finite number of critical
points and we can set:
\begin{equation}\label{mona}
\mathcal H \setminus \Gamma_0= \{ \widetilde{x}_1, ...,
\widetilde{x}_{\ell} \},\quad \mathcal H \cap \Gamma_0=\{
\widetilde{x}_{\ell+1}, ..., \widetilde{x}_{\ell+\ell'} \}.
\end{equation}
The following proposition  was proved in \cite{IUY}.

\begin{proposition}\label{Proposition -1}
Let $\widetilde x$ be an arbitrary point in $\Omega.$ There exists a
sequence of functions $\{\Phi_\epsilon\}_{\epsilon\in(0,1)}$ satisfying
(\ref{zzz}) such that all the critical points of $\Phi_\epsilon$ are
nondegenerate and there exists a
sequence $\{\widetilde x_\epsilon\}, \epsilon\in (0,1)$ such that
$$
\widetilde x_\epsilon \in \mathcal H_\epsilon = \{z\in\overline\Omega
\vert
\frac{\partial \Phi_\epsilon}{\partial z}(z)=0 \},\quad
\widetilde x_\epsilon\rightarrow \widetilde x\,\,\mbox{ as}\,\,
\epsilon\rightarrow +0.
$$
Moreover for any $j$ from $\{1,\dots,\mathcal N\}$ we have
$$
\mathcal H_\epsilon\cap\gamma_j=\emptyset\quad\mbox{if}\,\,
\gamma_j\cap \widetilde \Gamma\ne\emptyset,
$$
$$
\mathcal H_\epsilon\cap\gamma_j\subset \Gamma_0 \quad\mbox{if}\,\,
\gamma_j\cap \widetilde \Gamma = \emptyset,
$$
$$
\mbox{Im}\,\Phi_\epsilon(\widetilde x_\epsilon)\notin \{\mbox{Im}\,
\Phi_\epsilon(x)\vert x\in \mathcal H_\epsilon\setminus
\{\widetilde{x_\epsilon}\}\}
\,\,\mbox{and} \,\,\mbox{Im}\,\Phi_\epsilon(\widetilde x_\epsilon)
\ne 0.
$$
\end{proposition}

In order to prove (\ref{zona}) we need the following proposition.
\begin{proposition}\label{Proposition -2}
Let $\hat\Gamma_*\subset\subset\widetilde \Gamma$ be an arc with left
endpoint $x_-$ and  right endpoint $x_+$ oriented clockwise.
For any $\widehat x\in Int\,\Gamma_*$  there exists a function
$\Phi(z)$ which satisfies (\ref{zzz}), (\ref{mika}),
$\mbox{Im}\,\Phi\vert_{\partial\Omega\setminus \Gamma_{*}}=0$ and
\begin{equation}\label{lana0}
\widehat x\in\mathcal
G=\{x\in\hat\Gamma_*\,\vert\quad
\frac{\partial  \mbox{Im}\,\Phi}{\partial \vec\tau}( x)=0\},
\quad card\,
\mathcal G <\infty,
\end{equation}
\begin{equation}
(\frac{\partial}{\partial \vec\tau})^2 \mbox{Im}\,\Phi(x)\ne 0
\quad \forall
x\in \mathcal G\setminus\{x_-,x_+\},
\end{equation}
Moreover
\begin{equation}\label{lana1}
\mbox{Im}\,\Phi(\widehat x)\ne \mbox{Im}\,\Phi(x)\quad \forall
x\in \mathcal G\setminus
\{\widehat x\} \mbox { and }\quad  \mbox{Im}\,\Phi(\widehat x)\neq 0.
\end{equation}
\begin{equation}\label{zopa}{\mathbf D}_+(x_-)
(\frac{\partial}{\partial \vec \tau})^6\mbox{Im}\,\Phi\ne 0,
\quad {\mathbf D}_-(x_+)(\frac{\partial}{\partial \vec \tau})^6
\mbox{Im}\,\Phi\ne 0.
\end{equation}
\end{proposition}
\begin{proof}
Denote $\hat\Gamma_0^*=\partial\Omega\setminus\hat \Gamma_*.$
Let $\widehat x_-, \widehat x_+\in \partial\Omega$ be points such
that the arc $[\widehat x_-,\widehat x_+]\subset (x_-,x_+)$ and
$\widehat x\in (\widehat x_-,\widehat x_+)$ be an arbitrary point and
$x_0$  be another fixed point from the interval
$(\widehat x,\widehat x_+).$
We claim that there exists  a pair $(\varphi,\psi)\in
C^6(\overline\Omega)
\times C^6(\overline \Omega)$ which solves the system of
Cauchy-Riemann equations in $\Omega$
such that
\begin{eqnarray}
{\mbox A})\quad  \psi\vert_{\hat\Gamma_0^*}=0, \vert\frac{\partial
\varphi}{\partial\vec\tau}\vert_{\gamma_j\setminus\hat\Gamma_*}>0\,\,
\mbox{if}\,\, \gamma_j\cap \hat\Gamma_*\ne \emptyset,
\frac{\partial\psi}{\partial\vec \tau}(\widehat x)=0,
(\frac{\partial}{\partial \vec\tau})^2 \psi(\widehat x)\ne 0,
\quad\quad\,\nonumber\\
\quad{\mbox A}')\quad
{\mathbf D}_+(x_-)(\frac{\partial}{\partial \vec \tau})^6\psi (x)
\ne 0,\quad
{\mathbf D}_-(x_+)(\frac{\partial}{\partial \vec \tau})^6\psi(x)
\ne 0,\quad\quad\quad\quad\quad\quad\quad\quad\quad\quad
\,\,\,\,\,\,\nonumber
\end{eqnarray}

$$
{\mbox B})\quad\mbox{The restriction of the function $\psi$ to the
arc $[\widehat x_-,\widehat
x_+]$ is a Morse function,}
$$
$$
{\mbox C})\quad\frac{\partial\psi}{\partial \vec\tau}>0 \,\,\mbox{on}
\quad (x_-,\widehat x_-],\quad   \frac{\partial\psi}
 {\partial\vec \tau}<0 \,\,\mbox{ on } [ \widehat x_+, x_+),
\quad\quad\quad\quad\quad\quad\quad\quad\quad\quad\quad\quad
$$
$$
{\mbox D})\quad
\psi(\widehat x)\notin\{\psi(x)\vert x\in \partial\Omega
\setminus\{\widehat x\}, \frac{\partial\psi}{\partial \vec\tau}(x)=0\},
\quad\quad\quad\quad\quad\quad\quad\quad\quad\quad\quad\quad\quad\quad
$$
$$
{\mbox E})\quad \mbox{if}\,\, \gamma_j\cap  \hat\Gamma_* = \emptyset,
\thinspace \mbox{then the restriction of the function}\,\,
\varphi\quad\mbox{on } \gamma_j\quad\quad\quad\quad\quad\quad
$$
$$ \,\,\mbox{has only two nondegenerate critical points.}
\quad\quad\quad\quad\quad\quad\quad\quad\quad\quad\quad\quad
$$

Such a pair of functions may be constructed in the following way.
Let $ \gamma_1\cap \hat\Gamma_*\ne \emptyset$ and $\gamma_j\cap
\hat\Gamma_*=\emptyset$ for all $j\in\{2,\dots\mathcal N\}.$ First, by
Corollary \ref{MMM1I} in the  Appendix, for some $\alpha\in (0,1),$
there exists a solution
$(\widetilde \varphi, \widetilde \psi)\in
C^{6+\alpha}(\overline\Omega)\times C^{6+\alpha}(\overline \Omega)$
to the Cauchy-Riemann
equations with the following boundary data
$$
\widetilde
\psi\vert_{\partial\Omega\setminus [x_0,\widehat x_{+}]}=\psi_*,\quad
\frac{\partial\widetilde\varphi}{\partial\vec \tau}\vert
_{\gamma_0\setminus
[x_0,\widehat x_{+}]}<\beta<0
$$
and  such that if $\gamma_j\cap\hat \Gamma_* = \emptyset$  the function
$\widetilde\varphi$ has only two nondegenerate critical points located on
the contour $\gamma_j$.
The function $\psi_*$ has the following properties:
$\psi_*\vert_{\hat\Gamma_0^*}=0$, $\frac{\partial\psi_*}{\partial
\vec\tau}>0 \,\,\mbox{on}\,\, (x_-,\widehat x_{-}],\quad
\frac{\partial\psi_*}{\partial \vec\tau}<0 \,\,\mbox{on}\,\,
[ \widehat x_+, x_+).$
The function $\psi_*$ on the set $[\widehat x_{-},x_{0}]$ has only
one critical point $\widehat x$ and $\psi_*(\widehat x)\ne 0.$  On the set
$(x_0,\widehat x_+)$ the Cauchy data is not fixed.
The restriction of the function $\widetilde
\psi$ on $ [x_{0},\widehat x_{+}]$ can be approximated in the space
$C^{6+\alpha}(\overline{[x_{0},\widehat x_{+}]})$ by a sequence of
Morse functions $\{g_\epsilon\}_{\epsilon\in (0,1)}$ such that
$$
(\frac{\partial}{\partial \vec\tau})^k\widetilde\psi(x)=
(\frac{\partial}{\partial \vec\tau})^k g_\epsilon(x)\quad
x\in \{\widehat x_+,x_0\},\quad k\in\{0,1,\dots,6\},
$$
and
$$
\psi_*(\widehat x)\notin \{g_\epsilon (x)\vert \frac{\partial g_\epsilon (x)}
{\partial \vec \tau}=0\}.
$$
Let us consider some arc $\mathcal J\subset\subset(x_-,\widehat x_{-})$.
On this arc we have $\frac{\partial\widetilde
\psi}{\partial\vec \tau}>0$, say,
\begin{equation}\label{Nn}
\frac{\partial\widetilde \psi}{\partial\vec \tau}>\beta'>0\quad
\mbox{on}\,\,\mathcal J \quad \mbox{for some positive}\quad \beta'.
\end{equation}
Let $(\varphi_\epsilon,\psi_\epsilon)\in C^{6+\alpha}
(\overline \Omega)\times C^{6+\alpha}(\overline\Omega)$ be a solution
to the Cauchy-Riemann equations with boundary data
$\psi_\epsilon=0$ on $\partial\Omega\setminus (\mathcal J\cup
[x_{0},\widehat x_+])$ and
$\psi_\epsilon=g_\epsilon-\widetilde \psi$ {on} $\quad [x_{0},
\widehat x_+]$
{and} on $\mathcal J$ the Cauchy data is chosen in such a way that
\begin{equation}\label{NN}\Vert \psi_\epsilon
\Vert_{C^{6+\alpha}({\partial\Omega})}+\Vert \varphi_\epsilon
\Vert_{C^{6+\alpha}({\partial\Omega})}\rightarrow 0\,\,\mbox{as}
\quad\Vert
g_\epsilon-\widetilde \psi\Vert_{C^{6+\alpha}([x_{0},\widehat x_+])}
\rightarrow 0.
\end{equation}
By (\ref{Nn}), (\ref{NN}) for all small positive $\epsilon$,
the restriction of the function $\widetilde \psi+\psi_\epsilon$ to
$\partial\Omega$ satisfies
$$
(\widetilde \psi+\psi_\epsilon)\vert_{\Gamma^*_0}=0,\quad
\frac{\partial(\widetilde \psi+\psi_\epsilon)}{\partial
\nu}\vert_{\gamma_0\setminus [x_0,\widehat x_+]}<0,\quad
\frac{\partial (\widetilde
\psi+\psi_\epsilon)}{\partial\vec\tau}>0\,\, \mbox{on}\,\,
[x_-,\widehat x_-],
$$
$$\quad \frac{\partial (\widetilde
\psi+\psi_\epsilon)}{\partial\vec\tau}<0\,\, \mbox{on}\,\,
[\widehat x_+,x_+],\quad
(\widetilde \psi+\psi_\epsilon)\vert_{[x_{0},\widehat x_+]}=g_\epsilon,
\quad
(\widetilde \psi+\psi_\epsilon)\vert_{[\widehat x_{-}, x_0]}=\psi_*.
$$
 If $j\ge 2$  then the restriction of the function $\varphi_\epsilon
+\widetilde \varphi$  on $\gamma_j$ has only two critical points
located on the  contour $\gamma_j\subset\hat\Gamma_0^*$.
These critical points are nondegenerate if $\epsilon$ is
sufficiently small.

Therefore the restriction of the function $(\widetilde \psi+\psi_\epsilon)$
on $\hat\Gamma_*$ has a finite number of  a critical points.
Some of these points may be the critical points of $(\widetilde \psi
+\psi_\epsilon)$  considered as a function on $\overline\Omega.$ We
change slightly the function $(\widetilde \psi+\psi_\epsilon)$ such that
all of its critical points are in $\Omega.$
Suppose that function $\widetilde \psi+\psi_\epsilon$ has critical
points on $\hat\Gamma_*.$ Then these critical points should be among
the set of critical points of the function $g_\epsilon$,
otherwise it would be
the point $\widehat x.$ We denote these points by $\widehat x_1,\dots,
\widehat x_m.$
Let $(\widehat\varphi,\widehat\psi)\in C^{6+\alpha}(\overline\Omega)
\times C^{6+\alpha}(\overline\Omega)$ be a solution to the Cauchy-Riemann
problem  (\ref{(4.112)}) with the  following boundary data
$$
\widehat\psi\vert_{\Gamma_0^*}=0,\,\, \widehat\psi(\widehat x)=1,\,\,
\widehat\psi\vert_{\mathcal G\setminus \{\widehat x\}}=0, \quad
\vert\frac{\partial\widehat\psi}{\partial\nu}\vert\vert
_{\overline{\gamma_0\setminus\mathcal
J}}>0.
$$
For all small positive $\epsilon_1$ the function $\widetilde
\psi+\psi_\epsilon+\epsilon_1\widehat \psi$ does not have a critical
point on $\partial\Omega$ and the restriction of this function on
$\widetilde \Gamma$ has a finite number of nondegenerate
critical points.  Therefore we take
$(\widetilde\varphi+\varphi_\epsilon+\epsilon_1\widehat \varphi,
\widetilde\psi+\psi_\epsilon+\epsilon_1\widehat \psi)$ as the pairs of
functions satisfying A) - E).

The function $\varphi+i\psi$ with pair $(\varphi,\psi)$ satisfying
conditions A)-E) satisfies all the hypotheses of Proposition
\ref{Proposition -2} except that some of its critical points might
possibly be degenerate. In order to fix this problem we consider a
perturbation of the function $\varphi+i\psi$ which is constructed
in the following way.
By Proposition \ref{balalaika1},
there exists a holomorphic function $w$ in $\Omega$
such that
\begin{equation}\label{zozo}
\mbox{Im}\,w\vert_{\Gamma_0^*}=0,\,\,w\vert_{\mathcal H_{0}}=
\frac{\partial w}{\partial z}\vert_{\mathcal H_{0}}=0,\quad
\frac{\partial^2 w}{\partial z^2}\vert_{\mathcal H_{0}}\ne 0.
\end{equation}
Denote $\Phi_\delta=\varphi+i\psi+\delta w.$ For all sufficiently
small positive  $\delta$, we have
$$
\mathcal H_{0}\subset\mathcal H_\delta \equiv \{x\in\Omega\vert
\frac{\partial}{\partial z}\Phi_\delta(x)=0\}.
$$

We now show that for all sufficiently small positive $\delta$, all
critical
points of the function $\Phi_\delta$ are nondegenerate. Let
$\widetilde x$ be a  critical point of the function $\varphi+i\psi.$
If $\widetilde x$ is a nondegenerate critical point, by the implicit
function theorem, there exists a ball $B(\widetilde x,\delta_1)$
such that the function $\Phi_\delta$ in this ball has only one
nondegenerate critical point for all small $\delta.$ Let $\widetilde
x$ be a degenerate critical point of $\varphi+i\psi.$ Without loss
of generality we may assume that $\widetilde x=0$. In some
neighborhood of $0$, we have
$\frac{\partial\Phi_\delta}{\partial z}=\sum_{k=1}^\infty c_k
z^{k+\widehat k}-\delta
\sum_{k=1}^\infty b_k z^k$ for some natural number $\widehat k$ and some
$c_1\ne 0$.  Moreover (\ref{zozo}) implies $b_1\ne 0.$  Let
$(x_{1,\delta},x_{2,\delta})\in \mathcal H_\delta$ and
$z_\delta=x_{1,\delta}+ix_{2,\delta}\rightarrow 0.$ Then either
\begin{equation}\label{noraa} z_\delta=0\,\,\mbox{ or}\,\,\,
z_\delta^{\widehat k}=\delta b_1/c_1+o(\delta)\quad \mbox{as}\,\,
\delta\rightarrow 0.
\end{equation}
Therefore $\frac{\partial^2\Phi_\delta}{\partial z^2}(z_\delta)\ne 0$
for all sufficiently small
$\delta.$
\end{proof}

Let $\alpha\in (0,1)$ and $\mathcal A,\mathcal B\in C^{6+\alpha}
(\overline\Omega)$ be two complex-valued  solutions to the boundary
value problem
\begin{equation}\label{(2.15)}
2\frac{\partial \mathcal A}{\partial \overline z}= -A\quad
\mbox{in}\,\,\Omega,\,\,\mbox{Im}\,{\mathcal A}\vert_{\Gamma_0}=0,\quad
2\frac{\partial \mathcal{B}}{\partial z}
=-{B}\quad
\mbox{in}\,\,\Omega,\,\,\mbox{Im}\,{\mathcal B}\vert_{\Gamma_0}=0.
\end{equation}

Consider the following boundary value problem
\begin{equation}\label{xoxol}
\frac{\partial a}{\partial \overline z} =0\quad\mbox{in}\,\Omega,\quad
\frac{\partial d}{\partial z} =0\quad\mbox{in}\,\Omega,\quad
(ae^{\mathcal A}+de^{\mathcal B})\vert_{\Gamma_0}=\beta.
\end{equation}
The existence of such functions $a(z)$ and $d(\overline z)$ is given
by the following proposition.

\begin{proposition} \label{zika}
Let $\alpha\in (0,1)$,  $\mathcal A$ and $\mathcal B$ be as
in (\ref{(2.15)}). If $\beta\in C^{5+\alpha}(\overline \Gamma_0)$
(\ref{xoxol}) has at least one solution $(a, d) \in
C^{5+\alpha}(\overline\Omega) \times C^{5+\alpha}(\overline\Omega)$
such that
\begin{equation}\label{novik}
\Vert (a,d)\Vert_{C^{5+\alpha}(\overline \Omega)
\times C^{5+\alpha}(\overline \Omega)}
\le C_{1}\Vert \beta\Vert_{C^{5+\alpha}(\overline \Gamma_0)}.
\end{equation}
If $\beta\in H^\frac 12(\Gamma_0)$, then (\ref{xoxol})
has at least one solution $(a, d)\in H^1(\Omega)\times H^1(\Omega)$
such that
\begin{equation}\label{novikov}
\Vert (a,d)\Vert_{H^1(\Omega)\times H^1(\Omega)}
\le C_{2}\Vert \beta\Vert_{H^\frac 12(\Gamma_0)}.
\end{equation}
\end{proposition}

\begin{proof}Let $\widetilde \Omega$ be a domain in $\Bbb R^2$ with smooth
boundary such that $\Omega\subset\widetilde\Omega$ and there exits an
open subdomain
$\widetilde \Gamma_0 \subset \partial \widetilde \Omega$ satisfying
$\overline \Gamma_0\subset \widetilde \Gamma_0$. Denote $\Gamma^*
=\partial\widetilde\Omega\setminus\widetilde\Gamma_0.$ We extend $\mathcal A,
\mathcal B$ to $\widetilde \Gamma_0$ keeping the regularity and we extend
$\beta$ to $\widetilde \Gamma_0$  in  such a way that
$\Vert \beta\Vert_{H^\frac 12(\widetilde\Gamma_0)}\le C_{3} \Vert
\beta\Vert_{H^\frac 12(\Gamma_0)}$ or $\Vert \beta\Vert_{C^{5+\alpha}
(\overline{\widetilde\Gamma_0})}\le C_{3} \Vert \beta\Vert_{C^{5+\alpha}
(\overline{\Gamma_0})}$
where the constant $C_{3}$ is independent of $\beta.$
By the trace theorem there exist a constant $C_{4}$ independent of
$\beta$,
and a pair $(r,\widetilde r)$ such that
$(re^{\mathcal A}+\widetilde r e^{\mathcal B})\vert_{\widetilde\Gamma_0}=\beta$
and if $\beta\in H^\frac 12(\widetilde\Gamma_0)$ then  $(r,\widetilde r)\in
H^1(\Omega)\times H^1(\Omega)$ and
$$
\Vert (r,\widetilde r)\Vert_{ H^1(\Omega)\times H^1(\Omega)}
\le C_{4}\Vert \beta\Vert_{H^\frac 12(\Gamma_0)}.$$
Similarly if
$\beta\in C^{5+\alpha}(\widetilde\Gamma_0)$ then  $(r,\widetilde r)\in
C^{5+\alpha}(\overline\Omega)\times C^{5+\alpha}(\overline\Omega)$ and
$$
\Vert (r,\widetilde r)\Vert_{ C^{5+\alpha}(\overline\Omega)\times
C^{5+\alpha}(\overline\Omega)}\le C_{5}\Vert \beta\Vert
_{ C^{5+\alpha}(\Gamma_0)}.
$$

Let $f=\frac{\partial r}{\partial z }$ and $\widetilde f
=\frac{\partial \widetilde r}{\partial \overline z }.$
For any $\epsilon$ from $(0,1)$ consider the extremal problem
$$
J_{\epsilon}(p,\widetilde p)
=\Vert(p,\widetilde p)\Vert^2_{L^2(\widetilde\Omega)}+\frac 1\epsilon \Vert
\frac{\partial p}{\partial z}-f\Vert^2_{L^2(\widetilde\Omega)}
+\frac 1\epsilon \Vert \frac{\partial\widetilde p}{\partial \overline z}
-\widetilde f\Vert^2_{L^2(\widetilde\Omega)}\rightarrow \inf,\quad (p,\widetilde p)
\in
\mathcal K,
$$
where
$\mathcal K=\{(h_1,h_2)\in L^2(\widetilde\Omega)\times L^2(\widetilde\Omega)
\vert (h_1
e^{\mathcal A}+h_2e^{\mathcal B})\vert_{\widetilde\Gamma_0}=0\}.$
Denote the solution to this extremal problem as $(p_\epsilon,\widetilde
p_\epsilon).$ Then
$$
J'_{\epsilon}(p_\epsilon,\widetilde p_\epsilon)(\delta,\widetilde\delta)=0
\quad\forall (\delta,\widetilde\delta)\in {\mathcal K}.
$$
Hence
\begin{equation}\label{npma}
((p_\epsilon,\widetilde p_\epsilon),(\delta,\widetilde\delta))_{L^2(\widetilde\Omega)}
+\frac 1\epsilon ( \frac{\partial p_\epsilon}{\partial z}-f,
\frac{\partial \delta}{\partial z})_{L^2(\widetilde\Omega)}
+\frac 1\epsilon ( \frac{\partial \widetilde p_\epsilon}{\partial \overline z}
-\widetilde f,\frac{\partial\widetilde \delta}{\partial \overline z})
_{L^2(\widetilde\Omega)}
=0 \quad\forall (\delta,\widetilde\delta)\in {\mathcal K}.
\end{equation}
Denote $P_\epsilon=-\frac 1\epsilon ( \frac{\partial p_\epsilon}
{\partial z}-f),\widetilde P_\epsilon=-\frac 1\epsilon
( \frac{\partial \widetilde p_\epsilon}{\partial \overline z}-\widetilde f).$
From (\ref{npma}) we obtain
\begin{equation}\label{ika}
\frac{\partial P_\epsilon}{\partial \overline z}=p_\epsilon,\quad
\frac{\partial\widetilde P_\epsilon}{\partial z}=\widetilde p_\epsilon,
\quad P_\epsilon\vert_{\Gamma^*}=\widetilde P_\epsilon\vert_{\Gamma^*}=0,
((\nu_1+i\nu_2)P_\epsilon e^{\mathcal B}-(\nu_1-i\nu_2)
\widetilde P_\epsilon e^{\mathcal A})\vert_{\widetilde \Gamma_0}=0.
\end{equation}
We claim that there exists a constant $C_{6}$ independent of
$\epsilon$ such that
\begin{equation}\label{nika}
\Vert (P_\epsilon,\widetilde P_\epsilon)\Vert_{H^1(\widetilde\Omega)}\le
C_{6}(\Vert (p_\epsilon,\widetilde p_\epsilon)\Vert_{L^2(\widetilde\Omega)}
+\Vert (P_\epsilon,\widetilde P_\epsilon)\Vert_{L^2(\widetilde\Omega)}).
\end{equation}
It clearly suffices to prove the estimate (\ref{nika}) locally assuming
that $\supp\,(p_\epsilon,\widetilde p_\epsilon)$ is in a small neighborhood
of zero and
the vector $(0,1)$ is orthogonal to $\partial\Omega$  on the intersection
of this neighborhood with the boundary. Using a conformal transformation
we may assume
that $\partial\Omega\cap \mbox{supp}\, P_\epsilon,
\partial\Omega\cap \mbox{supp}\, \widetilde P_\epsilon\subset \{ x_1=0\}.$
In order to prove (\ref{nika}) we consider the system of equations
\begin{equation}\label{gorko}
\frac{\partial \bold u}{\partial x_2}+\widehat B\frac{\partial \bold u}
{\partial x_1}={\bold F},\quad \mbox{supp}\,\bold u\subset B(0,\delta)
\cap \{x\vert x_2\ge 0\}.
\end{equation}
Here ${\bold u}=(u_1,u_2,u_3,u_4)=(\mbox {Re}\, P_\epsilon, \mbox{Im}\,
P_\epsilon, \mbox {Re}\,\widetilde P_\epsilon, \mbox{Im}\, \widetilde P_\epsilon),
$ ${\bold F}=2(\mbox{Re} \,p_\epsilon,\mbox{Im} \,p_\epsilon,\mbox{Re}
\,\widetilde p_\epsilon,\mbox{Im} \,\widetilde p_\epsilon)$,
 $\widehat B=\begin{pmatrix}0& 1 & 0 & 0\\-1& 0& 0& 0\\
0& 0& 0 & -1\\0 & 0& 1&0\end{pmatrix}.$
The matrix $\widehat B$ has two eigenvalues $ \pm i$ and four linearly
independent eigenvectors:
$$
{\bold q}_3=(0,0,1,i), {\bold q}_4=(1,-i,0,0) \quad\mbox{corresponding
to the eigenvalue}\,\, -i ,
$$
$$
{\bold q}_1=(1,i,0,0), {\bold q}_2=(0,0,1,-i)\quad \mbox{ corresponding
to the eigenvalue}\,\, i.
$$
We set $\bold r_1=(\nu_1e^{\mathcal B},-\nu_2e^{\mathcal B},
-\nu_1e^{\mathcal A},-\nu_2e^{\mathcal A}),$
$\bold r_2=(\nu_2e^{\mathcal B},\nu_1e^{\mathcal B},
\nu_2e^{\mathcal A},-\nu_1e^{\mathcal A}).$
Consider the matrix $D=\{d_{j\ell}\}$ where $d_{j\ell}=r_j\cdot q_{\ell}$.
We have
$$
D=\begin{pmatrix} (\nu_1-i\nu_2)e^{\mathcal B} & -(\nu_1-i\nu_2)
e^{\mathcal A}\\
(\nu_2+i\nu_1)e^{\mathcal B}& (\nu_2+i\nu_1)e^{\mathcal A}
\end{pmatrix}.
$$
Since the Lopatinski determinant $det\, D \ne 0 $ we  obtain (\ref{nika})
(see e.g. \cite{Wendland}).

Next we need to get rid of the second term in the right hand side of (\ref{nika}).
Suppose that for any $\widetilde C$ one can find $\epsilon$ such that
the estimate
$$
\Vert (P_\epsilon,\widetilde P_\epsilon)\Vert_{H^1(\widetilde\Omega)}
\le \widetilde C\Vert (p_\epsilon,\widetilde p_\epsilon)\Vert_{L^2(\widetilde\Omega)}
$$
fails.  That is, there exist a sequence $\epsilon_k\rightarrow +\infty$ and a sequence $\{C_{\epsilon_k}\}$ such that
 $\lim_{\epsilon_k\to +0}C_{\epsilon_k} =+
\infty$ and
$$
\left\Vert \frac{(p_{\epsilon_k}, \widetilde p_{\epsilon_k})}
{\Vert (P_{\epsilon_k}, \widetilde P_{\epsilon_k})\Vert_{H^1(\widetilde{\Omega})}}
\right\Vert_{L^2(\widetilde{\Omega})} < \frac{1}{C_{\epsilon_k}}.
$$
We set $(Q_{\epsilon_k},\widetilde Q_{\epsilon_k})= (P_{\epsilon_k},\widetilde P_{\epsilon_k})
/\Vert (P_{\epsilon_k},\widetilde P_{\epsilon_k})\Vert_{H^1(\widetilde\Omega)}$
and $(q_{\epsilon_k},\widetilde q_{\epsilon_k})= (p_{\epsilon_k},\widetilde p_{\epsilon_k})
/\Vert (P_{\epsilon_k},\widetilde P_{\epsilon_k})\Vert_{H^1(\widetilde\Omega)}$.
Then $\Vert (q_{\epsilon_k},\widetilde q_{\epsilon_k})\Vert_{L^2(\widetilde\Omega)}
 \to 0$ as $\epsilon_k \to 0$.
Passing to the limit in (\ref{ika}) we have
$$
\frac{\partial Q}{\partial \overline z}=0\quad\mbox{in} \,\widetilde\Omega,\quad
\frac{\partial\widetilde Q}{\partial z}=0\quad\mbox{in} \,\widetilde\Omega,
\quad Q\vert_{\Gamma^*}
=\widetilde Q\vert_{\Gamma^*}=0.
$$
By the uniqueness for the Cauchy problem for the operator  $\partial_z$
we conclude $Q\equiv\widetilde Q\equiv 0.$
On the other hand, since $\Vert (Q_{\epsilon_k},
\widetilde Q_{\epsilon_k})\Vert_{H^1(\widetilde\Omega)} = 1$, we can extract a
subsequence, denoted the same, which is convergent
in $L^2(\widetilde\Omega)$.
Therefore the sequence $\{(Q_{\epsilon_k},\widetilde Q_{\epsilon_k})\}$ converges
to zero in $L^2(\widetilde\Omega)\times L^2(\widetilde\Omega).$
By (\ref{nika}), we have
$1/C_{7}\le \Vert (q_{\epsilon_k},\widetilde q_{\epsilon_k})\Vert_{L^2(\widetilde\Omega)}
+ \Vert (Q_{\epsilon_k},\widetilde Q_{\epsilon_k})\Vert_{L^2(\widetilde\Omega)}$.
Therefore $\lim\inf_{\epsilon_k\to 0} \Vert (Q_{\epsilon_k},\widetilde Q_{\epsilon_k})
\Vert_{L^2(\widetilde\Omega)} \ne 0$, and this is a contradiction.
Hence
$$
\Vert (P_{\epsilon_k},\widetilde P_{\epsilon_k})\Vert_{H^1(\widetilde\Omega)}
\le C_{8}\Vert (p_{\epsilon_k},\widetilde p_{\epsilon_k})\Vert_{L^2(\widetilde\Omega)},
\quad \forall \epsilon>0.
$$

Let us plug in (\ref{npma}) the function $(p_{\epsilon_k},\widetilde p_{\epsilon_k})$
instead of $(\delta,\widetilde \delta)$.   Then, by the above
inequality, in view of the definitions of $P_{\epsilon_k}$ and
$\widetilde P_{\epsilon_k}$, we have
$$
\Vert (p_{\epsilon_k},\widetilde p_{\epsilon_k})\Vert^2
_{L^2(\widetilde\Omega)} \le C_{9}((f, \widetilde f),
(P_{\epsilon_k}, \widetilde P_{\epsilon_k}))_{L^2(\widetilde\Omega)}
\le C_{10}\Vert(f,\widetilde f)\Vert_{L^2(\widetilde\Omega)}
\Vert (P_{\epsilon_k},\widetilde P_{\epsilon_k})\Vert_{L^2(\widetilde\Omega)}
$$
$$
\le C_{11} \Vert(f,\widetilde f)\Vert_{L^2(\widetilde\Omega)}
\Vert (p_{\epsilon_k},\widetilde p_{\epsilon_k})\Vert_{L^2(\widetilde\Omega)}.
$$
This inequality implies that the sequence $ (p_{\epsilon_k},\widetilde p_{\epsilon_k})$
is bounded in $L^2(\widetilde\Omega)$ and
$$
(\frac{\partial p_{\epsilon_k}}{\partial z},\frac{\partial\widetilde p_{\epsilon_k}}
{\partial \overline z})\rightarrow (f,\widetilde f)\quad \mbox{in}\,\,
L^2(\widetilde\Omega)\times L^2(\widetilde\Omega).
$$
Then we construct a solution to (\ref{xoxol}) such that
\begin{equation}\label{lida}
\Vert (p,\widetilde p)\Vert_{L^2(\widetilde\Omega)}\le C_{12}
\Vert (f,\widetilde f)\Vert_{L^2(\widetilde\Omega)}.
\end{equation}
Observe that we can write the boundary value problem
$$
\frac{\partial p}{\partial z}=f\quad\mbox{in}\,\Omega,\quad
\frac{\partial\widetilde p}{\partial \overline z}=\widetilde f\quad\mbox{in}\,
\Omega, \quad
(p e^{\mathcal A}+\widetilde p e^{\mathcal B})\vert_{\widetilde\Gamma_0}=0
$$
in the form of system (\ref{gorko})  with ${\bf u}=(\mbox{Re}\,p,
\mbox{Im}\, p, \mbox{Re}\,\widetilde p,\mbox{Im}\,\widetilde p),
{\bf F}=2(\mbox{Re}\,f,\mbox{Im}\,f, \mbox{Re}\,\widetilde f,
\mbox{Im}\,\widetilde f).$
We set $\bold r_1=(e^{\mathcal A},-e^{\mathcal A}, -e^{\mathcal B},
-e^{\mathcal B}),$  $\bold r_2=(e^{\mathcal A},e^{\mathcal A},
e^{\mathcal B},-e^{\mathcal B}).$
Consider the matrix $D=\{d_{j\ell}\}$
where $d_{j\ell}=r_j\cdot q_\ell$.
We have
$$
D=\begin{pmatrix} e^{\mathcal B} & -e^{\mathcal A}\\
e^{\mathcal B}& e^{\mathcal A}
\end{pmatrix}.
$$
Since the Lopatinski determinant $det\, D \ne 0 $ the estimate (\ref{lida})
imply (\ref{novik}) and (\ref{novikov}) (see e.g.,
\cite{Wendland} Theorem 4.1.2).
This completes the proof of the proposition.
\end{proof}
The following proposition was proven in \cite{IUY}:
\begin{proposition} \label{Proposition 2.1}
Let $\Phi$ satisfy (\ref{zzz}),(\ref{mika}) and the function $C=C_1+iC_2$ belongs to $C^1(\overline\Omega).$ Let $\widetilde
f\in L^2(\Omega)$, and $\widetilde v\in H^1(\Omega)$ be a solution to
\begin{equation}\label{zina}
2 \frac{\partial}{\partial{z}}\widetilde v -\tau\frac{\partial\Phi}
{\partial{z}}\widetilde v+C \widetilde v
=\widetilde f\quad \mbox{in }\,\Omega
\end{equation}
or $\widetilde v$ be a solution to
\begin{equation}\label{zina1}
2\frac{\partial}{\partial{\overline z}}\widetilde v
- \tau \frac{\partial\overline\Phi}{\partial{\overline z}}
\widetilde v+ C\widetilde v =\widetilde f\quad\mbox{ in }\,\Omega.
\end{equation}
In the case (\ref{zina}) we have
\begin{eqnarray}\label{vika1}
\Vert
\frac{\partial \widetilde{v}}{\partial {x_1}}-i\mbox{Im}(\tau\frac{\partial\Phi}
{\partial{z}}-C )\widetilde v\Vert^2
_{L^2(\Omega)} -
\int_{\partial\Omega}(\tau(\nabla\varphi,\nu)-(\nu_1C_1+\nu_2C_2))\vert
\widetilde v\vert^2d\sigma-\int_\Omega (\frac{\partial C_1}{\partial x_1}+\frac{\partial C_2}{\partial x_2})\vert
\widetilde v\vert^2dx\nonumber\\
+ \mbox{Re}\int_{\partial\Omega}i\left(\left(\nu_2
\frac{\partial}{\partial x_1}-\nu_1 \frac{\partial}{\partial
x_2}\right) \widetilde v\right)\overline{\widetilde v}d\sigma \nonumber\\+
\Vert -\frac 1i \frac{\partial \widetilde{v}}{\partial{x_2}}-\mbox{Re}(\tau\frac{\partial\Phi}
{\partial{z}}-C )\widetilde v\Vert^2
_{L^2(\Omega)} = \Vert\widetilde f\Vert^2_{L^2(\Omega).}
\end{eqnarray}
In the case (\ref{zina1}) we have
\begin{eqnarray} \label{vika2}
\Vert \frac{\partial \widetilde{v}}{\partial{x_1}}-i\mbox{Im}(\tau\frac{\partial\overline\Phi}
{\partial{\bar z}}-C )\widetilde v\Vert
_{L^2(\Omega)}
- \int_{\partial\Omega}(\tau(\nabla\varphi,\nu)-(\nu_1C_1-\nu_2C_2))\vert \widetilde
v\vert^2d\sigma -\int_\Omega (\frac{\partial C_1}{\partial x_1}-\frac{\partial C_2}{\partial x_2})\vert
\widetilde v\vert^2dx\nonumber\\+ \mbox{Re}\int_{\partial\Omega}i\left(\left(-\nu_2
\frac{\partial}{\partial x_1}+\nu_1 \frac{\partial}{\partial
x_2}\right)\widetilde v\right)\overline{\widetilde
v}d\sigma\nonumber\\
+ \Vert\frac 1i\frac{\partial \widetilde{v}}{\partial{x_2}}-\mbox{Re}(\tau\frac{\partial\overline\Phi}
{\partial{\bar z}}-C )\widetilde v
\Vert^2
_{L^2(\Omega)} = \Vert\widetilde f\Vert^2_{L^2(\Omega)}.
\end{eqnarray}
\end{proposition}
Consider the boundary value problem
$$\left\{\aligned
&\mathcal{K}(x,{D})u =
(4 \frac{\partial}{\partial{ z}}\frac{\partial}{\partial{\overline z}}
+ 2 {A}\frac{\partial}{\partial{z}}
+2{B} \frac{\partial}{\partial\overline z})u = f
\quad \text{in} \quad \Omega, \\
&u \vert_{\partial \Omega} = 0.
\endaligned\right.
$$

For this problem we have the following Carleman estimate with boundary
terms.

\begin{proposition}\label{Theorem 2.1}
Suppose that $\Phi$ satisfies (\ref{zzz}), (\ref{mika}),
$u\in H^1_0(\Omega)$ and $\Vert A\Vert
_{L^\infty(\overline\Omega)}+\Vert B\Vert
_{L^\infty(\overline\Omega)}\le K$.
Then there exist $\tau_0=\tau_0(K,\Phi)$ and $C_{13}=C_{13}(K,\Phi)$ independent
of $u$ and $\tau$ such that for all $\vert\tau\vert>\tau_0$

\begin{eqnarray}\label{suno4} \vert \tau\vert\Vert
ue^{\tau\varphi}\Vert^2_{L^2(\Omega)}+\Vert
ue^{\tau\varphi}\Vert^2_{H^1(\Omega)}+\Vert\frac{\partial
u}{\partial\nu}e^{\tau\varphi}\Vert^2_{L^2(\Gamma_0)}+
\tau^2\Vert\vert\frac{\partial\Phi}{\partial z} \vert
ue^{\tau\varphi}\Vert^2_{L^2(\Omega)}\nonumber \\
\le C_{13}(\Vert ({\mathcal K }(x,D)
u)e^{\tau\varphi}\Vert^2_{L^2(\Omega)}+\vert\tau\vert
\int_{\widetilde\Gamma^*}\vert
\frac{\partial u}{\partial\nu}\vert^2e^{2\tau\varphi}d\sigma).
\end{eqnarray}
\end{proposition}

\begin{proof}
Denote $\widetilde v=ue^{\tau\varphi},{\mathcal K}(x,D) u=f.$
Observe that
$\varphi(x_1,x_2)=\frac{1}{2}(\Phi(z)+\overline{\Phi(z)})$.
Therefore
\begin{eqnarray}\nonumber
e^{\tau\varphi}\Delta (e^{-\tau\varphi} \widetilde v)= (
2\frac{\partial}{\partial
z}-{\tau}\frac{\partial\Phi}{\partial z})(
2\frac{\partial}{\partial\overline z}
-{\tau}\frac{\partial\overline\Phi} {\partial \overline
z})\widetilde v=\\(2\frac{\partial}{\partial\overline z}
-{\tau}\frac{\partial\overline\Phi} {\partial \overline
z})(2\frac{\partial}{\partial
z}-{\tau}\frac{\partial\Phi}{\partial z})\widetilde
v
=\tilde f=(f-2B\frac{\partial u}{\partial \overline z}-2A\frac{\partial u}{\partial z})e^{\tau\varphi}.
\end{eqnarray}
Assume now that $u$ is a real valued function.
Denote $\widetilde w=(
2\frac{\partial}{\partial\overline z}
-{\tau}\frac{\partial\overline\Phi} {\partial \overline
z})\widetilde v.
$

Thanks to the zero Dirichlet boundary condition for $u$  we have $$
\widetilde w\vert_{\partial\Omega}=
2\frac{\partial\widetilde
v}{\partial \overline z}\vert_{\partial\Omega}=(\nu_1+i\nu_2)\frac{\partial \widetilde v}{\partial
\nu}\vert_{\partial\Omega}.\,\,
$$
Let $\mathcal C$ be  some smooth function in  $\Omega$ such that $$2\frac{\partial \mathcal C}{\partial z}=C(x)=C_1(x)+iC_2(x)\quad \mbox{in}\,\,\Omega, \quad \mbox{Im} \,{\mathcal C}=0 \quad \mbox  {on}\,\,\Gamma_0,$$ where $\vec C=(C_1,C_2)$ is the smooth function in $\Omega$  such that
$$
div\, \vec C=1\quad \mbox{in}\,\,\Omega,\quad (\nu,\vec C)=-1\quad \mbox{on}\,\,\Gamma_0.
$$

 By Proposition \ref{Proposition 2.1} we have the following integral
equalities:
\begin{eqnarray}\label{ipolit} \Vert \frac{\partial (\widetilde
we^{N \mathcal C})}{\partial{x_1}}-i\mbox{Im}(\tau\frac{\partial\overline\Phi}
{\partial{\bar z}}+NC ) (\widetilde
we^{N \mathcal C})\Vert
_{L^2(\Omega)}-
\int_{\partial\Omega}(\tau(\nabla\varphi,\nu)+N(\nu_1C_1+\nu_2C_2))\vert
\frac{\partial\widetilde v}{\partial\nu}e^{N\mathcal
C}\vert^2d\sigma\nonumber\\
+N\int_\Omega\vert \widetilde
we^{N \mathcal C}\vert^2dx+\mbox{Re}\int_{\partial\Omega}i((\nu_2
\frac{\partial}{\partial x_1}-\nu_1 \frac{\partial}{\partial
x_2})(\widetilde we^{N \mathcal C}))\overline{\widetilde
we^{ N\mathcal
C}}d\sigma+\\
+\Vert -\frac 1i \frac{\partial (\widetilde
we^{N \mathcal C})}{\partial{x_2}}-\mbox{Re}(\tau\frac{\partial\Phi}
{\partial{z}}+NC )(\widetilde we^{N \mathcal C})\Vert^2
_{L^2(\Omega)}=\Vert
\tilde fe^{\tau\varphi+N \mathcal C}\Vert^2_{L^2(\Omega)}.\nonumber
\end{eqnarray}

We now simplify the integral $\mbox{Re}\,i\int_{\partial\Omega}((\nu_2
\frac{\partial}{\partial x_1}-\nu_1 \frac{\partial}{\partial
x_2})(\widetilde we^{N\mathcal C}))\overline{\widetilde w_1
e^{N\mathcal C}}d\sigma.$ We recall that $\widetilde v=ue^{\tau
\varphi}$ and $\widetilde w=(\nu_1+i\nu_2)
\frac{\partial \widetilde v}{\partial
\nu}=(\nu_1+i\nu_2)\frac{\partial u}{\partial
\nu} e^{\tau \varphi}.$ Denote
$R+iP=(\nu_1+i\nu_2)e^{N \mbox{Im}\mathcal C}.$ Therefore
\begin{eqnarray} \label{leonid}\mbox{Re}\int_{\partial\Omega}i((\nu_2
\frac{\partial}{\partial x_1}-\nu_1 \frac{\partial}{\partial
x_2})(\widetilde we^{ N\mathcal C}))\overline{\widetilde
we^{ N\mathcal
C}}d\sigma=\\
\mbox{Re}\int_{\partial\Omega}i((\nu_2 \frac{\partial}{\partial
x_1}-\nu_1 \frac{\partial}{\partial x_2})[(R+iP)\frac{\partial
u}{\partial\nu}
e^{\tau \varphi+N\mbox{Re}\,{\mathcal C}}])(R-iP)\frac{\partial u}{\partial\nu}
e^{\tau\varphi+N\mbox{Re}\,{\mathcal C}}d\sigma=\nonumber\\
\mbox{Re}\int_{\partial\Omega}i[(\nu_2 \frac{\partial}{\partial
x_1}-\nu_1 \frac{\partial}{\partial x_2})(R+iP)] \vert\frac{\partial
(\widetilde v e^{N\tilde{\mathcal C}})}{\partial\nu}\vert^2
(R-iP)d\sigma.\nonumber
\end{eqnarray}


Using the above formula we obtain
\begin{eqnarray} \label{suno2}
 \Vert \frac{\partial (\widetilde
we^{N \mathcal C})}{\partial{x_1}}-i\mbox{Im}(\tau\frac{\partial\overline\Phi}
{\partial{\bar z}}+NC ) (\widetilde
we^{N \mathcal C})\Vert
_{L^2(\Omega)}-
\int_{\partial\Omega}(\tau(\nabla\varphi,\nu)+N(\nu_1C_1+\nu_2C_2))\vert
\frac{\partial\widetilde v}{\partial\nu}e^{N\mathcal
C}\vert^2d\sigma\nonumber\\
+N\int_\Omega\vert \widetilde
we^{N \mathcal C}\vert^2dx+\mbox{Re}\int_{\partial\Omega}i[(\nu_2 \frac{\partial}{\partial
x_1}-\nu_1 \frac{\partial}{\partial x_2})(R+iP)] \vert\frac{\partial
(\widetilde v e^{N\mbox{Re}\,{\mathcal C}})}{\partial\nu}\vert^2
(R-iP)d\sigma\nonumber\\
+\Vert -\frac 1i \frac{\partial (\widetilde
we^{N \mathcal C})}{\partial{x_2}}-\mbox{Re}(\tau\frac{\partial\Phi}
{\partial{z}}+NC )(\widetilde w e^{N \mathcal C})\Vert^2
_{L^2(\Omega)}=\Vert
\tilde fe^{\tau\varphi+N \mathcal C}\Vert^2_{L^2(\Omega)}
\end{eqnarray}

Taking the parameter $N$ sufficiently large positive and taking into account that the function $R+iP$ is independent of $N$ on $\Gamma_0$ we conclude from (\ref{suno2})
\begin{eqnarray} \label{suno5}
 -
\int_{\partial\Omega}(\tau(\nabla\varphi,\nu)+\frac{N}{2}(\nu_1C_1+\nu_2C_2))\vert
\frac{\partial\widetilde v}{\partial\nu}e^{N\mathcal
C}\vert^2d\sigma +
N\int_\Omega\vert \widetilde
we^{N \mathcal C}\vert^2dx\\
\le \Vert
\tilde fe^{\tau\varphi}e^{N \mathcal C}\Vert^2_{L^2(\Omega)}+C(N)\int_{\tilde\Gamma}\vert
\frac{\partial\widetilde v}{\partial\nu}e^{N\mathcal
C}\vert^2d\sigma\nonumber
\end{eqnarray}

A simple computation gives
\begin{eqnarray}4\Vert \frac{\partial (\widetilde v e^{N\mbox{Re}\,\mathcal C})}{\partial \bar z}\Vert^2_{L^2(\Omega)}
+\tau^2\Vert\overline{\frac{\partial\Phi}{\partial z}} (\widetilde v e^{N\mbox{Re}\,\mathcal C})\Vert^2_{L^2(\Omega)}=
\Vert 2\frac{\partial (\widetilde v e^{N\mbox{Re}\,\mathcal C})}{\partial \bar z}
-\tau\overline{\frac{\partial\Phi}{\partial z}} (\widetilde v e^{N\mbox{Re}\,\mathcal C})\Vert^2_{L^2(\Omega)}=
                                 \nonumber\\
  \Vert e^{N\mbox{Re}\,\mathcal C}(2\frac{\partial \widetilde v }{\partial \bar z}
-(\tau\overline{\frac{\partial\Phi}{\partial z}} +2\frac{\partial N\mbox{Re}\,\mathcal C}{\partial \bar z} ) \widetilde v)\Vert^2_{L^2(\Omega)}                                                    \nonumber\\
\le 2 \Vert \tilde w e^{N \mathcal C}\Vert^2_{L^2(\Omega)}+C(N) \Vert u e^{\tau\varphi}\Vert^2_{L^2(\Omega)}.
\end{eqnarray}

Since by assumption (\ref{mika}), the function $\Phi$ has zeros of
at most second order, there exists a constant $C_{14}>0$ independent
of $\tau$ such that
\begin{equation}\label{(2.23)}
\tau\Vert \widetilde v  e^{N\mbox{Re}\,\mathcal C}\Vert^2_{L^2(\Omega)}\le C_{14}(\Vert \widetilde
v  e^{N\mbox{Re}\,\mathcal C}\Vert^2_{H^1(\Omega)}+\tau^2\Vert\vert\frac{\partial\Phi}{\partial
z} \vert \widetilde v e^{N\mbox{Re}\,\mathcal C}\Vert^2_{L^2(\Omega)}).
\end{equation}
Therefore by (\ref{suno5})-(\ref{(2.23)}) there exists $N_0>0$ such that for any $N>N_0$ there exists $\tau_0(N)$ that
 \begin{eqnarray}\label{xxx}-
\int_{\partial\Omega}(\tau(\nabla\varphi,\nu)+\frac{N}{2}(\nu_1C_1+\nu_2C_2))\vert
\frac{\partial\widetilde v}{\partial\nu}e^{N\mathcal
C}\vert^2d\sigma +
\frac{N}{2}\int_\Omega\vert \widetilde
we^{N \mathcal C}\vert^2dx\nonumber\\
\tau\Vert \widetilde v  e^{N\mbox{Re}\,\mathcal C}\Vert^2_{L^2(\Omega)}+\Vert \widetilde
v  e^{N\mbox{Re}\,\mathcal C}\Vert^2_{H^1(\Omega)}+\tau^2\Vert\vert\frac{\partial\Phi}{\partial
z} \vert \widetilde v e^{N\mbox{Re}\,\mathcal C}\Vert^2_{L^2(\Omega)}\nonumber\\
\le \Vert
\tilde fe^{\tau\varphi}e^{N \mathcal C}\Vert^2_{L^2(\Omega)}+C_{15}(N)\int_{\tilde\Gamma^*}\vert
\frac{\partial\widetilde v}{\partial\nu}e^{N\mathcal
C}\vert^2d\sigma
\end{eqnarray}
In order to drop the assumption that $u$ is the real valued function we obtain the (\ref{xxx}) separately for the real and imaginary parts of $u$ and combine them.
This concludes the proof of the proposition.
\end{proof}
As a corollary we derive a Carleman inequality for
the function $u$ which satisfies the integral equality
\begin{equation}\label{nos}
(u, \mathcal K(x,D)^*w)_{L^2(\Omega)}+(f,w)_{H^{1,\tau}(\Omega)}
+ (g e^{\tau\varphi},e^{-\tau\varphi} w)_{H^{\frac 12,\tau}(\widetilde
\Gamma)}=0
\end{equation}
for all
$w\in \mathcal X=\{w\in H^1(\Omega)\vert w\vert_{\Gamma_0}=0,
\mathcal K(x,D)^*w\in L^2(\Omega)\}.$
We have
\begin{corollary}\label{mymy}
Suppose that $\Phi$ satisfies (\ref{zzz}), (\ref{mika}),
$f\in H^1(\Omega), g\in H^\frac 12 (\widetilde \Gamma),$
$u\in L^2(\Omega)$ and the coefficients $A,B\in\{ C\in
C^1(\overline\Omega)\vert \Vert C\Vert_{C^1(\overline\Omega)}\le K\}.$
Then there exist $\tau_0=\tau_0(K,\Phi)$ and $C_{16}=C_{16}(K,\Phi)$,
independent of $u$ and $\tau$, such that for solutions of (\ref{nos}):
\begin{eqnarray}\label{suno444}
\Vert
ue^{\tau\varphi}\Vert^2_{L^2(\Omega)} \le C_{16}\vert\tau\vert
(\Vert fe^{\tau\varphi}
\Vert^2_{H^{1,\tau}(\Omega)}+
\Vert g e^{\tau\varphi}\Vert^2_{H^{\frac 12,\tau}(\widetilde \Gamma)})
\quad \forall\vert\tau\vert\ge \tau_0.
\end{eqnarray}
\end{corollary}

\begin{proof} Let $\epsilon$ be some positive number and $d(x)$ be a smooth positive function of $\tilde \Gamma$ which blow up like $\frac{1}{\vert x- y\vert^8}$ for any $y\in\partial\tilde\Gamma.$
Consider the extremal problem
\begin{equation}\label{1}
J_\epsilon(w)=\frac{1}{2}\Vert we^{-\tau\varphi}\Vert^2_{L^2(\Omega)}
+\frac{1}{2\epsilon}\Vert \mathcal K(x,D)^*w-ue^{2\tau\varphi}\Vert^2
_{L^2(\Omega)}+\frac{1}{2\vert\tau\vert}\Vert w e^{-\tau\varphi}\Vert^2
_{L^2_d(\widetilde \Gamma)}\rightarrow \inf,
\end{equation}
\begin{equation}\label{2}
w\in \widehat{\mathcal X}=\{w\in H^\frac 12 (\Omega)\vert \mathcal
K(x,D)^*w\in L^2(\Omega), w\vert_{\Gamma_0}=0\}.
\end{equation}
There exists a unique solution to (\ref{1}), (\ref{2})  which we denote
by $\widehat w_\epsilon.$
By Fermat's theorem
$$
J_\epsilon'(\widehat w_\epsilon)[\delta]=0\quad \forall
\delta\in \widehat{\mathcal X}.
$$
Using the notation $p_\epsilon=\frac 1\epsilon (\mathcal K(x,D)^*
\widehat w_\epsilon-ue^{2\tau\varphi})$ this implies
\begin{equation}\label{3}
\mathcal K(x,D)p_\epsilon+\widehat w_\epsilon e^{-2\tau\varphi}=0
\quad\mbox{in}\,\,\Omega, \quad p_\epsilon\vert_{\partial\Omega}=0,
\quad \frac{\partial p_\epsilon}{\partial\nu}\vert_{\widetilde\Gamma}
=d\frac{\widehat w_\epsilon}{\vert\tau\vert} e^{-2\tau\varphi}.
\end{equation}
By Proposition \ref{Theorem 2.1} we have
\begin{eqnarray}\label{suno4445A} \vert \tau\vert\Vert
p_\epsilon e^{\tau\varphi}\Vert^2_{L^2(\Omega)}+\Vert
p_\epsilon e^{\tau\varphi}\Vert^2_{H^1(\Omega)}+\Vert\frac{\partial
p_\epsilon}{\partial\nu}e^{\tau\varphi}\Vert^2_{L^2(\Gamma_0)}+
\tau^2\Vert\vert\frac{\partial\Phi}{\partial z} \vert
p_\epsilon e^{\tau\varphi}\Vert^2_{L^2(\Omega)}\nonumber
\\
\le C_{11}(\Vert
\widehat w_\epsilon e^{-\tau\varphi}\Vert^2_{L^2(\Omega)}
+\frac{1}{\vert\tau\vert} \int_{\widetilde\Gamma^*}\vert\widehat
w_\epsilon\vert^2e^{-2\tau\varphi}d\sigma)
\le 2C_{17}J_\epsilon(\widehat w_\epsilon).
\end{eqnarray}

Taking the scalar product of equation (\ref{3}) with
$\widehat w_\epsilon$ we obtain
$$
2J_\epsilon(\widehat w_\epsilon)+(u  e^{2\tau\varphi}, p_\epsilon)
_{L^2(\Omega)}=0.
$$
Applying to the second term of the above equality estimate
(\ref{suno4445A}) we have
$$
{\vert\tau\vert}J_\epsilon(\widehat w_\epsilon)
\le C_{18}\Vert ue^{\tau\varphi}\Vert^2_{L^2(\Omega)}.
$$
Using this estimate we pass to the limit in (\ref{3}) as $\epsilon$
goes to zero. We obtain
\begin{equation}\label{13'}
\mathcal K(x,D)p+\widehat we^{-2\tau\varphi}=0\quad\mbox{in}
\,\,\Omega,\quad  p\vert_{\partial\Omega}=0, \quad
\frac{\partial  p}{\partial\nu}\vert_{\widetilde\Gamma}
= d\frac{\widehat w}{\vert \tau\vert}e^{-2\tau\varphi},
\end{equation}
\begin{equation}\label{003'}
\mathcal K(x,D)^*\widehat w- ue^{2\tau\varphi}=0\quad\mbox{in}
\,\,\Omega,\quad \widehat w\vert_{\Gamma_0}=0,
\end{equation}
and
\begin{equation}\label{omo}
\vert\tau\vert\Vert \widehat w e^{-\tau\varphi}\Vert^2_{L^2(\Omega)}
+\Vert \widehat w e^{-\tau\varphi}\Vert^2_{L^2(\widetilde\Gamma)}
\le C_{19}\Vert ue^{\tau\varphi}\Vert^2_{L^2(\Omega)}.
\end{equation}
Since $\widehat w\in L^2(\Omega)$ we have $p\in H^2(\Omega)$ and by
the trace theorem $\frac{\partial p}{\partial \nu}\in H^\frac
12(\partial\Omega).$ The relation (\ref{13'}) implies $\widehat w\in
H^\frac 12(\tilde \Gamma).$ But since $\widehat w\in L^2_d(\tilde
\Gamma)$ and $\widehat w\vert_{\Gamma_0}=0$ we have $\widehat w\in
H^\frac 12(\partial\Omega).$ By (\ref{suno4445A})-(\ref{omo}) we get
\begin{equation}\label{elkao}
\Vert \widehat w e^{-\tau\varphi}\Vert_{H^{\frac 12, \tau}
(\partial\Omega)}
\le  C_{20} \vert\tau\vert^\frac 12\Vert ue^{\tau\varphi}
\Vert_{L^2(\Omega)}.
\end{equation}
Taking the scalar product of (\ref{003'}) with
$\widehat w e^{-2\tau\varphi}$ and using the estimates
(\ref{elkao}), (\ref{omo}) we get
\begin{equation}\label{omm}
\frac{1}{\vert \tau\vert}\Vert \nabla \widehat w e^{-\tau\varphi}
\Vert^2_{L^2(\Omega)}+\vert\tau\vert\Vert \widehat w e^{-\tau\varphi}
\Vert^2_{L^2(\Omega)}+\frac{1}{\vert\tau\vert}
\Vert \widehat w e^{-\tau\varphi}\Vert^2
_{H^{\frac 12,\tau}(\widetilde\Gamma)}\le C_{21}\Vert ue^{\tau\varphi}
\Vert^2_{L^2(\Omega)}.
\end{equation}
From this estimate and a standard duality argument, the
statement of Corollary \ref{mymy} follows immediately.
\end{proof}

Consider the following problem
\begin{equation}\label{(2.26)}
{ L}(x,D)u=f e^{\tau\varphi}\quad \mbox{in}\,\,\Omega,\quad
u\vert_{\Gamma_0}=g e^{\tau\varphi}.
\end{equation}

We have
\begin{proposition}\label{Proposition 2.3}
Let $A,B\in C^{5+\alpha}(\overline \Omega), q\in L^\infty(\Omega)$ and
$\epsilon,\alpha$ be a small positive numbers. There exists $\tau_0>0$
such that for
all $\vert\tau\vert>\tau_0$  there exists a solution to the boundary
value problem
(\ref{(2.26)}) such that \begin{equation} \label{(2.27)}
\frac{1}{\root\of{\vert \tau\vert}}\Vert \nabla ue^{-\tau\varphi}
\Vert_{L^2(\Omega)}+\root\of{\vert \tau\vert}\Vert
ue^{-\tau\varphi}\Vert_{L^2(\Omega)}\le C_{22}(\Vert
f\Vert_{L^2(\Omega)}+\Vert
g\Vert_{H^{\frac 12, \tau}(\Gamma_0)}).
\end{equation}
Let $\epsilon$ be a sufficiently small positive number.
If $supp f\subset G_\epsilon=\{x\in \Omega\vert dist(x,\mathcal
H)>\epsilon\}$  and $g=0$ then there exists
$\tau_0>0$ such that for
all $\vert\tau\vert>\tau_0$  there exists a solution to the
boundary value
problem (\ref{(2.26)}) such that
\begin{equation} \label{(2.277)}
\Vert \nabla ue^{-\tau\varphi}\Vert_{L^2(\Omega)}+\vert \tau\vert \Vert
ue^{-\tau\varphi}\Vert_{L^2(\Omega)}\le C_{23}(\epsilon)\Vert
f\Vert_{L^2(\Omega)}.
\end{equation}
\end{proposition}

\begin{proof}
First we reduce the problem (\ref{(2.26)})  to the case $g=0.$
Let $r(z)$ be a holomorphic function and $\widetilde r(\overline z) $
be an antiholomorphic function  such that $(e^{\mathcal A}r+e^{\mathcal B}
\widetilde r)\vert_{\Gamma_0}=g$  where $\mathcal A, \mathcal B
\in C^{6+\alpha}(\overline \Omega)$ are defined as in (\ref{(2.15)}).
The existence of such functions $r,\widetilde r$ follows from
Proposition
\ref{zika}, and these functions can be chosen in such a way that
$$
\Vert r\Vert_{H^1(\Omega)}+\Vert\widetilde r\Vert_{H^1(\Omega)}
\le C_{24}\Vert g\Vert_{H^\frac 12(\Gamma_0)}.
$$

We look for a solution $u$ in the form
$$
u=(e^{\mathcal A+\tau\Phi}r
+e^{\mathcal B+\tau\overline\Phi}\widetilde r)+\widetilde u,
$$
where
\begin{equation}
{ L}(x,D)\widetilde u=\widetilde f e^{\tau\varphi}\quad \mbox{in}
\,\,\Omega,\quad \widetilde u\vert_{\Gamma_0}=0
\end{equation}
and $\widetilde f= f-(q-2\frac{\partial A}{\partial z}-AB)
e^{\mathcal A}re^{i\tau\psi}-(q-2\frac{\partial B}{\partial \overline z}
-AB)e^{\mathcal B}\widetilde re^{-i\tau\psi}.$

In order to prove (\ref{(2.27)}) we consider the following extremal
problem:
\begin{equation}\label{01'}
\widetilde I_\epsilon(u)=\frac{1}{2}\Vert ue^{-\tau\varphi}\Vert^2
_{H^{1,\tau}(\Omega)}
+ \frac{1}{2\epsilon}\Vert L(x,D)u-\widetilde fe^{\tau\varphi}
\Vert^2_{L^2(\Omega)}+\frac{1}{2}\Vert u e^{-\tau\varphi}
\Vert^2_{H^{\frac 12,\tau}(\widetilde \Gamma)}\rightarrow \inf,
\end{equation}
\begin{equation}\label{02'}
u\in {\mathcal Y}=\{w\in H^1(\Omega)\vert w\vert_{\Gamma_0}=0,
L(x,D)w\in L^2(\Omega)\}.
\end{equation}
There exists a unique solution to problem (\ref{01'}), (\ref{02'})
which we denote as $\widehat u_\epsilon.$
By Fermat's theorem
\begin{equation}\label{NANA}
\widetilde I_\epsilon'(\widehat u_\epsilon)[\delta]=0\quad
\forall \delta\in \mathcal Y.
\end{equation}

Let $p_\epsilon=\frac{1}{\epsilon}(L(x,D)\widehat u_\epsilon
-\widetilde fe^{\tau\varphi}).$ Applying  Corollary \ref{mymy}
we obtain from (\ref{NANA})
\begin{equation}\label{suno444}\frac{1}{ \vert \tau\vert}\Vert
p_\epsilon e^{\tau\varphi}\Vert^2_{L^2(\Omega)}
\le C_{25}(\Vert \widehat u_\epsilon e^{-\tau\varphi}\Vert^2
_{H^{1,\tau}(\Omega)}+\Vert \widehat u_\epsilon e^{-\tau\varphi}\Vert^2
_{H^{\frac 12,\tau}
(\widetilde \Gamma)})
\le 2C_{25}\widetilde I_\epsilon(\widehat u_\epsilon).
\end{equation}

Substituting in (\ref{NANA}) with  $\delta=\widehat u_\epsilon$
we get
$$
2\widetilde I_\epsilon(\widehat u_\epsilon)+(\widetilde f  e^{\tau\varphi},
p_\epsilon)_{L^2(\Omega)}=0.
$$
Applying to this equality estimate (\ref{suno444}) we have
$$
\widetilde I_\epsilon(\widehat u_\epsilon)\le C_{26}{\vert\tau\vert}
\Vert \widetilde f\Vert^2_{L^2(\Omega)}.
$$
Using this estimate we pass to the limit as $\epsilon \rightarrow +0$.
We obtain
\begin{equation}\label{03'}
L(x,D) u- \widetilde fe^{\tau\varphi}=0\quad\mbox{in}\,\,\Omega,\quad
u\vert_{\Gamma_0}=0,
\end{equation}
and
\begin{equation}\label{om63}
\Vert ue^{-\tau\varphi}\Vert^2_{H^{1,\tau}(\Omega)}+\Vert ue^{-\tau\varphi}
\Vert^2_{L^2(\widetilde\Gamma)}\le C_{27}\vert \tau\vert
\Vert\widetilde f\Vert^2_{L^2(\Omega)}.
\end{equation}
Since $\Vert\widetilde f\Vert_{L^2(\Omega)}\le
C_{28}(\Vert f\Vert_{L^2(\Omega)}+\Vert g\Vert
_{H^\frac 12(\Gamma_0)})$, inequality (\ref{om63})
implies (\ref{(2.27)}).

In order to prove (\ref{(2.277)}) we consider the following extremal
problem

\begin{equation}\label{1'}
\widetilde J_\epsilon(u)=\frac{1}{2}\Vert ue^{-\tau\varphi}\Vert^2
_{L^2(\Omega)}+\frac{1}{2\epsilon}\Vert L(x,D)u-fe^{\tau\varphi}\Vert^2
_{L^2(\Omega)}+\frac{1}{2\vert\tau\vert}\Vert u e^{-\tau\varphi}\Vert^2
_{L^2_d(\widetilde \Gamma)}\rightarrow \inf,
\end{equation}
\begin{equation}\label{2'}
u\in \widetilde{\mathcal X}=\{w\in H^\frac 12(\Omega)\vert L(x,D)w\in L^2(\Omega),  w\vert
_{\Gamma_0}=0\}.
\end{equation}(Here $d(x)$ is a smooth positive function of $\tilde \Gamma$ which blow up like $\frac{1}{\vert x- y\vert^8}$ for any $y\in\partial\tilde\Gamma.$)
There exists a unique solution to problem (\ref{1'}), (\ref{2'})
which we
denote as $\widehat u_\epsilon.$
By Fermat's theorem
$$\widetilde J_\epsilon'(\widehat u_\epsilon)[\delta]=0\quad
\forall \delta\in \widetilde{\mathcal X}.
$$
This equality implies
\begin{equation}\label{3A}
L(x,D)^*p_\epsilon+\widehat u_\epsilon e^{-2\tau\varphi}=0
\quad\mbox{in}\,\,\Omega,\quad \widehat p_\epsilon\vert
_{\partial\Omega}=0,\quad \frac{\partial p_\epsilon}{\partial\nu}
\vert_{\widetilde\Gamma}=d\frac{\widehat u_\epsilon}{\vert\tau\vert}
e^{-2\tau\varphi}.
\end{equation}
By Proposition \ref{Theorem 2.1}
$$
\frac{1}{\vert \tau\vert}\Vert
p_\epsilon e^{\tau\varphi}\Vert^2_{H^{1,\tau}(\Omega)}
+ \Vert\frac{\partial p_\epsilon}{\partial\nu}
e^{\tau\varphi}\Vert^2_{L^2(\Gamma_0)}+
\tau^2\Vert\vert\frac{\partial\Phi}{\partial z} \vert
p_\epsilon e^{\tau\varphi}\Vert^2_{L^2(\Omega)}
$$
\begin{equation}\label{veronika'}\le C_{29}(\Vert
\widehat u_\epsilon e^{-\tau\varphi}\Vert^2_{L^2(\Omega)}
+\frac{1}{\vert\tau\vert} \int_{\widetilde\Gamma^*}\vert\widehat
u_\epsilon\vert^2e^{-2\tau\varphi}d\sigma)\le C_{30}\widetilde
J_\epsilon(\widehat u_\epsilon).
\end{equation}

Taking the scalar product of equation (\ref{3A}) with
$\widehat u_\epsilon$ we obtain
$$
2\widetilde J_\epsilon(\widehat u_\epsilon)
+(f  e^{\tau\varphi}, p_\epsilon)_{L^2(\Omega)}=0.
$$
Applying to this equality estimate (\ref{veronika'}) we have
\begin{equation}\label{Zip}
{\vert\tau\vert^2}\widetilde J_\epsilon(\widehat u_\epsilon)\le C_{31}
\Vert f\Vert^2_{L^2(\Omega)}.
\end{equation}
Using this estimate we pass to the limit in (\ref{3A}). We conclude that
\begin{equation}\label{3'}
L(x,D)^*p+ue^{-2\tau\varphi}=0\quad\mbox{in}\,\,\Omega,\quad  p\vert
_{\partial\Omega}=0, \quad\frac{\partial p}{\partial\nu}\vert
_{\widetilde\Gamma}=\frac{u}{\vert \tau\vert}e^{-2\tau\varphi},
\end{equation}
\begin{equation}\label{03'}
L(x,D) u- fe^{\tau\varphi}=0\quad\mbox{in}\,\,\Omega, \quad u\vert
_{\Gamma_0}=0.
\end{equation}
Moreover (\ref{Zip}) implies
\begin{equation}\label{om}
\vert\tau\vert^2\Vert ue^{-\tau\varphi}\Vert^2_{L^2(\Omega)}
+\vert \tau\vert\Vert ue^{-\tau\varphi}\Vert^2_{L^2(\widetilde\Gamma)}
\le C_{32}\Vert f\Vert^2_{L^2(\Omega)}.
\end{equation}
This finishes the proof of the proposition.
\end{proof}
\section{Estimates and Asymptotics}

In this section we prove some estimates and obtain
asymptotic expansions needed in the construction of the complex
geometrical optics solutions in Section 4.

Consider the operator
\begin{eqnarray}\label{mol1}
{L}_1(x,{D})=4 \frac{\partial}{\partial{z}}
\frac{\partial}{\partial{\overline z}}
+2{A}_1 \frac{\partial}{\partial{ z}}+2{B}_1
\frac{\partial}{\partial{\overline z}}
+q_1=\nonumber\\(2\frac{\partial}{\partial{ z}} +{B}_1)
(2\frac{\partial}{\partial{\overline z}}+{A}_1)
+q_1-2 \frac{\partial {A}_1}{\partial z}
-{A}_1 {B}_1=\nonumber\\
(2\frac{\partial}{\partial{\overline z}}+{A}_1)
(2\frac{\partial}{\partial{z}}+{B}_1)
+q_1-2 \frac{\partial {B}_1}{\partial {\overline{z}}}
-{A}_1 {B}_1.
\end{eqnarray}
Let $\mathcal{A}_1, \mathcal{B}_1 \in
C^{6+\alpha}(\overline\Omega),$ with some $\alpha\in
(0,1),$ satisfy
\begin{equation}\label{001q}
2\frac{\partial \mathcal A_1}{\partial\overline z}= -A_1\quad
\mbox{in}\,\,\Omega,\,\,\mbox{Im}\,{\mathcal A_1}\vert_{\Gamma_0}=0,
\quad 2 \frac{\partial \mathcal{B}_1}{\partial z} = -{B}_1
\quad
\mbox{in}\,\,\Omega,\,\,\mbox{Im}\,
{\mathcal B_1}\vert_{\Gamma_0}=0
\end{equation}
and let $\mathcal{A}_2, \mathcal{B}_2 \in
C^{6+\alpha}(\overline\Omega)$ be defined similarly.
Observe that
$$
(2\frac{\partial}{\partial{\overline z}}+ A_1)e^{\mathcal A_1}=0\quad
\mbox{in}\,\,\Omega,\quad
(2\frac{\partial}{\partial{ z}}+{ B_1}) e^{{\mathcal B_1}}=0
\quad\mbox{in}\,\,\Omega.
$$
Therefore if $a(z), \Phi(z)$ are holomorphic functions and
$b(\overline z)$ is an antiholomorphic function, we have
$$
{L}_1(x,{D})(e^{\mathcal{A}_1} ae^{\tau\Phi})=
(q_1 -2 \frac{\partial {{A}_1}}{\partial z}
-{A}_1 {B}_1)e^{\mathcal{A}_1} a e^{\tau\Phi},
$$
$$
{L}_1(x,{D})(e^{\mathcal{B}_1} b e^{\tau\overline\Phi})=
(q_1 -2 \frac{\partial {B}_1}{\partial \overline{z}}
-{A}_1 {B}_1)e^{\mathcal{B}_1} b e^{\tau\overline\Phi}.
$$
Let us introduce the operators:
$$
\partial_{\overline z}^{-1}g=\frac{1}{2\pi i}\int_\Omega
\frac{g(\xi_1,\xi_2)}{\zeta-z} d\zeta\wedge d\overline\zeta=-\frac
1\pi\int_\Omega \frac{g(\xi_1,\xi_2)}{\zeta-z}d\xi_1d\xi_2,
$$
$$
\partial_{ z}^{-1}g=-\frac{1}{2\pi i}\overline{\int_\Omega
\frac{\overline g(\xi_1,\xi_2)}{\zeta-z} d\zeta\wedge
d\overline\zeta}=-\frac 1\pi\int_\Omega
\frac{g(\xi_1,\xi_2)}{\overline\zeta-\overline z}d\xi_1d\xi_2.
$$
We have (e.g., p.47, 56, 72 in \cite{VE}):
\begin{proposition}\label{Proposition 3.0}
{\bf A)} Let $m\ge 0$ be an integer number and $\alpha\in (0,1).$
Then $\partial_{\overline z}^{-1},\partial_{ z}^{-1}\in
\mathcal L(C^{m+\alpha}(\overline{\Omega}),C^{m+\alpha+1}
(\overline{\Omega})).$
\newline
{\bf B}) Let $1\le p\le 2$ and $ 1<\gamma<\frac{2p}{2-p}.$ Then
 $\partial_{\overline z}^{-1},\partial_{ z}^{-1}\in
\mathcal L(L^p( \Omega),L^\gamma(\Omega)).$
\newline
{\bf C})Let $1< p<\infty.$ Then  $\partial_{\overline z}^{-1},
\partial_{ z}^{-1}\in \mathcal L(L^p( \Omega),W^1_p(\Omega)).$
\end{proposition}

Assume that $\mathcal A,\mathcal B, A, B$
satisfy (\ref{(2.15)}). Setting
$T_{B}g=e^{\mathcal
B}\partial^{-1}_z(e^{-\mathcal B} g)$ and
$P_{A}g=e^{\mathcal A}\partial^{-1}_{\overline
z}(e^{-{\mathcal A}} g)$,  for any $g\in C^\alpha(\overline\Omega)$ we have
$$
(2\frac{\partial}{\partial{ z}}+ B)T_{B}g=g\quad\mbox{in}\,\,\Omega,\quad
(2\frac{\partial}{\partial{\overline z}}+{A})P_{A}g=g\quad\mbox{in}\,\,\Omega.
$$
We define two other operators:
\begin{equation}\label{anna}
\mathcal R_{\tau,A}g
= \frac 12e^{{\mathcal A}}e^{\tau(\overline {\Phi}-{\Phi})}
\partial_{\overline z}^{-1}(ge^{-{\mathcal A}}
e^{\tau({\Phi}-\overline {\Phi})}),\,\,
\widetilde {\mathcal R}_{\tau,B}g
= \frac 12e^{{\mathcal B}}e^{\tau(\overline {\Phi}-{\Phi})}
\partial_{ z}^{-1}(ge^{- {\mathcal B}}e^{\tau( {\Phi}
-\overline {\Phi})})
\end{equation}
for $\mathcal A,\mathcal B, A, B$ satisfying (\ref{(2.15)}).

The following proposition follows from straightforward calculations.

\begin{proposition}\label{Proposition 3.1} Let $g\in
C^\alpha(\overline\Omega)$ for some positive $\alpha.$ The function
$\mathcal R_{\tau,A}g$ is a solution to
\begin{equation}\label{(3.2)}
2 \frac{\partial}{\partial{\overline z}}\mathcal{R}_{\tau, {A}}g -2 \tau
\frac{\partial \overline{\Phi} }{\partial\overline{z}}
\mathcal{R}_{\tau, {A}}g
+ {A}\mathcal{R}_{\tau, {A}}g=g \quad \text{in} \quad \Omega.
\end{equation}
The function $\widetilde {\mathcal R}_{\tau,B}g$ solves
\begin{equation}\label{(3.3)}
2 \frac{\partial}{\partial{ z}}\widetilde{\mathcal{R}}_{\tau,{B}}g
+ 2 \tau\frac{\partial \Phi }{\partial {z}}
\widetilde{\mathcal{R}}_{\tau, {B}}g
+{B}\widetilde{\mathcal{R}}_{\tau,{B}}g=g\quad \mbox{in}\,\,\Omega.
\end{equation}
\end{proposition}
Using the stationary phase argument (e.g., Bleistein and Handelsman
\cite{BH}), we will show
\begin{proposition}\label{opa} Let $g\in
L^1(\Omega)$ and a function $\Phi$ satisfy (\ref{zzz}),(\ref{mika}).
Then
$$
\mbox{lim}_{\vert\tau\vert\rightarrow +\infty}\int_\Omega
ge^{\tau(\Phi(z)-\overline{\Phi(z)})}dx=0.
$$
\end{proposition}

\begin{proof} Let $\{ g_k\}^\infty_{k=1}\in C^\infty_0(\Omega)$ be
a sequence of
functions such that $g_k\rightarrow g$ in $L^1(\Omega).$ Let
$\epsilon >0$ be  arbitrary. Suppose that $\widehat j$ is
large enough such that $\Vert g-g_{\widehat j}\Vert_{L^1(\Omega)}\le
\frac{\epsilon}{2}.$ Then
$$
\vert\int_\Omega
ge^{\tau(\Phi(z)-\overline{\Phi(z)})}dx\vert\le \vert\int_\Omega(g-
g_{\widehat
j})e^{\tau(\Phi(z)-\overline{\Phi(z)})}dx\vert+\vert\int_\Omega
g_{\widehat j}e^{\tau(\Phi(z)-\overline{\Phi(z)})}dx\vert.
$$
The first term on the right-hand side of this inequality is less then
$\epsilon/2$ and the second goes to zero as  $\vert \tau\vert$
approaches to infinity by the stationary phase argument (see e.g.
\cite{BH}).
\end{proof}
We have
\begin{proposition}\label{Proposition 3.22}
Let $g\in C^2(\Omega), g\vert_{\mathcal O_\epsilon} = 0$ and
$g\vert_{\mathcal H}= 0$.  Then for any $1\le p<\infty$
\begin{equation}\label{(3.4A)}
\left\Vert \mathcal R_{\tau,A}g+\frac {g}{2\tau\overline{\partial_z \Phi}}
\right\Vert_{L^p(\Omega)}
+ \left\Vert\widetilde {\mathcal R}_{\tau,B} g-\frac {g}{2\tau{\partial_z
\Phi}}\right\Vert_{L^p(\Omega)}
= o\left(\frac 1\tau\right)\quad \mbox{as}\,\,
\vert\tau\vert\rightarrow +\infty.
\end{equation}
\end{proposition}

\begin{proof}
We give a proof of the asymptotic formula  for
$\widetilde {\mathcal R}_{\tau,B} g.$ The proof for
$\mathcal R_{\tau,A}g$ is similar.
Let $\widetilde g(\zeta, \overline{\zeta})=ge^{-{\mathcal
B}}.$ Then
\begin{eqnarray*}
&&2e^{-{\mathcal B}}\mathcal R_{\tau,B}g
= - \frac{e^{\tau(\overline{\Phi} - \Phi)}}{\pi}\int_{\Omega}
\frac{\widetilde g(\zeta, \overline{\zeta})}
{\overline{\zeta} - \overline{z}}
e^{\tau(\Phi(\zeta)-\overline{\Phi(\zeta)})}d\xi_1d\xi_2\\
&&=  - \frac{e^{\tau(\overline{\Phi}- \Phi)}}{\pi}
\lim_{\delta\to +0} \int_{\Omega\setminus B(z,\delta)}
\frac{\widetilde g(\zeta, \overline{\zeta})}
{\overline{\zeta} - \overline{z}}
e^{\tau(\Phi(\zeta)-\overline{\Phi(\zeta)})}d\xi_1d\xi_2.
\end{eqnarray*}
Let $z = x_1 + ix_2$ and $x=(x_1,x_2)$ be not a critical point of the
function $\Phi$.  Then
\begin{eqnarray}\label{lastochka}
&&2e^{-{\mathcal B}}\mathcal R_{\tau,B}g
= - \frac{e^{\tau(\overline{\Phi}- \Phi)}}{\pi\tau}
\lim_{\delta\to +0} \int_{\Omega\setminus B(x,\delta)}
\frac{\widetilde g(\zeta, \overline{\zeta})}
{\overline{\zeta} - \overline{z}}
\frac{\partial_{\zeta}
e^{\tau(\Phi(\zeta)-\overline{\Phi(\zeta)})}}
{\partial_{\zeta}\Phi(\zeta)}d\xi_1d\xi_2\\
&&= \frac{e^{\tau(\overline{\Phi} - \Phi)}}{\pi\tau}
\lim_{\delta\to +0} \int_{\Omega\setminus B(x,\delta)}
\frac{1}{\overline{\zeta}-\overline{z}}
\frac{\partial}{\partial\zeta}
\left( \frac{\widetilde g(\zeta, \overline{\zeta})}
{\partial_{\zeta}\Phi(\zeta)}\right)
e^{\tau(\Phi(\zeta)-\overline{\Phi(\zeta)})}d\xi_1d\xi_2\nonumber\\
- &&\frac{e^{\tau(\overline{\Phi}-\Phi)}}{\pi\tau}
\lim_{\delta\to +0}\int_{S(x,\delta)}
\frac{\widetilde g(\zeta,\overline{\zeta})}{\overline{\zeta} - \overline{z}}
\frac{(\widetilde{\nu}_1 - i\widetilde{\nu}_2)}{2\partial_{\zeta}
\Phi(\zeta)}e^{\tau(\Phi(\zeta)-\overline{\Phi({\zeta})})}
d\xi_1d\xi_2.\nonumber
\end{eqnarray}
Since $\widetilde g\vert_{\mathcal H} = 0$, we have
\begin{equation}\label{zizo}
\left\vert \frac{\partial}{\partial \zeta}\left(
\frac{\widetilde g(\zeta,\overline{\zeta})}{\partial_{\zeta}\Phi(\zeta)}
\right) \right\vert
\le C\sum_{k=1}^{\ell} \frac{\Vert \widetilde g\Vert_{C^1(\overline{\Omega})}}
{\vert \zeta - \widetilde{x}_k\vert}\in L^p(\Omega)\quad
\forall p\in (1,2).
\end{equation}
Hence, passing to the limit in (\ref{lastochka})  we get
$$
2e^{-{\mathcal B}}\mathcal R_{\tau,B}g
= \frac{e^{\tau(\overline{\Phi}-\Phi)}}
{\pi\tau}\int_{\Omega} \frac{1}{\overline{\zeta}-\overline{z}}
\frac{\partial}{\partial \zeta}\left(
\frac{\widetilde g(\zeta,\overline{\zeta})}{\partial_{\zeta}\Phi(\zeta)}
\right)
e^{\tau(\Phi(\zeta) - \overline{\Phi(\zeta)})}d\xi_1 d\xi_2
- \frac{\widetilde g(z,\overline z)}{\tau\partial_{z}\Phi(z)}.
$$

Denote
$
G_{\tau}(x) =
\int_{\Omega}
\frac{1}{\overline{\zeta}-\overline{z}}
\frac{\partial}{\partial \zeta}\left(
\frac{\widetilde g(\zeta,\overline{\zeta})}{\partial_{\zeta}\Phi(\zeta)}
\right)e^{\tau(\Phi(\zeta)-\Phi(\overline{\zeta}))}d\xi_1d\xi_2.$
By Proposition \ref{opa}, we see that
\begin{equation}\label{luna1}
G_{\tau}(x)
\longrightarrow 0\,\,\mbox{
as}\,\, \vert\tau\vert \to +\infty\quad \forall x\in\overline\Omega.
\end{equation}

Denote
$$
T(\xi_1,\xi_2) = \left\vert\frac{\partial_{\overline{\zeta}}
\widetilde g(\zeta,\overline\zeta)}
{\partial_{\zeta}\Phi(\zeta)}\right\vert \chi_\Omega,
$$
where $\chi_\Omega$ is the characteristic function of $\Omega.$

Clearly
\begin{equation}\label{lunal}
\vert G_\tau(x)\vert\le \int_{\Omega} \frac{\vert T(\xi_1,\xi_2)\vert}
{\vert z-\zeta\vert}
d\xi_1d\xi_2 \quad a.e. \,\,\mbox{in}\,\, \Omega\quad \forall \tau.
\end{equation}
 By (\ref{zizo}) $ T$ belongs to $L^p(\Bbb R^2)$ for any $p\in (1,2).$
For any  $f \in L^p(\R^2)$, we set
$$
I_rf(z) = \int_{\R^2}\vert z-\zeta\vert^{-\frac{2}{r}}
f(\zeta,\overline \zeta)
d\xi_1d\xi_2.
$$
Then, by the Hardy-Littlewood-Sobolev inequality,
if $r > 1$ and $\frac{1}{r} = 1 - \left(\frac{1}{p}
- \frac{1}{q}\right)$ for $1 < p < q < \infty$, then
$$
\Vert I_rf\Vert_{L^q(\R^2)} \le C_{p,q}\Vert f\Vert_{L^p(\R^2)}.
$$
Set $r=2$.  Then we have to choose $\frac{1}{p} - \frac{1}{q}
= \frac{1}{2}$, that is, we can arbitrarily choose $p > 2$ close to $2$,
so that $q$ is arbitrarily large.  Hence $\int_{\Omega} \frac{T}
{\vert z-\zeta\vert} d\xi_1d\xi_2 $ belongs to $L^q(\Omega)$ with positive $q.$
By (\ref{luna1}), (\ref{lunal})  and the dominated convergence theorem
$$
G_\tau\rightarrow 0 \quad\mbox{in} \,\,L^q(\Omega)\,\,\quad
\forall q\in(1,\infty).
$$
The proof of the proposition is finished.
\end{proof}

We now consider the contribution from the critical points.

\begin{proposition}\label{Proposition 3.223}
Let $\Phi$ satisfy (\ref{zzz}) and (\ref{mika}).
Let $g\in C^{4+\alpha}(\overline\Omega)$ for some $\alpha>0,$
$g\vert_{\mathcal O_\epsilon} = 0$ and
$g\vert_{\mathcal H}= 0$.  Then there exist constants $p_k$ such that
\begin{equation}\label{Masa}
\int_\Omega ge^{\tau(\Phi(z)-\overline{\Phi(z)})}dx=\frac{1}{\tau^2}
\sum_{k=1}^\ell p_k e^{2\tau i\psi(\widetilde x_k)}+o(\frac {1}{\tau^2})
\quad \mbox{as}\,\,\vert \tau\vert\rightarrow +\infty.
\end{equation}
\end{proposition}
\begin{proof} Let $\delta>0$ be a sufficiently small number and $\widetilde e_k
\in C_0^\infty(B(\widetilde x_k,\delta)), \widetilde e_k\vert
_{B(\widetilde x_k,\delta/2)}\equiv 1.$ By the stationary
phase argument
\begin{eqnarray*}
I(\tau)=\int_\Omega ge^{\tau(\Phi-\overline\Phi)}dx=\sum_{k=1}^\ell
\int_{B(\widetilde x_k,\delta)}\widetilde e_k ge^{\tau(\Phi-\overline\Phi)}dx
+o(\frac {1}{\tau^2})= \nonumber\\\sum_{k=1}^\ell
e^{2i\tau\psi(\widetilde x_k)}
\int_{B(\widetilde x_k,\delta)} \widetilde e_k ge^{\tau(\Phi-\overline\Phi)
-2i\tau\psi(\widetilde x_k)}dx+o(\frac {1}{\tau^2})\quad \mbox{as}
\,\,\vert \tau\vert\rightarrow +\infty.
\end{eqnarray*}
Since all the critical points of $\Phi$ are nondegenerate,
in some neighborhood of $\widetilde x_k$ one can take local coordinates
such that $\Phi-\overline\Phi-2i\tau\psi(\widetilde x_k)
= z^2-\overline z^2$.  Therefore
$$
I(\tau)=\sum_{k=1}^\ell  e^{2i\tau\psi(\widetilde x_k)} \int_{B(0,\delta')}
q_ke^{\tau(z^2-\overline z^2)}dx+o(\frac {1}{\tau^2})\quad
\mbox{as}\,\,
\vert \tau\vert\rightarrow +\infty,
$$
where $q_k\in C_0^4(B(0,\delta'))$ and $q_k(0)=0.$ Hence there exist functions
$r_{1,k},r_{2,k}\in C^3_0(B(0,\delta'))$ such that
   $q_k=2 zr_{1,k}+2\overline zr_{2,k}. $ Integrating by parts,
one can decompose $I(\tau)$ as
\begin{eqnarray}
I(\tau)=-\frac 1\tau\sum_{k=1}^\ell  e^{2i\tau\psi(\widetilde x_k)}
\int_{B(0,\delta')} (\frac{\partial r_{1,k}}{\partial z}
-\frac{\partial r_{2,k}}{\partial \overline z})
e^{\tau(z^2-\overline z^2)}dx+o(\frac {1}{\tau^2})= \nonumber\\
-\frac 1\tau\sum_{k=1}^\ell  e^{2i\tau\psi(\widetilde x_k)}
\int_{B(0,\delta')} (\frac{\partial r_{1,k}}{\partial z}
-\frac{\partial r_{2,k}}{\partial \overline z})(0)\chi(x)
e^{\tau(z^2-\overline z^2)}dx\nonumber\\
-\frac 1\tau\sum_{k=1}^\ell  e^{2i\tau\psi(\widetilde x_k)}
\int_{B(0,\delta')} \widetilde q_ke^{\tau(z^2-\overline z^2)} dx
+o(\frac {1}{\tau^2})\quad \mbox{as}\,\,\vert \tau\vert\rightarrow
+\infty,\nonumber
\end{eqnarray}
where $\chi,\widetilde q_k\in C_0^2(B(0,\delta')),
\chi\vert_{B(0,\delta'/2)}\equiv 1$ and $\widetilde q_k(0)=0.$
Hence there exist functions
$\widetilde r_{1,k},\widetilde r_{2,k}\in C^1_0(B(0,\delta'))$
such that $\widetilde q_k=2 z\widetilde r_{1,k}+2\overline z
\widetilde r_{2,k}. $ Integrating by parts and applying Proposition
\ref{Proposition 3.22} we obtain
 $$
 \lim_{\vert\tau\vert\rightarrow+\infty}\tau\int
_{B(0,\delta')} \widetilde q_ke^{\tau(z^2-\overline z^2)} dx
=-\sum_{k=1}^\ell  e^{2i\tau\psi(\widetilde x_k)}
\lim_{\vert\tau\vert\rightarrow+\infty}\int_{B(0,\delta')}
(\frac{\partial \widetilde r_{1,k}}{\partial z}
-\frac{\partial \widetilde r_{2,k}}{\partial \overline z})
e^{\tau(z^2-\overline z^2)}dx=0.
 $$
Therefore (\ref{Masa}) follows from a standard application of
stationary phase. The proof of the proposition  is completed.
\end{proof}

\begin{proposition}\label{Proposition 3.2234}
Let $0<\epsilon'<\epsilon,$ a function $\Phi$ satisfy (\ref{zzz}),
(\ref{mika}) and $\overline {\mathcal O}_\epsilon \cap
(\mathcal H\setminus \Gamma_0)=\emptyset$.
Suppose that $g\in C^{\alpha}(\overline\Omega)\cap H^1(\Omega)$
for some $\alpha \in (0,1)$, $g\vert_{\mathcal O_\epsilon} = 0$ and
$g\vert_{\mathcal H}= 0$.
 Then
\begin{equation}\label{mis}
\vert\tau\vert\Vert \widetilde{\mathcal R}_{\tau,B} g\Vert
_{L^\infty (\mathcal O_{\epsilon'})}+\Vert \nabla\widetilde{\mathcal R}
_{\tau,B} g\Vert_{L^\infty (\mathcal O_{\epsilon'})}
\le C_1(\epsilon',\alpha)\Vert g\Vert
_{C^\alpha(\overline\Omega)\cap H^1(\Omega)}.
\end{equation}
Moreover
\begin{equation}\label{miss1}
\Vert \nabla\widetilde{\mathcal R}_{\tau,B} g\Vert_{L^2 (\Omega)}
+\vert\tau\vert^\frac 12\Vert \widetilde{\mathcal R}_{\tau,B}
g\Vert_{L^2 (\Omega)}+\vert\tau\vert\Vert \frac{\partial \Phi}
{\partial z}\widetilde{\mathcal R}_{\tau,B} g\Vert
_{L^2 (\Omega)}
\le C_2(\epsilon',\alpha)\Vert g\Vert
_{C^\alpha(\overline\Omega)\cap H^1(\Omega)}.
\end{equation}
\end{proposition}

\begin{proof} Denote $\widetilde g=g e^{-\mathcal B}.$ Let $x=(x_1,x_2)$ be an
arbitrary point from $\mathcal O_{\epsilon'}$ and $z=x_1+ix_2.$  Then
\begin{eqnarray}
-\pi \partial_z^{-1} (e^{\tau(\Phi-\overline \Phi)}\widetilde g)
=\int_\Omega\frac{ \widetilde g  e^{\tau(\Phi-\overline \Phi)}}
{\overline \zeta-\overline z} d\xi_1d\xi_2=\lim_{\delta\rightarrow +0}
\sum_{k=1}^\ell\int_{\Omega\setminus B(\widetilde x_k,\delta)}
\frac{\widetilde g e^{\tau(\Phi-\overline \Phi)}}
{\overline \zeta-\overline z} d\xi_1d\xi_2.\nonumber
\end{eqnarray}

Integrating by parts and taking $\delta$ sufficiently small we have
\begin{eqnarray}\label{omon}
-\pi \partial_z^{-1} (e^{\tau(\Phi-\overline \Phi)}\widetilde g)
=-\frac{1}{\tau}\lim_{\delta\rightarrow + 0}\int_{\Omega\setminus
\cup_{k=1}^\ell B(\widetilde x_k,\delta)}
\frac{ \frac{\partial \widetilde g}{\partial\zeta}}
{(\overline \zeta-\overline z)\frac{\partial\Phi}{\partial\zeta}}
e^{\tau(\Phi-\overline \Phi)} d\xi_1d\xi_2\nonumber\\+\frac{1}{\tau}
\lim_{\delta\rightarrow +0}
\int_{\Omega\setminus \cup_{k=1}^\ell B(\widetilde x_k,\delta)}
\frac{ \widetilde g\frac{\partial^2\Phi}{\partial\zeta^2}}
{(\overline \zeta-\overline z)(\frac{\partial\Phi}{\partial\zeta})^2}
e^{\tau(\Phi-\overline \Phi)} d\xi_1d\xi_2\nonumber\\
+\frac{1}{2\tau}\lim_{\delta\rightarrow +0}
\int_{ \cup_{k=1}^\ell S(\widetilde x_k,\delta)}
(\widetilde\nu_1-i\widetilde\nu_2)\frac{ \widetilde g}
{(\overline \zeta-\overline z)\frac{\partial\Phi}{\partial\zeta}}
e^{\tau(\Phi-\overline \Phi)}d\sigma.
\end{eqnarray}
Since $g\vert_{\mathcal H}=0$  we have that $\Vert g\Vert
_{C^0( S(\widetilde x_k,\delta))}
\le \delta^\alpha\Vert g\Vert_{C^\alpha(\overline \Omega)}.$
Using this inequality and the fact that  all the critical points of
$\Phi$ are nondegenerate we obtain
$$
\frac{1}{2\tau}\lim_{\delta\rightarrow +0}
\int_{ \cup_{k=1}^\ell S(\widetilde x_k,\delta)}(\widetilde\nu_1
-i\widetilde\nu_2)\frac{ \widetilde g}{(\overline \zeta-\overline z)
\frac{\partial\Phi}{\partial\zeta}} e^{\tau(\Phi-\overline \Phi)}d\sigma=0.
$$
Since $\vert\frac{ \widetilde g\frac{\partial^2\Phi}{\partial\zeta^2}}
{(\frac{\partial\Phi}{\partial\zeta})^2}(\zeta,\overline\zeta) \vert
\le C_3\Vert\widetilde  g\Vert_{C^\alpha(\overline\Omega)}
\sum_{k=1}^\ell\frac{1}{\vert \xi-\widetilde x_k\vert^{2-\alpha}}
$ we see that $\frac{ \widetilde g\frac{\partial^2\Phi}{\partial\zeta^2}}
{(\frac{\partial\Phi}{\partial\zeta})^2} (\zeta,\overline \zeta)\in L^1(\Omega)$ and
\begin{eqnarray}\label{omon1}
-\pi \partial_z^{-1} (e^{\tau(\Phi-\overline \Phi)}\widetilde g)
=-\frac{1}{\tau}\int_{\Omega}\frac{\frac{ \partial \widetilde g}
{\partial \zeta}}{(\overline \zeta-\overline z)\frac{\partial\Phi}
{\partial\zeta}}e^{\tau(\Phi-\overline \Phi)} d\xi_1d\xi_2\nonumber\\
+\frac{1}{\tau}
\int_{\Omega}\frac{ \widetilde g\frac{\partial^2\Phi}{\partial\zeta^2}}
{(\overline \zeta-\overline z)(\frac{\partial\Phi}{\partial\zeta})^2}
e^{\tau(\Phi-\overline \Phi)} d\xi_1d\xi_2.
\end{eqnarray}
From this equality, Proposition \ref{opa} and definition (\ref{anna}) of the
operator $\widetilde{\mathcal R}_{\tau,B}$, the estimate (\ref{mis}) follows
immediately.
To prove (\ref{miss1}) we observe
$$
\frac{\partial \widetilde{\mathcal R}_{\tau,B}g}{\partial\overline z}
= \frac{\partial \mathcal B}{\partial \overline z}\widetilde{\mathcal R}
_{\tau,B}g+\widetilde{\mathcal R}_{\tau,B}\{\frac{\partial g}
{\partial \overline z}-\frac{\partial\mathcal B}{\partial\overline z}g\}
+\frac{\tau}{2\pi}e^{\tau(\overline\Phi-\Phi)+\mathcal B}
\int_\Omega\frac{\frac{\partial \overline \Phi(\zeta)}
{\partial \overline \zeta}
-\frac{\partial \overline \Phi(z)}{\partial\overline z}}
{\overline \zeta-\overline z}\widetilde ge^{\tau(\Phi-\overline\Phi)}d\xi_1d\xi_2.
$$
By Proposition \ref{Proposition 3.0}
$$
\Vert \frac{\partial\mathcal  B}{\partial \overline z}
\widetilde{\mathcal R}_{\tau,B}g+\widetilde{\mathcal R}_{\tau,B}
\{\frac{\partial g}{\partial \overline z}
- \frac{\partial\mathcal B}{\partial\overline z}g\}
\Vert_{L^2(\Omega)}\le C_4\Vert g\Vert_{H^1(\Omega)}.
$$

Using arguments similar to (\ref{omon}), (\ref{omon1}) we obtain
$$
\Vert \frac{\tau}{2\pi}\int_\Omega\frac{
\frac{\partial \overline \Phi(\zeta)}
{\partial \overline \zeta}-\frac{\partial \overline \Phi(z)}
{\partial\overline z}}{\overline \zeta-\overline z}
\widetilde ge^{\tau(\Phi-\overline\Phi)}d\xi_1d\xi_2\Vert_{L^2(\Omega)}
\le C_5
\Vert g\Vert_{C^\alpha(\overline\Omega)\cap H^1(\Omega)}.
$$
Hence
$$
\Vert\frac{\partial \widetilde{\mathcal R}_{\tau,B}g}
{\partial\overline z}\Vert_{L^2(\Omega)}
\le C_6\Vert g\Vert_{C^\alpha(\overline\Omega)\cap H^1(\Omega)}.
$$
Combining this estimate with (\ref{mis}) we conclude
$$
 \Vert\nabla \widetilde{\mathcal R}_{\tau,B}g\Vert_{L^2(\Omega)}
\le C_7\Vert g\Vert_{C^\alpha(\overline\Omega)\cap H^1(\Omega)}.
$$
Using this estimate and equation (\ref{(3.2)}) we have
$$
\vert\tau\vert \Vert\frac{\partial\Phi}{\partial z}
\widetilde{\mathcal R}
_{\tau,B}g\Vert_{L^2(\Omega)}\le C_8\Vert g\Vert
_{C^\alpha(\overline\Omega)\cap H^1(\Omega)},
$$
finishing the proof of the proposition.
\end{proof}

Let $e_1,e_2 \in C^\infty(\overline \Omega)$ be functions such that
\begin{equation}\label{TT}
e_1+e_2=1\quad \mbox{in}\,\,\Omega,
\end{equation}
$e_2$ vanishes in some
neighborhood of $\mathcal H\setminus\Gamma_0$ and $e_1$ vanishes in a
neighborhood of
$\partial\Omega.$
\begin{proposition}\label{Proposition 00}
Let for some $\alpha\in (0,1)$ $A,B\in C^{5+\alpha}(\overline\Omega)$,
and the functions $\mathcal A,\mathcal B\in C^{6+\alpha}(\overline\Omega)$
satisfy (\ref{(2.15)}). Let $e_1,e_2$ be defined as in
(\ref{TT}).  Let $g \in L^p(\Omega)$ for some $p>2$,
$\mbox{supp}\, g\subset\subset \mbox{supp}\, e_1$ and $dist (\Gamma_0, supp\,g)>0.$
We define $u$ by
$$
u=\widetilde {\mathcal R}_{\tau, B}(e_1(P_{A} g
- \widetilde M e^{\mathcal A}))
+ \frac{e_2(P_{A} g-\widetilde M e^{\mathcal A})}{2\tau\partial_z\Phi},
$$
where $\widetilde{M} = \widetilde M(z)$ is a polynomial such that
$\frac{\partial^k}{\partial z^k}(P_{A} g-\widetilde M
e^{\mathcal A})\vert_{\mathcal H}=0$ for any  $k$ from $\{0,\dots, 6\}.$
Then we have
\begin{equation}\label{oP}
\mathcal P(x,D)(ue^{\tau\Phi})\triangleq (2\frac{\partial}
{\partial \overline z}+A)(2\frac{\partial}{\partial  z}+B)
(ue^{\tau\Phi})=g e^{\tau \Phi}
+ \frac{e^{\tau\varphi}}{\vert\tau\vert} h_\tau
\quad \mbox{as}\,\,\vert\tau\vert\rightarrow +\infty,
\end{equation} where
$$
\Vert h_\tau\Vert_{L^\infty(\Omega)}\le C_9(p)
\Vert g\Vert_{L^p(\Omega)}
$$
and for some sufficiently small positive $\epsilon'$ we have:
\begin{equation}\label{PP10}
\frac{1}{\vert\tau\vert^\frac 12}\Vert \nabla u\Vert_{L^2(\Omega)}
+\vert\tau\vert^\frac 12\Vert u\Vert_{L^2(\Omega)}+\Vert u\Vert
_{H^{1,\tau}(\mathcal O_{\epsilon '})}\le C_{10} \Vert g\Vert
_{L^p(\Omega)}.
\end{equation}
\end{proposition}
\begin{proof}
By Proposition \ref{Proposition 3.0} $ P_{A} g$ belongs to
$W^1_p(\Omega).$ Since $p>2,$ by the Sobolev embedding theorem
there exists $\alpha>0$ such that $ P_{A} g \in C^\alpha(\overline
\Omega).$  By properties of elliptic operators and the fact
that $\mbox{supp} \, e_2\cap \mbox{supp}\, g=\{\emptyset\}$
we have that $ P_{A} g\in C^5(supp\,e_2).$
The estimate (\ref{PP10}) follows from Proposition
\ref{Proposition 3.2234}. Short calculations give
\begin{equation}\label{mina}
\mathcal P(x,D)(ue^{\tau\Phi})=g e^{\tau \Phi}+\frac{e^{\tau\Phi}}
{\tau}\mathcal P(x,D)\left(\frac{e_2(P_{A} g
-\widetilde Me^{\mathcal A})}{2\partial_z\Phi}\right).
\end{equation}
This formula implies (\ref{oP}) with $h_\tau=e^{i\tau\psi}
\mathcal P(x,D)\left(\frac{e_2(P_{A} g-\widetilde
Me^{\mathcal A})}{2\partial_z\Phi}\right)/\mbox{sign}\, \tau.$
\end{proof}

The following proposition will play a critically important role in the
construction of the complex geometric optic solutions.
\begin{proposition}\label{Proposition 0}
Let $f\in L^p(\Omega)$ for some $p>2$, $dist (\overline\Gamma_0,supp\,f)>0$, $q\in H^\frac 12(\Gamma_0),$
$\epsilon'$ be a small positive number such that $\mathcal
\overline{O_{\epsilon'}}\cap(\mathcal H\setminus\Gamma_0)=\emptyset.$
Then there exists $\tau_0$ such that  for all $\vert\tau\vert>\tau_0$
there exists
a solution to the boundary value problem
\begin{equation}\label{mimino}
L(x,D)w=fe^{\tau\Phi}\quad \mbox{in}\,\,\Omega, \quad w\vert_{\Gamma_0}
=qe^{\tau\varphi}/\tau
\end{equation}
such that
$$
\root\of{\vert\tau\vert}
\Vert w e^{-\tau\varphi}\Vert_{L^2(\Omega)}
+ \frac{1}{\root\of{\vert\tau\vert} }
\Vert (\nabla w) e^{-\tau\varphi}\Vert_{L^2(\Omega)}+\Vert we^{-\tau\varphi}
\Vert_{H^{1,\tau}(\mathcal O_{\epsilon'})}
\le C_{11}(\Vert f\Vert_{L^p(\Omega)}
+ \Vert q \Vert_{H^\frac 12(\Gamma_0)}).
  $$
\end{proposition}
\begin{proof} Let $\chi\in C_0^\infty(\Omega)$ be equal to one
in some neighborhood of the set $\mathcal H\setminus\Gamma_0.$
By Proposition \ref{Proposition 2.3}
there exists a solution to the problem (\ref{mimino})
with
inhomogeneous term $(1-\chi)f$ and boundary data $q/\tau$ such that
\begin{equation}\label{norma}
\Vert  w_1 e^{-\tau\varphi}\Vert_{H^{1,\tau}(\Omega)}
\le C_{12}(\Vert f \Vert_{L^2(\Omega)}
+ \Vert q \Vert_{H^\frac 12(\Gamma_0)}).
\end{equation}
Denote $w_2=\widetilde {\mathcal R}_{-\tau, B}(e_1(P_{A} (\chi f)
- \widetilde M e^{\mathcal A}))
+ \frac{e_2(P_{A} (\chi f)-\widetilde M e^{\mathcal A})}
{2\tau\partial_z\Phi}$ where $\widetilde{M} = \widetilde M(z)$ is a
polynomial such that
$\frac{\partial^k}{\partial z^k}(P_{A} (\chi f)-\widetilde M e^{\mathcal A})
\vert_{\mathcal H}=0$ for any  $k$ from $\{0,\dots, 6\}.
$ Let  $q_\tau $ be the restriction of $w_2$ to
$\Gamma_0.$
By (\ref{mis}) there exists a constant $C_{13}$ independent of
$\tau$ such that
\begin{equation}\label{zorka}
\vert \tau\vert^\frac 12\Vert q_\tau\Vert_{ H^\frac 12(\Gamma_0)}
\le C_{13}\Vert f\Vert_{L^p(\Omega)}.
\end{equation}
By Proposition \ref{Proposition 00} there exists a constant
$C_{14}$ independent of $\tau$ such that
\begin{equation}\label{norma1}
\root\of{\vert\tau\vert} \Vert w_2 e^{-\tau\varphi}\Vert_{L^2(\Omega)}
+ \frac{1}{\root\of{\vert\tau\vert} }\Vert \nabla w_2 e^{-\tau\varphi}\Vert
_{L^2(\Omega)}
+ \Vert w_2e^{-\tau\varphi}\Vert_{H^{1,\tau}(\mathcal O_{\epsilon'})}
\le C_{14}\Vert f\Vert_{L^p(\Omega)}.
\end{equation}
Let $\widetilde a_\tau,\widetilde b_\tau\in H^1(\Omega)$ be holomorphic and
 antiholomorphic functions, respectively, such that $(\widetilde a_\tau
e^{\mathcal A}+\widetilde b_\tau e^{\mathcal B})\vert_{\Gamma_0}
=-q_\tau$. By (\ref{zorka}) and Proposition \ref{zika}
there exist constants
$C_{15},C_{16}$ independent of $\tau$ such that
\begin{equation}\label{norma2}
\Vert \widetilde a_\tau\Vert_{H^1(\Omega)}
+ \Vert \widetilde b_\tau\Vert_{H^1(\Omega)}
\le C_{15} \Vert q_\tau\Vert_{ H^\frac 12(\Gamma_0)}
\le C_{16}\frac{\Vert f\Vert_{L^p(\Omega)}}{\root\of{\vert\tau\vert}}.
\end{equation}
The function $W=(w_2+\widetilde a_\tau e^{\mathcal A}) e^{\tau\Phi}
+\widetilde b_\tau e^{\mathcal B+\tau\overline\Phi}$ satisfies
$$
L(x,D) W=\chi f e^{\tau\Phi}  +e^{\tau\varphi}\frac{\widetilde h_\tau}
{\root\of{\vert \tau\vert}}\quad\mbox{in}\,\,\Omega,\quad W\vert
_{\Gamma_0}=0,
$$
where
\begin{equation}\label{norma4}
\Vert \widetilde h_\tau\Vert_{L^2(\Omega)}\le C_{17}
\Vert f \Vert_{L^2(\Omega)}
\end{equation}
with some constant $C_{17}$ independent of $\tau.$
 By (\ref{norma1}), (\ref{norma2})
\begin{equation}\label{norma3}
\root\of{\vert\tau\vert} \Vert W e^{-\tau\varphi}\Vert
_{L^2(\Omega)}+\frac{1}{\root\of{\vert\tau\vert} }\Vert
\nabla W e^{-\tau\varphi}\Vert_{L^2(\Omega)}
+\Vert We^{-\tau\varphi}\Vert_{H^{1,\tau}(\mathcal O_{\epsilon'})}
\le C_{18}\Vert f\Vert_{L^p(\Omega)}.
\end{equation}
Let $\widetilde W$ be a solution to problem (\ref{(2.26)}) with
inhomogeneous term and boundary data $f=-\frac{\widetilde h_\tau}
{\root\of{\vert \tau\vert}}, g\equiv0$ respectively given by
Proposition \ref{Proposition 2.3}.  The estimate (\ref{(2.27)}) has
the form
\begin{equation}\label{norma4}
\Vert \widetilde W e^{-\tau\varphi}\Vert_{H^{1,\tau}(\Omega)}
\le C_{19}\Vert \widetilde h_\tau\Vert_{L^2(\Omega)}
\le C_{20}\Vert f\Vert_{L^2(\Omega)}.
\end{equation}
Then the function $w_1+W+\widetilde W$ solves (\ref{mimino}).
The estimate (\ref{mina}) follows form (\ref{norma}), (\ref{norma3})
and (\ref{norma4}).
The proof of the proposition is completed.
\end{proof}
\section{Complex Geometrical Optics Solutions}

For a complex-valued vector field $(A_1, B_1)$ and complex-valued
potential $q_1$ we will construct
solutions to the boundary value problem
\begin{equation}\label{(2.1III)}
{ L}_{1}(x,D)u_1=0\quad \mbox{in}\,\,\Omega,\quad
u_1\vert_{\Gamma_0}=0
\end{equation}
of the form
\begin{equation}\label{mozilaa}
u_1(x)=a_\tau(z)
e^{\mathcal A_1+\tau\Phi}+d_\tau(\overline z)e^{\mathcal B_1
+\tau\overline{\Phi}} +
u_{11}e^{\tau\varphi}+u_{12}e^{\tau\varphi}.
\end{equation}

Here $\mathcal A_1$ and $\mathcal B_1$ are defined by (\ref{001q})
respectively for $A_1$ and $B_1$,
$a_\tau(z)=a(z)+\frac{a_1(z)}{\tau}+\frac{a_{2,\tau}(z)}{\tau^2},$
$d_\tau(\overline z)={d(\overline z)}+\frac{d_1(\overline z)}{\tau}
+\frac{d_{2,\tau}(\overline z)}{\tau^2},$
\begin{equation}\label{iopa}
a,d\in C^{5+\alpha}(\overline\Omega), \,\,\,\, \frac{\partial a}
{\partial \overline z}=0\,\,\mbox{in}\,\Omega,\,\,
\frac{\partial d}{\partial z}=0\,\,\mbox{in}\,\Omega,
\end{equation}
\begin{equation}\label{ikaa}
(ae^{\mathcal A_1}+de^{\mathcal B_1})\vert_{\Gamma_0}=0.
\end{equation}

Let
$\widetilde x$ be some fixed point from $\mathcal H\setminus\partial\Omega.$
Suppose in addition that
\begin{equation}\label{nip}
\frac{\partial^k a}{\partial z^k} \vert_{\mathcal H\cap \partial\Omega}
=\frac{\partial^k d}{\partial \overline z^k}
\vert_{\mathcal H\cap \partial\Omega}=0 \quad \forall
k\in \{0,\dots, 5\}, \quad a\vert_{\mathcal H\setminus \{\widetilde x\}}
=d\vert_{\mathcal H\setminus \{\widetilde x\}}=0,\,\, a(\widetilde x)\ne 0,
d(\widetilde x)\ne 0.
\end{equation}
Such functions exists by Proposition \ref{balalaika1}.

Denote
$$
g_1=T_{{B}_1}
((q_1-2 \frac{\partial {B}_1}{\partial \overline z} - {A}_1 {B}_1)
d e^{\mathcal{B}_1})
-{M}_2(\overline{z})e^{\mathcal{B}_1},\quad g_2=P_{{A}_1}
((q_1-2 \frac{\partial {A}_1}{\partial {z}}
-{A}_1 {B}_1) ae^{\mathcal{A}_1})
-{M}_1(z)e^{{\mathcal A}_1},\quad
$$
where $M_1(z)$ and $M_2(\overline z)$ are polynomials such that
\begin{equation}\label{TTT}
\frac{\partial^k g_1}{\partial \overline z^k}\vert_{\mathcal H}
=\frac{\partial^k g_2}{\partial z^k}\vert_{\mathcal H}=0 \quad \forall
k\in \{0,\dots,6\},\quad\frac{\partial g_1}{\partial z}
=\frac{\partial g_2}{\partial \overline z}=0\,\,\mbox{on}\,\,
\mathcal H\setminus\{\widetilde x\}.
\end{equation}
Thanks to our assumptions on the regularity of $A_1,B_1$ and
$q$, $g_1,g_2$ belong to $C^{6+\alpha}(\overline\Omega).$

Note that by (\ref{TTT}), (\ref{nip})
\begin{equation}\label{TTT0}\frac{\partial^{k+j} g_1}
{\partial z^k\partial {\overline z}^j} \vert_{\mathcal H\cap\partial\Omega}
=\frac{\partial^{k+j} g_2}{\partial z^k\partial {\overline z}^j}
\vert_{\mathcal H\cap\partial \Omega}=0 \quad \mbox{if $k+j\le 6$}.
\end{equation}
The function $a_1(z)$ is  holomorphic in $\Omega$ and
$d_1(\overline z)$ is antiholomorphic in $\Omega$ and
\begin{equation}\label{AK0}
a_1e^{{\mathcal A_1}}+d_1e^{{\mathcal B_1}}=\frac{ g_1}
{2\overline{\partial_z
\Phi}}\nonumber+\frac{ g_2}{2\partial_z \Phi}\quad \mbox{on}\,
\Gamma_0.
\end{equation}
The existence of such functions is given again by Proposition
\ref{zika}.
Observe that by (\ref{TTT0}) the functions $\frac{e_2 g_1}
{\overline{\partial_z \Phi}}, \frac{e_2g_2}{{\partial_z \Phi}} \in
C^4(\overline\Omega).$
Let
$$
\widehat g_1=T_{{B}_1}
((q_1-2 \frac{\partial {B}_1}{\partial \overline z} - {A}_1 {B}_1)
d_1 e^{\mathcal{B}_1})
-\widehat {M}_2(\overline{z})e^{\mathcal{B}_1},\quad\widehat g_2=P_{{A}_1}
((q_1-2 \frac{\partial {A}_1}{\partial {z}}
-{A}_1 {B}_1) a_1e^{\mathcal{A}_1})
-\widehat {M}_1(z)e^{{\mathcal A}_1},\quad
$$
where $\widehat M_1(z)$ and $\widehat M_2(\overline z)$ are polynomials
such that
\begin{equation}
\frac{\partial^k \widehat g_1}{\partial \overline z^k}\vert_{\mathcal H}
=\frac{\partial^k \widehat g_2}{\partial z^k}\vert_{\mathcal H}
=0 \quad \forall k\in \{0,\dots,3\}.
\end{equation}

Henceforth we recall (\ref{anna}).
The function $u_{11}$ is given by
\begin{eqnarray}\label{wolf}
u_{11}=-e^{-i\tau\psi}\mathcal R_{-\tau, A_1}\left\{e_1(g_1+\widehat g_1/\tau)
\right\}-e^{-i\tau\psi}\frac{e_2 (g_1+\frac{\widehat g_1}{\tau})}
{2\tau\overline{\partial_z
\Phi}}+\frac{e^{-i\tau\psi}}{4\tau^2\overline{\partial_z\Phi}}
L_1(x,D)\left(\frac{e_2g_1}{\overline{\partial_z \Phi}}\right)\nonumber\\
-e^{i\tau\psi}\widetilde{\mathcal R}_{\tau,B_1}\left\{e_1(g_2+\widehat g_2/\tau)
\right\}-e^{i\tau\psi}\frac{e_2(g_2+\frac{\widehat g_2}{\tau})}
{2\tau\partial_z \Phi}+\frac{e^{i\tau\psi}}{4\tau^2{{\partial_z\Phi}}}
L_1(x,D)\left(\frac{e_2g_2}{{\partial_z \Phi}}\right).
\end{eqnarray}

Now let us determine the functions $u_{12}$, $a_2(z)$ and
$d_2(\overline z).$

First we can obtain the following asymptotic formulae for any point on
the boundary of $\Omega$:
\begin{equation}\label{AK1}
\mathcal R_{-\tau, A_1}\left\{e_1g_1\right\}
\vert_{\partial\Omega}
= \frac{e^{\mathcal A_1+2i\tau\psi}}
{2\tau^2\vert \mbox{det}\, \psi''({\widetilde x})\vert^\frac 12}
\left (\frac{e^{-2i\tau\psi({\widetilde x})} \sigma_{1}({\widetilde x})}
{(z-\widetilde z)^2}+\frac{e^{-2i\tau\psi({\widetilde x})} m_{1}({\widetilde x})}
{(\widetilde z-z)}\right )
+\mathcal{W}_{\tau,1},
\end{equation}
\begin{equation}\label{AKK1}
\widetilde{\mathcal R}_{\tau, B_1}\left\{e_1g_2\right\}
\vert_{\partial\Omega}
= \frac{e^{\mathcal B_1-2i\tau\psi}}
{2\tau^2\vert \mbox{det}\, \psi''({\widetilde x})\vert^\frac 12}
\left (\frac{e^{2i\tau\psi({\widetilde x})} \widetilde \sigma_{1}
({\widetilde x})}{(\overline z-\overline{\widetilde z})^2}
+\frac{e^{2i\tau\psi({\widetilde x})} \widetilde m_{1}({\widetilde x})}
{(\overline{\widetilde z}-\overline z)}\right )
+\mathcal{W}_{\tau,2},
\end{equation}
where
\begin{equation}\label{v1}
\sigma_{1}(\widetilde x)=\frac{\partial_{ z}\widetilde g_1 ({\widetilde x})}
{\partial^2_z\Phi({\widetilde x})}, \quad m_{1}(\widetilde x)
=\frac {1}{2}\left(\frac{\partial_{z}\widetilde g_1 ({\widetilde x})}
{{\partial^2_z\Phi({\widetilde x})}}\frac{\partial_z^3\Phi({\widetilde x})}
{\partial^2_z\Phi({\widetilde x})}
+
 \frac{\partial^2_{\overline z}\widetilde g_1 ({\widetilde x})}
{\overline{\partial^2_z\Phi({\widetilde x})}}-\frac{\partial^2_{ z}
\widetilde g_1 ({\widetilde x})}{{\partial^2_z\Phi({\widetilde x})}}\right ),
\end{equation}
\begin{equation}\label{v2}
\widetilde \sigma_{1}(\widetilde x)=\frac{\partial_{\overline  z}\widetilde g_2
({\widetilde x})}{\overline{\partial^2_z\Phi({\widetilde x})}},\quad \widetilde
m_{1}(\widetilde x)=\frac {1}{2}\left(\frac{\partial_{\overline z}\widetilde
g_2 ({\widetilde x})}{\overline{\partial^2_z\Phi({\widetilde x})}}\frac{\partial^3
_{\overline z}\overline\Phi({\widetilde x})}{\partial^2_{\overline z}
\overline\Phi({\widetilde x})} -
\frac{\partial^2_{\overline z}\widetilde g_2 ({\widetilde x})}
{\overline{\partial^2_z\Phi({\widetilde x})}}
+ \frac{\partial^2_{ z}\widetilde g_2 ({\widetilde x})}
{\partial^2_z\Phi({\widetilde x})}\right ),
\end{equation}
$\widetilde g_1=e^{-\mathcal A_1}g_1$, $\widetilde g_2=e^{-\mathcal B_1}g_2$
and $\mathcal W_{\tau,1},\mathcal W_{\tau,2}\in H^\frac 12(\Gamma_0)$
satisfy
\begin{equation}\label{AK2}
\Vert\mathcal{W}_{\tau,1}\Vert_{H^\frac 12(\Gamma_0)}+\Vert\mathcal{W}
_{\tau,2}\Vert_{H^\frac 12(\Gamma_0)}=o(\frac{1}{\tau^2})
\quad\mbox{as}\,\vert\tau\vert\rightarrow +\infty.
\end{equation}
The proof of (\ref{AK1}) and (\ref{AKK1}) is given in Section 8.

Denote
$$
p_+(x)=e^{\mathcal A_1(x)}\left (\frac{\sigma_{1}(\widetilde x)}{(z-\widetilde z)^2}
+\frac{m_{1}(\widetilde x)}{(\widetilde z-z)}\right ),\quad
p_-(x)=e^{\mathcal B_1(x)}\left(\frac{\widetilde \sigma_{1}({\widetilde x})}
{(\overline z-\overline{\widetilde z})^2}+ \frac{\widetilde m_{1}({\widetilde x})}
{(\overline{\widetilde z}-\overline z)}\right ).
$$
Thanks to Proposition \ref{zika} we can define functions $a_{2,\pm}(z)
\in C^2(\overline\Omega)$ and $d_{2,\pm}(\overline z)\in
C^2(\overline\Omega)$ satisfying
\begin{equation}\label{AK3}
a_{2,\pm}e^{{\mathcal A_1}}+d_{2,\pm}e^{{\mathcal B_1}}=p_\pm\quad
\mbox{on}\,\Gamma_0.
\end{equation}
Straightforward computations give
\begin{eqnarray}\label{mumy1}
&{L}_1(x,{D})((a+ \frac{a_1}{\tau})
e^{\mathcal{A}_1+\tau \Phi}
+({d}+\frac{d_1}{\tau})
e^{\mathcal{B}_1 +\tau \overline{ \Phi}}
+e^{\tau \varphi}u_{11})\nonumber \\
&=(q_1 -2 \frac{\partial {A}_1}{\partial z}
-{A}_1{B}_1)e^{\tau \Phi}
\biggl(-\widetilde{\mathcal{R}}_{\tau,{B}_1}\{e_1(g_2+\widehat g_2/\tau)\}
-\frac{e_2(g_2+\widehat g_2/\tau)}{2 \tau \partial_z \Phi}\biggl)
                                                     \nonumber\\
&+(q_1 -2 \frac{\partial {B}_1}{\partial \overline{z}}
-{A}_1{B}_1) e^{\tau \overline{\Phi}}
\biggl(-\mathcal{R}_{-\tau,{A}_1}\{e_1(g_1+\widehat g_1/\tau)\}
-\frac{e_2(g_1+\widehat g_1/\tau)}{2 \tau
\overline{\partial_z \Phi}}\biggl) \nonumber\\
&+\frac{e^{\tau {\Phi}}}{\tau^2} {L}_1(x,{D})\left(
\frac{1}{4{\partial_z\Phi}}L_1(x,D)\left(
\frac{e_2g_2}{\partial_z \Phi}\right)\right) + \frac{e^{\tau
\overline\Phi}}{\tau^2}{L}_1(x,{D})
\biggl(\frac{1}{4\overline{\partial_z\Phi}}L_1(x,D)\left(\frac{e_2g_1}
{\overline{\partial_z \Phi}}\right)\biggl).
\end{eqnarray}

Using Proposition \ref{Proposition 3.22} we transform the
right-hand
side of (\ref{mumy1}) as follows.
\begin{eqnarray}\label{ROMA}
&{L}_1(x,{D})((a+ \frac{a_1}{\tau})
e^{\mathcal{A}_1+\tau \Phi}
+({d}+\frac{d_1}{\tau})
e^{\mathcal{B}_1 +\tau \overline{ \Phi}}
+u_{11}e^{\tau \varphi})\\
&=-(q_1 -2 \frac{\partial {A}_1}{\partial z}
-{A}_1{B}_1)e^{\tau \Phi}
\frac{g_1}{2 \tau \partial_z \Phi}\nonumber \\
&-(q_1 -2 \frac{\partial {B}_1}{\partial \overline{z}}
-{A}_1{B}_1) e^{\tau \overline{\Phi}}
\frac{g_2}{2 \tau \overline{\partial_z \Phi}}
+\chi_{\mathcal O_\epsilon} O_{L^4(\Omega)}\biggl(\frac{1}{\tau^2}\biggl)++\chi_{\Omega\setminus\mathcal O_{\epsilon'}} o_{L^4(\Omega)}\biggl(\frac{1}{\tau^2}\biggl)\quad\mbox{as}\,
\vert\tau\vert\rightarrow +\infty.\nonumber
\end{eqnarray}

We are looking for $u_{12}$ in the form $u_{12}=u_0+u_{-1}.$
The function $u_{-1}$ is given by
\begin{eqnarray}\label{NNN}
u_{-1}&=\frac{e^{i \tau \psi}}{\tau}
\widetilde{\mathcal{R}}_{\tau,{B}_1}\{e_1g_5\}
+\frac{e^{-i \tau \psi}}{\tau}
\mathcal{R}_{-\tau,{A}_1}\{e_1g_6\}
+\frac{e_2g_5e^{i \tau \psi}}{2 \tau^2 \partial_z \Phi}
+\frac{e_2g_6e^{-i \tau \psi}}{2 \tau^2 \overline{\partial_z \Phi}},
\end{eqnarray}
where
\begin{equation}\label{NNN1}
g_5=\frac{P_{{A_1}}
((q_1 -2 \frac{\partial {A}_1}{\partial z}
-{A}_1{B}_1)g_1)-{M}_5(z)e^{\mathcal A_1}}
{2 \partial _z \Phi},
g_6=\frac{T_{{B}_1}
((q_1 -2 \frac{\partial {B}_1}{\partial \overline{z}}
-{A}_1{B}_1)g_2)-{M}_6(\overline{z})e^{\mathcal B_1}}
{2\overline{\partial _z \Phi}}.
\end{equation}
Here $M_5(z),M_6(\overline z)$ are polynomials such that
$$
g_5 \vert_{\mathcal{H}}=g_6 \vert_{\mathcal{H}}=\nabla g_5\vert
_{\mathcal{H}}=\nabla g_6\vert_{\mathcal{H}}=0.
$$
Using Proposition \ref{zika} we introduce functions $a_{2,0}, d_{2,0}
\in C^2(\overline \Omega)$ (holomorphic and antiholomorphic respectively)
such that
\begin{eqnarray}\label{IPA}
a_{2,0}e^{{\mathcal A_1}}+d_{2,0}e^{{\mathcal B_1}}=\frac{g_5}{2\partial_z\Phi}
+\frac{g_6}{2\overline{\partial_z\Phi}}\quad \mbox{on}\,\Gamma_0.
\end{eqnarray}
Next we claim that
\begin{eqnarray}\label{AKM}
\quad\quad
{\mathcal R}_{-\tau,A_1}\{e_1g_6\}\vert_{\Gamma_0}=o(\frac {1}{\tau})
\,\,\mbox{as}\,\vert\tau\vert\rightarrow +\infty,\,\,
\widetilde{\mathcal R}_{\tau,B_1}\{e_1 g_5\}\vert_{\Gamma_0}
=o(\frac {1}{\tau})\quad\mbox{as}\,\vert\tau\vert\rightarrow +\infty.
\end{eqnarray}
To see this, let us introduce the function $\mathcal F$
with domain $\Gamma_0$.
\begin{eqnarray}
\mathcal F=2e^{ -\mathcal
A_1}e^{\tau(\Phi-\overline{\Phi})}{\mathcal
R}_{-\tau,A_1}\left \{e_1 g_6\right\}=\nonumber\\
\partial_{\overline
z}^{-1} (e_1e^{-\mathcal
A_1+\tau(\Phi-\overline{\Phi})}\frac{T_{{B}_1} ((q_1 -2
\frac{\partial {B}_1}{\partial \overline{z}}
-{A}_1{B}_1)g_2)-{M}_6 e^{\mathcal B_1}}
{2\overline{\partial _z \Phi}}).\nonumber
\end{eqnarray}
Denoting $r(x)=e^{\mathcal A_1}\frac{T_{{B}_1}
((q_1 -2 \frac{\partial {B}_1}{\partial \overline{z}}
-{A}_1{B}_1)g_2)-{M}_6(\overline{z})e^{\mathcal B_1}}
{2\overline{\partial _z \Phi}}$ we have
$$
\mathcal F(x)
=-\frac{1}{\pi}\int_\Omega\frac{e_1(x)r(x)e^{2i\tau\psi}}
{\overline \zeta-\overline z}d\xi_1d\xi_2=\frac{1}{2i\pi\tau}
\int_\Omega\sum_{k=1}^2\frac{\partial}{\partial x_k}\left(
\frac{\partial\psi}{\partial x_k}\frac{e_1(x)}{\vert\nabla\psi\vert^2}
\frac{r(x)}{\overline \zeta-\overline z}\right)e^{2i\tau\psi}
d\xi_1d\xi_2.
$$
Since $\sum_{k=1}^2\frac{\partial}{\partial x_k}(\frac{\frac{\partial\psi}
{\partial x_k}}{\vert\nabla\psi\vert^2}\frac{e_1(x)r(x)}
{\overline \zeta-\overline z})\in L^1(\Omega)$, we have
$\mathcal F=o(\frac 1\tau)$. This proves (\ref{AKM}).

Now we finish the construction of  functions $a_{2,\tau}(z)$ and
$d_{2,\tau}(\overline z)$ by setting
$$
d_{2,\tau}(\overline z)=d_{2,0}(\overline z)+\frac{d_{2,+}(\overline z)
e^{2i\tau\psi(\widetilde x)} +d_{2,-}(\overline z)e^{-2i\tau\psi(\widetilde x)}}
{2\vert \mbox{det}\, \psi''({\widetilde x})\vert^\frac 12},
$$
$$
a_{2,\tau}(z)=a_{2,0}(z)+\frac{a_{2,+}( z)e^{2i\tau\psi(\widetilde x)}
+a_{2,-}( z)e^{-2i\tau\psi(\widetilde x)}}{2\vert \mbox{det}\, \psi''({\widetilde x})\vert^\frac 12},
$$
where $a_{2,\pm}, d_{2,\pm}$ satisfy (\ref{AK3}).
To complete the construction of a solution to  (\ref{(2.1III)})
we define  $u_0$ as the  solution to the inhomogeneous problem
\begin{equation}\label{(3.18)}
{L}_1(x,D)(u_{0}e^{\tau\varphi})=h_1e^{\tau \varphi}\quad
\mbox{in}\,\,\Omega,
\end{equation}
\begin{equation}\label{(3.188)}
u_{0}e^{\tau\varphi}=e^{\tau\varphi}{\bf m}_1
\quad \mbox{on $\Gamma_0$},
\end{equation}
where
$$
h_1(\tau)=-e^{-\tau\varphi}{L}_1(x,{D})(a_{\tau}
e^{\mathcal{A}_1+\tau \Phi}
+d_{\tau}e^{\mathcal{B}_1 +\tau \overline{\Phi}}
+u_{11}e^{\tau \varphi}+u_{-1}e^{\tau \varphi}),
$$
$$
{\bf m}_1=-e^{-\tau \varphi}(a_{\tau}
e^{\mathcal{A}_1+\tau \Phi}
+d_{\tau}e^{\mathcal{B}_1 +\tau \overline{\Phi}}
+u_{11}e^{\tau \varphi}+u_{-1}e^{\tau \varphi})
\vert _{\Gamma_0}.
$$

Observe that by (\ref{ROMA}) - (\ref{NNN1})  $h_1(\tau)$ can be represented in the form $h_1(\tau)=h_{11}+h_{12}$ where
\begin{equation}\label{(3.188)}
\Vert h_{11}\Vert_{L^4(\Omega)}=O(\frac {1}{\tau^2}),\quad\Vert  h_{12}\Vert_{L^4(\Omega)}=o(\frac 1\tau)\quad\mbox{as}\,\vert\tau\vert
\rightarrow +\infty
\end{equation} and for some positive $\epsilon$
$$
supp\, h_{11} \subset \mathcal O_\epsilon, \quad dist (supp \,h_{12},\partial\Omega)>0
$$
and by (\ref{AK0}),  (\ref{AK2}), (\ref{AK3}), (\ref{IPA})
\begin{equation}
\Vert u_0\Vert_{H^\frac 12(\Gamma_0)}=o(\frac 1{\tau^2})\quad\mbox{as}\,
\vert\tau\vert\rightarrow +\infty.
\end{equation}
By Proposition \ref{Proposition 2.3} and Proposition \ref{Proposition 0}
there exists a solution to  (\ref{(3.18)}), (\ref{(3.188)}) such that

\begin{equation} \label{NASA1}\frac{1}{\root\of{\vert \tau\vert}}
\Vert u_0\Vert_{H^1(\Omega)}+\root\of{\vert \tau\vert}\Vert
u_0\Vert_{L^2(\Omega)}+\Vert u_0\Vert_{H^{1,\tau}(\mathcal O_\epsilon)}
=o(\frac{1}{\tau}) \quad\mbox{as}\,\vert\tau\vert\rightarrow +\infty.
\end{equation}
\subsection{Complex geometrical optics  solutions for the
adjoint operator}

We now construction complex geometrical optics solutions for the adjoint
operator. This parallels the previous construction since the adjoint has
a similar form.

Consider the operator $
{L}_2(x,{D})
=4\frac{\partial}{\partial  z} \frac{\partial}{\partial \overline z}
+2{A}_2\frac{\partial}{\partial  z}
+2{B}_2\frac{\partial}{\partial \overline z}+q_2.
$
Its adjoint has the form
\begin{align*}
{L}_2(x,{D})^{*}
&=4\frac{\partial}{\partial z} \frac{\partial}{\partial \overline z}
-2\overline{{A}_2}\frac{\partial}{\partial \overline z}
-2\overline{{B}_2}\frac{\partial}{\partial  z}+\overline{q_2}
-2\frac{\partial \overline{A}_2}{\partial \overline z}
-\frac{\partial \overline{B}_2}{\partial z}\\
&=(2\frac{\partial}{\partial z}-\overline{{A}_2})
(2\frac{\partial}{\partial \overline z}-\overline{{B}_2})
+q_2 -2 \frac{\partial \overline{{A}_2}}{\partial \overline z}
-\overline{{A}_2}  \overline{{B}_2} \\
&=(2\frac{\partial}{\partial \overline z}-\overline{{B}_2})
(2\frac{\partial}{\partial  z}-\overline{{A}_2})
+q_2 -2 \frac{\partial \overline{{B}_2}}
{\partial {z}}
-\overline{{A}_2}  \overline{{B}_2}.
\end{align*}

Next we construct solution to the following boundary value problem:
\begin{equation}\label{(2.1I8)}
{L}_{2}(x,D)^*v=0\quad \mbox{in}\,\,\Omega,\quad
v\vert_{\Gamma_0}=0.
\end{equation}

We construct solutions to (\ref{(2.1I8)}) of the form
\begin{equation}\label{mozila1}
v(x)=b_{\tau}(z)e^{\mathcal{B}_2-\tau \Phi}
+c_{\tau}(\overline{z})
e^{\mathcal{A}_2-\tau \overline{\Phi}}
+v_{11}e^{-\tau \varphi}
+v_{12}e^{-\tau \varphi}, \quad
v\vert_{\Gamma_0}=0.
\end{equation}
Here $\mathcal A_2,\mathcal B_2\in C^{6+\alpha}(\overline \Omega)$
satisfy
\begin{equation}\label{ziz}
2\frac{\partial{\mathcal A_2}}{\partial z}=\overline{A_2} \,\,\quad
\mbox{in}\,\,\Omega,\quad \mbox{Im}\,\mathcal A_2\vert_{\Gamma_0}=0,
\quad
2\frac{\partial{\mathcal B_2}}{\partial \overline z}=\overline{B_2}
\,\,\quad\mbox{in}\,\,\Omega,\quad \mbox{Im}\,\mathcal B_2\vert
_{\Gamma_0}=0,
\end{equation}
and $b_\tau(z)=b(z)+\frac{b_1(z)}{\tau}+\frac{b_{2,\tau}(z)}{\tau^2},
c_\tau(\overline z)=c(\overline z)+\frac{c_1(\overline z)}{\tau}
+\frac{c_{2,\tau}(\overline z)}{\tau^2}$ and
\begin{equation}\label{PI5}
b,c\in C^{5+\alpha}(\overline\Omega),\,\, \frac{\partial b}
{\partial \overline z}=0\,\,\mbox{in}\,\Omega,\,\,
\frac{\partial c}{\partial z}=0\,\,\mbox{in}\,\Omega,
\end{equation}
\begin{equation}\label{ikaaa}
(be^{\mathcal{B}_2}
+ ce^{\mathcal{A}_2})\vert_{\Gamma_0}=0,
\end{equation}
\begin{equation}\label{iksa}
\frac{\partial^k b}{\partial z^k} \vert_{\mathcal H\cap \partial\Omega}
=\frac{\partial^k c}{\partial \overline z^k}
\vert_{\mathcal H\cap \partial\Omega}=0\quad \forall k\in\{0,\dots,5\},
b\vert_{\mathcal H\setminus\{\widetilde x\}}=c\vert
_{\mathcal H\setminus\{\widetilde x\}}=0,\,\, b(\widetilde x)
\ne 0, c(\widetilde x) \ne 0.
\end{equation}
The existence of the functions $b$ and $c$ is given by
Proposition \ref{balalaika1} .  Denote
$$
g_3=P_{-\overline{{B}_2}}
((\overline q_2-2\frac{\partial \overline{{A}_2}}
{\partial \overline {z}}-\overline{{A}_2}
\overline{{B}_2})be^{\mathcal B_2})-{M}_3(z) e^{\mathcal B_2},\quad
g_{4}=T_{-\overline{{A}_2}}
((\overline q_2-2\frac{\partial \overline{{B}_2}}
{\partial z}-\overline{{A}_2}
\overline{{B}_2})ce^{\mathcal A_2})-{M}_4(\overline z) e^{\mathcal A_2},
$$
where the polynomials $M_3(z), M_4(\overline z)$ are chosen such that
\begin{equation}\label{Vova}
 \frac{\partial^k g_3}{\partial z^k}\vert_{\mathcal H}=\frac{\partial^k g_4}
{\partial \overline z^k}\vert_{\mathcal H}=0\quad\forall k\in\{0,\dots,6\}.
 \end{equation}
 By (\ref{Vova}), (\ref{iksa})
 \begin{equation}\label{Veal}
 \frac{\partial^{k+j}g_3}{\partial z^k\partial \overline z^j}\vert
_{\mathcal H\cap \partial\Omega}=\frac{\partial^{k+j} g_4}
{\partial z^k\partial \overline z^j}\vert_{\mathcal H\cap \partial\Omega}
=0\quad \forall k+j\le 6,\quad\frac{\partial g_4}{\partial z}
=\frac{\partial g_3}{\partial \overline z}=0\,\,\mbox{on}\,\,
\mathcal H\setminus\{\widetilde x\} .
 \end{equation}
Observe that by (\ref{Veal}) $\frac{g_3}{ \partial _z \Phi},
\frac{ g_4}{ \overline{ \partial _z \Phi}}\in C^{4+\alpha}
(\overline\Omega).$
Using Proposition \ref{zika} we introduce a holomorphic function
$b_1(z)\in C^2(\overline \Omega)$ and an antiholomorphic function
$c_1(\overline z)\in C^2(\overline\Omega)$ such that
\begin{equation}
b_1e^{{\mathcal B_2}}+c_1 e^{{\mathcal A_2}}=\frac{e_2 g_3}
{2{\partial_z \Phi}}+
\frac{e_2 g_4}{2\overline{\partial_z \Phi}}\quad \mbox{on}\,\Gamma_0.
\end{equation}
Let
$$
\widehat g_3=P_{-\overline{{B}_2}}
((\overline q_2-2\frac{\partial \overline{{A}_2}}
{\partial \overline {z}}-\overline{{A}_2}
\overline{{B}_2})b_1e^{\mathcal B_2})-\widehat {M}_3(z) e^{\mathcal B_2},
\quad
\widehat g_{4}=T_{-\overline{{A}_2}}
((\overline q_2-2\frac{\partial \overline{{B}_2}}
{\partial z}-\overline{{A}_2}
\overline{{B}_2})c_1e^{\mathcal A_2})-\widehat {M}_4(\overline z)
e^{\mathcal A_2},
$$
where the polynomials $\widehat M_3(z),\widehat M_4(\overline z)$ are chosen
such that
\begin{equation}
\frac{\partial^k \widehat g_3}{\partial z^k}\vert_{\mathcal H}
= \frac{\partial^k \widehat g_4}{\partial \overline z^k}\vert_{\mathcal H}=0
\quad\forall k\in\{0,\dots,3\}.
\end{equation}
The function $v_{11}$ is defined by
\begin{eqnarray}\label{MMM}
v_{11}&=-e^{-i \tau \psi}
\widetilde{\mathcal{R}}_{-\tau,-\overline{{A}_2}}\{e_1(g_3+\widehat g_3/\tau)\}
+\frac{e^{-i\tau \psi}e_2(g_3+\widehat g_3/\tau)}{2 \tau \partial _z \Phi}
-e^{i \tau \psi}
\mathcal{R}_{\tau,-\overline{{B}_2}}\{e_1(g_4+\widehat g_4/\tau)\}
+\frac{e^{i \tau \psi}e_2(g_4+\widehat g_4/\tau)}{2\tau\overline{\partial_z \Phi}}
\nonumber \\
&-\frac{e^{-i \tau \psi}}{4\tau^2 \partial_z \Phi}L_2(x,D)^*
\biggl(\frac{e_2g_3}{\partial_z \Phi}\biggl)
-\frac{e^{i \tau \psi}}{4 \tau^2 \overline{\partial_z \Phi}} L_2(x,D)^*
\biggl(\frac{e_2g_4}{ \overline{\partial_z \Phi}}\biggl).
\end{eqnarray}
Here we set
$$
\mathcal R_{\tau,-\overline{B_2}}\{g\}
= \frac{1}{2}e^{\mathcal B_2}e^{\tau(\overline\Phi-\Phi)}
\partial _{\overline z}^{-1}(ge^{-\mathcal B_2}
e^{\tau(\Phi-\overline\Phi)})
$$
$$
\widetilde{\mathcal R}_{-\tau,-\overline{A_2}}\{g\}
= \frac{1}{2}e^{\mathcal A_2}e^{\tau(\Phi-\overline\Phi)}
\partial _{z}^{-1}(ge^{-\mathcal A_2}e^{\tau(\overline\Phi-\Phi)})
$$
provided that $A_2, B_2, \mathcal{A}_2, \mathcal{B}_2$ satisfy
(\ref{ziz}).
By Proposition \ref{zinka} the following asymptotic formulae hold:
\begin{equation}\label{victory1}
\widetilde{\mathcal R}_{-\tau,-\overline A_2}\left\{e_1g_3\right\}\vert
_{\partial\Omega} = \frac{1}{2\tau^2}\frac{e^{{\mathcal A_2}
+2\tau i\psi}}
{\vert \mbox{det}\, \psi''({\widetilde x})\vert^\frac 12}
\left (\frac{e^{-2i\tau\psi({\widetilde x})} r_{1}({\widetilde x})}
{(\overline z-\overline{\widetilde z})^2}+\frac{e^{-2i\tau\psi({\widetilde x})} t_{1}
({\widetilde x})}{(\overline{\widetilde z}-\overline z)}\right)
+\widetilde{\mathcal W}_{2,\tau}
\end{equation}
and
\begin{equation}\label{victory}
\mathcal R_{\tau, -\overline B_2}\left\{e_1g_4\right\}\vert
_{\partial\Omega}
= \frac{1}{2\tau^2}\frac{e^{{\mathcal B_2}-2\tau i\psi}}
{\vert \mbox{det}\, \psi''({\widetilde x})\vert^\frac 12}
\left ( \frac{e^{2i\tau\psi({\widetilde x})}\widetilde r_{1}({\widetilde x})}
{(z-\widetilde z)^2}+\frac{e^{2i\tau\psi({\widetilde x})}\widetilde t_{1}
({\widetilde x})}{(\widetilde z-z)}\right )
+\widetilde{\mathcal W}_{1,\tau},
\end{equation}
where
\begin{equation}\label{v3}
r_{1}(\widetilde x)=\frac{\partial_{ \overline z}
\widetilde g_3({\widetilde x})}{\overline{\partial^2_z
\Phi({\widetilde x})}},\quad t_{1}(\widetilde x)=\frac{1}{2}
\left (\frac{\partial_{ \overline z}\widetilde g_3 ({\widetilde x})}
{\overline{\partial^2_z\Phi({\widetilde x})}}\frac{\partial^3_{\overline z}
\overline\Phi({\widetilde x})}{\partial^2_{\overline z}\overline\Phi({\widetilde x})}
+ \frac{\partial^2_{ z}\widetilde g_3 ({\widetilde x})}
{{\partial^2_z\Phi({\widetilde x})}}
-\frac{\partial^2_{ \overline z}\widetilde g_3 ({\widetilde x})}
{\overline{\partial^2_z\Phi({\widetilde x})}}\right ),
\end{equation}
\begin{equation}\label{v4}
\widetilde r_{1}(\widetilde x)=\frac{\partial_z\widetilde g_4 ({\widetilde x})}
{\partial^2_z\Phi({\widetilde x})},\quad \widetilde t_{1}(\widetilde x)
=\frac{1}{2}\left( \frac{\partial_{z}\widetilde g_4 ({\widetilde x})}
{{\partial^2_z\Phi({\widetilde x})}}\frac{{\partial_{ z}^3
\Phi({\widetilde x})}}{{\partial^2_z\Phi({\widetilde x})}}
- \frac{\partial^2_{ z}\widetilde g_4 ({\widetilde x})}
{{\partial^2_z \Phi({\widetilde x})}}
+ \frac{\partial^2_{\overline z}\widetilde g_4 ({\widetilde x})}
{\overline{\partial^2_z\Phi({\widetilde x})}}\right),
\end{equation}
$\widetilde g_3=e^{-{\mathcal A_2}}g_3, \widetilde g_4
=e^{-{\mathcal B_2}}g_4.$
Here the functions  $\widetilde{\mathcal W}_{\tau,1},\widetilde{\mathcal W}
_{\tau,2}\in H^\frac 12(\Gamma_0)$ satisfy
\begin{equation}\label{PI4}
\Vert\widetilde{\mathcal{W}}_{\tau,1}\Vert_{H^\frac 12(\Gamma_0)}
+\Vert\widetilde{\mathcal W}_{\tau,2}\Vert_{H^\frac 12(\Gamma_0)}
=o(\frac{1}{\tau^2})\quad\mbox{as}\,\vert\tau\vert\rightarrow +\infty.
\end{equation}

Using Proposition \ref{zika} we define the holomorphic  functions
$b_{2,\pm}(z)\in C^2(\overline \Omega)$ and antiholomorphic
$c_{2,\pm}(\overline z)\in C^2(\overline \Omega)$
such that
\begin{equation}\label{PI3}
b_{2,\pm}e^{{\mathcal B_2}}+c_{2,\pm}e^{{\mathcal A_2}}
=\widetilde p_\pm\quad \mbox{on}\,\Gamma_0,
\end{equation}
where $\widetilde p_k$ is defined as
$$
\widetilde p_+(x)=e^{{\mathcal B_2}}\left(\frac{\widetilde r_{1}
({\widetilde x})}{(z-\widetilde z)^2}+\frac{\widetilde t_{1}({\widetilde x})}
{(\widetilde z-z)}\right),\quad
\widetilde p_-(x)= e^{{\mathcal A_2}}
\left (\frac{ r_{1}(\widetilde x)}{(\overline z
-\overline{\widetilde z})^2}+\frac{ t_{1}(\widetilde x)}
{(\overline{\widetilde z}-\overline z)}\right).
$$
Similarly to (\ref{ROMA}), there exist
two positive numbers $\epsilon $ and $\epsilon'$ such that
\begin{eqnarray}\label{O1}
&{L}_2(x,{D})^{*}((b+\frac{b_1}{\tau})
e^{\mathcal{B}_2-\tau \Phi}
+(\overline{b}+ \frac{c_1}{\tau})
e^{\mathcal{A}_2 -\tau \overline{\Phi}}
+v_{11}e^{-\tau \varphi})\nonumber \\
&=\frac{g_3 e^{-\tau \Phi}}{2 \tau \partial _z \Phi}
(\overline q_2 - 2\frac{\partial \overline{{A}_2}}
{\partial {\overline z}}
-\overline{{A}_2}\overline{{B}_2})
-\frac{g_4 e^{-\tau \overline{\Phi}}}{2 \tau
\overline{\partial _z \Phi}}
(\overline q_2 -2\frac{\partial \overline{{B}_2}}{\partial z}
-\overline{{A}_2}\overline{{B}_2})\\
& +\chi_{\Omega\setminus{\mathcal O}_\epsilon'}o_{L^4(\Omega)}
+ \chi_{{\mathcal O}_\epsilon}O_{L^4(\Omega)}
\biggl(\frac{1}{\tau} \biggl).
\end{eqnarray}

We are looking for $v_{12}$ in the form $v_{12}=v_0+v_{-1}.$
The function $v_{-1}$ is given by

\begin{equation}\label{O2}
v_{-1}=-\frac{e^{\tau i{\psi}}}{\tau}
\mathcal{R}_{\tau,-\overline{{B}}_2}\{e_1g_7\}
-\frac{e^{-\tau i\psi}}{\tau}\widetilde{\mathcal{R}}_{-\tau,-\overline{{A}}_2}
\{e_1g_8\}
+\frac{e_2 g_7}{2 \tau^2 \overline{\partial_z \Phi}}
+\frac{e_2 g_8}{2 \tau^2 \partial_z \Phi},
\end{equation}
where
\begin{equation}\label{O3}
g_7=\frac{P_{-\overline{{B}_2}}
((q_2 -2\frac{\partial \overline{{B}_2}}{\partial z}
-\overline{{A}_2}\overline{{B}_2})g_3)-{M}_7(z)e^{\mathcal B_2}}
{2 \partial_z \Phi},\,\,
g_8=\frac{T_{-\overline{{A}_2}}
((q_2 - 2\frac{\partial \overline{{A}_2}}{\partial\overline z}
-\overline{{A}_2}\overline{{B}_2})g_4)-{M}_8(\overline z) e^{\mathcal A_2}}
{2 \overline{\partial_z \Phi}},
\end{equation}
and $M_7(z),M_8(\overline z)$ are polynomials such that
\begin{equation}\label{TTH}
g_7\vert_{\mathcal H}=g_8\vert_{\mathcal H}=\nabla g_7\vert
_{\mathcal H}=\nabla g_8\vert_{\mathcal H}=0.
\end{equation}
Using Proposition \ref{zika} we introduce functions $b_{2,0}, c_{2,0}
\in C^2(\overline \Omega)$ such that
\begin{eqnarray}
b_{2,0}e^{{\mathcal B_2}}+c_{2,0}e^{{\mathcal A_2}}=
\frac{g_7}{2\overline{\partial_z\Phi}}+
\frac{g_8}{2{\partial_z\Phi}}\quad \mbox{on}\,\Gamma_0.
\end{eqnarray}
Similarly to (\ref{AKM}) we have
\begin{equation}\label{PI2}
(\frac{1}{\tau}
{\mathcal R}_{\tau,-\overline B_2} \{e_1g_8\}\nonumber\\+\frac{1}{\tau}
\widetilde{\mathcal R}_{-\tau,-\overline A_2}\{e_1 g_7\})\vert_{\Gamma_0}
=o(\frac {1}{\tau^2})\quad\mbox{as}\,\vert\tau\vert\rightarrow +\infty.
\end{equation}

Now we finish the construction of  $b_{2,\tau}(z)$ and
$c_{2,\tau}(\overline z)$ by setting
\begin{equation}\label{PI}
b_{2,\tau}(\overline z)=b_{2,0}(\overline z)+\frac{b_{2,+}(\overline z)
e^{2i\tau\psi(\widetilde x)}+b_{2,-}(\overline z)e^{-2i\tau\psi(\widetilde x)}}
{2\vert \mbox{det}\, \psi''({\widetilde x})\vert^\frac 12}
\end{equation}
and
\begin{equation}\label{PI1}
c_{2,\tau}(z)=c_{2,0}(z)+\frac{c_{2,+}( z)e^{2i\tau\psi(\widetilde x)}
+c_{2,-}( z)e^{-2i\tau\psi(\widetilde x)}}{2\vert \mbox{det}\,
\psi''({\widetilde x})\vert^\frac 12} ,
\end{equation}
where $b_{2,+},c_{2,-}$ are defined in (\ref{PI3}).

Consider the following boundary value problem
\begin{equation}\label{NINA}
{L}_2(x,{D})^* (e^{-\tau \varphi}v_0)
=h_2 e^{-\tau \varphi} \quad \text{in}~ \Omega,
\end{equation}
\begin{equation}\label{NINA1}
e^{-\tau \varphi}v_0 \vert_{\Gamma_0}
={\bf m}_2 e^{-\tau\varphi},
\end{equation}
where
$$
h_2=-e^{\tau \varphi} {L}_2(x,{D})^* (b_{\tau}e^{\mathcal{B}_2
-\tau \Phi} +c_{\tau}e^{\mathcal{A}_2-\tau
\overline{\Phi}} +v_{11}e^{-\tau \varphi}+ v_{-1}e^{-\tau \varphi})
$$
and
$$
{\bf m}_2=-e^{\tau \varphi}(b_{\tau}e^{\mathcal{A}_2-\tau \Phi}
+c_{\tau}e^{\mathcal{B}_2-\tau \overline{\Phi}}
+v_{11}e^{-\tau \varphi}+v_{-1}e^{-\tau \varphi}).
$$

By (\ref{O1})-(\ref{O3})  represent the function  $h_2$ in the form $h_2=h_{21}+h_{22}$ where for some positive $\epsilon$
$$
supp\, h_{21}\subset \mathcal O_\epsilon,\quad dist (supp \, h_{22}, \partial\Omega)>0.
$$ The norms of the functions $h_{2j}$ are estimated as
\begin{equation}\label{osel}
\Vert h_{21}\Vert_{L^4(\Omega)}=O(\frac 1{\tau^2}),\quad\Vert h_{22}\Vert_{L^4(\Omega)}=o(\frac 1{\tau})\quad\mbox{as}
\,\vert\tau\vert\rightarrow +\infty.
\end{equation}
By (\ref{PI}), (\ref{PI1}),  (\ref{PI3}), (\ref{PI4}), (\ref{PI5})
we have
\begin{equation}\label{kozel}
\Vert v_0\Vert_{H^\frac 12(\Gamma_0)}=o(\frac 1{\tau^2}) \quad\mbox{as}
\,\vert\tau\vert\rightarrow +\infty.
\end{equation}
Thanks to (\ref{osel}), (\ref{kozel}), by Proposition
\ref{Proposition 2.3} and Proposition \ref{Proposition 0} for sufficiently
small positive $\epsilon$  there exists a
solution to problem (\ref{NINA}), (\ref{NINA1}) such that
\begin{equation}\label{nona1}
\frac{1}{\root\of{\vert \tau\vert}}\Vert v_0\Vert_{H^1(\Omega)}
+\root\of{\vert \tau\vert}\Vert
v_0\Vert_{L^2(\Omega)}+\Vert v_0\Vert_{H^{1,\tau}
(\mathcal O_\epsilon)}=o(\frac{1}{\tau})\quad\mbox{as}\,
\vert\tau\vert\rightarrow +\infty.
\end{equation}

\section{Proof of Theorem \ref{main}.}

Let $u_1$ be a complex geometrical optics  solution as
in (\ref{mozilaa}).
Let $u_2$ be a solution to the following boundary value problem
\begin{equation}\label{(2.1I)}
{ L}_{2}(x,D)u_2=0\quad \mbox{in}\,\,\Omega,\quad
u_2\vert_{\partial\Omega}=u_1\vert_{\partial\Omega},
\quad \frac{\partial u_2}{\partial \nu}\vert_{\widetilde \Gamma}
=\frac{\partial u_1}{\partial\nu}\vert_{\widetilde \Gamma}.
\end{equation}

Setting $u=u_1-u_2, q=q_1-q_2$ we have
\begin{equation}
{L}_2(x,{D})u
+2({A}_1-{A}_2)\frac{\partial u_1}{\partial{z}}
+2({B}_1-{B}_2)\frac{\partial u_1}{\partial{\overline z}}
+qu_1=0 \quad \text{in}~ \Omega, \label{mn}
\end{equation}
\begin{equation}
u \vert_{\partial\Omega}
=0, \quad \frac{\partial u}{\partial \nu} \vert_{\widetilde \Gamma}
=0.
\end{equation}
Let $v$ be a solution to (\ref{(2.1I8)}) in the form
(\ref{mozila1}).  Taking the scalar product of (\ref{mn}) with
$\overline v$  in $L^2(\Omega)$ we obtain
\begin{equation}\label{ippolit}
0= \int_{\Omega}(2({A}_1-{A}_2)\frac{\partial u_1}{\partial{z}}
+2({B}_1-{B}_2)\frac{\partial u_1}{\partial \overline{z}}
+qu_1 )\overline v dx.
\end{equation}
Our goal is to obtain the asymptotic formula for the right hand side of
(\ref{ippolit}).
We have
\begin{proposition}\label{Nova}
The following asymptotic formula is valid as
$\vert \tau\vert\rightarrow +\infty$:
\begin{eqnarray}\label{nonsence1}
I_0=(qu_1,v)_{L^2(\Omega)}=\int_\Omega (qa\overline ce^{({\mathcal A}_1
+\overline{\mathcal A_2})}+q d\overline b e^{({\mathcal B_1}
+\overline{\mathcal B_2})})dx\\+\int_\Omega(
\frac q\tau(a_1b+a\overline{c_1})e^{({\mathcal A_1}
-\overline{\mathcal B_2})}+\frac q\tau
(\overline a\overline{b_1}+\overline b d_1)e^{({\mathcal A_2}
-\overline{\mathcal B_1})})dx
\nonumber \\
+\frac 1\tau\int_\Omega\left(\frac{a\overline{g_4}
e^{\mathcal A_1}}{2{\partial_z\Phi}}-\frac{\overline c{g_2}
e^{\overline{\mathcal A_2}}}{2{\partial_z\Phi}}
-\frac{\overline b{g_1}e^{\overline{\mathcal B_2}}}{2{\partial_{\overline z}
\overline\Phi}}+\frac{d\overline{g_3}e^{{\mathcal B_1}}}
{2\overline{\partial_z\Phi}}\right)dx\nonumber\\
+2\pi \frac{(qa\overline b)(\widetilde x) e^{(\mathcal A_1
+\overline{\mathcal B_2}+2\tau i\psi)(\widetilde x)}
+(q d\overline c)(\widetilde x)e^{({\mathcal B_1}
+\overline{\mathcal A_2}-2i\tau\psi)(\widetilde x)}}
{\tau\vert \mbox{det}\, \psi(\widetilde x)\vert^\frac 12}
                                               \nonumber\\
+\frac {1}{2\tau i}\int_{\partial\Omega} qa\overline
be^{\mathcal A_1+\overline{\mathcal B_2}+2\tau i\psi}
\frac{(\nu,\nabla\psi)}{\vert\nabla\psi\vert^2}d\sigma -
\frac {1}{2\tau i}\int_{\partial\Omega} qd\overline c
e^{\mathcal B_1+\overline{\mathcal A_2}-2\tau i\psi}
\frac{(\nu,\nabla\psi)}{\vert\nabla\psi\vert^2}d\sigma +o(\frac 1\tau).
                                       \nonumber
\end{eqnarray}
\end{proposition}
\begin{proof}
By (\ref{mozilaa}), (\ref{wolf}), (\ref{NASA1}) and Proposition
\ref{Proposition
3.22} we have
\begin{equation}\label{P1}
u_1(x)=(a(z)+\frac{a_1(z)}{\tau})e^{\mathcal
A_1+\tau\Phi}+({d(\overline z)}+\frac{d_1(\overline z)}{\tau})e^{\mathcal
B_1+\tau\overline\Phi}-\frac{g_1e^{\tau\overline\Phi}}{2\tau\overline
{\partial_z\Phi}}-\frac{g_2e^{\tau\Phi}}
{2\tau{\partial_z\Phi}}+o_{L^2(\Omega)}(\frac 1\tau).
\end{equation}
Using (\ref{mozila1}), (\ref{MMM}), (\ref{nona1}) and Proposition
\ref{Proposition 3.22} we get
\begin{equation}\label{P2}
v(x)=(b(z)+\frac{b_1(z)}{\tau})e^{\mathcal
B_2-\tau\Phi}+({c(\overline z)}+\frac{c_1(\overline z)}{\tau})
e^{\mathcal
A_2-\tau\overline\Phi}+\frac{g_4e^{-\tau\overline\Phi}}
{2\tau\overline{\partial_z\Phi}}+\frac{g_3e^{-\tau\Phi}}
{2\tau{\partial_z\Phi}}+o_{L^2(\Omega)}(\frac 1\tau).
\end{equation}

By (\ref{P1}), (\ref{P2}) we obtain
\begin{eqnarray}\label{nonsence}
(qu_1,v)_{L^2(\Omega)}
= (q((a+\frac{a_1}{\tau})e^{\mathcal
A_1+\tau\Phi}+({d}+\frac{d_1}{\tau})e^{\mathcal
B_1+\tau\overline\Phi}-\frac{g_1e^{\tau\overline\Phi}}
{2\tau\overline{\partial_z\Phi}}-\frac{g_2e^{\tau\Phi}}
{2\tau{\partial_z\Phi}}+o_{L^2(\Omega)}(\frac 1\tau)),\nonumber\\
(b+\frac{b_1}{\tau})e^{\mathcal
B_2-\tau\Phi}+({c}+\frac{c_1}{\tau})e^{\mathcal
A_2-\tau\overline\Phi}+\frac{g_4e^{-\tau\overline\Phi}}
{2\tau\overline{\partial_z\Phi}}+\frac{g_3e^{-\tau\Phi}}
{2\tau{\partial_z\Phi}}+o_{L^2(\Omega)}(\frac 1\tau))
_{L^2(\Omega)}=\nonumber\\
\int_\Omega \left( q  (d\overline b+\frac 1\tau(d_1\overline b
+d\overline b_1))e^{\mathcal B_1+\overline{\mathcal B_2}}
+q(a\overline c+\frac{1}{\tau}(a\overline{c_1}+a_1\overline c))
e^{\mathcal A_1+\overline{\mathcal A_2}}\right )dx\nonumber\\
+\frac 1\tau\int_\Omega\left(\frac{a\overline{g_4}
e^{\mathcal A_1}}{2{\partial_z\Phi}}-\frac{\overline c{g_2}
e^{\overline{\mathcal A_2}}}{2{\partial_z\Phi}}
-\frac{\overline b{g_1}e^{\overline{\mathcal B_2}}}
{2{\partial_{\overline z}\overline\Phi}}+\frac{d\overline{g_3}
e^{{\mathcal B_1}}}{2\overline{\partial_z\Phi}}\right)dx\nonumber\\
+\int_\Omega q \left(a\overline b e^{\mathcal A_1+\overline{B_2}
+\tau(\Phi-\overline\Phi)}+ d\overline c e^{\mathcal B_1+\overline{\mathcal A_2}
+\tau(\overline \Phi-\Phi)}\right)dx
+o(\frac{1}{\tau}).\nonumber
\end{eqnarray}
Applying the stationary phase argument to the last integral on the right
hand side of this formula we finish the proof of Proposition \ref{Nova}.
\end{proof}
We set
\begin{equation}\label{narkotic}
\mathcal{U}(x)=a_{\tau}(z)e^{\mathcal{A}_1(x)+ \tau \Phi(z)}
+d_{\tau}(\overline{z})e^{\mathcal{B}_1(x)+\tau \overline{\Phi(z)}},
\,\,
\mathcal{V}(x)=b_{\tau}(z)e^{\mathcal{B}_2(x)- \tau \Phi(z)}
+c_{\tau}(\overline z)e^{\mathcal{A}_2(x)-\tau \overline{\Phi(z)}}.
\end{equation}
Short calculations give:
\begin{eqnarray}\label{MAHA}
I_1 \equiv 2(({A}_1-{A}_2)\frac{\partial \mathcal{U}}{\partial z},\mathcal{V})
_{L^2(\Omega)} \nonumber\\
=  (2({A}_1-{A}_2)(
\biggl(\frac{\partial\mathcal{A}_1}{\partial z}
+\tau \frac{\partial \Phi}{\partial z}\biggl)
a_{\tau}+\frac{\partial a_\tau}{\partial z})e^{\mathcal{A}_1+\tau \Phi}
+d_{\tau}\frac{\partial \mathcal{B}_1}{\partial z}
e^{\mathcal{B}_1+\tau \overline{\Phi}},\nonumber \\
b_{\tau} e^{\mathcal{B}_2-\tau \Phi} +c_{\tau}
e^{\mathcal{A}_2-\tau \overline{\Phi}})_{L^2(\Omega)}                              \nonumber\\
=\sum_{k=1}^3\tau^{2-k}\kappa_k
-\int_{\Omega} ({A}_1-{A}_2){B}_1
d_\tau \overline{c_{\tau}}
e^{\mathcal{B}_1+\overline{\mathcal{A}_2}-2i \tau \psi}dx
\nonumber\\
- \biggl(2\frac{\partial}{\partial z}({A}_1-{A}_2)
a_\tau e^{\mathcal{A}_1 + \tau \Phi},
b_{\tau}e^{\mathcal{B}_2-\tau \Phi}\biggl)_{L^2(\Omega)}+\frac{1}{\tau}
\mathcal I_1(\partial\Omega)
\nonumber\\
-  (2({A}_1-{A}_2)a_\tau e^{\mathcal{A}_1 + \tau \Phi},
\frac{\partial \mathcal{B}_2}{\partial \overline{z}}b_{\tau}
e^{\mathcal{B}_2-\tau \Phi})_{L^2(\Omega)}\nonumber \\
+\int_{\partial\Omega}(A_1-A_2)(\nu_1-i\nu_2)a_\tau\overline{b_\tau}
e^{\mathcal A_1+\overline{\mathcal B_2}+2i\tau\psi}d\sigma
+ o\left( \frac{1}{\tau}\right)                       \nonumber\\
= \sum_{k=1}^3 \tau^{2-k}\kappa_k
+\int_{\Omega} \biggl\{-({A}_1-{A}_2){B}_1
d_{\tau}\overline{c_{\tau}}
e^{\mathcal{B}_1+\overline{\mathcal{A}_2}-2i \tau \psi} \nonumber\\
\left. -({A}_1-{A}_2)B_2
a_{\tau}\overline{b_{\tau}}
e^{\mathcal{A}_1+\overline{\mathcal{B}_2} +2i \tau \psi}
-2 \frac{\partial}{\partial z}({A}_1-{A}_2)
a_{\tau}\overline{b_{\tau}}
e^{\mathcal{A}_1+\overline{\mathcal{B}_2} +2i \tau \psi}\right\}dx
                                               \nonumber\\
+\int_{\partial\Omega}(A_1-A_2)(\nu_1-i\nu_2)
a_\tau\overline{b_\tau}e^{\mathcal A_1
+\overline{\mathcal B_2}+2i\tau\psi}d\sigma+\frac{1}{\tau}
\mathcal I_1(\partial\Omega) +  o\left( \frac{1}{\tau}\right)
\end{eqnarray}
\nopagebreak
and
\begin{eqnarray}\label{MAHA1}
I_2 \equiv ( ({B}_1-{B}_2)\frac{\partial
\mathcal{U}}{\partial \overline z},\mathcal{V})
_{L^2(\Omega)}                                   \nonumber\\
= (2({B}_1-{B}_2)(a_{\tau} e^{\mathcal{A}_1 + \tau \Phi}
\frac{\partial \mathcal{A}_1}{\partial \overline{z}}
+ \frac{\partial}{\partial \overline{z}}
\biggl( d_{\tau}e^{\mathcal{B}_1+\tau \overline{\Phi}}\biggl)),
 b_{\tau} e^{\mathcal{B}_2 -\tau \Phi}
+ c_{\tau}e^{\mathcal{A}_2-\tau \overline{\Phi}})_{L^2(\Omega)}
                                   \nonumber \\
=\sum_{k=1}^3\tau^{2-k}\widetilde\kappa_k
+ \int_{\Omega} 2({B}_1-{B}_2)
\frac{\partial \mathcal{A}_1}{\partial \overline{z}}
a_{\tau}\overline{b_{\tau}}
e^{\mathcal{A}_1+\overline{\mathcal{B}_2}+2 \tau i \psi}dx
-\biggl(2\frac{\partial}{\partial \overline{z}}
({B}_1-{B}_2)
d_{\tau}e^{\mathcal{B}_1+\tau \overline\Phi}, \nonumber\\
c_{\tau}e^{\mathcal{A}_2 -\tau \overline{\Phi}}\biggl)
_{L^2(\Omega)}
-\biggl(2({B}_1-{B}_2)d_{\tau}
e^{\mathcal{B}_1+\tau \overline\Phi},
\frac{\partial \mathcal{A}_2}{\partial z}
c_{\tau}e^{\mathcal{A}_2 -\tau \overline{\Phi}}\biggl)
_{L^2(\Omega)}
                                          \nonumber \\
+\int_{\partial\Omega}(B_1-B_2)(\nu_1+i\nu_2)d_\tau
\overline{c_\tau}e^{\mathcal B_1+\overline{\mathcal A_2}
-2i\tau\psi}d\sigma+\frac{1}{\tau}
\mathcal I_2(\partial\Omega) + o\left( \frac{1}{\tau}\right)
                                                   \nonumber\\
= \sum_{k=1}^3\tau^{2-k}\widetilde\kappa_k
+ \int_{\Omega}\left\{ -({B}_1-{B}_2)
{A}_1
a_{\tau}\overline{b_{\tau}}
e^{\mathcal{A}_1+\overline{\mathcal{B}_2}+2 \tau i \psi}\right.\nonumber \\
-\left.2\frac{\partial}{\partial \overline{z}}
({B}_1-{B}_2)
d_{\tau} \overline{c_{\tau}}
e^{\mathcal{B}_1 + \overline{\mathcal{A}_2}-2i \tau \psi}
- ({B}_1-{B}_2) d_{\tau}
\overline{c_{\tau}}
{A}_2 e^{\mathcal{B}_1+\overline{\mathcal{A}_2}-2i \tau \psi} \right\}dx
                            \nonumber \\
+\int_{\partial\Omega}(B_1-B_2)(\nu_1+i\nu_2)d_\tau
\overline{c_\tau}e^{\mathcal B_1+\overline{\mathcal A_2}
- 2i\tau\psi}d\sigma+\frac{1}{\tau}\mathcal I_2(\partial\Omega)
+ o\left( \frac{1}{\tau}\right).
\end{eqnarray}

Here $\kappa_k,\widetilde \kappa_k$ are some constants independent of
$\tau$ but may  depend on $A_j, B_j$, $\Phi.$
The terms $\mathcal I_1(\partial\Omega),\mathcal I_2(\partial\Omega)$ are
given by
\begin{eqnarray}\label{nina}
\mathcal I_1(\partial\Omega)=\int_\Omega (A_1-A_2)e^{\mathcal A_1
+\overline{\mathcal A_2}}\frac{\partial\Phi}{\partial z}
\overline{c}\left ( \frac{a_{2,+}e^{2\tau i\psi({\widetilde x})}
+a_{2,-}e^{-2\tau i\psi({\widetilde x})}}{\vert \mbox{det}\,
\psi''({\widetilde x})\vert^\frac 12} \right )dx\nonumber\\
+ \int_\Omega (A_1-A_2)e^{\mathcal A_1+\overline{\mathcal A_2}}
\frac{\partial\Phi}{\partial z}a\left (\overline{\frac{ c_{2,+}
e^{2\tau i\psi({\widetilde x})}+c_{2,-}e^{-2\tau i\psi({\widetilde x})} }
{\vert \mbox{det}\, \psi''({\widetilde x})\vert^\frac 12}}
\right)dx=\nonumber\\
-2\int_\Omega \frac{\partial}{\partial \overline z} e^{\mathcal A_1
+\overline{\mathcal A_2}}\frac{\partial\Phi}{\partial z}
\overline{c}\left ( \frac{a_{2,+}e^{2\tau i\psi({\widetilde x})}
+ a_{2,-}e^{-2\tau i\psi({\widetilde x})}}{\vert \mbox{det}
\, \psi''({\widetilde x})\vert^\frac 12}\right )dx\nonumber\\
-2\int_\Omega
\frac{\partial}{\partial \overline z}e^{\mathcal A_1+\overline{\mathcal A_2}}
\frac{\partial\Phi}{\partial z}a\left (\overline{\frac{c_{2,+}
e^{2\tau i\psi({\widetilde x})}+c_{2,-}e^{-2\tau i\psi({\widetilde x})}}
{\vert \mbox{det}\, \psi''({\widetilde x})\vert^\frac 12}}\right )dx=
\nonumber\\
-\int_{\partial\Omega} (\nu_1+i\nu_2) e^{\mathcal A_1
+\overline{\mathcal A_2}}\frac{\partial\Phi}{\partial z}\overline{c}
\left (\frac{a_{2,+}e^{2\tau i\psi({\widetilde x})}+a_{2,-}
e^{-2\tau i\psi({\widetilde x})}}{\vert \mbox{det}
\, \psi''({\widetilde x})\vert^\frac 12}\right )d\sigma\nonumber\\
- \int_{\partial\Omega} (\nu_1+i\nu_2)e^{\mathcal A_1
+\overline{\mathcal A_2}}\frac{\partial\Phi}{\partial z}a\left
(\overline{\frac{c_{2,+}e^{2\tau i\psi({\widetilde x})}+c_{2,-}
e^{-2\tau i\psi({\widetilde x})}}{\vert \mbox{det}
\, \psi''({\widetilde x})\vert^\frac 12}}\right )d\sigma
\end{eqnarray}
and
\begin{eqnarray}\label{vik}\mathcal
I_2(\partial\Omega)=\int_\Omega (B_1-B_2)
e^{\mathcal B_1+\overline{\mathcal B_2}}\frac{\partial\overline\Phi}
{\partial \overline z}\overline{b}\left (\frac{d_{2,+}
e^{2\tau i\psi({\widetilde x})}
+d_{2,-}e^{-2\tau i\psi({\widetilde x})} }{\vert \mbox{det}
\, \psi''({\widetilde x})\vert^\frac 12}  \right )dx\nonumber\\
+ \int_\Omega (B_1-B_2)e^{\mathcal B_1+\overline{\mathcal B_2}}
\frac{\partial\overline \Phi}{\partial\overline  z}d\left (\overline{
\frac{ b_{2,+}e^{2\tau i\psi({\widetilde x})}+b_{2,-}
e^{-2\tau i\psi({\widetilde x})}}{\vert \mbox{det}
\, \psi ''({\widetilde x})\vert^\frac 12} }\right )dx=\nonumber
\\
-2\int_\Omega \frac{\partial}{\partial z}
e^{\mathcal B_1+\overline{\mathcal B_2}}\frac{\partial\overline\Phi}
{\partial \overline z}\overline{b}\left (\frac{d_{2,+}
e^{2\tau i\psi({\widetilde x})}+d_{2,-}e^{-2\tau i\psi({\widetilde x})}}
{\vert \mbox{det}\, \psi ''({\widetilde x})\vert^\frac 12}
  \right  )dx\nonumber\\
  -2\int_\Omega \frac{\partial}{\partial  z}
e^{\mathcal B_1+\overline{\mathcal B_2}}\frac{\partial\overline\Phi}
{\partial\overline z}d\left (\overline{ \frac{b_{2,+}e^{2\tau i\psi({\widetilde x})}
+b_{2,-}e^{-2\tau i\psi({\widetilde x})}}{\vert \mbox{det}
\, \psi ''({\widetilde x})\vert^\frac 12}}\right )dx=\nonumber\\
-\int_{\partial\Omega} (\nu_1-i\nu_2) e^{\mathcal B_1
+\overline{\mathcal B_2}}\frac{\partial\overline\Phi}
{\partial \overline z}\overline{b}\left (\frac{d_{2,+}
e^{2\tau i\psi({\widetilde x})}+d_{2,-}e^{-2\tau i\psi({\widetilde x})}}
{\vert \mbox{det}\, \psi ''({\widetilde x})\vert^\frac 12}
\right)d\sigma\nonumber\\
- \int_{\partial\Omega} (\nu_1-i\nu_2)e^{\mathcal B_1
+\overline{\mathcal B_2}}\frac{\partial\overline\Phi}
{\partial \overline z}d\left (\overline{\frac{b_{2,+}
e^{2\tau i\psi({\widetilde x})}+b_{2,-}e^{-2\tau i\psi({\widetilde x})}}
{\vert \mbox{det}\, \psi ''({\widetilde x})\vert^\frac 12}
}\right )d\sigma.
\end{eqnarray}

Denote
$$
U_1=-e^{\tau\overline\Phi}\mathcal R_{-\tau, A_1}\{e_1g_1\},\quad
U_2=-e^{\tau\Phi}\widetilde{\mathcal R}_{\tau,B_1}\{e_1g_2\}.
$$
A short calculation gives
\begin{equation}\label{OO}
{2}\frac{\partial U_1}{\partial\overline z}=(-e_1g_1+{A_1}
\mathcal R_{-\tau, A_1}\{e_1 g_1\})e^{\tau\overline\Phi}
\end{equation}
and
\begin{equation}\label{OOO}
{2}\frac{\partial U_2}{\partial z}=(-e_1g_2+{B_1}\widetilde{\mathcal R}
_{\tau, B_1} \{e_1g_2\})e^{\tau\Phi}.
\end{equation}
We have
\begin{eqnarray}\label{yap}
\frac{\partial }{\partial z}\mathcal R_{-\tau, A_1} \{e_1g_1\}
=\frac{\partial {\mathcal A_1}}{\partial z}\mathcal R_{-\tau, A_1}
\{e_1g_1\}+\tau\frac{\partial \Phi}{\partial z}\mathcal R_{-\tau, A_1}
\{e_1g_1\}+\mathcal R_{-\tau, A_1}\{\frac{\partial (e_1g_1)}
{\partial z}\}
\nonumber\\
=-\mathcal R_{-\tau, A_1} \{e_1g_1\frac{\partial \mathcal A_1}
{\partial z}\}
-\tau\mathcal R_{-\tau, A_1} \{\frac{\partial \Phi}{\partial z}
e_1g_1\}+\mathcal R_{-\tau, A_1}\{\frac{\partial
(e_1g_1)}{\partial z}\}\nonumber\\
+\tau\frac{e^{\mathcal A_1}}{2\pi}
e^{-\tau(\overline{\Phi}-\Phi)}\int_\Omega
\frac{\frac{\partial\Phi}{\partial
\zeta}(\zeta)-{\frac{\partial\Phi}{\partial z}(z)}}{\zeta-z}
(e_1g_1e^{-\mathcal A_1})(\xi_1,\xi_2)
e^{\tau(\overline{\Phi(\zeta)}-\Phi(\zeta))}d\xi_1d\xi_2\nonumber\\
+\frac{e^{\mathcal A_1}}{2\pi}
e^{-\tau(\overline{\Phi}-\Phi)}\int_\Omega
\frac{\frac{\partial\mathcal A_1}{\partial
\zeta}(\zeta,\overline \zeta)-{\frac{\partial\mathcal A_1}{\partial z}(z,\overline z)}}
{\zeta-z}
(e_1g_1e^{-\mathcal A_1})(\xi_1,\xi_2)
e^{\tau(\overline{\Phi(\zeta)}-\Phi(\zeta))}d\xi_1d\xi_2.
\end{eqnarray}
Let
$$
\frak G(x,g, \mathcal A, \tau)=-\frac{1}{2\pi}\int_{\Omega}
\frac{(\tau\frac{\partial\Phi( \zeta)}{\partial \zeta}
+\frac{\partial{\mathcal A}(\zeta,\overline \zeta)}{\partial \zeta})
-(\tau\frac{\partial \Phi( z)}{\partial z}+\frac{\partial{\mathcal A}
(z,\overline z)}{\partial z})}{\zeta-z}e_1 g e^{-\mathcal A}
e^{\tau(\overline\Phi-\Phi)}d\xi_1d\xi_2.
$$

We set
\begin{eqnarray*}
&&\frak G_1(x,\tau)=\frak G( x,g_1, \mathcal A_1, \tau), \quad
\frak G_2(x,\tau)=\overline{\frak G(x,\overline g_2,
\overline{\mathcal B_1},\tau)}, \\
&&\frak G_3(x,\tau)=\overline{\frak G(x,\overline{g_3},
\overline{\mathcal A_2},-\tau)},
\quad \frak G_4(x,\tau) = \frak G(x,g_4,
\mathcal B_2, -\tau).
\end{eqnarray*}

By (\ref{TT}), (\ref{TTT}), (\ref{yap}) and  Proposition
\ref{Proposition 3.22}, we obtain
\begin{equation}\label{001'}
\frac{\partial }{\partial z}\mathcal R_{-\tau, A_1} \{e_1g_1\}
=\mathcal R_{-\tau, A_1}\{\frac{\partial (e_1g_1)}{\partial z}\}
-e^{\mathcal A_1}
e^{-\tau(\overline{\Phi}-\Phi)}\frak G_1(\cdot,\tau)
+o_{L^2(\Omega)}(\frac
1\tau).
\end{equation}
Simple computations provide the formula
\begin{eqnarray}\label{yop}
\frac{\partial }{\partial \overline z}\widetilde{\mathcal R}
_{\tau, B_1} \{e_1g_2\}=\frac{\partial {\mathcal B_1}}{\partial
\overline z}\widetilde{\mathcal R}_{\tau, B_1} \{e_1g_2\}
+\tau\frac{\partial \overline\Phi}{\partial \overline z}
\widetilde{\mathcal R}_{\tau, B_1} \{e_1g_2\}+\widetilde{\mathcal R}
_{\tau, B_1}\{\frac{\partial (e_1g_2)}{\partial \overline z}\}
\nonumber\\
-\tau\widetilde{\mathcal R}_{\tau, B_1} \{\frac{\partial \overline\Phi}
{\partial
\overline z}e_1g_2\}-\widetilde{\mathcal R}_{\tau, B_1} \{\frac{\partial
\mathcal B_1}{\partial
\overline z}e_1g_2\}=\widetilde{\mathcal
R}_{\tau, B_1}\{\frac{\partial (e_1g_2)}{\partial \overline
z}\}\nonumber\\+\tau\frac{e^{\mathcal B_1}}{2\pi}
e^{\tau(\overline{\Phi}-\Phi)}\int_\Omega
\frac{\frac{\partial\overline\Phi}{\partial \overline
\zeta}(\overline\zeta)-{\frac{\partial\overline\Phi}
{\partial\overline  z}(\overline
z)}}{\overline \zeta-\overline z} (e_1g_2e^{-\mathcal
B_1})(\xi_1,\xi_2)
e^{\tau(\Phi(\zeta)-\overline{\Phi(\zeta)})}d\xi_1d\xi_2\nonumber\\
+\frac{e^{\mathcal B_1}}{2\pi}
e^{\tau(\overline{\Phi}-\Phi)}\int_\Omega
\frac{\frac{\partial\mathcal B_1}{\partial \overline
\zeta}(\zeta,\overline\zeta)-\frac{\partial\mathcal B_1}
{\partial\overline  z}(z,\overline
z)}{\overline \zeta-\overline z} (e_1g_2e^{-\mathcal
B_1})(\xi_1,\xi_2)
e^{\tau(\Phi(\zeta)-\overline{\Phi(\zeta)})}d\xi_1d\xi_2.
\end{eqnarray}
By (\ref{TT}), (\ref{TTT}), (\ref{yop}), Proposition
\ref{Proposition 3.22} we have
\begin{equation}\label{0011A}
\frac{\partial }{\partial \overline z}\widetilde{\mathcal R}_{\tau, B_1}
\{e_1g_2\}=\widetilde{\mathcal R}
_{\tau, B_1}\{\frac{\partial (e_1g_2)}{\partial \overline z}\}
- e^{\mathcal B_1}e^{\tau(\overline{\Phi}-\Phi)}\frak G_2(\cdot,\tau)
+o_{L^2(\Omega)}(\frac
1\tau).
\end{equation}
Denote
\begin{equation}\nonumber
 V_1=-e^{-\tau{\Phi}}\widetilde{\mathcal R}_{-\tau, -\overline A_2}
\{e_1g_3\},\, V_{2}=-e^{-\tau\overline{\Phi}}
\mathcal R_{\tau, -\overline B_2}\{e_1g_4\},
\mathcal P(x,D)=2(A_1-A_2)\frac{\partial}{\partial z}
+2(B_1-B_2)\frac{\partial}{\partial\overline z},
\end{equation}
\begin{align*}
&\mathcal Q_+=-({B}_1-{B}_2)A_1
-({A}_1-{A}_2)
B_2
-2\frac{\partial}{\partial z}({A}_1-{A}_2)+(q_1-q_2), \\
&\mathcal Q_-=-({A}_1-{A}_2)
B_1
-({B}_1-{B}_2)A_2
-2\frac{\partial}{\partial \overline{z}}
({B}_1-{B}_2)+(q_1-q_2).
\end{align*}
The following proposition is proved in Section 8.

\begin{proposition}\label{olimpus}
There exist  two numbers $\kappa,\kappa_0$ independent of $\tau$
such that the following asymptotic formula holds true:
\begin{eqnarray}\label{nika1}
(\mathcal P(x,D)(U_1+U_2),b_\tau e^{\mathcal{B}_2-\tau \Phi}
+ c_\tau e^{\mathcal{A}_2 -\tau \overline{\Phi}} )_{L^2(\Omega)}
+(\mathcal P(x,D)(a_\tau
e^{\mathcal A_1+\tau\Phi}+d_\tau e^{\mathcal B_1
+\tau\overline{\Phi}}),V_1+V_2)_{L^2(\Omega)}=\nonumber\\
\kappa+\frac{\kappa_0}{\tau}-2\int_{\partial\Omega}(\nu_1+i\nu_2)
e^{\mathcal A_1+\overline{\mathcal A_2}}
\overline{c_\tau(\overline z)}\frak G_1(x,\tau)d\sigma
+2\int_{\partial\Omega}(\nu_1-i\nu_2)e^{\mathcal B_1+\overline{B_2}}
d_\tau(\overline z)\overline{\frak G_4(x,\tau)}d\sigma\nonumber
\end{eqnarray}
\begin{equation}\label{vik11}+2\int_{\partial\Omega}a_\tau(z)
\overline{\frak G_3(x,\tau)}(\nu_1+i\nu_2)
e^{\mathcal A_1+\overline{\mathcal A_2}}d\sigma-
2\int_{\partial\Omega}(\nu_1-i\nu_2)e^{\mathcal B_1
+\overline{\mathcal B}_2} \overline{b_\tau(z)}
\frak G_2(x,\tau)d\sigma\nonumber
\end{equation}
\begin{eqnarray}
& +&\frac{ e^{-2i\tau\psi(\widetilde x)}}{\tau\vert\mbox{det}\,\psi^{''}(\widetilde x)\vert^\frac 12}\overline{\left\{\frac{\partial g_4(\widetilde x)}{\partial
z}\right\}}e^{-\overline{\mathcal B_2} (\widetilde x)}\int_{\partial\Omega} \frac{(\nu_1-i\nu_2)d e^{\mathcal B_1+\overline{\mathcal B_2}}}{\overline {\widetilde z}- \overline z}d\sigma \nonumber\\
&-&\frac{ e^{-2i\tau\psi(\widetilde x)}}{\tau\vert\mbox{det}\,\psi^{''}(\widetilde x)\vert^\frac 12}\frac{\partial g_1(\widetilde x)}{\partial z}e^{-\mathcal A_1(\widetilde x)}\int_{\partial\Omega} \frac{(\nu_1+i\nu_2)\overline{c} e^{\mathcal A_1+\overline{{\mathcal A}_2}}}{\widetilde z-z}d \sigma\nonumber\\
&+&\frac{ e^{2i\tau\psi(\widetilde x)}}{\tau\vert\mbox{det}\,\psi^{''}(\widetilde x)\vert^{\frac 12}}\overline{\frac{\partial g_3(\widetilde x)}{\partial \overline
z}}e^{-\overline{\mathcal A_2}(\widetilde x)}\int_{\partial\Omega}\frac{(\nu_1+i\nu_2)a e^{\mathcal A_1+\overline{\mathcal A_2}}}{ \widetilde z-z}d\sigma\nonumber\\
&-&\frac{e^{2i\tau\psi(\widetilde x)}}{\tau\vert \mbox{det}\,\psi^{''}(\widetilde x)\vert^\frac 12} \frac{\partial g_2(\widetilde x)}{\partial \overline z}e^{-\mathcal B_1(\widetilde x)}\int_{\partial\Omega}\frac{(\nu_1-i\nu_2)\overline b e^{\mathcal{B}_1+\overline{\mathcal B_2} }}{\overline{\widetilde z}-\overline z}d\sigma\nonumber\\
&-& \frac{2\pi (\mathcal Q_+a\overline be^{\mathcal A_1+\overline{\mathcal B_2}+2i\tau\psi}+\mathcal Q_-\overline c de^{\mathcal B_1+\overline{\mathcal A_2}-2i\tau\psi})(\widetilde x)}{\tau\vert\mbox{det}\,\psi^{''}(\widetilde x)\vert^{\frac 12}}.
\end{eqnarray}
\end{proposition}

By (\ref{OO}), (\ref{OOO}), (\ref{001'}), (\ref{0011A})
and Proposition \ref{opa} there exists a constant $\mathcal C_0$
independent of $\tau$ such that
\begin{eqnarray}\label{Avrora}
&(&{\mathcal P}(x,D)(U_1+U_2),V_1+V_2)_{L^2(\Omega)}
= ((A_1-A_2)(-2\left(
 \mathcal R_{-\tau, A_1}\{\frac{\partial
(e_1g_1)}{\partial z}\}\right. \nonumber\\
&-& \left. e^{\mathcal A_1} e^{-\tau(\overline \Phi-\Phi)}\frak G_1
+ o_{L^2(\Omega)}(\frac 1\tau)\right)e^{\tau\overline\Phi}                                                           \nonumber\\
&+& (-e_1g_2 + \frac{B_1e_1g_2}{2\tau\partial_z\Phi}
+ o_{L^2(\Omega)}(\frac 1\tau))e^{\tau\Phi}), V_1+V_2)_{L^2(\Omega)}
                                          \nonumber\\
&+&((B_1-B_2)(-e_1g_1+{A_1}\frac{e_1g_1}
{2\tau\overline{\partial_z\Phi}}
+o_{L^2(\Omega)}(\frac 1\tau))e^{\tau\overline\Phi},
V_1+V_2)_{L^2(\Omega)}\nonumber\\
&-&(2(B_1-B_2)(\widetilde{\mathcal R}
_{\tau, B_1}\{\frac{\partial (e_1g_2)}{\partial \overline z}\}-e^{\mathcal B_1}
e^{\tau(\overline \Phi-\Phi)}\frak G_2+o_{L^2(\Omega)}(\frac 1\tau))e^{\tau\Phi},
V_1+V_2)_{L^2(\Omega)}\nonumber\\
&=&\frac{\mathcal C_0}{\tau}+ o(\frac 1\tau)\quad
\mbox{as}\,\,\vert\tau\vert\rightarrow +\infty.\nonumber
\end{eqnarray}

Next we claim that
\begin{equation}\label{nokia}
({\mathcal
P}(x,D)(u_0e^{\tau\varphi}), v)_{L^2(\Omega)}
=o(\frac{1}{\tau})\quad \mbox{as}\,\,\vert\tau\vert\rightarrow
+\infty,
\end{equation}
and
\begin{equation}\label{nokia1}
({\mathcal
P}(x,D)u, v_0e^{-\tau\varphi})_{L^2(\Omega)}=o(\frac{1}{\tau})
\quad \mbox{as}\,\,\vert\tau\vert\rightarrow +\infty.
\end{equation}
Let us first  prove (\ref{nokia1}).
By (\ref{NASA1}), (\ref{nona1}), (\ref{001'}), (\ref{0011A}),
Proposition 3.2 and Proposition \ref{Proposition 3.22} we have
\begin{equation}\label{nokia10}
({\mathcal P}(x,D)u, v_0e^{-\tau\varphi})_{L^2(\Omega)}=({\mathcal
P}(x,D)\mathcal U,v_0e^{-\tau\varphi})_{L^2(\Omega)}+o(\frac{1}{\tau})
\quad\mbox{as}\,\vert\tau\vert\rightarrow +\infty.
\end{equation}
We remind that the function $\mathcal U$ and $\mathcal V$ are
defined by (\ref{narkotic}).
By (\ref{nona1}) we obtain from (\ref{nokia10})
\begin{equation}\label{nokia100}
({\mathcal
P}(x,D)u, v_0e^{-\tau\varphi})_{L^2(\Omega)}
=\tau \int_\Omega 2\chi(\frac{\partial \Phi}{\partial z}(A_1-A_2)
a e^{\mathcal A_1+i\tau\psi}+\frac{\partial \overline\Phi}{\partial
\overline z}(B_1-B_2)be^{\mathcal B_1-i\tau\psi})\overline{v_0}  dx
+o(\frac{1}{\tau})
\end{equation}
as $\vert\tau\vert\rightarrow +\infty$.
Here $\chi\in C^\infty_0(\overline \Omega)$ is a function such that
$\chi\equiv 1$ in some neighborhood of $\mbox{supp}\, e_2$ and
$\mathcal H\setminus\partial\Omega\subset supp \, e_2.$

By (\ref{(3.18)}) and (\ref{NINA}) the functions $v_{0,+}=e^{-i\tau\psi}
\overline{v_0}$ and
$v_{0,-}=e^{i\tau\psi}\overline{v_0}$ satisfy
$e^{\tau\Phi}\overline{L_2(x,D)^*(e^{-\tau\overline\Phi}v_{0,+})}
=\overline h_2e^{i\tau\psi}$ and
$e^{\tau\overline\Phi}\overline{L_2(x,D)^*(e^{-\tau\Phi}v_{0,-})}
= \overline h_2e^{-i\tau\psi}$.  More explicitly, there exist two first-order
operators $\mathcal {P}_k(x,D)$ such that
$$
e^{\tau\Phi}\overline{L_2(x,D)^*(e^{-\tau\overline\Phi}v_{0,+})}
=\Delta\overline{v_{0,+}}-2\tau\frac{\partial \Phi}{\partial z}
(2\frac{\partial \overline{v_{0,+}}}{\partial \overline z}-A_2
\overline v_{0,+})
+\mathcal P_1(x,D)\overline{v_{0,+}}=o_{L^2(\Omega)}(\frac 1\tau)
\quad\mbox{as}\,\vert\tau\vert\rightarrow +\infty
$$

and
$$
e^{\tau\overline\Phi}\overline{L_2(x,D)^*(e^{-\tau\Phi}v_{0,-})}
=\Delta\overline{v_{0,-}}-2\tau\frac{\partial \overline\Phi}
{\partial \overline z}(2\frac{\partial \overline{v_{0,-}}}{\partial z}
-B_2\overline v_{0,-})+\mathcal P_2(x,D)\overline{v_{0,-}}
=o_{L^2(\Omega)}(\frac 1\tau)\quad\mbox{as}\,\vert\tau\vert
\rightarrow +\infty.
$$
In the above equalities we used (\ref{(3.188)}) and (\ref{osel}).

Let $\chi_1\in C^\infty_0(\overline \Omega)$ be a function such
that $\chi_1\equiv 1$ on $\supp\, \chi$ and $g\in C^2(\overline\Omega).$
Taking the scalar product of the first equation with $\chi_1 g$ we obtain
$$
 \int_\Omega 2\tau\frac{\partial \Phi}{\partial z}
\overline{v_{0,+}}\chi_1 (2\frac{\partial}{\partial \overline z}
+A_2)gdx=o(\frac 1\tau)-\int_\Omega(
 \overline{v_{0,+}}(\Delta +\mathcal P_1(x,D)^*)(\chi_1g)
+2\tau\overline{v_{0,+}}\frac{\partial \Phi}{\partial z}g
(\frac{\partial}{\partial \overline z}+A_2)\chi_1)dx.
$$
By (\ref{nona1}) we have
\begin{equation}\label{opaal}
\int_\Omega \tau\frac{\partial \Phi}{\partial z}
\overline{v_{0,+}}\chi_1 (2\frac{\partial}{\partial \overline z}+A_2)
gdx=o(\frac 1\tau)\quad \mbox{as}\,\,\vert\tau\vert\rightarrow
+\infty.
\end{equation}

Taking the scalar product of the second equation with $ \chi_1 g$
where $g\in C^2(\overline\Omega)$ we have
$$
 \int_\Omega 2\tau\frac{\partial \Phi}{\partial \overline z}
\overline{v_{0,-}}\chi_1 (2\frac{\partial}{\partial z}+B_2)
gdx=o(\frac 1\tau)
-\int_\Omega\left (
 \overline{v_{0,-}}(\Delta +\mathcal P_2(x,D)^*)(\chi_1g)
+2\tau\overline{v_{0,-}}g\frac{\partial \Phi}
{\partial \overline z}(2\frac{\partial}{\partial z}
+B_2)\chi_1\right )dx.
$$
By (\ref{nona1}) we get
\begin{equation}\label{opaNA}
 \int_\Omega 2\tau\frac{\partial \Phi}{\partial \overline z}
\overline{v_{0,-}}\chi_1(2\frac{\partial}{\partial z}+B_2) gdx
=o(\frac 1\tau)\quad \mbox{as}\,\,\vert\tau\vert\rightarrow +\infty.
 \end{equation}

Taking $g$ such that $(2\frac{\partial}{\partial \overline z}+A_2)g
=(A_1-A_2)e^{\mathcal A_1}a(z) $ in (\ref{opaal}) and $g$ such that
$(2\frac{\partial}{\partial z}+B_2)g=(B_1-B_2)b(\overline z)
e^{\mathcal B_1} $ in (\ref{opaNA})  from (\ref{nokia10})
we obtain (\ref{nokia1}).

In order to prove (\ref{nokia}) we observe that
\begin{eqnarray}
(\mathcal P(x,D)(u_0e^{\tau\varphi}),v)_{L^2(\Omega)}
=(\mathcal P(x,D)(u_0e^{\tau\varphi}),\mathcal{V})_{L^2(\Omega)}
+o(\frac{1}{\tau})=(\mathcal P(x,D)(u_0e^{\tau\varphi}),
\chi\mathcal{V})_{L^2(\Omega)}+o(\frac{1}{\tau})\nonumber\\
=(u_0e^{\tau\varphi},\mathcal P(x,D)^*(\chi\mathcal{V}))
_{L^2(\Omega)}+o(\frac{1}{\tau})
\quad\mbox{as}\,\,\vert\tau\vert\rightarrow +\infty.
\end{eqnarray}
Then we can finish the proof of (\ref{nokia}) using arguments
similar to
(\ref{nokia100})-(\ref{opaal}).

Denote $ \mathcal
M_1=\frac{1}{4{\partial_{\overline z}\overline\Phi}}L_1(x,D)(\frac{e_2g_1}
{\overline{\partial_z \Phi}}),\,\, \mathcal
M_2=\frac{1}{4{{\partial_z\Phi}}}L_1(x,D)(\frac{e_2g_2}{{\partial_z
\Phi}}), $ $\mathcal M_3=-\frac{1}{4 \partial_z
\Phi}L_2(x,D)^* \biggl(\frac{e_2g_3}{\partial_z
\Phi}\biggl),$\newline $ \mathcal M_4=-\frac{1}{4
\overline{\partial_z \Phi}} L_2(x,D)^* \biggl(\frac{e_2g_4}{
{\partial_{\overline z} \overline\Phi}}\biggl) .$
Then there exists a constant $\mathcal C$ independent of $\tau$ such that
\begin{equation}\label{ZONA1}
(\mathcal P(x,D) (e^{\tau\Phi}\frac{\mathcal M_1}{\tau^2}
+ e^{\tau\overline\Phi}\frac{\mathcal M_2}{\tau^2}), v)
_{L^2(\Omega)}
+ (\mathcal P(x,D)u, e^{-\tau\Phi}\frac{\mathcal M_3}{\tau^2}
+ e^{-\tau\overline\Phi}\frac{\mathcal M_4}{\tau^2})
_{L^2(\Omega)}=\frac{\mathcal C}{\tau}+o(\frac 1\tau)\quad\mbox{as}
\,\vert\tau\vert\rightarrow +\infty.
\end{equation}

Denote
$\mathcal X_1=-\frac{e_2g_1}{2{\partial_{\overline z}\overline\Phi}},\,\,
\mathcal X_2=-\frac{e_2g_2}{2{{\partial_z\Phi}}}$,
$\mathcal X_3=\frac{e_2g_3}{2 \partial_z \Phi},
\mathcal X_4=\frac{e_2g_4}{2  {\partial_{\overline z} \overline \Phi}}.
$
Then, using the stationary phase argument we conclude
\begin{eqnarray}\label{ZONA}
(\mathcal P(x,D) (e^{\tau\Phi}\frac{\mathcal X_2}{\tau}
+ e^{\tau\overline\Phi}\frac{\mathcal X_1}{\tau}), v)
_{L^2(\Omega)}
+ (\mathcal P(x,D) u, e^{-\tau\Phi}\frac{\mathcal X_3}{\tau}
+e^{-\tau\overline\Phi}\frac{\mathcal X_4}{\tau})_{L^2(\Omega)}=
                                             \\
\mathcal C_0+\frac{\mathcal C_{-1}}\tau+\frac 1\tau\int_{\widetilde\Gamma}
((A_1-A_2)\frac{\partial \Phi}{\partial  z}
\mathcal X_2\overline be^{\overline{\mathcal B_2}}e^{2\tau i\psi}
-(B_1-B_2)\frac{\partial \overline \Phi}{\partial \overline z}
\mathcal X_1 be^{\overline{\mathcal A_2}}e^{-2\tau i\psi})
\frac{(\nabla\psi,\nu)}{2i\vert \nabla\psi\vert^2}d\sigma\nonumber\\
+\frac 1\tau\int_{\widetilde\Gamma}((A_1-A_2)\frac{\partial \Phi}
{\partial  z}\overline{\mathcal X_3}
ae^{{\mathcal A_1}}e^{2\tau i\psi}-(B_1-B_2)
\frac{\partial \overline \Phi}{\partial \overline z}
\overline{\mathcal X_4} \overline ae^{{\mathcal B_1}}
e^{-2\tau i\psi})\frac{(\nabla\psi,\nu)}{2i\vert \nabla\psi\vert^2}
d\sigma
+o(\frac 1\tau)\quad\mbox{as}\,\vert\tau\vert\rightarrow +\infty.
\nonumber
\end{eqnarray}

Next we show  that
\begin{proposition}\label{zoma}
Under the conditions of Theorem \ref{main}
\begin{equation}\label{POPAP}A_1=A_2,\quad B_1=B_2 \quad\mbox{on}
\,\,\widetilde \Gamma\end{equation}
and for any function $\Phi$ satisfying (\ref{zzz}), (\ref{mika})
and for any functions $a,b,c,d$  satisfying (\ref{iopa}),
(\ref{ikaa}), (\ref{PI5}), (\ref{ikaaa})  we have
\begin{equation}\label{inna}
\frak I(\Phi,a,b,c,d)=\int_{\widetilde\Gamma}\left\{(\nu_1+i\nu_2)
\frac{\partial \Phi}{\partial z} a(z)\overline{c(\overline z)}
e^{\mathcal A_1+\overline{\mathcal A_2}}
+(\nu_1-i\nu_2)\frac{\partial \overline\Phi}{\partial\overline z}
d(\overline z)\overline{b(z)} e^{\mathcal B_1+\overline{ \mathcal B_2}}
\right\} d\sigma=0.
\end{equation}
\end{proposition}
\begin{proof}
Let $\widehat x$ be an arbitrary point from $Int\,\widetilde \Gamma$ and
$\Gamma_*$ be an arc containing $\widehat x$  such that
$\Gamma_*\subset \subset \widetilde\Gamma.$  By Proposition
\ref{Proposition -2} there exists a weight function $\Phi$
satisfying (\ref{lana0}) and (\ref{lana1}).
Then the boundary integrals in (\ref{MAHA}), (\ref{MAHA1}) have
the following asymptotic:
\begin{eqnarray}\label{ZINA}
&&\int_{\widetilde \Gamma}(B_1-B_2)d_\tau(\overline z)
\overline{c_\tau(\overline z)}e^{\mathcal B_1
+ \overline{\mathcal A_2}-2i\tau\psi}d\sigma
+ \int_{\widetilde\Gamma}
(A_1-A_2)(\nu_1-i\nu_2)a_\tau(z)\overline{b_\tau(z)}e^{\mathcal
A_1+\overline{\mathcal B_2}+2i\tau\psi}d\sigma\nonumber\\
&= &\sum_{x\in\mathcal G\setminus\{x_-,x_+\}}
\Biggl\{\left( \frac{2\pi}{i\frac{\partial^2\psi}
{\partial\vec \tau^2}(x)}\right )^\frac 12(\overline cd(B_1-B_2))
(x)\frac{e^{(\mathcal B_1+\overline{\mathcal A_2}-2\tau
i\psi)(x)}}{\root\of{\tau}}  \nonumber\\
&+& \left (\frac{2\pi}{-i\frac{\partial^2\psi}
{\partial\vec \tau^2}(x)}\right )^\frac 12 (a\overline b(A_1-A_2))( x)
\frac{e^{(\mathcal A_1+\overline{\mathcal B_2}+2\tau
i\psi)(x)}}{\root\of{\tau}}\Biggr\}\nonumber\\
&&\qquad \qquad \qquad
+ O(\frac{1}{\tau})\quad \mbox{as}\,\,\vert\tau\vert\rightarrow
+\infty.
\end{eqnarray}
We remind that the set $\mathcal G$ is introduced in (\ref{lana0}). Moreover,
in order to avoid the contribution from the points
$x_\pm$ functions $a,b$ are chosen in such a way that
\begin{equation}\label{elka}
\frac{\partial^{\vert\beta\vert}a}{\partial x_1^{\beta_1}
\partial x_2^{\beta_2}} (x_\pm)=\frac{\partial^{\vert\beta\vert}b}
{\partial x_1^{\beta_1}\partial x_2^{\beta_2}} (x_\pm)
=\frac{\partial^{\vert\beta\vert}c}{\partial x_1^{\beta_1}
\partial x_2^{\beta_2}} (x_\pm)=\frac{\partial^{\vert\beta\vert}d}
{\partial x_1^{\beta_1}\partial x_2^{\beta_2}} (x_\pm)=0\quad
\forall \vert \beta\vert\in\{0,\dots, 5\}.
\end{equation}
Let $\widetilde \chi_1\in C^\infty(\partial\Omega)$  be a function such that
it is equal $1$ near points $x_\pm$ and has support located in a
small neighborhood of these points. Then
$$\int_{\Gamma_*}\widetilde \chi_1(B_1-B_2)d_\tau\overline{c_\tau}
e^{\mathcal B_1+\overline{\mathcal A_2}-2i\tau\psi}d\sigma+
\int_{
\Gamma_*}\widetilde \chi_1(A_1-A_2)(\nu_1-i\nu_2)a_\tau\overline{b_\tau}
e^{\mathcal
A_1+\overline{\mathcal B_2}+2i\tau\psi}d\sigma=
$$
$$
\int_{\Gamma_*}\frac{\widetilde \chi_1(B_1-B_2)d_\tau\overline{c_\tau}
e^{\mathcal B_1+\overline{\mathcal A_2}}}{-2i\tau\frac{\partial\psi}
{\partial \vec\tau}}\frac{\partial e^{-2i\tau\psi}}{\partial \vec \tau}
d\sigma+\int_{
\Gamma_*}\frac{\widetilde \chi_1(A_1-A_2)(\nu_1-i\nu_2)a_\tau\overline{b_\tau}
e^{\mathcal
A_1+\overline{\mathcal B_2}}}{2i\tau\frac{\partial \psi}
{\partial\vec \tau}}\frac{\partial e^{2i\tau\psi}}{\partial \vec \tau}
d\sigma=
$$
$$
\int_{\Gamma_*}\frac{\partial}{\partial \vec \tau}
\left(\frac{\widetilde \chi_1(B_1-B_2)d_\tau\overline{c_\tau }
e^{\mathcal B_1+\overline{\mathcal A_2}}}{2i\tau
\frac{\partial\psi}{\partial \vec\tau}}\right)e^{-2i\tau\psi}d\sigma
$$
$$
-\int_{
\Gamma_*}\frac{\partial}{\partial \vec \tau}
\left(\frac{\widetilde \chi_1(A_1-A_2)(\nu_1-i\nu_2)
a_\tau\overline{b_\tau}e^{\mathcal
A_1+\overline{\mathcal B_2}}}{2i\tau\frac{\partial \psi}
{\partial\vec \tau}}\right)e^{2i\tau\psi}d\sigma=O(\frac 1\tau).
$$
In order to obtain the last equality we used that by (\ref{elka}) and
(\ref{zopa}) the functions
$$
\frac{\partial}{\partial \vec \tau}
\left(\frac{\widetilde \chi_1(B_1-B_2)d_\tau \overline{c_\tau }
e^{\mathcal B_1+\overline{\mathcal A_2}}}{2i\frac{\partial\psi}
{\partial \vec\tau}}\right), \quad
\frac{\partial}{\partial \vec \tau}\left(
\frac{\widetilde \chi_1(A_1-A_2)(\nu_1-i\nu_2)a_\tau\overline{b_\tau}e^{\mathcal
A_1+\overline{\mathcal B_2}}}{2i\frac{\partial \psi}{\partial\vec \tau}}
\right)
$$
are bounded.
By (\ref{nonsence1}), (\ref{MAHA})-(\ref{vik}),
(\ref{vik11})-(\ref{nokia1}), (\ref{ZONA1}) - (\ref{ZONA})
and (\ref{ZINA}),
we can represent the right-hand side of (\ref{ippolit}) as
\begin{eqnarray}
O(\frac{1}{\tau})=\tau F_1+F_0
+ \sum_{x\in\mathcal G\setminus\{x_-,x_+\}}\left(
\left(\frac{2\pi}{i\frac{\partial^2\psi}
{\partial\vec \tau^2}(x)}\right)^{\frac 12}
(\overline c d(B_1-B_2))( x)\frac{e^{(\mathcal B_1+\overline{\mathcal A_2}-2\tau i\psi)(x)}}
{\root\of{\tau}}\right.
\nonumber\\
+ \left(\frac{2\pi}{-i\frac{\partial^2\psi}
{\partial\vec \tau^2}(x)}\right)^\frac 12
\left.(a\overline b(A_1-A_2))( x)\frac{e^{(\mathcal B_1+\overline{\mathcal A_2}+2\tau i\psi)(x)}}{\root\of{\tau}}
\right).\nonumber
\end{eqnarray}
Taking into account that $F_1$ is equal to the left-hand side of (\ref{inna})
we obtain the equality (\ref{inna}).
Using (\ref{lana1}) and applying Bohr's theorem (e.g., \cite{BS},
p.393), we obtain (\ref{POPAP}).
\end{proof}

Thanks to (\ref{nonsence1}), (\ref{MAHA})-(\ref{vik}),
(\ref{vik11})-(\ref{nokia1}), (\ref{ZONA1})-(\ref{POPAP}),
(\ref{ZINA}) we can write down the right-hand side
of (\ref{ippolit}) as
\begin{eqnarray}\label{finally}
I_0+I_1&+&I_2 = \sum_{k=1}^3\tau^{2-k}(\kappa_k+ \widetilde\kappa_k)
+ \kappa\nonumber\\
&+&\int_{\Gamma_0}(A_1-A_2)(\nu_1-i\nu_2)a_\tau
\overline{b_\tau}e^{\mathcal A_1+\overline{\mathcal B_2}}d\sigma
+\int_{\Gamma_0}(B_1-B_2)(\nu_1+i\nu_2)d_\tau\overline{c_\tau}
e^{\mathcal B_1+\overline{\mathcal A_2}}d\sigma\nonumber\\
&-&\frac {1}{2\tau i}\int_{\Gamma_0}{\mathcal Q}_+\overline{a}
be^{\mathcal B_1+\overline{\mathcal A_2}}\frac{(\nabla\psi,\nu)}
{\vert\nabla\psi\vert^2}d\sigma
-\frac {1}{2\tau i}\int_{\Gamma_0}{\mathcal Q}_-a\overline{b} e^{\mathcal A_1
+\overline{\mathcal B_2}}\frac{(\nabla\psi,\nu)}{\vert\nabla\psi\vert^2}
d\sigma
\nonumber\\
&-&\pi \frac{({\mathcal Q}_+a\overline b)(\widetilde x)
e^{(\mathcal A_1+\overline{\mathcal B_2}+2\tau i\psi)
(\widetilde x)}+({\mathcal Q}_-
 d\overline c)(\widetilde x)e^{({\mathcal B_1}
+\overline{\mathcal A_2}-2i\tau\psi)(\widetilde x)}}
{\tau\vert \mbox{det}\, \psi''(\widetilde x)\vert^\frac 12}
                                  \nonumber \\
&+& \frac{1}{\tau } (\mathcal I_1(\partial\Omega)
+\mathcal I_2(\partial\Omega))
- 2\int_{\partial\Omega}(\nu_1+i\nu_2)
e^{\mathcal A_1+\overline{\mathcal A_2}}
\overline{c_\tau(\overline z)}
\frak G_1(x,\tau)d\sigma\nonumber\\
&+& 2\int_{\partial\Omega}(\nu_1-i\nu_2)
e^{\mathcal B_1+\overline{\mathcal B_2}}
d_\tau(\overline z)\overline{\frak G_4(x,\tau)}d\sigma
+2\int_{\partial\Omega}(\nu_1+i\nu_2)
e^{\mathcal A_1+\overline{\mathcal A_2}}
a_\tau(z)\overline{\frak G_3(x,\tau)}d\sigma        \nonumber\\
&&-2\int_{\partial\Omega}(\nu_1-i\nu_2)
e^{\mathcal B_1+\overline{\mathcal B}_2} \overline{b_\tau(z)}
\frak G_2(x,\tau)d\sigma\nonumber\\
& +&\frac{ e^{-2i\tau\psi(\widetilde x)}}{\tau\vert
\mbox{det}\,\psi^{''}(\widetilde x)\vert^\frac 12}
\overline{\left\{\frac{\partial g_4(\widetilde x)}
{\partial z}\right\}}e^{-\overline{\mathcal B_2} (\widetilde x)}
\int_{\partial\Omega} \frac{(\nu_1-i\nu_2)d
e^{\mathcal B_1+\overline{\mathcal B_2}}}
{\overline {\widetilde z}- \overline z}
d\sigma \nonumber\\
&-&\frac{ e^{-2i\tau\psi(\widetilde x)}}{\tau\vert\mbox{det}\,\psi^{''}(\widetilde x)\vert^\frac 12}\frac{\partial g_1(\widetilde x)}{\partial z}e^{-\mathcal A_1(\widetilde x)}\int_{\partial\Omega} \frac{(\nu_1+i\nu_2)\overline{c} e^{\mathcal A_1+\overline{{\mathcal A}_2}}}{\widetilde z-z}d \sigma\nonumber\\
&+&\frac{ e^{2i\tau\psi(\widetilde x)}}{\tau\vert\mbox{det}\,\psi^{''}(\widetilde x)\vert^{\frac 12}}\overline{\frac{\partial g_3(\widetilde x)}{\partial \overline
z}}e^{-\overline{\mathcal A_2}(\widetilde x)}\int_{\partial\Omega}\frac{(\nu_1+i\nu_2)a e^{\mathcal A_1+\overline{\mathcal A_2}}}{ \widetilde z-z}d\sigma\nonumber\\
&-&\frac{e^{2i\tau\psi(\widetilde x)}}{\tau\vert \mbox{det}\,\psi^{''}(\widetilde x)\vert^\frac 12} \frac{\partial g_2(\widetilde x)}{\partial \overline z}e^{-\mathcal B_1(\widetilde x)}\int_{\partial\Omega}\frac{(\nu_1-i\nu_2)\overline b e^{\mathcal{B}_1+\overline{\mathcal B_2} }}{\overline{\widetilde z}-\overline z}d\sigma
+ o(\frac 1\tau) \quad \mbox{as}\,\,\vert \tau\vert\rightarrow +\infty.
\end{eqnarray}
We note that $\kappa_k$ and $\widetilde{\kappa}_k$ denote generic
constants which are independent of $\tau$.
In order to transform some terms in the above equality, we need
the following proposition:
\begin{proposition}\label{balda} There exist a holomorphic function
$\Theta\in H^\frac 12(\Omega)$ and an antiholomorphic function
$\widetilde \Theta\in H^\frac 12(\Omega)$ such that
\begin{equation}\label{volt}
\Theta\vert_{\widetilde \Gamma}=e^{\mathcal A_1+\overline{\mathcal A_2}},
\quad \widetilde\Theta\vert_{\widetilde \Gamma}=e^{\mathcal B_1
+\overline{\mathcal B_2}}
\end{equation}
and
\begin{equation}\label{volt1}
e^{\mathcal B_1+\overline{\mathcal B_2}}\Theta-e^{\mathcal A_1
+\overline{\mathcal A_2}}\widetilde \Theta=0 \quad \mbox{on}\,\,\Gamma_0.
\end{equation}
\end{proposition}
\begin{proof}
Consider the extremal problem:
\begin{equation}\label{AX}
J(\Psi,\widetilde \Psi)=\Vert  e^{\mathcal A_1+\overline{\mathcal A_2}}
\frac{\partial\Phi}{\partial z} a\overline c-\Psi\Vert^2_{L^2(\widetilde \Gamma)}
+\Vert  e^{\mathcal B_1+\overline{\mathcal B_2}}
\frac{\partial \overline \Phi}{\partial \overline z}
\overline b d-{\widetilde\Psi}\Vert^2_{L^2(\widetilde \Gamma)}\rightarrow \inf,
\end{equation}
\begin{equation}\label{voron}
\frac{\partial\Psi}{\partial\overline z}=0\quad\mbox{in}\,\Omega,
\quad \frac{\partial\widetilde \Psi}{\partial z}=0\quad\mbox{in}\,\Omega,
\quad ((\nu_1+i\nu_2)\Psi+(\nu_1-i\nu_2)\widetilde \Psi)
\vert_{\Gamma_0}=0.
\end{equation}
Here the functions $a, b,c,d$ satisfy  (\ref{iopa}), (\ref{ikaa}),
(\ref{PI5}) and (\ref{ikaaa}).
Denote the unique solution to this extremal problem (\ref{AX}),
(\ref{voron}) as $(\widehat \Psi,\widehat{ \widetilde \Psi})$. Applying
Lagrange's principle we obtain
\begin{equation}\label{ipoa}
\mbox{Re} (e^{\mathcal A_1+\overline{\mathcal A_2}}\frac{\partial\Phi}
{\partial z} a\overline c-\widehat \Psi,\delta)_{L^2(\widetilde\Gamma)}+\mbox{Re}
(e^{\mathcal B_1+\overline{\mathcal B_2}}\frac{\partial\overline \Phi}
{\partial\overline z} \overline b d-\widehat {\widetilde \Psi},\widetilde \delta)
_{L^2(\widetilde\Gamma)}=0\quad
\end{equation}
for any $\delta,\widetilde\delta$  from $ H^\frac 12(\Omega)$ such that
$$
\,\,\frac{\partial\delta}{\partial\overline z} =0 \quad\mbox{in}\,\Omega,
\quad \frac{\partial\widetilde \delta}{\partial z}= 0\quad\mbox{in}
\,\Omega,\quad (\nu_1+i\nu_2)\delta\vert_{\Gamma_0}=-(\nu_1-i\nu_2)
\widetilde \delta\vert_{\Gamma_0}
$$
and there exist two functions $P,\widetilde P\in H^\frac 12(\Omega)$ such that
\begin{equation}\label{020}
\frac{\partial P}{\partial \overline z}=0\quad\mbox{in}\,\,\Omega,\quad
\frac{\partial\widetilde P}{\partial z}=0\quad\mbox{in}\,\,\Omega,
\end{equation}
\begin{equation}\label{021}
(\nu_1+i\nu_2) P=\overline{e^{\mathcal A_1+\overline{\mathcal A_2}}
\frac{\partial \Phi}{\partial z} a\overline c-\widehat \Psi}\quad\mbox{on}
\,\widetilde \Gamma,\quad (\nu_1-i\nu_2) \widetilde P
=\overline{e^{\mathcal B_1+\overline{\mathcal B_2}}
\frac{\partial\overline \Phi}{\partial\overline z} \overline b d
-\widehat {\widetilde\Psi}}\quad\mbox{on}\,\widetilde \Gamma,
\end{equation}
\begin{equation}\label{ona}
(P-\widetilde P)\vert_{\Gamma_0}=0.
\end{equation}
Denote $\Psi_0(z)=\frac{1}{2i}(P(z)-\overline{\widetilde P(\overline z)}),
\Phi_0(z)=\frac 12(P(z)+\overline {\widetilde P(\overline z)}).$ By (\ref{ona})
\begin{equation}\label{chech}
\mbox{Im}\,\Psi_0\vert_{\Gamma_0}=\mbox{Im}\,\Phi_0\vert_{\Gamma_0}=0.
\end{equation}
Hence
\begin{equation}\label{vodka}
P=(\Phi_0+i\Psi_0), \quad\overline {\widetilde P}=(\Phi_0-i\Psi_0).
\end{equation}
From (\ref{ipoa}) taking $\delta=\widehat \Psi$ and $\widetilde \delta
=\widehat{\widetilde \Psi}$ we have
\begin{equation}\label{Pona}
\mbox{Re} (e^{\mathcal A_1+\overline{\mathcal A_2}}
\frac{\partial \Phi}{\partial z} a\overline c-\widehat \Psi,{\widehat \Psi})
_{L^2(\widetilde\Gamma)}
+\mbox{Re} (e^{\mathcal B_1+\overline{\mathcal B_2}}
\frac{\partial \overline \Phi}{\partial \overline z}
\overline b d-\widehat {\widetilde \Psi},\widehat {\widetilde \Psi})_{L^2(\widetilde\Gamma)}
=0.
\end{equation}
By (\ref{020}), (\ref{021}) and (\ref{vodka}), we have
\begin{eqnarray}
H_1=\mbox{Re} (e^{\mathcal A_1+\overline{\mathcal A_2}}
\frac{\partial \Phi}{\partial z} a\overline c-\widehat \Psi,e^{\mathcal A_1
+\overline{\mathcal A_2}}\frac{\partial \Phi}{\partial z} a\overline c)
_{L^2(\widetilde\Gamma)}
+\mbox{Re} (e^{\mathcal B_1+\overline{\mathcal B_2}}
\frac{\partial\overline \Phi}{\partial\overline z}\overline b d
-\widehat {\widetilde \Psi},e^{\mathcal B_1+\overline{\mathcal B_2}}
\frac{\partial\overline \Phi}{\partial\overline z} \overline b d)
_{L^2(\widetilde\Gamma)}\nonumber\\
=\mbox{Re} ((\nu_1-i\nu_2)\overline P,e^{\mathcal A_1
+\overline{\mathcal A_2}}
\frac{\partial \Phi}{\partial z}  a\overline c)_{L^2(\widetilde\Gamma)}
+\mbox{Re} ((\nu_1+i\nu_2)\overline{\widetilde P},
e^{\mathcal B_1+\overline{\mathcal B_2}}
\frac{\partial\overline \Phi}{\partial\overline z} \overline b d)
_{L^2(\widetilde\Gamma)}
                                   =\nonumber\\
2\mbox{Re} ((\nu_1-i\nu_2)\overline{(\Phi_0+i\Psi_0)},
e^{\mathcal A_1+\overline{\mathcal A_2}}\frac{\partial \Phi}
{\partial z}  a\overline c)_{L^2(\widetilde\Gamma)}+2\mbox{Re}
((\nu_1+i\nu_2){(\Phi_0-i\Psi_0)},
e^{\mathcal B_1+\overline{\mathcal B_2}}
\frac{\partial\overline \Phi}{\partial\overline z}\overline b d)
_{L^2(\widetilde\Gamma)}.
                                          \nonumber
\end{eqnarray}
We can rewrite
$$
2\mbox{Re} ((\nu_1-i\nu_2)\overline{\Phi_0},
e^{\mathcal A_1+\overline{\mathcal A_2}}\frac{\partial \Phi}
{\partial z}  a\overline c)_{L^2(\widetilde\Gamma)}+2\mbox{Re}
((\nu_1+i\nu_2) {\Phi_0},e^{\mathcal B_1+\overline{\mathcal B_2}}
\frac{\partial\overline \Phi}{\partial\overline z} \overline b d)
_{L^2(\widetilde\Gamma)}=
$$
\begin{equation}\label{nona}
\frak I(\Phi,\Phi_0 a,b,c,\Phi_0 d)
+\overline{\frak I(\Phi,\Phi_0 a,b,c,\Phi_0 d)}
\end{equation}
and
\begin{eqnarray}\label{nona1A}
2\mbox{Re} ((\nu_1-i\nu_2)\overline{(i\Psi_0)},
e^{\mathcal A_1+\overline{\mathcal A_2}}
\frac{\partial \Phi}{\partial z}  a\overline c)_{L^2(\widetilde\Gamma)}
+2\mbox{Re} ((\nu_1+i\nu_2) {(-i \Psi_0)},e^{\mathcal B_1
+\overline{\mathcal B_2}}\frac{\partial\overline \Phi}
{\partial\overline z} \overline b d)_{L^2(\widetilde\Gamma)}=\nonumber\\
-2\mbox{Im} ((\nu_1-i\nu_2)\overline a c\overline{ \Psi_0},
e^{\mathcal A_1+\overline{\mathcal A_2}}\frac{\partial \Phi}{\partial z}
)_{L^2(\widetilde\Gamma)}-2\mbox{Im}
((\nu_1+i\nu_2)b\overline d {\Psi_0},
e^{\mathcal B_1+\overline{\mathcal B_2}}
\frac{\partial\overline \Phi}{\partial\overline z})_{L^2(\widetilde\Gamma)}
                           =\nonumber\\
-\frac 1i\int_{\widetilde\Gamma}((\nu_1-i\nu_2)\overline a c
\overline{ \Psi_0}\frac{\partial\overline \Phi}{\partial\overline z}
e^{\overline{\mathcal A_1}-{\mathcal A_2}}-(\nu_1+i\nu_2) a
\overline c{ \Psi_0} \frac{\partial \Phi}{\partial z}
e^{\mathcal A_1+\overline{\mathcal A_2}})d\sigma\nonumber\\
-\frac 1i \int_{\widetilde\Gamma}((\nu_1+i\nu_2)b\overline d
{\Psi_0}\frac{\partial \Phi}{\partial z}  e^{\overline{\mathcal B_1}
+{\mathcal B_2}}-(\nu_1-i\nu_2)\overline bd \overline {\Psi_0}
\frac{\partial\overline \Phi}{\partial\overline z}
e^{\mathcal B_1+\overline{\mathcal B_2}})d\sigma=\nonumber\\
\frac 1i(\frak I(\Phi,a\Psi_0,b,c,d\Psi_0)
-\overline{\frak I(\Phi,a\Psi_0,b,c,d\Psi_0)}).
\end{eqnarray}

Then by (\ref{chech}), (\ref{nona}), (\ref{nona1A}) and Proposition \ref{zoma}, $H_1=0.$
Taking into account (\ref{Pona}) we obtain that $J(\widehat \Psi,
\widehat{\widetilde\Psi})=0.$
Consequently, setting $\Theta=\widehat\Psi/(\partial_z\Phi a\overline c)$ and $\widetilde
\Theta =\widehat{\widetilde\Psi}/(\overline{\partial_z\Phi} d\overline b)$ we obtain
(\ref{volt}).

Observe that
\begin{equation}\label{nono}
(\nu_1+i\nu_2)\frac{\partial\Phi}{\partial z}
=-(\nu_1-i\nu_2)\frac{\partial\overline \Phi}{\partial\overline z} \quad\mbox{on}
\,\,\Gamma_0.
\end{equation}
In order to see this we argue as follows. We have that
$\frac{\partial}{\partial\overline z}
=\frac 12(\frac{\partial}{\partial x_1}+i\frac{\partial}{\partial x_2})
(\varphi+i\psi)=\frac 12 (\frac{\partial\varphi}{\partial x_1}
-\frac{\partial\psi}{\partial x_2})+\frac i2
(\frac{\partial\varphi}{\partial x_2}+\frac{\partial\psi}
{\partial x_1})$.
Hence $\frac{\partial\varphi}{\partial x_1}
=\frac{\partial\psi}{\partial x_2}$,
$\frac{\partial\varphi}{\partial x_2}
=-\frac{\partial\psi}{\partial x_1}$,
$\frac{\partial\varphi}{\partial\nu}
=-\frac{\partial\psi}{\partial\vec \tau}$ and
$\frac{\partial\psi}{\partial \nu}=\frac{\partial\varphi}
{\partial\vec\tau}$.
Observe that
$(\nu_1+i\nu_2)\frac{\partial}{\partial z}
=\frac 12(\nu_1\frac{\partial}{\partial x_1}+\nu_2\frac{\partial}
{\partial x_2})+\frac i2 (\nu_2\frac{\partial}{\partial x_1}
-\nu_1\frac{\partial}{\partial x_2})=\frac 12(\frac{\partial}
{\partial \nu}+i\frac{\partial}{\partial\vec\tau})
$ and $(\nu_1-i\nu_2)\frac{\partial}{\partial \overline z}
= \frac 12(\frac{\partial}{\partial \nu}
-i\frac{\partial}{\partial\vec\tau}).
$
Hence
$$
(\nu_1+i\nu_2)\frac{\partial\Phi}{\partial z}
=\frac 12(\frac{\partial}{\partial \nu}
+i\frac{\partial}{\partial\vec\tau})(\varphi+i\psi)
=\frac 12(\frac{\partial\varphi}{\partial \nu}
-\frac{\partial\psi}{\partial\vec \tau})
+\frac i2 (\frac{\partial\varphi}{\partial\vec \tau}
+\frac{\partial\psi}{\partial \nu})
= -\frac{\partial\psi}{\partial \vec \tau}
+i\frac{\partial\varphi}{\partial \vec \tau}.
$$
Therefore
$$
(\nu_1-i\nu_2)\frac{\partial\overline\Phi}{\partial \overline z}
= \overline{\frac{\partial\psi}{\partial \vec \tau}
+i\frac{\partial\varphi}{\partial\vec \tau}}
= -\frac{\partial\psi}{\partial\vec \tau}
-i\frac{\partial\varphi}{\partial\vec \tau}.
$$
Taking into account that $\psi\vert_{\Gamma_0}=0$ we obtain
(\ref{nono}).

From (\ref{nono}), (\ref{iopa}), (\ref{ikaa}),
(\ref{PI5}), (\ref{ikaaa}) and (\ref{voron}) we obtain (\ref{volt1}).
The proof of the proposition is completed. In general
$\Phi, a,b,c,d$ may have a finite number of zeros
in $\Omega.$ At these zeros $\Theta,\widetilde \Theta$
may have poles. On the other hand observe
that $\Theta,\widetilde \Theta$ are independent of
a particular  choice of the functions $\Phi, a,b,c, d.$  Making small
perturbations of these functions we can shift the position of the zeros. Therefore we may assume that there are no  poles for  $\Theta,\widetilde \Theta.$
\end{proof}
Thanks to Proposition \ref{balda},
we can rewrite (\ref{finally}) as
\begin{eqnarray}\label{visokk}
& & \quad\quad\quad\quad o(\frac 1\tau)=\sum_{k=1}^3\tau^{2-k}
\widetilde F_k\\
&-& \frac{ \pi}{\tau\vert \mbox{det}\, \psi ''(\widetilde x)
\vert^\frac 12}
\left\{({\mathcal Q}_+a\overline b)(\widetilde x)
e^{(\mathcal A_1+\overline{\mathcal B_2}
+2\tau i\psi)(\widetilde x)}
+ ({\mathcal Q}_- d\overline c)
(\widetilde x)e^{({\mathcal B_1}+\overline{\mathcal A_2}
-2i\tau\psi)(\widetilde x)}\right\}\nonumber\\
&-& \frac{1}{\tau \vert \mbox{det}\, \psi ''({\widetilde x})\vert
^\frac 12}\int_{\Gamma_0}
(\nu_1+i\nu_2)(e^{\mathcal A_1+\overline{\mathcal A_2}}-\Theta)
\frac{\partial\Phi}{\partial z}a(z)\overline{ \left (c_{2,+}
e^{2\tau i\psi({\widetilde x})}+ c_{2,-}e^{-2\tau i\psi({\widetilde x})}
\right )} d\sigma\nonumber\\
&-& \frac{1}{\tau \vert \mbox{det}\, \psi ''({\widetilde x})
\vert^\frac 12}\int_{\Gamma_0} (\nu_1-i\nu_2)(e^{\mathcal B_1
+\overline{\mathcal B_2}}-\widetilde \Theta)\frac{\partial\overline\Phi}
{\partial \overline z}\overline{b(z)}{\left ( d_{2,+}e^{2\tau i\psi({\widetilde x})}
+d_{2,-}e^{-2\tau i\psi({\widetilde x})} \right )}d\sigma             \nonumber\\
&-& \frac{1}{\tau \vert \mbox{det}\, \psi ''({\widetilde x})
\vert^\frac 12}\int_{\Gamma_0} (\nu_1+i\nu_2)
(e^{\mathcal A_1+\overline{\mathcal A_2}}-\Theta)\frac{\partial\Phi}
{\partial z}\overline{c(\overline z)}{ \left (a_{2,+}
e^{2\tau i\psi({\widetilde x})}+ a_{2,-}e^{-2\tau i\psi({\widetilde x})}\right )}
d\sigma\nonumber\\
&-& \frac{1}{\tau \vert \mbox{det}\, \psi ''({\widetilde x})
\vert^\frac 12}\int_{\Gamma_0} (\nu_1-i\nu_2)
(e^{\mathcal B_1+\overline{\mathcal B_2}}-\widetilde \Theta)
\frac{\partial\overline\Phi}{\partial \overline z}d(\overline z)
\overline{\left ( b_{2,+}e^{2\tau i\psi({\widetilde x})}
+ b_{2,-}e^{-2\tau i\psi({\widetilde x})} \right )}d\sigma
                                                 \nonumber\\
&-& 2\int_{\Gamma_0}(\nu_1+i\nu_2)
 (e^{\mathcal A_1+\overline{\mathcal A_2}}-\Theta)\overline{c_\tau(\overline z)}
\frak G_1(x,\tau)d\sigma
+2\int_{\Gamma_0}(\nu_1-i\nu_2)(e^{\mathcal B_1+\overline{\mathcal B_2}}
-\widetilde \Theta)d_\tau(\overline z)\overline{\frak G_4(x,\tau)}d\sigma
                                                     \nonumber\\
&+&2\int_{\Gamma_0}(\nu_1+i\nu_2)(e^{\mathcal A_1
+\overline{\mathcal A_2}}-\Theta)a_\tau(z)\overline{\frak G_3(x,\tau)}d\sigma-
2\int_{\Gamma_0}(\nu_1-i\nu_2)(e^{\mathcal B_1+\overline{\mathcal B}_2}
-\widetilde \Theta)\overline{b_\tau(z)} \frak G_2(x,\tau)d\sigma\nonumber\\
&+ &\frac{ e^{-2i\tau\psi(\widetilde x)}}{\tau\vert\mbox{det}\,\psi^{''}(\widetilde x)\vert^\frac 12}\overline{\frac{\partial g_4(\widetilde x)}{\partial
z}}e^{-\overline{\mathcal B_2} (\widetilde x)}\left(\int_{\Gamma_0} \frac{(\nu_1-i\nu_2)d (e^{\mathcal B_1+\overline{\mathcal B_2}}-\widetilde\Theta)}{\overline {\widetilde z}- \overline z}d\sigma -2\pi (d\widetilde\Theta)(\widetilde x) \right )\nonumber\\
&-&\frac{ e^{-2i\tau\psi(\widetilde x)}}{\tau\vert\mbox{det}\, \psi^{''}(\widetilde x)\vert^\frac 12}\frac{\partial g_1(\widetilde x)}{\partial z}e^{-\mathcal A_1(\widetilde x)}\left (\int_{\Gamma_0} \frac{(\nu_1+i\nu_2)\overline{c} (e^{\mathcal A_1+\overline{{\mathcal A}_2}}-\Theta)}{\widetilde z-z}d \sigma-2\pi(\overline c\Theta)(\widetilde x)\right )\nonumber\\
&+&\frac{ e^{2i\tau\psi(\widetilde x)}}{\tau\vert\mbox{det}\,\psi^{''}(\widetilde x)\vert^{\frac 12}}\overline{\frac{\partial g_3(\widetilde x)}{\partial \overline
z}}e^{-\overline{\mathcal A_2}(\widetilde x)}\left (\int_{\Gamma_0}\frac{(\nu_1+i\nu_2)a (e^{\mathcal A_1+\overline{\mathcal A_2}}-\Theta)}{ \widetilde z-z}d\sigma-2\pi (a\Theta)(\widetilde x)\right)\nonumber\\
&-&\frac{e^{2i\tau\psi(\widetilde x)}}{\tau\vert
\mbox{det}\,\psi^{''}(\widetilde x)\vert^\frac 12} \frac{\partial
g_2(\widetilde x)}{\partial \overline z}e^{-\mathcal B_1(\widetilde
x)}\left (\int_{\Gamma_0} \frac{(\nu_1-i\nu_2)\overline b
(e^{\mathcal{B}_1+\overline{\mathcal B_2} }-\widetilde\Theta)}
{\overline{\widetilde z}-\overline z}d\sigma-2\pi (\overline
b\Theta)(\widetilde x)\right).\nonumber
\end{eqnarray}
Here $\widetilde{F}_k$ are some constants independent of
$\tau$.

Then, using (\ref{AK3}), (\ref{volt1}), (\ref{nono}), on $\Gamma_0$
we have
\begin{eqnarray}\label{nosok}
- (\nu_1+i\nu_2)\frac{\partial\Phi}{\partial z} (e^{\mathcal A_1
+\overline{\mathcal A_2}}-\Theta)\overline{c}\left (
a_{2,+}e^{2\tau i\psi({\widetilde x})}+ a_{2,-}e^{-2\tau i\psi({\widetilde x})}
\right )\\
- (\nu_1-i\nu_2)\frac{\partial\overline\Phi}{\partial \overline z}
(e^{\mathcal B_1+\overline{\mathcal B_2}}-\widetilde \Theta)
\overline{b}\left ( d_{2,+}e^{2\tau i\psi({\widetilde x})}
+ d_{2,-}e^{-2\tau i\psi({\widetilde x})}\right )=\nonumber\\
- (\nu_1+i\nu_2)\frac{\partial\Phi}{\partial z}\overline c\left
((e^{\mathcal A_1+\overline{\mathcal A_2}}-\Theta) a_{2,+}
e^{2\tau i\psi({\widetilde x})}+(e^{\mathcal B_1+\overline{\mathcal A_2}}
-\widetilde \Theta e^{\overline{\mathcal A_2}-\overline{\mathcal B_2}})
d_{2,+}e^{2\tau i\psi({\widetilde x})}\right )\nonumber\\
-(\nu_1-i\nu_2)\frac{\partial\overline\Phi}{\partial \overline z}
\overline{b}\left (( e^{\mathcal A_1+\overline{\mathcal B_2}}
-\Theta e^{\overline{\mathcal B_2}-\overline{\mathcal A_2}})a_{2,-}
e^{-2\tau i\psi({\widetilde x})}+(e^{\mathcal B_1+\overline{\mathcal B_2}}
-\widetilde \Theta) d_{2,-}e^{-2\tau i\psi({\widetilde x})}\right )=\nonumber\\
- (\nu_1+i\nu_2)\frac{\partial\Phi}{\partial z}\overline c
\left ((e^{\mathcal A_1+\overline{\mathcal A_2}}-\Theta) a_{2,+}
e^{2\tau i\psi({\widetilde x})}+(e^{\mathcal B_1+\overline{\mathcal A_2}}
- \Theta e^{-{\mathcal A_1}+{\mathcal B_1}}) d_{2,+}
e^{2\tau i\psi({\widetilde x})}\right )\nonumber\\
- (\nu_1-i\nu_2)\frac{\partial\overline\Phi}{\partial \overline z}
\overline{b}\left (( e^{\mathcal A_1+\overline{\mathcal B_2}}
-\widetilde\Theta e^{{\mathcal A_1}-{\mathcal B_1}})a_{2,-}
e^{-2\tau i\psi({\widetilde x})}+(e^{\mathcal B_1+\overline{\mathcal B_2}}
-\widetilde \Theta)d_{2,-}e^{-2\tau i\psi({\widetilde x})}\right )=\nonumber\\
- (\nu_1+i\nu_2)\frac{\partial\Phi}{\partial z} (e^{\overline{\mathcal A_2}}
-\Theta e^{-\mathcal A_1})\overline{c} p_{+}e^{2\tau i\psi({\widetilde x})}
-(\nu_1-i\nu_2)\frac{\partial\overline\Phi}{\partial \overline z}
(e^{\overline{\mathcal B_2}}-\widetilde \Theta e^{-\mathcal B_1})
\overline{b}p_- e^{-2\tau i\psi({\widetilde x})}\nonumber
\end{eqnarray}
and
\begin{eqnarray}\label{nosok1}
-(\nu_1+i\nu_2)(e^{\mathcal A_1+\overline{\mathcal A_2}}
-\Theta)\frac{\partial\Phi}{\partial z}a\left (\overline{ c_{2,+}
e^{2\tau i\psi({\widetilde x})}+c_{2,-}e^{-2\tau i\psi({\widetilde x})}}\right )
\\
-  (\nu_1-i\nu_2)(e^{\mathcal B_1+\overline{\mathcal B_2}}-\widetilde \Theta)
\frac{\partial\overline\Phi}{\partial \overline z}d\left (\overline{ b_{2,+}
e^{2\tau i\psi({\widetilde x})}+ b_{2,-}e^{-2\tau i\psi({\widetilde x})}}\right )
=\nonumber\\
-  (\nu_1+i\nu_2)\frac{\partial\Phi}{\partial z}a\left ((e^{\mathcal A_1
+\overline{\mathcal A_2}}-\Theta)\overline{c_{2,-}
e^{-2\tau i\psi({\widetilde x})}}+(e^{\mathcal A_1+\overline{\mathcal B_2}}
-e^{-\mathcal B_1+\mathcal A_1}\widetilde \Theta)\overline{b_{2,-}
e^{-2\tau i\psi({\widetilde x})}}\right )\nonumber\\
-  (\nu_1-i\nu_2)\frac{\partial\overline {\Phi}}{\partial \overline z}
d\left ((e^{\mathcal B_1+\overline{\mathcal A_2}}
-\Theta e^{-\mathcal A_1+\mathcal B_1})\overline{c_{2,+}
e^{-2\tau i\psi({\widetilde x})}}+(e^{\mathcal B_1+\overline{\mathcal B_2}}
-\widetilde \Theta)\overline{b_{2,+}e^{2\tau i\psi({\widetilde x})}}\right )
                                           \nonumber=\\
-  (\nu_1+i\nu_2)\frac{\partial\Phi}{\partial z}a\left
((e^{\mathcal A_1+\overline{\mathcal A_2}}-\Theta)\overline{c_{2,-}
e^{-2\tau i\psi({\widetilde x})}}+(e^{\mathcal A_1+\overline{\mathcal B_2}}
-e^{\overline{\mathcal B_2}-\overline{\mathcal A_2}} \Theta)
\overline{b_{2,-}e^{-2\tau i\psi({\widetilde x})}}\right )\nonumber\\
- (\nu_1-i\nu_2)\frac{\partial\overline {\Phi}}{\partial \overline z}
d\left ((e^{\mathcal B_1+\overline{\mathcal A_2}}
-\Theta e^{\overline{\mathcal A_2}-\overline{\mathcal B_2}})
\overline{c_{2,+}
e^{2\tau i\psi({\widetilde x})}}+(e^{\mathcal B_1+\overline{\mathcal B_2}}
-\widetilde \Theta)\overline{b_{2,+}e^{2\tau i\psi({\widetilde x})}}\right )
                                               \nonumber=\\
-  (\nu_1+i\nu_2)\frac{\partial\Phi}{\partial z}
(e^{\mathcal A_1}-\Theta e^{-\overline{\mathcal A_2}})
a\overline{\widetilde p_{-}e^{-2\tau i\psi({\widetilde x})}}
-  (\nu_1-i\nu_2)\frac{\partial\overline {\Phi}}{\partial \overline z}
(e^{\mathcal B_1}-\widetilde \Theta e^{-\overline{\mathcal B_2}}) d
\overline{\widetilde p_+e^{2\tau i\psi({\widetilde x})}} .\nonumber
\end{eqnarray}
Using (\ref{nosok}), (\ref{nosok1}) and Proposition \ref{nuvo}
we rewrite (\ref{visokk}) as
\begin{eqnarray}\label{PPP}
&& o(\frac 1\tau)=\sum_{k=1}^3\tau^{2-k} \widetilde F_k\\
&&-
\frac{\pi}{\tau\vert \mbox{det}\, \psi ''(\widetilde x)
\vert^\frac 12}
\left\{({\mathcal Q}_+a\overline b)
(\widetilde x) e^{(\mathcal A_1+\overline{\mathcal B_2}
+2\tau i\psi )(\widetilde x)}+ ({\mathcal Q}_-
 d\overline c)(\widetilde x)e^{({\mathcal B_1}+\overline{\mathcal A_2}
-2i\tau\psi)(\widetilde x)}\right\}\nonumber\\
&&\frac{-e^{2i\tau\psi({\widetilde x})}}{\tau \vert \mbox{det}\,
\psi ''({\widetilde x})\vert^\frac 12}\int_{\Gamma_0}
(\nu_1+i\nu_2)(e^{\mathcal A_1+\overline{\mathcal A_2}}
-\Theta )a\overline{\frac{(\partial_z g_3
e^{-\mathcal A_2})({\widetilde x})}{\widetilde z-z}}d\sigma\nonumber\\
&&+\frac{e^{-2i\tau\psi({\widetilde x})}}{\tau \vert \mbox{det}\,
\psi ''({\widetilde x})\vert^\frac 12}\int_{\Gamma_0}
(\nu_1+i\nu_2)
(e^{\mathcal A_1+\overline{\mathcal A_2}}-\Theta)\overline{c(\overline z)}
\frac{(\partial_zg_1e^{-\mathcal A_1})({\widetilde x})}{\widetilde z-z}
d\sigma\nonumber\\
&&-\frac{e^{-2i\tau\psi({\widetilde x})}}{\tau \vert \mbox{det}
\, \psi ''({\widetilde x})\vert^\frac 12}\int_{\Gamma_0}
(\nu_1-i\nu_2)(e^{\mathcal B_1+\overline{\mathcal B_2}}
-\widetilde\Theta)d\overline{\frac{(\partial_{\overline z}g_4e^{-\mathcal B_2})
({\widetilde x})}{\overline{\widetilde z}-\overline z}}d\sigma\nonumber\\
&&+\frac{e^{2i\tau\psi({\widetilde x})}}{\tau \vert
\mbox{det}\, \psi ''({\widetilde x})\vert^\frac 12}
\int_{\Gamma_0}(\nu_1-i\nu_2)
(e^{\mathcal B_1+\overline{\mathcal B}_2}
-\widetilde \Theta) \overline{b(z)}
\frac{(\partial_{\overline z}g_2e^{-\mathcal B_1})({\widetilde x})}
{\overline{\widetilde z}-\overline z}d\sigma\nonumber\\
&&+ \frac{ e^{-2i\tau\psi(\widetilde x)}}{\tau\vert\mbox{det}\,\psi^{''}(\widetilde x)\vert^\frac 12}\overline{\frac{\partial g_4(\widetilde x)}{\partial
z}}e^{-\overline{\mathcal B_2} (\widetilde x)}\left(\int_{\Gamma_0} \frac{(\nu_1-i\nu_2)d (e^{\mathcal B_1+\overline{\mathcal B_2}}-\widetilde\Theta)}{\overline {\widetilde z}- \overline z}d\sigma -2\pi (d\widetilde\Theta)(\widetilde x) \right )\nonumber\\
&&-\frac{ e^{-2i\tau\psi(\widetilde x)}}{\tau\vert\mbox{det}\,\psi^{''}(\widetilde x)\vert^\frac 12}\frac{\partial g_1(\widetilde x)}{\partial z}e^{-\mathcal A_1(\widetilde x)}\left (\int_{\Gamma_0} \frac{(\nu_1+i\nu_2)\overline{c} (e^{\mathcal A_1+\overline{{\mathcal A}_2}}-\Theta)}{\widetilde z-z}d \sigma-2\pi(\overline c\Theta)(\widetilde x)\right )\nonumber\\
&&+\frac{ e^{2i\tau\psi(\widetilde x)}}{\tau\vert\mbox{det}\,\psi^{''}(\widetilde x)\vert^{\frac 12}}\overline{\frac{\partial g_3(\widetilde x)}
{\partial \overline z}}e^{-\overline{\mathcal A_2}
(\widetilde x)}\left (\int_{\Gamma_0}\frac{(\nu_1+i\nu_2)a
(e^{\mathcal A_1+\overline{\mathcal A_2}}-\Theta)}{ \widetilde z-z}d\sigma-2\pi (a\Theta)(\widetilde x)\right)\nonumber\\
&&-\frac{e^{2i\tau\psi(\widetilde x)}}
{\tau\vert \mbox{det}\,\psi^{''}(\widetilde x)\vert^\frac 12}
\frac{\partial g_2(\widetilde x)}{\partial \overline z}
e^{-\mathcal B_1(\widetilde x)}\left
(\int_{\Gamma_0}\frac{(\nu_1-i\nu_2)\overline b
(e^{\mathcal{B}_1+\overline{\mathcal B_2} }-\widetilde\Theta)}
{\overline{\widetilde z}-\overline z}d\sigma
-2\pi (\overline b\widetilde\Theta)(\widetilde x)\right ).\nonumber
\end{eqnarray}

Let $\eta$ be a smooth function such that $\eta$ is zero in some neighborhood of $\partial\Omega$ and $\eta(\widetilde x)\ne 0$. Observe that the partial Cauchy data of the operators $L_2(x,D)$ and the operator $e^{-s\eta}L_1(x,D)e^{s\eta}$ are exactly the same. Therefore we have the analog of  (\ref{PPP}) for these two operators with $\mathcal A_1$ and $\mathcal B_1$ replaced by  $\mathcal A_1-s\eta$ and $\mathcal B_1-s \eta$.
The coefficients $A_1,B_1$ should be replaced by $A_1+2s\frac{\partial \eta}{\partial \overline z},B_1+2s\frac{\partial \eta}{\partial z}.$ The functions $\mathcal Q_\pm$ will not change. The function $q_1$ should be replaced by
$q_1 +s\Delta\eta+s^2\vert\nabla \eta\vert^2
+ 2sA_1\frac{\partial\eta}{\partial z}
+ 2sB_1\frac{\partial\eta}{\partial\overline z}.$ This immediately
implies that $({\mathcal Q}_+a\overline b)(\widetilde x)
=({\mathcal Q}_- d\overline c)(\widetilde x)=0.$
The proof of the theorem is completed. $\square$

\section{ Proof of Theorem 1.1}
Suppose that we have two operators
$$
L_1(x,D)=\Delta_{g_1}+2A_1\frac{\partial }{\partial z}
+2B_1\frac{\partial }{\partial \overline z}+q_1
$$
and
$$
L_2(x,D)=\Delta_{g_2}+2 A_2\frac{\partial }{\partial z}
+2B_2\frac{\partial }{\partial \overline z}+ q_2
$$
with the same partial Cauchy data. Multiplying
the metric $g_2$, if necessary, by some positive smooth function
$\widetilde \beta$, we may assume that
\begin{equation}\label{guilty}
\frac{\partial^\ell}{\partial\nu^\ell}( g_1^{jk}-
\, g_2^{jk})\vert_{\widetilde \Gamma}=0,\quad \ell\in\{0,1\}.
\end{equation}
We note that $\{g_1^{jk}\}$ denotes the inverse matrix to $g_1 =
\{g_{1,jk}\}$. Without loss of generality, we may assume that there
exists a smooth positive function $ \mu_2$ such that $g_2=\mu_2 I.$
Indeed, using  isothermal coordinates we make a change of variables
in the operator  $L_2(x,D)$ such that $g_2=\mu_2 I.$  Then we make
the same changes of variables in the operator $L_1(x,D).$ The
partial Cauchy data of both operators obtained by this change of
variables are the same.

Let $\omega$ be a subdomain in $\Bbb R^2$ such that
$\Omega\cap\omega
=\emptyset,$ $\partial\omega\cap\partial\Omega=
\widetilde \Gamma$ and the boundary of the domain
$\widetilde \Omega=\mbox{Int} (\Omega\cup\omega)$ is smooth.
We extend $\mu_2$ in $\widetilde \Omega $ as a smooth positive function and
set $ g_1^{jk}=\frac {1}{\mu_2} I$ in
$\omega.$ By (\ref{guilty}) $g_1\in C^1(\overline\Omega).$
There exists an isothermal mapping $\chi_1=(\chi_{1,1},\chi_{1,2})$
such that the operator $L_1(x,D)$ is  transformed to
\begin{equation}\label{bistro}
Q_1(y,D)=\frac{1}{\mu_1}\Delta+2 C_1\frac{\partial }
{\partial z}+2D_1\frac{\partial }{\partial \overline z}
+ r_1\quad y\in \chi_1(\widetilde\Omega),
\end{equation}
where $\mu_1$ is a smooth positive function in $\chi_1(\widetilde \Omega)$ and $C_1,D_1,r_1$ are some smooth complex valued functions.
Consider a solution to the boundary value  problem
$$
Q_1(y,D)w=0\quad\mbox{in}\,\, \chi_1(\widetilde\Omega),
\,\,\,w\vert_{\chi_1(\Gamma_0)}=0
$$
of the form (\ref{mozilaa}) with a holomorphic weight function
$\Phi_1.$ Then the function $ u_1(x)=w(\chi_1(x))$ is solution to
$$
L_1(x,D)u_1=0\quad\mbox{in}\,\, \widetilde \Omega, \quad u_1\vert_{\Gamma_0}=0.
$$
Since the partial Cauchy data for the operators $L_1(x,D)$ and
$L_2(x,D)$ are the same, there exists a function $u_2$ such
that
\begin{equation}\label{ih}
 L_2(x,D)u_2=0\quad\mbox{in}\,\, \Omega, \quad  u_2\vert
_{\Gamma_0}=0, \,(\frac{\partial u_1}{\partial \nu_{g_1}}
-\frac{\partial  u_2}{\partial\nu_{g_{2}}})\vert_{\widetilde\Gamma}=0.
\end{equation}
Using (\ref{guilty}) and (\ref{ih}) we extend $u_2\in
H^1(\widetilde{\Omega})$ in
$\widetilde \Omega$ such that
\begin{equation}\label{magnus}
u_1\vert_\omega= u_2\vert_\omega.
\end{equation}
Let $\varphi_2$ be a harmonic function in $\widetilde\Omega$ such that
$$
\frac{\partial\varphi_2}{\partial\nu}\vert_{\Gamma_0}=0,\quad
\varphi_2=\mbox{Re}\, \Phi_1\circ\chi_1\quad\mbox{on}\,\,
\partial \widetilde
\Omega\setminus \Gamma_0.
$$
We claim that
\begin{equation}\label{magnus1}
 \varphi_2=\mbox{Re}\, \Phi_1\circ\chi_1\quad
\mbox{in}\,\,\omega.
\end{equation}

A difficulty comes from the fact that the function $\varphi_2$ is
continuous on $\overline \Omega$ (see e.g \cite{SS}) but the
derivatives of $\varphi_2$ may be discontinuous at some points of
$\partial\Gamma_0.$ First we observe that it is suffices to prove
(\ref{magnus1}) for the functions such that $Im \,\Phi_1=0$ on some
open set $\mathcal O_{\Phi_1}\subset
\partial\chi_1(\widetilde\Omega)$ such that
$\overline{\chi_1(\Gamma_0)}\subset{\mathcal O}_{\Phi_1}.$ Indeed,
without loss of generality, assume that $\partial\widetilde
\Omega\setminus\Gamma_0$ is an arc with two endpoints $x_\pm.$ Let
the sequences  $\{x_{\epsilon,-}\}, \{x_{\epsilon,+}\} \subset
\partial\widetilde \Omega\setminus\Gamma_0$ be such that
$x_{\epsilon,\pm}\rightarrow x_\pm$ as $\epsilon\rightarrow 0.$
Consider a sequence of holomorphic  functions
$\{\Phi_{1,\epsilon}\}_{\epsilon \in (0,1)}$ such that
$$
\frac{\partial \Phi_{1,\epsilon}}{\partial \overline z}=0\quad\mbox{in}\,\,\chi_1(\widetilde \Omega),
 \quad \mbox{Im}\,\Phi_{1,\epsilon}\vert_{\chi_1(\Gamma_{0,\epsilon})}=0,
$$
$$
\Phi_{1,\epsilon}\rightarrow \Phi_1 \quad\mbox{in}\quad C^1(\overline{\widetilde\Gamma_\epsilon}),
$$
where $\overline{\widetilde\Gamma_\epsilon}\subset\widetilde \Gamma$ is the arc between points $x_{\epsilon,-},
 x_{\epsilon,+}$ and $\Gamma_{0,\epsilon}=\partial\tilde\Omega\setminus \widetilde\Gamma_\epsilon.$
We define $\varphi_{2,\epsilon}$ as
$$
\frac{\partial\varphi_{2,\epsilon}}{\partial\nu}\vert_{\Gamma_{0,\epsilon}}=0,\quad
\varphi_{2,\epsilon}=\mbox{Re}\, \Phi_{1,\epsilon}\circ\chi_1\quad\mbox{on}\,\, \partial\widetilde
\Omega\setminus\Gamma_{0,\epsilon}.
$$
First we assume that
\begin{equation}\label{mika0}
\varphi_{2,\epsilon}=\mbox{Re}\, \Phi_{1,\epsilon}\circ\chi_1\quad\mbox{on}\,\,\omega.
\end{equation}
Passing to the limit in the above equality we obtain
(\ref{magnus1}).

Now we concentrate on the proof of (\ref{mika0}). Let $\Phi_{1,\epsilon}$ be one of  the functions in the sequence
 $\{\Phi_{1,\epsilon}\}_{\epsilon \in (0,1)}$.
Consider a sequence of domains
$\widetilde \Omega_\varepsilon$ such that
$\widetilde \Omega_\varepsilon\subset\widetilde\Omega$,
$\partial\widetilde{\Omega}_\varepsilon\cap
\partial\widetilde{\Omega}=\Gamma_0$
and $dist (\partial\widetilde\Omega_\varepsilon\setminus \Gamma_0, \widetilde\Gamma)\rightarrow 0$
as $\varepsilon\rightarrow 0.$ Then the function $\varphi_{2,\epsilon}$ is smooth on $\widetilde\Omega_\varepsilon.$
Let us take as $u_1$ the CGO solution constructed in the previous sections.
Thanks to the Carleman estimate (\ref{suno4}) there exists $\tau_0=\tau_0(\varepsilon)$ such that
\begin{equation}\label{voldemort} \Vert e^{-\tau\varphi_{2,\epsilon}}  u_2\Vert
_{L^2(\widetilde\Omega_\varepsilon)}\le C_0\vert \tau e^{\delta_\varepsilon\vert\tau\vert}\vert\quad \forall \vert\tau\vert\ge\tau_0,
\end{equation}
where $C_0=C_0(\varepsilon)$ is independent of $\tau$ and $\delta_\varepsilon\rightarrow 0$ as $\varepsilon \rightarrow 0.$   On the other hand
$u_1=e^{\tau \mbox{Re}\, \Phi_{1,\epsilon}\circ\chi_1}((a_\tau e^{{\mathcal C}_1+i\tau
\mbox{Im}\, \Phi_{1,\epsilon}}+b_\tau e^{{\mathcal D}_1-i\tau
\mbox{Im}\, \Phi_{1,\epsilon}})\circ\chi_1 +O(\frac 1\tau)).$
Here we note that $\mathcal{C}_1, \mathcal{D}_1
\in C^{6+\alpha}(\overline{\widetilde{\Omega}
_{\varepsilon}})$ are defined
similarly to (\ref{001q}):
$$
2\frac{\partial \mathcal C_1}{\partial\overline z} = -C_1\quad
\mbox{in}\,\,\widetilde{\Omega}_{\epsilon}, \,\,\mbox{Im}\,
{\mathcal C_1}\vert_{\Gamma_0}=0, \quad 2 \frac{\partial
\mathcal{D}_1}{\partial z} = -{D}_1 \quad
\mbox{in}\,\,\widetilde{\Omega}_{\varepsilon}, \,\,\mbox{Im}\,
{\mathcal D_1}\vert_{\Gamma_0}=0.
$$
Then by (\ref{magnus}) the following holds true:
\begin{equation}\label{horn1}
e^{\tau\varphi_{2,\epsilon}} (e^{-\tau\varphi_{2,\epsilon}} u_2)=e^{\tau \mbox{Re}\, \Phi_{1,\epsilon}
\circ\chi_1}((a_\tau e^{{\mathcal C}_1+i\tau
\mbox{Im}\, \Phi_{1,\epsilon}}+b_\tau e^{{\mathcal D}_1-i\tau
\mbox{Im}\, \Phi_{1,\epsilon}})\circ\chi_1
+O(\frac 1\tau))\quad \forall x\in\omega.
\end{equation}
This equality implies (\ref{magnus1}) immediately. Indeed, let for
some point $\widehat x $ from $\omega$
\begin{equation}\label{rollex}
\varphi_{2,\epsilon}(\widehat x)\ne \mbox{Re}\, \Phi_{1,\epsilon}\circ\chi_1(\widehat x).
 \end{equation} Then there exists a ball $B(\widehat x,\delta')
\subset \omega$ such that
\begin{equation}\label{zumba}
\vert \varphi_{2,\epsilon}( x)- \mbox{Re}\, \Phi_{1,\epsilon}\circ\chi_1( x)\vert >
\alpha'>0\quad \forall x\in \overline{B(\widehat x,\delta')}.
\end{equation}
Let us fix positive $\varepsilon_1$ such that $\overline{B(\widehat x,\delta')}\subset\Omega_{\varepsilon_1}$ and $2\delta_{\varepsilon_1}<\alpha'.$
Form (\ref{horn1}) by (\ref{voldemort}) and (\ref{zumba}) we have
\begin{eqnarray}
C' e^{\vert \tau\vert\alpha'} Vol(B(\widehat x,\delta'))^\frac 12
\le \Vert e^{\tau (\mbox{Re}\, \Phi_{1,\epsilon}\circ\chi_1-\varphi_{2,\epsilon})}
(((a_\tau e^{{\mathcal C}_1+i\tau
\mbox{Im}\, \Phi_{1,\epsilon}}+b_\tau e^{\mathcal D-i\tau
\mbox{Im}\, \Phi_{1,\epsilon}})\circ\chi_1
+O(\frac 1\tau))\Vert_{L^2(B(\widehat x,\delta'))}\nonumber\\
=\Vert e^{-\tau\varphi_{2,\epsilon}} u_2\Vert_{L^2(B(\widehat x,\delta'))}
\le C_0\vert \tau\vert e^{\delta_\epsilon\vert\tau\vert},\nonumber
\end{eqnarray}
where $\tau>\tau_0$ if $ \varphi_{2,\epsilon}(\widehat x)< \mbox{Re}\,
\Phi_{1,\epsilon}\circ\chi_1( \widehat x)$ and  $\tau<-\tau_0$ if
$\varphi_{2,\epsilon}(\widehat x)> \mbox{Re}\, \Phi_{1,\epsilon}\circ\chi_1( \widehat x).$
The above inequality contradicts (\ref{rollex}).

Let $\Xi=\chi_{1,1}+i\chi_{1,2}.$
Using the Cauchy-Riemann equations, we construct
the multivalued function $\psi_2$ such that $\Phi_2=\varphi_2+i\psi_2$ is
holomorphic on the Riemann surface associated with $\widetilde\Omega.$
Moreover we take the function $\Phi_1$ which can be
holomorphically extended in some domain $\mathcal O$ such that $\chi_1(\widetilde \Omega)\subset \mathcal O.$ Observe that
$$
\Phi_2=\Phi_1\circ \Xi \quad\mbox{on}\,\omega.
$$
Then $\Xi= \Phi_1^{-1}\circ\Phi_2$ in $\omega.$

We claim that the function
$\Xi$ can be extended up to a single valued
holomorphic  function  $\widetilde \Xi$ on
$\widetilde\Omega$ such that $\widetilde\Xi: \widetilde\Omega\rightarrow \chi_1(\widetilde\Omega),$  $\widetilde\Xi(\widetilde \Omega)=\chi_1(\widetilde \Omega)$ and $\partial_z\tilde \Xi\ne 0.$ First we show that the function $\Xi$ can be extended along any curve connecting two points in $\Omega.$ Our proof is by contradiction. Let $\gamma$ be such a continuous curve connecting a point $z_1$ in $\omega$ and a point $z_2$ in $\Omega$ such that  the function $\Xi$ can not be extended along $\gamma.$ Consider the parametrization of the curve $\gamma$ such that we are moving from the point $z_1=\gamma(0)$ to the point $z_2=\gamma(1)$. Let $\widehat z=\gamma(\kappa)$ be the first point on $\gamma$ around which the holomorphic continuation of the function $\Xi$  is impossible. Consider the function $\Phi_1$ such that
$\{z\vert \partial_z\Phi_1=0\}\cap \{z\vert z=\tilde \Xi(\gamma (s))\quad s\in [0,\kappa]\}=\emptyset.$
Observe that $$\Phi_2(\gamma(s))=\Phi_1\circ \Xi (\gamma(s))\quad \forall s\in [0,\kappa].$$
Indeed let $\hat s=\sup_{s\in X} \,s$ where $ X=\{s\vert \mbox{there exists}\,\,  \delta(s)>0\, \mbox{such that}\, \Phi_2(z)= \Phi_1\circ \Xi (z) \quad\forall z\in B(\gamma(s),\delta)\}.$
 Let $\hat s<\kappa.$ Since $\partial_z\Phi_1(\gamma (\hat s))\ne 0$  $\Phi_1^{-1}\circ \Phi_2$ is holomorphic with a domain which contains the ball centered at $\gamma(\hat s).$ Since $\tilde \Xi =\Phi_1^{-1}\circ \Phi_2$ on some open set this equality holds true on this ball which is a contradiction with the definition of $\hat s.$

Now we consider the situation at the point $\hat z$. Since we can not extend $\Xi$ around this point we would have  $\Pi=\{\widetilde z\vert \Phi_1(\widetilde z)=\Phi_2(\widehat z)\}\subset \chi_1(\partial\widetilde\Omega).$
 Since $\partial_z\Phi_1(\widetilde z)\ne 0$ we can extend $\Xi$ on some ball
centered at $\widehat z.$ (Of course such extension might not be the one we are looking for since  $\widetilde\Xi: \widetilde\Omega\rightarrow \chi_1(\widetilde\Omega)$ might not be valid.)
Consider a perturbation of the function $\Phi_1$: $\Phi_1+\epsilon\Psi_1,$  where $\Psi_1$ is a smooth holomorphic function in
$\mathcal O$ such that $Im\Psi_1\vert_{\chi_1(\Gamma_0)\cup \Pi}=0.$ This perturbation
generates a perturbation of the  function $\Phi_2$: $\Phi_2+\epsilon \Psi_2,$ where
\begin{equation}\label{mina}
\Delta \mbox{Re}\,\Psi_2=0\quad \mbox{in}\,\,\widetilde\Omega, \frac{\partial \mbox{Re}\,\Psi_2}{\partial\nu}\vert_{\Gamma_0}=0,\quad
\mbox{Re}\,\Psi_2=\mbox{Re}\, \Psi_1\circ\chi_1\quad\mbox{on}\,\, \partial \widetilde
\Omega\setminus \Gamma_0.
\end{equation}

For these new functions we still have
$$
\Phi_2+\epsilon \Psi_2=(\Phi_1+\epsilon\Psi_1)\circ \Xi \quad\mbox{on}\,\omega.
$$
For all sufficiently small $\epsilon$,
the function $(\Phi_1+\epsilon\Psi_1)$ does not have a critical point on $\widetilde \Omega.$ Therefore the function  $(\Phi_1+\epsilon\Psi_1)^{-1}\circ (\Phi_2+\epsilon \Psi_2)$ can be holomorphically
continued along of $\gamma$ up to the point $\widehat z.$ Denote this extension on some ball centered at $\widehat z$ as
$\widetilde\Xi_\epsilon.$ Obviously $\widetilde\Xi_\epsilon=\widetilde\Xi.$  Making a choice of $\Psi _1$ in such a way that $Im \Psi_2(\widehat z)\ne Im \Psi_1(\tilde \Xi (\widehat z))$ we obtain that this equality is impossible.

Let us show that the function $\widetilde \Xi$
does not have a critical points in $\widetilde \Omega. $ Our proof is again  by contradiction.
Suppose that $\widehat z$ is a critical point of $\widetilde \Xi.$ If such critical points exist these points are critical points of the function $\Phi_2$. Consider the perturbation of the function $\Phi_1$:  $\Phi_1+\epsilon\Psi_1,$  where $\Psi_1$ is a smooth holomorphic function in $\chi_1(\widetilde\Omega)$ such that $Im\Psi_1\vert_{\chi_1(\Gamma_0)}=0$  and such that for the function  $\Phi_2$ given by (\ref{mina}) we have $\partial_z\Psi_2(\widehat z)\ne 0.$ The mapping $\widetilde \Xi$ is still the same but  a position of the critical point for the function
$\Phi_2+\epsilon\Psi_2$ changes which is
a contradiction.

Let us show that $\tilde \Xi$ is the single valued function. Our proof is by contradiction.
Let $\tilde \Xi$ be the multivalued function around of some point $\tilde z.$ Then there exists holomorphic $\Phi_1$ such that $\partial_z\Phi_1(\tilde z)=0$ and $\Phi_2=\Phi_1\circ\tilde \Xi.$
Obviously \begin{equation}\label{zlo}\{z\vert z=\tilde \Xi(\hat z)\}\subset \{z\vert\partial_z\Phi_1=0\}.
\end{equation}
Let  $\Psi_1$ is a smooth holomorphic function in
$\mathcal O$ such that $Im\Psi_1\vert_{\chi_1(\Gamma_0)\cup \Pi}=0$ and $\partial\Psi_1 \ne 0$ for all $z\in \{z\vert\partial_z\Phi_1=0\}.$ Then for the function $\Phi_1+\epsilon\Psi_1$ we again should have $\tilde \Xi(\hat z)\}\subset \{z\vert\partial z(\Phi_1+\epsilon \Psi_1)=0\}.$ This contradicts to (\ref{zlo}).

If $\widetilde\Xi(\widetilde \Omega)\ne \chi_1(\widetilde \Omega)$ we still have that $\widetilde\Xi(\widetilde \Omega)\subset \chi_1(\widetilde \Omega)$.
On the other hand, on the boundary of
$\chi_1(\widetilde \Omega)\setminus \widetilde\Xi(\widetilde \Omega)$
the imaginary part of  the  function $\Phi_1$ is zero.  This is impossible.

In the domain $\Omega$ consider the new infinitesimal
coordinates for the operator $P_1$ given
by the mapping  $\widetilde\Xi^{-1}\circ\Xi(x).$
In these coordinates, the operator $P_1(x,D)$ has the form
\begin{equation}\label{bistro1}
\widetilde Q(x,D)=\frac{1}{\widetilde \mu_1}\Delta
+2 \widetilde A_1\frac{\partial }{\partial z}
+2\widetilde B_1\frac{\partial }{\partial \overline z}+\widetilde q_1.
\end{equation}
Since $\widetilde\Xi^{-1}\circ\Xi\vert_{\widetilde\Gamma}=Id$,
the Cauchy data for the
operators $L_2$ and $\widetilde Q$ are exactly the same.
The operators $L_2$ and
$\widetilde Q$ are particular cases of the operator (\ref{OMX}).
Since $(\mu_2-\widetilde\mu_1)\vert_{\widetilde\Gamma}=0$ the partial Cauchy data ${\mathcal C}_{\mu_2 I, A_2,B_2,q_2}$ and
${\mathcal C}_{\widetilde \mu_1 I,\widetilde A_1,\widetilde B_1,\widetilde q_1}$ are equal.
We multiply the operator $Q$ by the function $\widetilde\mu_1/\mu_2$
and denote the resulting operator as
$\widehat Q(x,D)=\frac{\widetilde\mu_1}{\mu_2}
\widetilde Q(x,D).$
Therefore by Corollary \ref{simplyconnected} there exists
a function  $\eta$ which satisfies (\ref{victoryy}) such that
$L_2(x,D)=e^{-\eta}\widehat Q(x,D)e^{\eta}.$  The proof of the theorem
is completed.  $\square$

{\bf  Proof of Theorem \ref{magia}.}
First we observe that in order to prove the statement of this theorem it suffices to prove it in the case when instead of the
 whole $\tilde \Gamma$ the input and output both are measured on an arbitrary small neighborhood  of the point $\hat x$.

  Now one  can consider only the case when $\tilde \Gamma\subset \{x_1=0\}$ is a small neighborhood of the point $0.$ We observe that if
\begin{equation}\label{isa}
(\sigma^{22}_1-\sigma^{22}_2) \vert_{\tilde \Gamma}= \frac{\partial}{\partial x_2}(\sigma^{22}_1-\sigma^{22}_2)\vert_{\tilde \Gamma}=0
\end{equation}
then repeating the proof of [21] we obtain
\begin{equation}\label{vaz}\sigma_1=\sigma_2\quad \mbox{on}\,\,\tilde \Gamma;\quad \frac{\partial\sigma_1}{\partial x_2}=\frac{\partial\sigma_2}{\partial x_2}\quad \mbox{on}\,\,\tilde \Gamma.
\end{equation}

Let us show that there exists a diffeomorphism
\begin{equation}\label{isa0}F:\Omega\rightarrow \Omega , \quad F(x)\vert_{\tilde \Gamma}=x
 \end{equation}such that for the metric $\tilde \sigma_1=\vert det\, DF^{-1}\vert F^*\sigma_1$ we have
$$
\tilde \sigma_1^{22}=\sigma_2^{22}\quad \mbox{on}\,\,\tilde \Gamma.
$$
First assume that we are already have
\begin{equation}\label{001}
 \sigma_1^{22}(0)=\sigma_2^{22}(0)\quad \mbox{and}\quad \frac{\partial}{\partial x_2}(\sigma^{22}_1-\sigma^{22}_2)(0)=0.
\end{equation}
Let $y=y(x)$ be some diffeomorphism of $\Omega$ into itself. By $x=x(y)$ we denote the inverse mapping.
Then
$$
\tilde \sigma^{22}_1=\sigma_1^{11}(\frac{\partial x_1}{\partial y_2}
)^2/\frac{\partial x_2}{\partial y_2}+2\sigma_1^{12}\frac{\partial
x_1}{\partial y_2}+\sigma^{22}_1\frac{\partial x_2}{\partial y_2}.
$$
This equality and (\ref{001}) immediately imply that as the  perturbation of
 the identity mapping one can construct the diffeomorphism of $\Omega$ into itself which satisfy (\ref{isa0}) such that (\ref{isa}) hold true .

Let us construct the diffeomorphism which satisfies (\ref{isa}) such
that (\ref{001}) holds true. Let $\rho$ be a smooth function such
that $\rho\vert_{\partial\Omega}=0$ and $\rho$ is strictly positive
in $\Omega$ and
$\frac{\partial\rho}{\partial\nu}\vert_{\partial\Omega}<0.$ Consider
the system of ODE
$$
\frac{ dy_1}{dt}=\rho(y)f_1(y), \quad \frac{
dy_2}{dt}=\rho(y)f_2(y).
$$ The  corresponding phase flow $g^s$ is a diffeomorphism of
$\Omega$ into itself such that $g^s(x)=x$ on $\partial\Omega.$ Let
$f_1(0)\ne 0$ and $f_2(0)\ne 0.$  With  appropriate choice of
$f_1,f_2$ one can arrange that $\frac{\partial x_1(0)}{\partial
y_2}=0$ and $\frac{\partial x_2(0)}{\partial
y_2}=\sigma^{22}_2(0)/\sigma^{22}_1(0). $Then the first equality in
(\ref{001}) holds true. Adjusting the second derivatives of $f_1,
f_2$ we can arrange that $\frac{\partial^2 x_1(0)}{\partial
y_2^2}=0$ and $\frac{\partial^2 x_2(0)}{\partial
y_2^2}=\frac{1}{\sigma^{22}_1(0)}(\frac{\partial
\sigma^{22}_2(0)}{\partial y_2}-\frac{\partial
\sigma^{22}_1(0)}{\partial y_2}\frac{\partial x_2(0)}{\partial
y_2}).$ Then the second equality in (\ref{001}) holds true.

Now (\ref{vaz}) is established. The rest of the proof of Theorem 1.2
as in the proof of Theorem 1.1.$\square$

\section{Appendix I}

Consider the following problem for the Cauchy-Riemann equations
\begin{eqnarray}\label{(4.112)}
L(\phi,\psi)=(\frac{\partial \phi}{\partial x_1}-\frac{\partial
\psi}{\partial x_2},\frac{\partial \phi}{\partial
x_2}+\frac{\partial \psi}{\partial x_1}) =0\quad\mbox{in}\,\,\Omega,
\quad \left(\phi,\psi\right)\vert _{\Gamma_0} =(b_1(x),b_2(x)),\\
\frac{\partial^l}{\partial z^l}(\phi+i\psi)(\widehat x_j)=c_{0,j},
\quad \forall j\in \{1,\dots N\}\quad\mbox{and}
\,\, \forall l\in\{0,\dots, 5\}.\nonumber
\end{eqnarray}
Here $\widehat x_1,\dots \widehat x_N$ be an arbitrary fixed points in
$\Omega.$ We consider the pair $b_1,b_2$ and complex numbers
$\vec C=(c_{0,1},c_{1,1},c_{2,1},c_{3,1},c_{4,1},c_{5,1}\dots
c_{0,N},c_{1,N},c_{2,N},c_{3,N},c_{4,N},c_{5,N}) $ as  initial
data for  (\ref{(4.112)}).
The following proposition establishes the solvability of
(\ref{(4.112)}) for a dense set of Cauchy data.
\begin{proposition}\label{MMM1}
There exists a set $\mathcal O\subset C^6(\overline{\Gamma_0})
\times C^6(\overline{\Gamma_0})\times
\Bbb C^{6N}$ such
that for each $(b_1,b_2,\vec C)\in \mathcal O,$  (\ref{(4.112)})
has at least
one solution $(\phi,\psi)\in C^6(\overline\Omega)\times
C^6(\overline{\Omega})$ and
$\overline{\mathcal O}=C^6(\overline{\Gamma_0})\times
C^6(\overline{\Gamma_0})\times \Bbb C^{6N}.$
\end{proposition}

\begin{proof} Denote $B=(b_1,b_2)$ an arbitrary element of the space
$C^7(\overline{\Gamma_0})\times C^7(\overline{\Gamma_0}).$
Consider the following extremal problem
\begin{eqnarray}\label{extr}
J_\epsilon(\phi,\psi)=\left\Vert (\phi,\psi) -B\right\Vert^4
_{B_4^{\frac {27}{4}}(\Gamma_0)}
+ \epsilon \sum_{k=0}^3\Vert\frac{\partial^k (\phi,\psi)}{\partial\nu^k}\Vert_{B_4^{\frac {27}{4}-k}(\partial\Omega)}^4  \\
+ \frac
1\epsilon\left\Vert\Delta^3 L(\phi,\psi)\right\Vert^4_{L^4(\Omega)}
+\sum_{j=1}^N\sum_{k=0}^5\vert\frac{\partial^k}{\partial z^k}
 (\phi+i\psi)(\widehat x_j)-c_{k,j}\vert^2\nonumber
\rightarrow \inf,
\end{eqnarray}
\begin{equation}\label{extr1}
(\phi,\psi)\in W^7_4(\Omega)\times W^7_4(\Omega).
\end{equation}
Here  $B_k^l$ denotes the Besov space of the corresponding orders.

For each $\epsilon>0$ there exists a unique solution to
(\ref{extr}), (\ref{extr1}) which we denote as $(\widehat
\phi_\epsilon,\widehat \psi_\epsilon)$.
This fact can be proved by standard arguments. We fix $\epsilon >0$.
Denote by $\mathcal U_{ad}$ the set of admissible elements of the problem
(\ref{extr}), (\ref{extr1}), namely
$$
\mathcal U_{ad}=\{ (\phi,\psi)\in W^7_4(\Omega)\times W^7_4(\Omega)
\vert J_\epsilon(\phi,\psi)<\infty\}.
$$
Denote $\widehat J_\epsilon=\inf_{(\phi,\psi)\in W^7_4(\Omega)\times W^7_4(\Omega)}
J_\epsilon(\phi,\psi).$
Clearly the pair $(0,0)\in \mathcal U_{ad}.$ Therefore there exists
a minimizing sequence $\{(\phi_k,\psi_k)\}_{k=1}^{\infty}\subset
W^7_4(\Omega)\times W^7_4(\Omega)$ such that
$$
\widehat J_\epsilon=\lim_{k\rightarrow +\infty}J_\epsilon(\phi_k,\psi_k).
$$
Observe that the minimizing sequence is bounded in $W_4^7(\Omega)
\times W_4^7(\Omega).$  Indeed, since the sequence
$\{ \Delta^3 L(\phi_k,\psi_k),L(\phi_k,\psi_k)\vert_{\partial\Omega},
\dots, \frac{\partial^3}{\partial\nu^3} L(\phi_k,\psi_k)\vert
_{\partial\Omega}\}$ is bounded in $L^4(\Omega)\times L^4(\Omega)
\times \Pi_{k=0}^3 B_4^{\frac {27}{4}-k}(\partial\Omega)
\times B_4^{\frac {27}{4}-k}(\partial\Omega) $
the standard elliptic $L^p$-estimate implies that the sequence
$\{ L(\phi_k,\psi_k)\}$ is bounded in the space $W^6_4(\Omega)
\times W^6_4(\Omega).$ Taking into account that  the sequence
traces of the  functions $(\phi_k,\psi_k)$  is bounded in the
Besov space  $B_4^{\frac {27}{4}}(\partial\Omega)\times
B_4^{\frac {27}{4}}(\partial\Omega)$ and applying the estimates
for elliptic operators one more time we obtain that
$\{(\phi_k,\psi_k)\}$ bounded in $W_4^7(\Omega)\times W_4^7(\Omega).$
By the Sobolev imbedding theorem the sequence  $\{(\phi_k,\psi_k)\}$
is bounded in $C^{6}(\overline\Omega)\times C^{6}(\overline\Omega) .$
Then taking if necessary a subsequence, (which we denote again as
$\{(\phi_k,\psi_k)\}$)    we obtain
$$
(\phi_k,\psi_k)\rightarrow (\widehat \phi_\epsilon,\widehat\psi_\epsilon)
\quad \mbox{weakly in}\, W^7_4(\Omega)\times W^7_4(\Omega),
$$
$$
(\frac{\partial^j\phi_k}{\partial\nu^j},
\frac{\partial^j \psi_k}{\partial\nu^j})\rightarrow
(\frac{\partial^j\widehat \phi_\epsilon}{\partial\nu^j},
\frac{\partial^j\widehat\psi_\epsilon}{\partial\nu^j})
\,\, \mbox{weakly in}\,\, B_4^{\frac {27}{4}-j}
(\partial\Omega)\times B_4^{\frac{27}{4}-j}(\partial\Omega)
\quad \forall j\in\{0,1,2,3\},
$$
$$
\frac{\partial^k}{\partial z^k} (\phi+i\psi)(\widehat x_j)-c_{k,j}\rightarrow
C_{k,j,\epsilon}, \quad k\in \{0,\dots,5\},
$$
$$
\Delta^3 L(\phi_k,\psi_k)\rightarrow r_\epsilon\,\, \mbox{weakly in}\,
L^4(\Omega)\times L^4(\Omega),\quad L(\phi_k,\psi_k)\rightarrow
\widetilde r_\epsilon\,\, \mbox{weakly in}\, W^6_4(\Omega)\times
W^6_4(\Omega).
$$
Obviously, $r_\epsilon=\Delta^3 L(\widehat \phi_\epsilon,
\widehat\psi_\epsilon), \widetilde r_\epsilon=L(\widehat \phi_\epsilon,
\widehat\psi_\epsilon).$
Then, since the norms in the spaces $L^4(\Omega)$ and
$B_4^{\frac {27}{4}-k}(\partial\Omega)$ are lower semicontinuous
with respect to  weak convergence
we obtain that
$$
J_\epsilon(\widehat \phi_\epsilon,\widehat\psi_\epsilon)
\le \lim_{k\rightarrow +\infty} J_\epsilon(\phi_k,\psi_k)
=\widehat J_\epsilon.
$$
Thus the pair $(\widehat \phi_\epsilon,\widehat\psi_\epsilon)$ is a
solution to the extremal problem (\ref{extr}), (\ref{extr1}).
Since the set of an admissible elements is convex and the functional
$J_\epsilon$ is strictly convex, this solution is unique.

By Fermat's theorem   we have
$$
J_\epsilon '(\widehat \phi_\epsilon,\widehat \psi_\epsilon)
[\widetilde\delta]=0,\quad \forall\widetilde\delta\in W^7_4(\Omega)\times
W^7_4(\Omega).
$$

This equality can be written in the form
\begin{eqnarray}\label{ip} I_{\Gamma_0,\frac{27}{4}}'
((\widehat \phi_\epsilon,\widehat \psi_\epsilon)-B)[\widetilde\delta]
+ \epsilon \sum_{k=0}^3 I_{\partial\Omega,\frac{27}{4}-k}'
(\frac{\partial^k}{\partial\nu^k}
(\widehat \phi_\epsilon,\widehat \psi_\epsilon))
[\frac{\partial^k}{\partial\nu^k}\widetilde\delta]
+ (p_\epsilon,\Delta^3L\widetilde\delta)_{L^2(\Omega)}\quad\quad\nonumber\\
+ \sum_{j=1}^N\sum_{k=0}^5
(\frac{\partial^k}{\partial z^k}
(\widehat\phi_\epsilon+i\widehat\psi_\epsilon)(\widehat x_j)-c_{k,j})
\overline{\frac{\partial^k}{\partial z^k}
(\widetilde\delta_1+i\widetilde\delta_2)}(\widehat x_j)\nonumber\\
+ \overline{(\frac{\partial^k}{\partial z^k}
(\widehat\phi_\epsilon+i\widehat\psi_\epsilon)
(\widehat x_j)-c_{k,j})}\frac{\partial^k}{\partial z^k}
(\widetilde\delta_1+i\widetilde\delta_2)(\widehat x_j)=0,
\end{eqnarray}
where
$p_\epsilon=\frac{4}{\epsilon} ((\Delta^3(\frac{\partial\widehat \phi_\epsilon}
{\partial x_1}-\frac{\partial
\widehat\psi_\epsilon}{\partial x_2}))^3,(\Delta^3(\frac{\partial
\widehat\phi_\epsilon}{\partial
x_2}+\frac{\partial \widehat\psi_\epsilon}{\partial x_1}))^3)$ and
$I_{\Gamma^*,\kappa}'(\widehat w)$ denotes the derivative of the
functional $w\rightarrow \left\Vert w\right\Vert^4
_{B_4^{\kappa}(\Gamma^*)}$ at $\widehat w.$

Observe that the pair
$J_\epsilon(\widehat\phi_\epsilon,\widehat\psi_\epsilon)\le J_\epsilon(0,0)
=\left\Vert B\right\Vert^4 _{B_4^{\frac {27}{4}}(\Gamma_0)}
+\sum_{j=1}^N\sum_{k=0}^5\vert c_{k,j}\vert^2.$
This  implies that the sequence $\{(\widehat
\phi_\epsilon,\widehat \psi_\epsilon)\}$ is bounded in
$B_4^{\frac {27}{4}}(\Gamma_0)\times B_4^{\frac {27}{4}}(\Gamma_0) $,
the sequences
$\{ \frac{\partial^k}{\partial z^k} (\widehat\phi_\epsilon
+i\widehat\psi_\epsilon)(\widehat x_j)-c_{k,j}\}$ are bounded in $\Bbb C,$
the sequence $\epsilon \sum_{k=0}^3 I_{\partial\Omega,
\frac{27}{4}-k}'(\frac{\partial^k}{\partial\nu^k}(\widehat
\phi_\epsilon,\widehat \psi_\epsilon))[\frac{\partial^k}{\partial\nu^k}
\widetilde\delta]$ converges to zero for any $\widetilde \delta$ from
$B_4^{\frac {27}{4}}(\partial\Omega)\times B_4^{\frac {27}{4}}
(\partial\Omega).$  Then (\ref{ip}) implies that the sequence
$\left\{p_\epsilon\right\}_{\epsilon\in(0,1)}$ is bounded
in $L^\frac43(\Omega)\times L^\frac43(\Omega).$

Therefore there exist $\mathcal B\in B_4^{\frac {27}{4}}(\Gamma_0)\times
B_4^{\frac {27}{4}}(\Gamma_0),$
$C_{0,j},C_{1,j},\dots, C_{5,j}\in \Bbb C$ and
$p = (p_1,p_2)\in L^\frac 43(\Omega)\times L^\frac 43(\Omega)$
such that
\begin{equation}\label{zopak}
(\widehat \phi_\epsilon,\widehat \psi_\epsilon)-B
\rightharpoonup \mathcal B
\quad \mbox{ weakly in} \,\, B_4^{\frac {27}{4}}(\Gamma_0)\times
B_4^{\frac {27}{4}}(\Gamma_0),
\quad p_\epsilon \rightharpoonup p \quad \mbox{ weakly in}
\,\, L^\frac 43(\Omega)\times L^\frac 43(\Omega),
\end{equation}
\begin{equation}\label{zopaa}
\frac{\partial^k}{\partial z^k}
(\widehat\phi_\epsilon+i\widehat\psi_\epsilon)(\widehat x_j)-c_{k,j}
\rightharpoonup C_{k,j}\quad  k\in\{0,1,\dots, 5\}, \,j\in\{1,\dots,N\}.
\end{equation}

Passing to the limit in (\ref{ip}) we obtain
\begin{equation}\label{zima}
I_{\Gamma_0,\frac{27}{4}}'(\mathcal B)[\widetilde\delta]+ (p,\Delta^3
L\widetilde\delta)_{L^2(\Omega)}
+ 2\mbox{Re}\,\sum_{j=1}^N \sum_{k=0}^5
C_{k,j}\overline{\frac{\partial^k}{\partial z^k}
(\widetilde\delta_1+i\widetilde\delta_2)}(\widehat x_j)
= 0\quad \forall \widetilde\delta\in W^7_4(\Omega)\times W^7_4(\Omega).
\end{equation}

Next we claim that
 \begin{equation}\label{Laplacee}
 \Delta^3 p=0\quad\mbox{in}\,\,\Omega\setminus\cup_{j=1}^N\{\widehat x_j\}
 \end{equation}
in the sense of distributions. Suppose that (\ref{Laplacee}) is
already proved. This implies
$$
(p,\Delta^3
L\widetilde\delta)_{L^2(\Omega)}+2\mbox{Re}\, \sum_{j=1}^N\sum_{k=0}^5C_{k,j}
\overline{\frac{\partial^k}{\partial z^k}
(\widetilde\delta_1+i\widetilde\delta_2)}(\widehat x_j)=0 \quad\forall
\widetilde\delta_1,\widetilde\delta_2\in C^\infty_0(\Omega).
$$
If $p=(p_1,p_2)$, denoting $P=p_1-ip_2$, we have
$$
\mbox{Re}\,(\Delta^3 P,\partial_{\overline z}(\widetilde\delta_1+i\widetilde\delta_2))
_{L^2(\Omega)}
+ \mbox{Re}\, \sum_{j=1}^N\sum_{k=0}^5\overline{C_{k,j}}
{\frac{\partial^k}{\partial z^k}
(\widetilde\delta_1+i\widetilde\delta_2)}(\widehat x_j)=0\quad
\forall \widetilde\delta_1,\widetilde\delta_2\in C^\infty_0(\Omega).
$$
Since  by (\ref{Laplacee}) $\mbox{supp}\,\,\Delta^3 P\subset
\cup_{j=1}^N\{\widehat x_j\}$  there exist some constants
$m_{\beta,j}$ and $\widehat \ell_j$ such that $ \Delta^3 P=\sum_{j=1}^N
\sum_{\vert\beta\vert=1}^{\widehat \ell} m_{\beta,j} D^\beta
\delta(x-\widehat x_j).$ The above equality can be written in the form
$$
-\sum_{\vert\beta\vert=1}^{\widehat \ell_j} m_{\beta,j}
\frac{\partial}{\partial \overline z}D^{\beta} \delta(x-\widehat x_j)
= \sum_{k=0}^5(-1)^k\overline{C_{k,j}}
\frac{\partial^k}{\partial z^k}\delta(x-\widehat x_j).
$$
From this we obtain
\begin{equation}\label{qqq} C_{0,j}=C_{1,j}=\dots =C_{5,j}=0
\quad j\in\{1,\dots,N\}.
\end{equation} Therefore
\begin{equation}\label{Laplace}
 \Delta^3 p=0\quad\mbox{in}\,\,\Omega.
 \end{equation}
This implies
$$
(p,\Delta^3 L\widetilde\delta)_{L^2(\Omega)}=0\quad \forall \widetilde\delta
\in W_4^7(\Omega)\times W_4^7(\Omega),
\quad L\widetilde\delta\vert_{\partial\Omega}=
 \frac{\partial L\widetilde\delta}{\partial\nu}\vert_{\partial\Omega}
=\dots
=  \frac{\partial^5 L\widetilde\delta}{\partial\nu^5}\vert
_{\partial\Omega} =0.
$$
This equality and (\ref{zima}) yield
 \begin{equation}\label{zima1}
I_{\Gamma_0,\frac{27}{4}}'(\mathcal B)[\widetilde\delta]=0\quad
\forall \widetilde\delta\in
W_4^7(\Omega)\times W_4^7(\Omega),
\quad L\widetilde\delta\vert_{\partial\Omega}=
 \frac{\partial L\widetilde\delta}{\partial\nu}\vert
_{\partial\Omega}
=\dots  \frac{\partial^5 L\widetilde\delta}{\partial\nu^5}
\vert_{\partial\Omega} =0.
\end{equation}
Then using the trace theorem we conclude that $\mathcal B =0$. Using this
and (\ref{zopak})
we obtain
\begin{equation}\label{II}
(\widehat \phi_{\epsilon_k},\widehat
\psi_{\epsilon_k})-B\rightharpoonup 0\quad\mbox{weakly in}\,\,
{B_4^{\frac {27}{4}}(\Gamma_0)}\times {B_4^{\frac {27}{4}}(\Gamma_0)}.
\end{equation}
From (\ref{zopaa}) and (\ref{qqq}) we obtain
\begin{equation}\nonumber
\frac{\partial^k}{\partial z^k}
(\widehat\phi_\epsilon+i\widehat\psi_\epsilon)(\widehat x)
\rightharpoonup c_{k,j}\quad  k\in\{0,1,\dots, 5\},\,\,j
\in\{1,\dots,N\}.
\end{equation}
By the Sobolev embedding theorem $B_4^{\frac {27}{4}}(\Gamma_0)
\subset \subset C^5(\overline{\Gamma_0})$. Therefore (\ref{II})
implies
\begin{equation}
(\widehat \phi_{\epsilon_k},\widehat \psi_{\epsilon_k})
-B\rightarrow
0\quad\mbox{in}\,\, C^5(\overline{\Gamma_0})\times
C^5(\overline{\Gamma_0}).
\end{equation}
Let  the pair $(\widetilde \phi_{\epsilon_k},\widetilde \psi_{\epsilon_k})$
be a solution to the boundary value problem
\begin{equation}\label{Yava}
L(\widetilde \phi_{\epsilon_k},\widetilde \psi_{\epsilon_k})= L(\widehat
\phi_{\epsilon_k},\widehat\psi_{\epsilon_k})\quad \mbox{in}
\,\,\Omega, \quad \widetilde \psi_{\epsilon_k}\vert_{\partial\Omega}
=\psi_{\epsilon_k}^*.
\end{equation}
Here $\psi_{\epsilon_k}^*$ is a smooth function such that
$\psi_{\epsilon_k}^*\vert_{\Gamma_0}=0$ and the pair
$(L(\widehat \phi_{\epsilon_k},\widehat \psi_{\epsilon_k}),\psi_{\epsilon_k}^*)$
is orthogonal to all solutions of the adjoint problem (see \cite{VE}).
Moreover since $L(\widetilde \phi_{\epsilon_k},\widetilde \psi_{\epsilon_k})
\rightarrow 0$ in $W^6_4(\Omega)\times W^6_4(\Omega)$ we may assume
$\psi_{\epsilon_k}^*\rightarrow 0$ in $C^6(\partial\Omega)\times
C^6(\partial\Omega).$
Among all possible solutions to problem (\ref{Yava}) (clearly
there is no unique solution to this problem) we choose one
such that $\int_\Omega \widetilde \phi_{\epsilon_k} dx=0.$
Thus we obtain
\begin{equation}
(\widetilde \phi_{\epsilon_k},\widetilde \psi_{\epsilon_k})\rightarrow 0
\quad\mbox{in}\,\, W^7_4(\Omega)\times  W^7_4(\Omega).
\end{equation}
Therefore the sequence $\{(\widehat \phi_{\epsilon_k}
-\widetilde \phi_{\epsilon_k},\widehat \psi_{\epsilon_k}-\widetilde
\psi_{\epsilon_k})\}$
represents the desired approximation to the solution of the Cauchy
problem (\ref{(4.112)}).

Now we prove (\ref{Laplacee}). Let $\widetilde x$ be an arbitrary
point in $\Omega\setminus\cup_{j=1}^N\{\widehat x_j\}$ and let $\widetilde
\chi$ be a smooth function
such that it is zero in some neighborhood of
$ \Gamma_0\cup\cup_{j=1}^N \{\widehat x_j\}$ and the set $\mathcal
D=\{x\in \Omega\vert \widetilde \chi(x)=1\}$ contains an open
connected subset $\mathcal F$ such that $\widetilde x\in \mathcal F$
and $ \widetilde \Gamma\cap \overline{\mathcal F}$ is an open set
in $\partial\Omega$.
In addition we assume that Int$(\supp\,\chi)$ is a simply connected
domain. By (\ref{zima}) we have
\begin{equation}\label{vorona}
0=(p,\Delta^3 L(\widetilde\chi\widetilde\delta))_{L^2(\Omega)}
=(\widetilde \chi
p,\Delta^3 L\widetilde \delta)_{L^2(\Omega)} +(p,[\Delta^3
L,\widetilde\chi]\widetilde\delta)_{L^2(\Omega)}\quad \forall
\widetilde\delta\in W^7_4(\Omega)\times  W^7_4(\Omega).
\end{equation}

Denote $L\widetilde\delta=\widehat \delta.$
Consider the functional mapping $\widehat \delta \in
W^2_4(\supp \widetilde \chi)$ to
$(p,[\Delta^3 L,\widetilde\chi]\widetilde\delta)_{L^2(\Omega)}$, where
$$
L\widetilde\delta=\widehat \delta \quad \mbox{in}\,\Omega, \quad \mbox{Im}\,
\widetilde\delta\vert_{\mathcal S}=0, \int_{\supp \widetilde \chi}\mbox{Re}
\,\widetilde \delta dx=0,
$$
where $\mathcal S$ denotes the boundary of $\supp \widetilde \chi.$
For each $\widehat \delta \in W^2_4(\supp \widetilde \chi)\times
W^2_4(\supp \widetilde \chi) $, there exists
a unique solution $\widetilde\delta \in W^3_4(\supp \widetilde \chi)
\times  W^3_4(\supp \widetilde \chi)$.
Hence the functional is well-defined and continuous on
$W^2_4(\supp \widetilde \chi)$. Therefore there exist
${\bf q}, {\bf r}, q_0\in L^\frac 43(\supp \widetilde \chi)$ such that
$\int_{\supp \widetilde \chi} (\sum_{j,k=1}^2r_{jk}
\frac{\partial^2}{\partial x_j\partial x_k}\widehat \delta
+ ({\bf q},\widehat \delta)+q_0 \widehat\delta )dx=(p,[\Delta^3 L,
\widetilde\chi]\widetilde\delta)_{L^2(\supp \widetilde \chi)}$.

Consider the boundary value problem
$$
\Delta^3 \widetilde P=\widetilde f\quad\mbox{in}\,\, \supp \chi,
\quad\widetilde P\vert_{\mathcal S}
= \frac{\partial\widetilde P}{\partial\nu}\vert_{\mathcal S}
= \frac{\partial^2\widetilde P}{\partial\nu^2}\vert_{\mathcal S}=0.
$$
Here $\widetilde f=2\mbox{div }(\nabla\widetilde {\bf q})
- q_0-\sum_{j,k=1}^2 \frac{\partial^2}{\partial x_j\partial x_k}
r_{jk}$.  A solution to this problem exists
and is unique, since $\widetilde f\in (\mathring W_4^2(\supp
\widetilde \chi))'.$ Then $ P\in \mathring W^1_\frac 43(\supp
\widetilde \chi)\times \mathring W^1_\frac 43(\supp \widetilde \chi).$
On the other hand, thanks to (\ref{vorona}), $P=\widetilde \chi p\in
\mathring W^1_\frac 43(\supp \widetilde \chi)\times
\mathring W^1_\frac 43(\supp \widetilde \chi).$


Next we take another smooth cut off function $\widetilde\chi_1$ such
that $\mbox{supp}\,\widetilde\chi_1\subset\mathcal D$ and
Int $(\supp \,\chi_1)$ is a simply connected domain.  A
neighborhood of $\widetilde x$ belongs to  $\mathcal D_1=\{x\vert
\widetilde\chi_1=1\}$, the interior of $\mathcal D_1$ is connected,
and $\overline{\mbox{ Int }\mathcal D_1}\cap \widetilde\Gamma$
contains an open subset $\mathcal O$ in $\partial\Omega.$  Similarly to
(\ref{vorona}) we have
\begin{equation}\nonumber
(\widetilde \chi_1 p,\Delta^3 L\widetilde\delta)_{L^2(\Omega)}-(p,[\Delta^3
L,\widetilde\chi_1]\widetilde\delta)_{L^2(\Omega)}=0 \quad \forall \widetilde\delta
\in W^7_4(\Omega)\times  W^7_4(\Omega).
\end{equation}

This equality implies that $\widetilde\chi_1 p\in W^2_{\frac 43}(\Omega)
\times W^2_{\frac 43}(\Omega)$, using a similar argument to
the one above.

Next we take another smooth cut off function $\widetilde\chi_2$ such
that $\mbox{supp}\,\widetilde\chi_2\subset\mathcal D_2$ and
Int $(\supp \,\chi_2)$ is a simply connected domain.
A neighborhood of $\widetilde x$ belongs to  $\mathcal D_3=\{x\vert
\widetilde\chi_2=1\}$, the interior of $\mathcal D_1$ is connected,
and $\overline{\mbox{ Int }\mathcal D_3}\cap \widetilde\Gamma$ contains an
open subset $\mathcal O$ in $\partial\Omega.$  Similarly to
(\ref{vorona}) we have
\begin{equation}\nonumber
(\widetilde \chi_2 p,\Delta^3 L\widetilde\delta)_{L^2(\Omega)}
-(p,[\Delta^3
L,\widetilde\chi_2]\widetilde\delta)_{L^2(\Omega)}=0 \quad \forall
\widetilde\delta\in W^7_4(\Omega)\times  W^7_4(\Omega).
\end{equation}

This equality implies that $\widetilde\chi_2 p\in W^3_{\frac 43}
(\Omega)\times W^3_{\frac 43}(\Omega)$, using a similar argument
to the one above. Let
$\omega$ be a domain such that $\omega\cap\Omega=\emptyset$,
$\partial\omega\cap\partial\Omega\subset\mathcal O$ contains an
open set in $\partial\Omega.$

We extend $p$ on $\omega$ by zero. Then
$$
(\Delta^3(\widetilde \chi_2
p),L\widetilde\delta)_{L^2(\Omega\cup\omega)}+(p,[\Delta^3
L,\widetilde\chi_2]\widetilde\delta)_{L^2(\Omega\cup\omega)}=0.
$$
Hence, since $[\Delta^3
L,\widetilde\chi_2]\vert_{\mathcal D_1}=0$ we have
$$
L^*\Delta^3(\widetilde\chi_2 p)=0 \quad\mbox{in} \,\,\mbox{ Int
}\mathcal D_2\cup\omega,\quad p\vert_\omega=0.
$$
By Holmgren's theorem $\Delta^3(\widetilde\chi_2 p)\vert _{\mbox
{Int }\mathcal D_1}=0$, that is, $(\Delta^3 p)(\widetilde x)=0.$
Thus (\ref{Laplacee}) is proved.
\end{proof}

Consider the Cauchy problem for the Cauchy-Riemann equations
\begin{eqnarray}\label{(4.112I)}
L(\phi,\psi)=(\frac{\partial \phi}{\partial x_1}-\frac{\partial
\psi}{\partial x_2},\frac{\partial \phi}{\partial
x_2}+\frac{\partial \psi}{\partial x_1}) =0\quad\mbox{in}\,\,\Omega,
\quad \left(\phi,\psi\right)\vert _{\Gamma_0} =(b(x),0),\\
\frac{\partial^l}{\partial z^l}(\phi+i\psi)(\widehat x_j)=c_{0,j},
\quad \forall j\in \{1,\dots N\}\quad\mbox{and} \,\, \forall
l\in\{0,\dots, 5\}.\nonumber
\end{eqnarray}
Here $\widehat x_1,\dots \widehat x_N$ be an arbitrary fixed points in $\Omega.$
We consider the function $b$ and complex numbers $\vec C=(c_{0,1},c_{1,1},
c_{2,1},c_{3,1},c_{4,1},c_{5,1}\dots c_{0,N},c_{1,N},c_{2,N},c_{3,N},
c_{4,N},c_{5,N}) $ as  initial data for  (\ref{(4.112)}).
The following proposition establishes the solvability of
(\ref{(4.112)}) for a dense set of Cauchy data.
\begin{corollary}\label{MMM1I}
There exists a set $\mathcal O\subset C^6(\overline{\Gamma_0})\times
\Bbb C^{6N}$ such
that for each $(b,\vec C)\in \mathcal O,$ problem  (\ref{(4.112I)}) has at
least
one solution $(\phi,\psi)\in C^6(\overline\Omega)
\times C^6(\overline\Omega)$ and
$\overline{\mathcal O}=C^6(\overline{\Gamma})\times \Bbb C^{6N}.$
\end{corollary}

The proof of Corollary \ref{MMM1I} is similar to the proof of
Proposition \ref{MMM1}. The only difference is that instead
of extremal problem considered there we have to use the following
extremal
problem
\begin{eqnarray}
J_\epsilon(\phi,\psi)=\left\Vert \phi -b\right\Vert^4
_{B_4^{\frac {27}{4}}(\Gamma_0)}
+\epsilon \sum_{k=0}^3\Vert\frac{\partial^k (\phi,\psi)}{\partial\nu^k}
\Vert_{B_4^{\frac {27}{4}-k}(\partial\Omega)}^4 \nonumber \\+ \frac
1\epsilon\left\Vert\Delta^3 L(\phi,\psi)\right\Vert^4_{L^4(\Omega)}
+ \sum_{j=1}^N\sum_{k=0}^5\vert \frac{\partial^k}{\partial z^k}
(\phi+i\psi)(\widehat x_j)-c_{k,j}\vert^2\nonumber
\rightarrow \inf,\nonumber
\end{eqnarray}
\begin{equation}\nonumber
(\phi,\psi)\in W^7_4(\Omega)\times  W^7_4(\Omega),\quad \psi\vert
_{\Gamma_0}=0.
\end{equation}

We have
\begin{proposition} \label{balalaika1}
Let $\alpha\in (0,1),$ $\mathcal A,\mathcal B\in C^{6+\alpha}
(\overline \Omega)$, $y_1,\dots, y_{\widehat k}$ be
some arbitrary points from $\Gamma_0$, $y_{\widehat k+1},
\dots, y_{\widetilde k}$ be
some arbitrary points from $\Omega$ and $\widetilde x$ be an arbitrary
point from $\Omega\setminus\{y_{1},\dots, y_{\widetilde k}\}.$
Then there exist
a holomorphic function $a \in C^{5+\alpha}(\overline \Omega)$ and
an antiholomorphic function $d \in C^{5+\alpha}(\overline \Omega)$
such that $(ae^{\mathcal A}+de^{\mathcal B})\vert_{\Gamma_0}=0,$
$$
\frac{\partial^{k+j} a}{\partial x_1^k\partial x_2^j} (y_{\ell})=0
\quad k+j\le 5, \,\,\forall \ell\in\{1,\dots,\widetilde k\},
$$
and
$$
a(\widetilde x)\ne 0\quad\mbox{and} \quad d({\widetilde x})\ne 0.
$$
\end{proposition}
\begin{proof}
Consider the operator
$$
R(\gamma)=(a(y_1),\dots
\frac{\partial^5 a}{\partial z^5}(y_1), \dots ,
a(y_{\widetilde k}),\dots,
\frac{\partial^5 a}{\partial z^5}(y_{\widetilde k}),\newline  d(y_1),
\dots \frac{\partial^5 d}{\partial\overline z^5}(y_1),
\dots, d(y_{\widetilde k}),\dots \frac{\partial^5 d}
{\partial\overline z^5}(y_{\widetilde k}), a(\widetilde x), d(\widetilde x)).
$$
Here $\gamma\in C^\infty_0(\widetilde \Gamma)$ and the functions
$a$ and $d$ are solutions to the following problem
$$
\frac{\partial a}{\partial \overline z}=0,\quad\mbox{in}\,\,\Omega,
\quad
\frac{\partial d}{\partial z} =0\quad\mbox{in}\,\,\Omega,
\quad (ae^{\mathcal A}+de^{\mathcal B})\vert
_{\partial\Omega}=\gamma.
$$
Consider the image of the operator $R$. Clearly it is closed.
Let us show that the point $(0,\dots, 0, 1,1)$ belongs to the
image of the operator $R.$
Let a holomorphic function $a$ satisfy
$$
\frac{\partial^\beta a}
{\partial {x_1}^{\beta_1} \partial {x_2}^{\beta_2}} (y_j)=0\quad
\forall \vert \beta\vert\in \{0,\dots, 5\},\quad j\in\{1,\dots,
\widetilde k\},\quad \vert a(\widetilde x)\vert>2.
$$
Consider the function $-e^{\mathcal A-\mathcal B}a(z)$ and
the pair $(b_1,b_2)=(\mbox{Re}\{-e^{\mathcal A-\mathcal B}a\},
\mbox{Im}\{e^{\mathcal A-\mathcal B}a\}).$  Using Proposition
\ref{MMM1} we solve problem
(\ref{(4.112)}) with $l=0$ approximately.
Let $(\phi_\epsilon,\psi_\epsilon)$
be a sequence of functions such that
$$
\frac{\partial}{\partial z}(\phi_\epsilon+i\psi_\epsilon)=0\quad
\mbox{in}\,\Omega, \quad (\phi_\epsilon,\psi_\epsilon)\vert_{\Gamma_0}
\rightarrow (b_1,b_2)\quad \mbox{in}\,\,
C^{5+\alpha}(\overline\Gamma_0), \quad (\phi_\epsilon+i\psi_\epsilon)
(\widetilde x)\rightarrow 1.
$$
Denote $d_\epsilon=\phi_\epsilon-i\psi_\epsilon,
\beta_\epsilon=ae^{\mathcal A}+d_\epsilon e^{\mathcal B}$.
Then the sequence $\{\beta_\epsilon\}$ converges to zero in the
space $C^{5+\alpha}(\Gamma_0).$

By Proposition \ref{zika} there exists a solution to problem
(\ref{xoxol}) with the initial data $\beta_\epsilon $, which
we denote as $\{\widetilde a_\epsilon, \widetilde d_\epsilon\}$ such that
the sequence $\{\widetilde a_\epsilon, \widetilde d_\epsilon\}$ converges
to zero in $(C^5(\overline\Omega))^2.$
Denote by $\gamma_\epsilon=(a+\widetilde a_\epsilon,d_\epsilon+\widetilde
d_\epsilon)\vert_{\Gamma_0}.$ Clearly $R(\gamma_\epsilon)$
converges to $(0,\dots, 0,1,1).$
The proof of the proposition is completed.
\end{proof}
\section{Appendix II. Asymptotic formulas}
%
Here we recall that we identify $x = (x_1,x_2) \in \R^2$
with $z = x_1 + ix_2 \in \C$.
\begin{proposition}\label{zinka}
Under the conditions of Theorem \ref{main}  for any point $x$  on the
boundary of $\Omega$ we have
\begin{eqnarray}\label{ikk}
-\frac{1}{\pi}\int_\Omega   \frac{e_1\widetilde g_1 e^{-\tau(\Phi(\zeta)
-\overline{\Phi(\zeta)})}}{\zeta- z}d\xi_1d\xi_2
= \frac{1}{\tau^2} \frac{e^{-2i\tau\psi({\widetilde x})}}
{\vert \mbox{det}\, \psi ''({\widetilde x})\vert^\frac 12}\left(\frac{\frac{\partial_{ z}
\widetilde g_1 ({\widetilde x})}{\partial^2_z\Phi({\widetilde x})} }
{(\widetilde z- z)^2}\right.\\
\left.
+\frac{\frac{\partial_{z}\widetilde g_1 ({\widetilde x})}
{{\partial^2_z\Phi({\widetilde x})}}\frac{\partial_z^3\Phi({\widetilde x})}
{\partial^2_z\Phi({\widetilde x})}
+ \frac{\partial^2_{\overline z}\widetilde g_1 ({\widetilde x})}
{\overline{\partial^2_z
\Phi({\widetilde x})}}-\frac{\partial^2_{ z}\widetilde g_1 ({\widetilde x})}
{{\partial^2_z\Phi({\widetilde x})}}}{2(\widetilde z- z)}
\right )+o(\frac{1}{\tau^2}) \quad \mbox{as}
\,\,\vert\tau\vert\rightarrow +\infty.\nonumber
\end{eqnarray}

\begin{eqnarray}\label{KKK}
-\frac{1}{\pi}\int_\Omega   \frac{e_1\widetilde g_2 e^{\tau(\Phi(\zeta)
-\overline{\Phi(\zeta)})}}{\overline\zeta- \overline z}d\xi_1d\xi_2=
\frac{1}{\tau^2} \frac{e^{2i\tau\psi({\widetilde x})}}
{\vert \mbox{det}\, \psi ''({\widetilde x})\vert^\frac 12}
\left (\frac{\frac{\partial_{\overline  z}
\widetilde g_2 ({\widetilde x})}{\overline{\partial^2_z\Phi({\widetilde x})}}}
{(\overline{\widetilde z}-\overline z)^2}\right. \\
\left.
+\frac{
\frac{\partial_{\overline z}\widetilde g_2 ({\widetilde x})}
{\overline{\partial^2_z\Phi({\widetilde x})}}\frac{\partial^3_{\overline z}
\overline\Phi({\widetilde x})}{\partial^2_{\overline z}
\overline\Phi({\widetilde x})}
- \frac{\partial^2_{\overline z}\widetilde g_2 ({\widetilde x})}
{\overline{\partial^2_z\Phi({\widetilde x})}}
+\frac{\partial^2_{ z}\widetilde g_2 ({\widetilde x})}
{\partial^2_z\Phi({\widetilde x})}}
{2(\overline {\widetilde z}-\overline z)} \right)
+ o(\frac{1}{\tau^2})
\quad\qquad\qquad \mbox{as}\,\,\vert\tau\vert\rightarrow
+\infty.\nonumber
\end{eqnarray}

\begin{eqnarray}\label{iks}
-\frac{1}{\pi}\int_\Omega
\frac{e_1\widetilde g_3 e^{-\tau(\Phi(\zeta)
-\overline{\Phi(\zeta)})}}{\overline\zeta- \overline z}d\xi_1d\xi_2=
\frac{1}{\tau^2}\frac{e^{-2i\tau\psi({\widetilde x})}}
{\vert \mbox{det}\, \psi ''({\widetilde x})\vert^\frac 12}
\left (\frac{\frac{\partial_{\overline  z}
\widetilde g_3 ({\widetilde x})}
{\overline{\partial^2_z\Phi({\widetilde x})}}}
{(\overline{\widetilde z}-\overline z)^2}\right.\\
\left.
+ \frac{
\frac{\partial_{\overline z}\widetilde g_3 ({\widetilde x})}
{\overline{\partial^2_{\overline z}\overline \Phi({\widetilde x})}}
\frac{\partial_{\overline z}^3\overline{\Phi}({\widetilde x})}
{\partial^2_{\overline z}\overline{\Phi}({\widetilde x})}
- \frac{\partial^2_{\overline z}\widetilde g_3 ({\widetilde x})}
{\overline{\partial^2_z\Phi({\widetilde x})}}
+ \frac{\partial^2_{ z}\widetilde g_3 ({\widetilde x})}
{\partial^2_z\Phi({\widetilde x})}}
{2(\overline{\widetilde z}- \overline z)}\right)
+ o(\frac{1}{\tau^2})
\qquad\qquad\quad \mbox{as}\,\,\vert\tau\vert\rightarrow
+\infty.\nonumber
\end{eqnarray}

\begin{eqnarray}\label{iks-1}
-\frac{1}{\pi}\int_\Omega   \frac{e_1\widetilde g_4
e^{\tau(\Phi(\zeta)-\overline{\Phi(\zeta)})}}{\zeta- z}d\xi_1d\xi_2
= \frac{1}{\tau^2} \frac{e^{2i\tau\psi({\widetilde x})}}
{\vert \mbox{det}\, \psi ''({\widetilde x})\vert^\frac 12}
\left(\frac{
\frac{\partial_{ z}\widetilde g_4({\widetilde x})}
{{\partial^2_z\Phi({\widetilde x})}}}
{(\widetilde z-z)^2} \right.
\nonumber\\
\left.
+ \frac{
\frac{\partial_{ z}\widetilde g_4 ({\widetilde x})}
{{\partial^2_z\Phi({\widetilde x})}}\frac{\partial^3_z\Phi({\widetilde x})}
{\partial_z^2\Phi({\widetilde x})}
+ \frac{\partial^2_{\overline z}\widetilde g_4 ({\widetilde x})}
{\overline{\partial^2_z\Phi({\widetilde x})}}
-\frac{\partial^2_{ z}\widetilde g_4 ({\widetilde x})}
{\partial^2_z\Phi({\widetilde x})}}
{2(\widetilde z-z)} \right)
+ o(\frac{1}{\tau^2})
\qquad\qquad\quad \mbox{as}\,\,\vert\tau\vert\rightarrow +\infty.
\end{eqnarray}
\end{proposition}
\begin{proof}
Let $\delta>0$ be a sufficiently small number and $\widetilde e\in
C_0^\infty(B(\widetilde x,\delta)), \widetilde e\vert
_{B(\widetilde x,\delta/2)}\equiv 1$. Let
$\widetilde g\in C^2(\overline \Omega)$
be some function such that $\widetilde g(\widetilde x)
=\frac{\partial \widetilde g}
{\partial \overline z}(\widetilde x)=0.$
We compute the asymptotic formulae of the following integral as
$\vert\tau\vert$ goes to infinity.
\begin{eqnarray}
&-&\frac{1}{\pi}\int_\Omega   \frac{e_1\widetilde g
e^{-\tau(\Phi(\zeta)-\overline{\Phi(\zeta)})}}{\zeta- z}d\xi_1d\xi_2=
-\frac{1}{\pi} \int_{B({\widetilde x},\delta)}
\frac{\widetilde e \widetilde g
e^{-\tau(\Phi(\zeta)-\overline{\Phi(\zeta)})}}{\zeta- z}
d\xi_1d\xi_2 + o(\frac {1}{\tau^2})=\nonumber\\
&-&\frac{1}{\pi} \int_{B({\widetilde x},\delta)}
\widetilde e\left\{ \frac{ \partial_{z}\widetilde g ({\widetilde x})
(\zeta-{\widetilde z}) + \frac 12 \partial^2_{\overline z}\widetilde g
({\widetilde x})(\overline \zeta-\overline{\widetilde z})^2}{\zeta- z}\right.
\nonumber\\
&+&\left.\frac{\partial_{\overline z}\partial_{z}\widetilde g ({\widetilde x})
(\zeta-{\widetilde z}) (\overline \zeta-\overline{\widetilde z})
+ \frac 12\partial^2_{ z}\widetilde g ({\widetilde x}) (\zeta-{\widetilde z})^2}
{\zeta- z}\right\}e^{-\tau(\Phi(\zeta)-\overline{\Phi(\zeta)})}
d\xi_1d\xi_2 + o(\frac{1}{\tau^2})=\nonumber
\end{eqnarray}
\begin{eqnarray}
&-&\frac{1}{\pi}\int_{B({\widetilde x},\delta)}
\widetilde e\left\{ \frac{ \frac{\partial_{z}\widetilde g ({\widetilde x})}
{{\partial^2_z\Phi({\widetilde x})}}(\partial_{ \zeta}\Phi
- \frac 12\partial^3_{ z}\Phi({\widetilde x})( \zeta-{\widetilde z})^2)
+ \frac 12\frac{\partial^2_{\overline z}\widetilde g ({\widetilde x})}
{\overline{\partial^2_z\Phi({\widetilde x})}}
\partial_{\overline \zeta}\overline\Phi(\overline \zeta-\overline{\widetilde z})}
{\zeta- z}\right.\nonumber\\
&+&\left.\frac{\frac{\partial_{\overline z}\partial_{z}\widetilde g
({\widetilde x})}
{\overline{\partial^2_z\Phi({\widetilde x})}}(\zeta-{\widetilde z})
\partial_{\overline \zeta}\overline\Phi
+ \frac 12\frac{\partial^2_{ z}\widetilde g ({\widetilde x})}
{\partial^2_z\Phi({\widetilde x})}\partial_\zeta\Phi (\zeta-{\widetilde z}) }
{\zeta- z}\right\}e^{-\tau(\Phi(\zeta)-\overline{\Phi(\zeta)})}
d\xi_1d\xi_2 + o(\frac{1}{\tau^2})=\nonumber\\
&-&\frac{1}{\pi} \int_{B({\widetilde x},\delta)}\widetilde e
\left\{ \frac{ \frac{\partial_{z}\widetilde g ({\widetilde x})}
{{\partial^2_z\Phi({\widetilde x})}}(\partial_{ \zeta}\Phi-\frac 12
\frac{\partial^3_z\Phi({\widetilde x})}{\partial^2_z\Phi({\widetilde x})}
\partial_{ \zeta}\Phi( \zeta-{\widetilde z}))+\frac 12\frac{\partial^2
_{\overline z}
\widetilde g ({\widetilde x})}{\overline{\partial^2_z\Phi({\widetilde x})}}
\partial_{\overline \zeta}\overline\Phi(\overline \zeta-\overline{\widetilde z})}
{\zeta- z}\right.\nonumber\\
&+& \left.\frac{\frac{\partial_{\overline z}\partial_{z}\widetilde g
({\widetilde x})}{\overline{\partial^2_z\Phi({\widetilde x})}}
(\zeta-{\widetilde z})
\partial_{\overline \zeta}\overline\Phi
+ \frac 12\frac{\partial^2_{ z}\widetilde g ({\widetilde x})}
{\partial^2_z\Phi({\widetilde x})}\partial_\zeta\Phi
(\zeta-{\widetilde z}) }{\zeta- z}\right\}e^{-\tau(\Phi(\zeta)
-\overline{\Phi(\zeta)})} d\xi_1d\xi_2
+ o(\frac{1}{\tau^2})=\nonumber\\
&-&\frac{1}{\pi\tau} \int_{B({\widetilde x},\delta)}
\widetilde e\left\{\frac{
-\frac{\partial_{z}\widetilde g ({\widetilde x})}
{{\partial^2_z\Phi({\widetilde x})}}\frac{\partial_z^3\Phi({\widetilde x})}
{\partial^2_z\Phi({\widetilde x})}
- \frac{\partial^2_{\overline z}\widetilde g ({\widetilde x})}
{\overline{\partial^2_z\Phi({\widetilde x})}}}{2(\zeta- z)}
- \frac{\frac{\partial_{ z}\widetilde g ({\widetilde x})}
{{\partial^2_z\Phi({\widetilde x})}}(1-\frac 12\frac{\partial^3_z
\Phi({\widetilde x})}{\partial^2_z\Phi({\widetilde x})}(\zeta-{\widetilde z}))}
{(\zeta-z)^2}
\right.\nonumber\\
&+&\left.\frac{\frac{\partial^2_{ z}\widetilde g ({\widetilde x})}
{{\partial^2_z\Phi({\widetilde x})}} }{2(\zeta- z)}
- \frac{\frac{\partial^2_{ z}\widetilde g
({\widetilde x})}{\partial^2_z\Phi({\widetilde x})}
(\zeta-{\widetilde z})}{2(\zeta- z)^2}\right\}
e^{-\tau(\Phi(\zeta)-\overline{\Phi(\zeta)})}d\xi_1d\xi_2
+ o(\frac{1}{\tau^2}).\nonumber\\
\end{eqnarray}
Here we used
$$
\int_{B({\widetilde x},\delta)}\widetilde e \frac{ (\zeta-{\widetilde z})}
{(\zeta- z)^2}e^{-\tau(\Phi(\zeta)-\overline{\Phi(\zeta)})}
d\xi_1d\xi_2
= o(\frac{1}{\tau})\quad \mbox{as}\,\,\vert\tau\vert\rightarrow
+\infty,
$$
which is obtained by  stationary phase. Another asymptotic
calculation is

\begin{eqnarray}
-\frac{1}{\pi}\int_\Omega
\frac{e_1\widetilde g e^{\tau(\Phi(\zeta)-\overline{\Phi(\zeta)})}}
{\zeta- z}d\xi_1d\xi_2
= \frac{1}{\pi\tau}
\int_{B({\widetilde x},\delta)}\widetilde e
\frac{\frac{\partial_{z}\widetilde g ({\widetilde x})}
{\partial^2_z\Phi({\widetilde x})}}
{(\zeta- z)^2}e^{-\tau(\Phi(\zeta)-\overline{\Phi(\zeta)})}dx
           \nonumber\\
+\frac{1}{2\pi\tau} \int_{B({\widetilde x},\delta)}
\widetilde e\frac{\frac{\partial_{z}\widetilde g ({\widetilde x})}
{{\partial^2_z\Phi({\widetilde x})}}\frac{\partial_z^3\Phi({\widetilde x})}
{\partial^2_z\Phi({\widetilde x})}
+ \frac{\partial^2_{\overline z}\widetilde g ({\widetilde x})}
{\overline{\partial^2_z \Phi({\widetilde x})}}
- \frac{\partial^2_{ z}\widetilde g ({\widetilde x})}
{{\partial^2_z\Phi({\widetilde x})}}}{(\zeta- z)}
e^{-\tau(\Phi(\zeta)-\overline{\Phi(\zeta)})}dx+ o(\frac{1}{\tau^2})
=\nonumber\\
\frac{1}{\tau^2}\frac{e^{-2i\tau\psi({\widetilde x})}}
{\vert \mbox{det}\, \psi''({\widetilde x})\vert^\frac 12}
\left(\frac{\frac{\partial_{z}
\widetilde g ({\widetilde x})}{\partial^2_z\Phi({\widetilde x})} }
{(\widetilde z- z)^2}
+\nonumber
\frac{\frac{\partial_{z}\widetilde g ({\widetilde x})}{{\partial^2_z
\Phi({\widetilde x})}}\frac{\partial_z^3\Phi({\widetilde x})}
{\partial^2_z\Phi({\widetilde x})}
+
\frac{\partial^2_{\overline z}\widetilde g ({\widetilde x})}
{\overline{\partial^2_z
\Phi({\widetilde x})}}-\frac{\partial^2_{ z}\widetilde g ({\widetilde x})}
{{\partial^2_z\Phi({\widetilde x})}}}{2(\widetilde z- z)}\right)\\
+ o(\frac{1}{\tau^2})
 \quad \mbox{as}
\,\,\vert\tau\vert\rightarrow +\infty.
\end{eqnarray}
Taking $\widetilde g=g_1$ and $\widetilde g=\overline g_3$ we obtain
(\ref{ikk}) and (\ref{iks}) from the above formula.
Taking $\widetilde g=g_4 $ or $\widetilde g=\overline g_2$ and replacing  $\tau$
by $-\tau$ we obtain (\ref{iks-1}) and (\ref{KKK}).

\end{proof}
\begin{proposition}\label{nuvo}
For any $x$ from the boundary of $\Omega$,
the following asymptotic formulae holds true as $\vert \tau\vert$
goes to $+\infty:$
\begin{equation}\label{vova}
\frak G_1(x,\tau)=-
\frac{e^{-2i\tau\psi({\widetilde x})}}{2\tau\vert \mbox{det}\, \psi''({\widetilde x})\vert^\frac 12}\left(\frac{\frac{\partial
(g_1e^{-\mathcal A_1})}{\partial z}({\widetilde x})+\frac{\partial\Phi}
{\partial z} m_{1}(\widetilde x)}{\widetilde z-z}+\frac{\sigma_{1}(\widetilde x)\frac{\partial\Phi}
{\partial z} }{(\widetilde z-z)^2}\right )
+ o(\frac{1}{\tau}),
\end{equation}
\begin{equation}\label{vova1}
\frak G_2(x,\tau)=-\frac{e^{2i\tau\psi({\widetilde x})}}
{2\tau\vert \mbox{det}\, \psi''({\widetilde x})
\vert^\frac 12}\left (\frac{\frac{\partial
(g_2e^{-\mathcal B_1})}{\partial\overline z}({\widetilde x})
+\frac{\partial\overline\Phi}{\partial \overline z}
\widetilde m_{1}(\widetilde x)}{\overline{\widetilde z}-\overline z}
+\frac{\widetilde \sigma_{1}(\widetilde x)\frac{\partial\overline\Phi}
{\partial\overline z} }{(\overline{\widetilde z}-\overline z)^2}\right )
+o(\frac 1\tau),
\end{equation}
\begin{equation}\label{vova2}
\frak G_3(x,\tau)=-
\frac{e^{-2i\tau\psi({\widetilde x})}}{2\tau\vert \mbox{det}\,
\psi''({\widetilde x})\vert^\frac 12}\left(
\frac{\frac{\partial(g_3e^{-\mathcal A_2})}{\partial \overline z}
({\widetilde x})-\frac{\partial\overline\Phi}{\partial \overline z} t_{1}(\widetilde x)}
{\overline{\widetilde z}-\overline z}-\frac{r_{1}(\widetilde x)
\frac{\partial\overline\Phi}{\partial \overline z} }
{(\overline{\widetilde z}-\overline z)^2}\right)+o(\frac 1\tau),
\end{equation}
\begin{equation}\label{vova3}
\frak G_4( x,\tau)=-\frac
{e^{2i\tau\psi({\widetilde x})}}{2\tau\vert \mbox{det}\, \psi''({\widetilde x})\vert^\frac 12}\left (\frac{
\frac{\partial(g_4e^{-\mathcal B_2})}{\partial z}({\widetilde x})
-\frac{\partial\Phi}{\partial  z} \widetilde t_{1}(\widetilde x)}{\widetilde z-z}
-\frac{\widetilde r_{1}(\widetilde x)\frac{\partial\Phi}{\partial  z} }
{({\widetilde z}- z)^2}\right)
+ o(\frac 1\tau).
\end{equation}
Here $\widetilde z=\widetilde x_{1}+i\widetilde x_{2}$ and $m_1,\widetilde m_1,
\sigma_1,\widetilde \sigma_1, t_1,\widetilde t_1, r_1,\widetilde r_1$ are
introduced in (\ref{v1}),(\ref{v2}), (\ref{v3}) and (\ref{v4}).
Moreover for  a sufficiently small positive $\epsilon$ the
following asymptotic formula holds true
\begin{equation}\label{luna}
\Vert\frac{\partial\frak G_1(\cdot,\tau)}{\partial\overline z}\Vert
_{C(\overline{\mathcal O_\epsilon})}+\Vert\frac{\partial \frak G_2(\cdot,\tau)}{\partial z}\Vert
_{C(\overline{\mathcal O_\epsilon})}
+ \Vert\frac{\partial \frak G_3(\cdot,\tau)}{\partial z}\Vert
_{C(\overline{\mathcal O_\epsilon})} + \Vert
\frac{\partial\frak G_4(\cdot,\tau)}{\partial \overline z}
\Vert_{C(\overline{\mathcal O_\epsilon})}
= o(\frac 1\tau)\quad\mbox{as}\,\,\vert\tau\vert\rightarrow +\infty.
\end{equation}
\end{proposition}
\begin{proof}
Observe that the functions $\frak G_1(x,\tau),\dots,
\frak G_4(x,\tau)$ are given by
$$
\frak G_1(x,\tau)=-\frac{1}{2\pi}\int_\Omega
\frac{\frac{\partial\mathcal A_1(\zeta,\overline\zeta)}{\partial
\zeta}+\tau \frac{\partial\Phi}{\partial
\zeta}(\zeta)-(\frac{\partial\mathcal A_1(z,\overline z)}{\partial z}
+\tau \frac{\partial\Phi}{\partial z}(z))}{\zeta-z}
(e_1g_1e^{-\mathcal A_1})(\xi_1,\xi_2)
e^{\tau(\overline{\Phi(\zeta)}-\Phi(\zeta))}d\xi_1d\xi_2,
$$
$$
\frak G_2(x,\tau)=-\frac{1}{2\pi}
\int_\Omega
\frac{\frac{\partial\mathcal B_1(\zeta,\overline\zeta)}
{\partial \overline
\zeta}
+ \tau\overline{\frac{\partial\Phi}{\partial\zeta}(\zeta)}
- (\frac{\partial\mathcal B_1
(z,\overline z)}{\partial\overline  z}
+\tau\frac{\partial\overline\Phi}{\partial\overline  z}(\overline
z))}{\overline \zeta-\overline z} (e_1g_2e^{-\mathcal
B_1})(\xi_1,\xi_2)
e^{\tau(\Phi(\zeta)-\overline{\Phi(\zeta)})}d\xi_1d\xi_2,
$$
$$
\frak G_3(x,\tau)=\frac{1}{2\pi}\int_{\Omega}
\frac{\tau\frac{\partial\overline\Phi(\overline \zeta)}{\partial\overline \zeta}
-\frac{\partial{\mathcal A_2}(\zeta,\overline \zeta)}{\partial\overline \zeta}
-(\tau\frac{\partial \overline\Phi(\overline z)}{\partial \overline z}
-\frac{\partial{\mathcal A_2}(z,\overline z)}{\partial\overline z})}
{\overline \zeta-\overline z}(e_1 g_3e^{-\mathcal A_2})(\xi_1,\xi_2)
e^{\tau(\overline{\Phi(\zeta)}-\Phi(\zeta))}d\xi_1d\xi_2,
$$
$$
\frak G_4( x,\tau)=\frac{1}{2\pi}\int_{\Omega}
\frac{\tau\frac{\partial\Phi(\zeta)}{\partial\zeta}
-\frac{\partial {\mathcal B_2}(\zeta,\overline\zeta)}{\partial \zeta}
-(\tau\frac{\partial\Phi(z)}{\partial z}-\frac{\partial{\mathcal B_2}
(z,\overline z)}{\partial z})}{\zeta- z}(e_1 g_4 e^{-\mathcal B_2})
(\xi_1,\xi_2)
 e^{\tau(\Phi(\zeta)-\overline{\Phi(\zeta)})}d\xi_1d\xi_2.
$$
Let $z=x_1+ix_2$ where $x=(x_1,x_2)\in \partial\Omega.$
By Proposition \ref{Proposition 3.223}, the following asymptotics
are valid:
\begin{eqnarray}
\frac{\tau}{2\pi}\int_\Omega\frac{\partial_\zeta\Phi(\zeta)}
{\zeta-z}(e_1g_1e^{-\mathcal A_1})e^{\tau(\overline
{\Phi(\zeta)}-\Phi(\zeta))}
d\xi_1d\xi_2
= -\frac{1}{2\pi}\int_\Omega\frac{e_1g_1e^{-\mathcal A_1}}
{\zeta-z}\frac{\partial}{\partial\zeta}
e^{\tau(\overline{\Phi(\zeta)}-\Phi(\zeta))}d\xi_1d\xi_2=
\nonumber\\
\frac{1}{2\pi}\int_\Omega \frac{\partial}{\partial\zeta}
\left(\frac{e_1g_1e^{-\mathcal A_1}}{\zeta-z}\right )
e^{\tau(\overline{\Phi(\zeta)}-\Phi(\zeta))}d\xi_1d\xi_2
=\frac{ 1}{2\tau}\frac{e^{-2i\tau\psi({\widetilde x})}}
{\vert \mbox{det}\, \psi''({\widetilde x})\vert^\frac 12}
\frac{\partial_z (g_1e^{-\mathcal A_1})({\widetilde x})}{\widetilde z-z}
+o(\frac{1}{\tau}),\nonumber
\end{eqnarray}

\begin{eqnarray}
\frac{\tau}{2\pi}\int_\Omega\frac{\partial_{\overline \zeta}\overline{\Phi}
(\overline\zeta)}
{\overline \zeta-\overline z}(e_1g_2e^{-\mathcal B_1})e^{\tau(\Phi(\zeta)
-\overline{\Phi(\zeta)})}d\xi_1d\xi_2
= -\frac{1}{2\pi}\int_\Omega\frac{e_1g_2e^{-\mathcal B_1}}
{\overline \zeta-\overline z}\frac{\partial}{\partial\overline\zeta}
e^{\tau(\Phi(\zeta)
- \overline{\Phi(\zeta)})}d\xi_1d\xi_2=\nonumber\\
\frac{1}{2\pi}\int_\Omega\frac{\partial}{\partial\overline\zeta}
\left (\frac{e_1g_2e^{-\mathcal B_1}}{\overline \zeta-\overline z} \right )
e^{\tau(\Phi(\zeta)-\overline{\Phi(\zeta)})}
d\xi_1d\xi_2
= \frac {1}{2\tau}\frac{e^{2i\tau\psi({\widetilde x})}}
{\vert \mbox{det}
\, \psi''({\widetilde x})\vert^\frac 12}
\frac{\partial_{\overline z}(g_2e^{-\mathcal B_1})({\widetilde x})}
{\overline{\widetilde z}-\overline z}
+o(\frac{1}{\tau}),\nonumber
\end{eqnarray}
\begin{eqnarray}
\frac{\tau}{2\pi}\int_\Omega\frac{\partial_{\overline \zeta}
\overline{\Phi}}{\overline \zeta-\overline z}(e_1g_3e^{-\mathcal A_2})
e^{\tau(\overline{\Phi(\zeta)}-\Phi(\zeta))}d\xi_1d\xi_2
=\frac{1}{2\pi}\int_\Omega\frac{e_1g_3e^{-\mathcal A_2}}
{\overline \zeta-\overline z}\frac{\partial}{\partial\overline\zeta}
e^{\tau(\overline{\Phi(\zeta)}-\Phi(\zeta))}d\xi_1d\xi_2
=\nonumber\\
-\frac{1}{2\pi}\int_\Omega\frac{\partial}{\partial\overline\zeta}
\left(\frac{e_1g_3e^{-\mathcal A_2}}{\overline \zeta-\overline z} \right)
e^{\tau(\overline{\Phi(\zeta)}-\Phi(\zeta))}d\xi_1d\xi_2
= -\frac{1}{2\tau}
\frac{e^{-2i\tau\psi({\widetilde x})}}
{\vert \mbox{det}\, \psi''({\widetilde x})\vert^\frac 12}
\frac{\partial_{\overline z}(g_3e^{-\mathcal A_2})({\widetilde x})}
{\overline{\widetilde z}-\overline z}+o(\frac{1}{\tau}) \nonumber
\end{eqnarray}
and
\begin{eqnarray}
\frac{\tau}{2\pi}\int_\Omega\frac{\partial_\zeta\Phi(\zeta)}
{\zeta-z}(e_1g_4e^{-\mathcal B_2})e^{\tau(\Phi(\zeta)
-\overline{\Phi(\zeta)})}d\xi_1d\xi_2=\frac{1}{2\pi}
\int_\Omega\frac{e_1g_4e^{-\mathcal B_2}}{\zeta-z}
\frac{\partial}{\partial\zeta} e^{\tau(\Phi(\zeta)
-\overline{\Phi(\zeta)})}d\xi_1d\xi_2=\nonumber\\-\frac{1}{2\pi}
\int_\Omega\frac{\partial}{\partial\zeta}\left (\frac{e_1g_4
e^{-\mathcal B_2}}{\zeta-z}\right ) e^{\tau(\Phi(\zeta)
-\overline{\Phi(\zeta)})}d\xi_1d\xi_2=-\frac {1}{2\tau}
\frac{e^{2i\tau\psi({\widetilde x})}}{\vert \mbox{det}
\, \psi''({\widetilde x})\vert^\frac 12}
\frac{\partial_z (g_4e^{-\mathcal B_2})({\widetilde x})}
{\widetilde z-z}+o(\frac{1}{\tau}).         \nonumber
\end{eqnarray}

Noting that $\widetilde{g}_1 = e^{-\mathcal{A}_1}g_1$,
$\widetilde{g}_2 = e^{-\mathcal{B}_1}g_2$,
$\widetilde{g}_3 = e^{-\mathcal{A}_2}g_3$ and
$\widetilde{g}_4 = e^{-\mathcal{B}_2}g_4$,
and taking into account Proposition \ref{zinka},
we obtain (\ref{vova})-(\ref{vova3})
for the functions  $\frak G_1(x,\tau), ...,
\frak G_4(x,\tau)$.

For proving  the estimate (\ref{luna}), it suffices to show that
$$
\Vert\frac{\partial\frak G_1(\cdot,\tau)}
{\partial\overline z}\Vert_{C(\overline{\mathcal O_\epsilon})}
\le o\left(\frac{1}{\tau}\right).
$$
In fact,
$$
\partial_{\overline z}\frak G_1(x,\tau)=-\frac {1}{4\pi}
\frac{\partial A_1}{\partial z} \int_\Omega \frac{(e_1 g_1
e^{-\mathcal A_1})(\xi_1,\xi_2) e^{\tau(\overline{\Phi(\zeta)}
-\Phi(\zeta))}}{\zeta-z}d\xi_1d\xi_2.
$$
Then Proposition \ref{Proposition 3.223}  and (\ref{TTT}) imply
$\Vert\frac{\partial\frak G_1(\cdot,\tau)}{\partial\overline z}\Vert
_{C(\overline{\mathcal O_\epsilon})}=o(\frac 1\tau).$
\end{proof}

{\bf Proof of Proposition \ref{olimpus}.}
Using (\ref{OO}) and (\ref{001'}) we have
\begin{eqnarray}
\frak L_0 \equiv &&(2({ A_1}-{ A_2})\frac{\partial U_1}{\partial z},
b_\tau e^{\mathcal{B}_2-\tau \Phi}
+{c_\tau }e^{\mathcal{A}_2 -\tau \overline{\Phi}})_{L^2(\Omega)}
                                        \nonumber\\
&+& (2(B_1-B_2)\frac{\partial U_1}{\partial \overline z},
b_\tau e^{\mathcal{B}_2-\tau \Phi}
+ c_\tau e^{\mathcal{A}_2 -\tau \overline{\Phi}}
)_{L^2(\Omega)}                                  \nonumber\\
&=& 2(({ A_1}-{ A_2})e^{\tau\overline\Phi}(- \mathcal R_{-\tau, A_1} \{\frac{\partial (e_1g_1)
}{\partial z}\}+e^{\mathcal A_1-\tau(\overline\Phi-\Phi)}\frak G_1(\cdot,\tau)),
{c_\tau }
e^{\mathcal{A}_2 -\tau \overline{\Phi}})_{L^2(\Omega)}
                                            \nonumber\\
                                            &-&2(({ A_1}-{ A_2}){e^{\tau\overline\Phi}}\frac{\partial
}{\partial z}\mathcal R_{-\tau, A_1} \{e_1g_1\},
{b_\tau}
e^{\mathcal{B}_2 -\tau {\Phi}})_{L^2(\Omega)}\nonumber\\
&+& (({B_1}-{B_2})(-e_1g_1+{A_1}{\mathcal R}_{-\tau, A_1} \{e_1g_1\})
e^{\tau\overline\Phi}, b_\tau e^{\mathcal{B}_2-\tau \Phi}
+{c_\tau }e^{\mathcal{A}_2 -\tau \overline{\Phi}})_{L^2(\Omega)}+o(\frac 1\tau).
\end{eqnarray}
By (\ref{AK1}) and Proposition \ref{Proposition 3.22} we have
\begin{eqnarray}\label{lenaN}
2(({ A_1}-{ A_2}){e^{\tau\overline\Phi}}\frac{\partial
}{\partial z}\mathcal R_{-\tau, A_1} \{e_1g_1\},
{b_\tau}
e^{\mathcal{B}_2 -\tau {\Phi}})_{L^2(\Omega)}=2(({ A_1}-{ A_2})\frac{\partial
}{\partial z}\mathcal R_{-\tau, A_1} \{e_1g_1\},
{b_\tau}
e^{\mathcal{B}_2})_{L^2(\Omega)}\nonumber
= \\(({ A_1}-{ A_2})(\nu_1-i\nu_2)\mathcal R_{-\tau, A_1} \{e_1g_1\},
{b_\tau}
e^{\mathcal{B}_2})_{L^2(\partial\Omega)}
-2(\mathcal R_{-\tau, A_1} \{e_1g_1\},\frac{\partial
}{\partial\overline  z}\{
{b_\tau}
e^{\mathcal{B}_2}\overline{({ A_1}-{ A_2})}\})_{L^2(\Omega)}\nonumber\\
=-(\frac{e_1g_1}{\tau\overline{\partial_z\Phi}},\frac{\partial
}{\partial \overline z}\{
{b_\tau}
e^{\mathcal{B}_2}\overline{({ A_1}-{ A_2})}\})_{L^2(\Omega)}
+ o(\frac 1\tau)\quad \mbox{as}\,\,\vert\tau\vert\rightarrow +\infty.
\end{eqnarray}

Using the stationary phase argument  we obtain
\begin{eqnarray}\label{lenaN}
-2(({ A_1}-{ A_2}){e^{\tau\overline\Phi}}\mathcal R_{-\tau, A_1}
\{\frac{\partial (e_1g_1)}{\partial z}\},
{c_\tau}
e^{\mathcal{A}_2 -\tau \overline{\Phi}})_{L^2(\Omega)}\nonumber\\
=-(({ A_1}-{ A_2})e^{\mathcal A_1
+\overline{{\mathcal A}_2}}\partial^{-1}_{\overline z}
\{\frac{\partial (e_1g_1)}{\partial z}
e^{-\mathcal A_1-\tau(\Phi-\overline\Phi)}\},c_\tau)
_{L^2(\Omega)}=\nonumber\\
\frac 1\pi \int_\Omega (A_1-A_2)
e^{\mathcal A_1+\overline{{\mathcal A}_2}}\overline{c_\tau}
\left(\int_\Omega \frac{\frac{\partial (e_1g_1)}{\partial \zeta}
e^{-\mathcal A_1-\tau(\Phi-\overline\Phi)}}{\zeta-z}
d\xi_1 d\xi_2 \right)dx_1 dx_2=\nonumber\\
\frac 1\pi \int_\Omega \frac{\partial (e_1g_1)}
{\partial \zeta}e^{-\mathcal A_1-\tau(\Phi-\overline\Phi)}
\left (\int_\Omega \frac{(A_1-A_2)\overline{c_\tau} e^{\mathcal A_1+\overline{{\mathcal A}_2}}}{\zeta-z}d x_1 d x_2\right )d\xi_1 d\xi_2= \nonumber\\ \frac{e^{-2i\tau\psi(\widetilde x)}}{\tau\vert\mbox {det} \, \psi^{''}(\widetilde x)\vert^\frac 12}\frac{\partial g_1}{\partial z}(\widetilde x)
e^{-\mathcal A_1(\widetilde x)}\left (\int_\Omega \frac{(A_1-A_2)\overline{c} e^{\mathcal A_1+\overline{{\mathcal A}_2}}}{\widetilde z-z}d x_1 d x_2\right )+
o(\frac{1}{\tau}).
\end{eqnarray}
Integrating by parts and using (\ref{luna}), (\ref{001q}),
$2\frac{\partial {\mathcal A}_1}{\partial \overline{z}}
= -A_1$ and
$2\frac{\partial {\mathcal A}_2}{\partial z}
= \overline{A_2}$, we have
\begin{eqnarray}\label{noch}
2((A_1-A_2)e^{\tau\overline{\Phi}} e^{\mathcal {A}_1}
e^{-\tau(\overline\Phi-\Phi)}\frak G_1,c_\tau
e^{\mathcal{A}_2 -\tau \overline{\Phi}})_{L^2(\Omega)}
=2((A_1-A_2)e^{\mathcal{ A}_1} \frak G_1,c_\tau
e^{\mathcal{A}_2})_{L^2(\Omega)}\\
=\int_\Omega 2(A_1-A_2)e^{\mathcal A_1+\overline{\mathcal A_2}}
\overline{c_\tau(\overline z)}\frak G_1(x,\tau)dx
=-4\int_\Omega\frac{ \partial}{\partial \overline z}
(e^{\mathcal A_1+\overline{\mathcal A_2}})
\overline{c_\tau(\overline z)}\frak G_1(x,\tau)dx
                           \nonumber\\
=4\int_\Omega e^{\mathcal A_1+\overline{\mathcal A_2}}
\overline{c_\tau(\overline z)}\frac{ \partial}{\partial \overline z}
\frak G_1(x,\tau)dx-2\int_{\partial\Omega}(\nu_1+i\nu_2)
e^{\mathcal A_1+\overline{\mathcal A_2}}\overline{c_\tau(\overline z)}
\frak G_1(x,\tau)d\sigma+o(\frac 1\tau)\quad\mbox{as}
\,\,\vert\tau\vert\rightarrow +\infty.\nonumber
\end{eqnarray}

Since $e^{\mathcal A_1}\frac{ \partial}{\partial \overline z}\frak G_1(x,\tau)
=-\frac {e^{\mathcal A_1}}{4\pi}\frac{\partial A_1}{\partial z}
\int_\Omega \frac{(e_1 g_1 e^{-\mathcal A_1})(\xi_1,\xi_2)
e^{\tau(\overline{\Phi(\zeta)}-\Phi(\zeta))}}{\zeta-z}d\xi_1d\xi_2
=\frac 12 \frac{\partial A_1}{\partial z}e^{\tau(\overline\Phi-\Phi)}\mathcal
R_{-\tau, A_1}\{e_1g_1\},$ applying the Proposition
\ref{Proposition 3.22}  and Proposition \ref{opa}
we obtain
\begin{equation}\label{ee}
\int_\Omega e^{\mathcal A_1+\overline{\mathcal A_2}}
\overline{c_\tau(\overline z)}\frac{ \partial}{\partial \overline z}
\frak G_1(x,\tau)dx=o(\frac 1\tau)\quad\mbox{as}\,\,\vert
\tau\vert\rightarrow +\infty.
\end{equation}
By (\ref{vova}), (\ref{lenaN})-(\ref{ee}) and Propositions
\ref{Proposition 3.22} and \ref{Proposition 3.223}, we conclude
\begin{eqnarray}\label{vikana}
\frak L_0
&=& (({B_1}-{B_2})(-e_1g_1+\frac{{ A_1}e_1g_1}
{2\tau\overline{\partial_z\Phi}}),b_\tau
e^{\mathcal{B}_2})_{L^2(\Omega)}\nonumber\\
&+& (\frac{e_1g_1}{\tau\overline{\partial_z\Phi}},\frac{\partial
}{\partial \overline z}\{
{b}
e^{\mathcal{B}_2}\overline{({ A_1}-{ A_2})}\})_{L^2(\Omega)}\nonumber\\
&+&\frac{ e^{-2i\tau\psi(\widetilde x)}}{\tau\vert\mbox {det}\, \psi^{''}(\widetilde x)\vert^\frac 12}\frac{\partial g_1}{\partial z}(\widetilde x)
e^{-\mathcal A_1(\widetilde x)}\left (\int_\Omega \frac{(A_1-A_2)\overline{c} e^{\mathcal A_1+\overline{{\mathcal A}_2}}}{\widetilde z-z}d x_1 d x_2\right )\nonumber\\
&-&2\int_{\partial\Omega}(\nu_1+i\nu_2)e^{\mathcal A_1
+\overline{\mathcal A_2}}\overline{c_\tau(\overline z)}\frak G_1(x,\tau)
d\sigma
+o(\frac 1\tau)\quad\mbox{as}\,\,\vert\tau\vert\rightarrow +\infty.
                               \nonumber\\
\end{eqnarray}

Using (\ref{OOO}) and (\ref{0011A}) we obtain after simple computations:
\begin{eqnarray}\label{BBB}
\frak L_1 \equiv &(&2({A_1}-{ A_2})
\frac{\partial U_2}{\partial z},b_\tau e^{\mathcal{B}_2-\tau \Phi}
+{c_\tau}e^{\mathcal{A}_2 -\tau \overline{\Phi}})_{L^2(\Omega)}
                                        \nonumber\\
&+& (2 (B_1-B_2)\frac{\partial U_2}{\partial \overline z},b_\tau
e^{\mathcal{B}_2-\tau \Phi}
+{c_\tau}e^{\mathcal{A}_2 -\tau \overline{\Phi}})_{L^2(\Omega)}
                                                 \nonumber\\
&=& (({A_1}-{A_2})(-e_1g_2 + {B_1}\widetilde{\mathcal R}_{\tau, B_1}
\{e_1g_2\})e^{\tau\Phi},b_\tau e^{\mathcal{B}_2-\tau \Phi}
+{c_\tau}
e^{\mathcal{A}_2 -\tau \overline{\Phi}})_{L^2(\Omega)}\nonumber\\
&+& 2((B_1- B_2)(-
\widetilde{\mathcal R}_{\tau, B_1} \{\frac{\partial (e_1g_2)}{\partial \overline z}\}e^{\tau\Phi}+e^{\mathcal B_1+\tau(\overline\Phi-\Phi)}\frak G_2(\cdot,\tau)),
b_\tau e^{\mathcal{B}_2-\tau \Phi})_{L^2(\Omega)}\nonumber\\
&-& 2((B_1- B_2)\frac{\partial }{\partial \overline z}
\widetilde{\mathcal R}_{\tau, B_1} \{e_1g_2\}e^{\tau\Phi},
{c_\tau}e^{\mathcal{A}_2 -\tau \overline{\Phi}})_{L^2(\Omega)}.
\end{eqnarray}

Integrating by parts  and using (\ref{AKK1}) we have
\begin{eqnarray}&-& 2((B_1- B_2)\frac{\partial }{\partial \overline z}
\widetilde{\mathcal R}_{\tau, B_1} \{e_1g_2\}e^{\tau\Phi},
 {c_\tau}e^{\mathcal{A}_2 -\tau \overline{\Phi}})_{L^2(\Omega)}=
 - 2((B_1- B_2)\frac{\partial }{\partial \overline z}
\widetilde{\mathcal R}_{\tau, B_1} \{e_1g_2\},
 {c_\tau}e^{\mathcal{A}_2 })_{L^2(\Omega)}=\nonumber\\
 &-& 2\int_\Omega
\widetilde{\mathcal R}_{\tau, B_1} \{e_1g_2\}\frac{\partial }{\partial \overline z}((B_1- B_2)
 \overline{c_\tau}e^{\overline{\mathcal{A}_2} }) dx-\int_{\partial\Omega}(B_1- B_2)(\nu_1+i\nu_2)
\widetilde{\mathcal R}_{\tau, B_1} \{e_1g_2\}
 \overline{c_\tau}e^{\overline{\mathcal{A}_2} }d\sigma=\nonumber\\
 &&\quad\quad\quad\int_{\Omega}\frac{e_1g_2}{\tau\partial_z\Phi}
 \frac{\partial }{\partial \overline z}((B_1- B_2)
 \overline{c}e^{\overline{\mathcal{A}_2} })d x
 +o(\frac 1\tau)\quad \mbox{as}\,\,\vert\tau\vert\rightarrow +\infty.
\end{eqnarray}
The stationary phase argument implies the formula
\begin{eqnarray}- 2((B_1- B_2)
\widetilde{\mathcal R}_{\tau, B_1} \{\frac{\partial (e_1g_2)}{\partial \overline z}\}e^{\tau\Phi},
b_\tau e^{\mathcal{B}_2-\tau \Phi}
)_{L^2(\Omega)}=- \int_\Omega(B_1- B_2)
\partial_z^{-1}\{\frac{\partial e_1g_2}{\partial \overline z}e^{-\mathcal B_1+\tau(\Phi-\overline\Phi)}\}
\overline{b_\tau} e^{\mathcal B_1+\overline{\mathcal{B}_2}}
dx=\nonumber\\
\frac{1}{\pi}\int_\Omega \frac{\partial (e_1g_2)}{\partial \overline \zeta}e^{-\mathcal B_1+\tau(\Phi-\overline\Phi)}\left(\int_\Omega\frac{(B_1-B_2)\overline {b_\tau} e^{\mathcal{B}_1+\overline{\mathcal B_2} }}{\overline \zeta-\overline z}dx\right) d\xi_1d\xi_2=\nonumber\\ \frac{e^{2i\tau\psi(\widetilde x)}}{\tau\vert \mbox{det}\, \psi^{''}(\widetilde x)\vert^\frac 12}\frac{\partial g_2(\widetilde x)}{\partial \overline z}e^{-\mathcal B_1(\widetilde x)}\int_\Omega\frac{(B_1-B_2)\overline {b_\tau} e^{\mathcal{B}_1-\overline{\mathcal B_2} }}{\overline{\widetilde z}-\overline z}dx+
o(\frac 1\tau).
\end{eqnarray}

By (\ref{luna}) we have the asymptotic formula
\begin{eqnarray}\label{Bip1}
2((B_1- B_2)(
e^{\mathcal B_1}e^{\tau\overline{\Phi}}\frak G_2
, b_\tau
e^{\mathcal{B}_2-\tau \Phi})_{L^2(\Omega)}\\
= 2((B_1- B_2)e^{\mathcal B_1}\frak G_2, b_\tau
e^{\mathcal{B}_2})_{L^2(\Omega)}
= 2\int_\Omega (B_1- B_2)e^{\mathcal B_1+\overline{\mathcal B}_2}
\overline{b_\tau(z)} \frak G_2(x,\tau)dx\nonumber\\
= -4\int_\Omega \frac{\partial}{\partial z}
e^{\mathcal B_1+\overline{\mathcal B}_2} \overline{b_\tau(z)}
\frak G_2(x,\tau)dx
= -2\int_{\partial\Omega}(\nu_1-i\nu_2)e^{\mathcal B_1
+\overline{\mathcal B}_2} \overline{b_\tau(z)} \frak G_2(x,\tau)
d\sigma \nonumber\\ +4\int_\Omega e^{\mathcal B_1
+\overline{\mathcal B}_2} \overline{b_\tau(z)} \frac{\partial}
{\partial z}\frak G_2(x,\tau) dx.\nonumber
\end{eqnarray}

Observe that $\frac{\partial}{\partial z}\frak G_2(x,\tau)
=-\frac{1}{4\pi}\frac{\partial B_1}{\partial \overline z}
\int_\Omega\frac{(e_1g_2 e^{-\mathcal B_1})(\xi_1,\xi_2)}
{\overline \zeta-\overline z}
e^{\tau(\Phi(\zeta)-\overline{\Phi(\zeta)}) }d\xi_1 d\xi_2
=\frac {e^{-\mathcal B_1}}{2} e^{\tau(\Phi-\overline\Phi)}
\widetilde{\mathcal R}_{\tau, B_1}\{e_1g_2\}.$
Then by Proposition 3.4
\begin{eqnarray}\label{NV11}
4\int_\Omega e^{\mathcal B_1
+ \overline{\mathcal B}_2} \overline{b_\tau(z)}
\frac{\partial}{\partial z}
\frak G_2(x,\tau) dx
= 2\int_\Omega e^{\mathcal B_1 + \overline{\mathcal B}_2}
\overline{b_\tau(z)} e^{-\mathcal B_1}
e^{\tau(\Phi-\overline\Phi)}
\widetilde{\mathcal R}_{\tau, B_1}\{e_1g_2\} dx
                                                     \nonumber\\
= \int_\Omega e^{\mathcal B_1
+\overline{\mathcal B}_2} \overline{b_\tau(z)}
e^{-\mathcal B_1}e^{\tau(\Phi-\overline\Phi)}
\frac{e_1g_2}{\tau\partial_z\Phi} dx
= o(\frac 1\tau) \quad \mbox{as}
\,\,\vert\tau\vert\rightarrow +\infty.
\end{eqnarray}
By (\ref{Bip1}) -(\ref{NV11}) we have
\begin{eqnarray}\label{vikana1UP}
\frak L_1
&=& (({A_1}-{A_2})(-e_1g_2+\frac{B_1e_1g_2}{2\tau{\partial_z\Phi}}),
c_\tau e^{\mathcal A_2})_{L^2(\Omega)}\nonumber\\
&+&\int_{\Omega}\frac{e_1g_2}{\tau\partial_z\Phi}
 \frac{\partial }{\partial \overline z}((B_1- B_2)
 \overline{c_\tau}e^{\overline{\mathcal{A}_2} })d x\nonumber\\
&  &\frac{e^{2i\tau\psi(\widetilde x)}}{\tau\vert \mbox{det}\,\psi^{''}(\widetilde x)\vert^\frac 12} \frac{\partial g_2(\widetilde x)}{\partial \overline z}e^{-\mathcal B_1(\widetilde x)}\int_\Omega\frac{(B_1- B_2)\overline b e^{\mathcal{B}_1+\overline{\mathcal B_2} }}{\overline{\widetilde z}-\overline z}dx\nonumber \\
&-&2\int_{\partial\Omega}(\nu_1-i\nu_2)e^{\mathcal B_1
+\overline{\mathcal B}_2} \overline{b_\tau(z)} \frak G_2(x,\tau)d\sigma
+o(\frac{1}{\tau})\quad \mbox{as}\,\,\vert\tau\vert\rightarrow +\infty.
\end{eqnarray}
Recall that $V_1 = -e^{-\tau\Phi}\widetilde{\mathcal R}
_{-\tau,-\overline{A_2}}\{e_1g_3\}$ and
$V_2 = -e^{-\tau\overline\Phi}{\mathcal R}_{\tau,-\overline{B_2}}
\{e_1g_4\}$.

By Proposition \ref{Proposition 3.1} we conclude
\begin{equation}\label{UU}
{2}\frac{\partial V_1}{\partial z}=(-e_1g_3+{\overline
A_2}\widetilde{\mathcal R}_{-\tau, \overline A_2}\{e_1g_3\})
e^{-\tau\Phi}
\end{equation}
and
\begin{equation}\label {UUU}
{2}\frac{\partial V_2}{\partial \overline z}
=(-e_1g_4+\overline{B_2}{\mathcal R}_{\tau,-\overline
B_2}\{e_1g_4\})e^{-\tau\overline\Phi}.
\end{equation}

Similarly to (\ref{yap}) and (\ref{yop})
we calculate $\frac{\partial V_1}{\partial\overline z}$ and
$\frac{\partial V_2}{\partial z}:$
\begin{equation}\label{ono}
\frac{\partial V_1}{\partial\overline z}
=- e^{-\tau\Phi}\widetilde{\mathcal R}_{-\tau,-\overline A_2}
\left\{\frac{\partial (e_1g_3)}{\partial \overline
z}\right\} + e^{-\tau\overline\Phi+{\mathcal A_2}}\frak G_3(\cdot,\tau)
\end{equation}
and
\begin{equation} \label{ono1}
 \frac{\partial V_2}{\partial
z}=-e^{-\tau\overline\Phi}{\mathcal
R}_{\tau,-\overline B_2}\left\{\frac{\partial(e_1g_4)}{\partial
z}\right\} + e^{-\tau\Phi+{\mathcal B_2}}\frak G_4(\cdot,\tau).
 \end{equation}
Using (\ref{001q}) and integrating by parts we obtain
\begin{eqnarray}\label{QQ}
\frak L_2 \equiv&(& 2({ A_1}-{ A_2})\frac{\partial }{\partial z}
(a_\tau e^{\mathcal A_1+\tau\Phi}+d_\tau e^{\mathcal B_1
+\tau\overline{\Phi}}),V_1+V_2)_{L^2(\Omega)}       \nonumber\\
&=& -(({A_1}-{ A_2})d_\tau B_1 e^{{\mathcal B_1}}e^{\tau\overline\Phi},
V_1+V_2)_{L^2(\Omega)}                                     \nonumber\\
+&(&(A_1-A_2)(\nu_1-i\nu_2)a_\tau e^{\mathcal A_1+\tau\Phi},
V_1+V_2)_{L^2(\partial\Omega)}\nonumber\\
-&(&2\frac{\partial }{\partial z}(A_1-A_2)a_\tau
e^{\mathcal A_1+\tau\Phi},V_1+V_2)_{L^2(\Omega)}     \nonumber\\
-&(&(A_1-A_2)a_\tau e^{\mathcal A_1+\tau\Phi}, 2(\frac{\partial
V_1}{\partial \overline z}+\frac{\partial V_2}{\partial \overline z})
)_{L^2(\Omega)}.                            \nonumber
\end{eqnarray}

We observe that by (\ref{ono}), (\ref{luna}),
Proposition \ref{Proposition 3.22} and Proposition \ref{opa} and
 the equality $\frac{\partial \frak G_3}{\partial z}=-\frac 12  e^{-\mathcal A_2}\frac{\partial \overline A_2}{\partial\overline z} e^{\tau(\overline\Phi-\Phi)} \widetilde {\mathcal R}_{-\tau, -\overline{A_2}}(e_1 g_3)$  :
\begin{eqnarray}
((A_1-A_2)a_\tau e^{\mathcal A_1+\tau\Phi}, 2\frac{\partial
V_1}{\partial \overline z})_{L^2(\Omega)}
=- 4\int_\Omega a_\tau(z)\frac{\partial}{\partial \overline z}
e^{\mathcal A_1+\overline{\mathcal A_2}}
\overline{\frak G_3(x,\tau)}dx\nonumber\\+((A_1-A_2)a_\tau
e^{\mathcal A_1+\tau\Phi}, - e^{-\tau\Phi}\widetilde{\mathcal R}_{-\tau,-\overline A_2}
\left\{\frac{\partial (e_1g_3)}{\partial \overline
z}\right\} )_{L^2(\Omega)}+o(\frac 1\tau)\nonumber=\\
-2\int_{\partial\Omega}a_\tau(z)\overline{\frak G_3(x,\tau)}
(\nu_1+i\nu_2)e^{\mathcal A_1+\overline{\mathcal A_2}}d\sigma
\nonumber\\
+ \frac 1\pi\int_\Omega \overline{\frac{\partial (e_1 g_3)}
{\partial \overline \zeta}}e^{\tau( \Phi-\overline\Phi)-\overline{\mathcal A_2}}\left(\int_\Omega\frac{(A_1-A_2)a_\tau e^{\mathcal A_1+\overline{\mathcal A_2}}}{ \zeta-z}dx\right )d\xi_1d\xi_2
+ o(\frac 1\tau)=\nonumber\\
\frac{e^{2i\tau\psi(\widetilde x)}}{\tau\vert\mbox{det}\,\psi^{''}(\widetilde x)\vert^{\frac 12}}\overline{\frac{\partial g_3(\widetilde x)}{\partial \overline
z}}e^{-\overline{\mathcal A_2}(\widetilde x)}\int_\Omega\frac{(A_1-A_2)a e^{\mathcal A_1+\overline{\mathcal A_2}}}{ \widetilde z-z}dx\nonumber\\
-2\int_{\partial\Omega}a_\tau(z)\overline{\frak G_3(x,\tau)}
(\nu_1+i\nu_2)e^{\mathcal A_1+\overline{\mathcal A_2}}d\sigma+ o(\frac 1\tau)\quad\mbox{as}\,\,\vert\tau\vert\rightarrow +\infty.
\end{eqnarray}
Hence, using (\ref{UUU}) we have
\begin{eqnarray}
\frak L_2=&(&(A_1-A_2)d_\tau B_1 e^{{\mathcal B_1}},
\widetilde{\mathcal R}_{-\tau, -\overline A_2}\{e_1g_3\})_{L^2(\Omega)}
                             \nonumber\\
+&(&(A_1-A_2)(\nu_1-i\nu_2)a_\tau e^{\mathcal A_1+\tau\Phi},V_1+V_2)
_{L^2(\partial\Omega)}\nonumber\\
+&(&2\frac{\partial }{\partial z}(A_1-A_2)a_\tau
e^{{\mathcal A_1}},{\mathcal R}_{\tau,-\overline B_2}\{e_1g_4\}
)_{L^2(\Omega)}\nonumber\\
&-&\frac{ e^{2i\tau\psi(\widetilde x)}}{\tau\vert\mbox{det}\,\psi^{''}(\widetilde x)\vert^{\frac 12}}\overline{\frac{\partial g_3(\widetilde x)}{\partial \overline
z}}e^{-\overline{\mathcal A_2}(\widetilde x)}\int_\Omega\frac{(A_1-A_2)a e^{\mathcal A_1+\overline{\mathcal A_2}}}{ \widetilde z-z}dx\nonumber\\
+&(&(A_1-A_2)a_\tau e^{{\mathcal A_1}},-e_1g_4-\overline{ B_2}
{\mathcal R}_{\tau, -\overline B_2}\{e_1g_4\})_{L^2(\Omega)}+o(\frac 1\tau)
                                                   \nonumber\\
&+&2\int_{\partial\Omega}a_\tau(z)\overline{\frak G_3(x,\tau)}
(\nu_1+i\nu_2)e^{\mathcal A_1+\overline{\mathcal A_2}}d\sigma
\quad \mbox{as}\,\,\vert \tau\vert\rightarrow +\infty.\nonumber
\end{eqnarray}
By (\ref{victory}) and (\ref{victory1})  we obtain
$$
((A_1-A_2)(\nu_1-i\nu_2)a_\tau e^{\mathcal A_1+\tau\Phi},V_1+V_2)
_{L^2(\partial\Omega)}=o(\frac{1}{\tau}) \quad \mbox{as}
\,\,\vert\tau\vert\rightarrow +\infty.
$$
Therefore, by Proposition \ref{Proposition 3.22}
\begin{eqnarray}\label{vikana2}
\frak L_2=-&(&(A_1-A_2)d_\tau B_1e^{{\mathcal B_1}},\frac{e_1g_3}
{2\tau{\partial_z\Phi}})_{L^2(\Omega)}     \nonumber\\
- &(&\frac{\partial }{\partial z}(A_1-A_2)a_\tau e^{{\mathcal A_1}},
\frac{e_1g_4}{\tau\overline{\partial_z\Phi}})_{L^2(\Omega)}
                                                 \nonumber\\
+ &(&(A_1-A_2)a_\tau e^{{\mathcal A_1}},
-e_1g_4+\overline{ B_2}\frac{e_1g_4}{2\tau\overline{\partial_z\Phi}}
)_{L^2(\Omega)} \nonumber\\
&-&\frac{ e^{2i\tau\psi(\widetilde x)}}{\tau\vert\mbox{det}\,\psi^{''}(\widetilde x)\vert^{\frac 12}}\overline{\frac{\partial g_3(\widetilde x)}{\partial \overline
z}}e^{-\overline{\mathcal A_2}(\widetilde x)}\int_\Omega\frac{(A_1-A_2)a e^{\mathcal A_1+\overline{\mathcal A_2}}}{ \widetilde z-z}dx\nonumber\\
&+&2\int_{\partial\Omega}a_\tau(z)\overline{\frak G_3(x,\tau)}
(\nu_1+i\nu_2)e^{\mathcal A_1+\overline{\mathcal A_2}}d\sigma
+ o(\frac 1\tau)\quad\mbox{as}\,\,\vert\tau\vert\rightarrow +\infty.
\end{eqnarray}
Integrating by parts we compute
\begin{eqnarray}\label{QQQ}
\frak L_3 \equiv&(&2({B_1}-{ B_2})\frac{\partial }{\partial \overline z}
(a_\tau e^{\mathcal A_1+\tau\Phi}+d_\tau e^{\mathcal B_1
+ \tau\overline{\Phi}}),V_1+V_2)_{L^2(\Omega)}=    \nonumber\\
&-&(2 \frac{\partial }{\partial \overline z}(B_1-B_2)d_\tau
e^{{\mathcal B_1}+\tau\overline\Phi},V_1+V_2)_{L^2(\Omega)}
                                     \nonumber\\
-&(&({B_1}-{B_2}){ A_1}a_\tau e^{\mathcal A_1+\tau\Phi},V_1+V_2)
_{L^2(\Omega)}                         \nonumber\\
+&(&(\nu_1+i\nu_2)({B_1}-{ B_2})d_\tau e^{\mathcal B_1+\tau\Phi},V_1+V_2)
_{L^2(\partial\Omega)}\nonumber\\
- &(&(B_1-B_2)d_\tau e^{\mathcal B_1+\tau\overline\Phi},
2(\frac{\partial V_1}{\partial  z} +\frac{\partial V_2}{\partial
z}))_{L^2(\Omega)}.\nonumber
\end{eqnarray}

We observe that by (\ref{luna}), (\ref{ono1}), Proposition
\ref{Proposition 3.22} and Proposition \ref{opa} and the equality $\frac{\partial \frak G_4}{\partial \overline z}=-\frac 12  e^{-\mathcal B_2}\frac{\partial\overline B_2}{\partial z} e^{\tau(\Phi-\overline\Phi)} {\mathcal R}_{\tau,- \overline{B_2}}(e_1 g_4)$ :
\begin{eqnarray}
2((B_1-B_2)d_\tau e^{\mathcal B_1+\tau\overline\Phi},
\frac{\partial V_2}{\partial z})_{L^2(\Omega)}
=-4\int_\Omega \frac{\partial}{\partial z}
e^{{\mathcal B}_1+\overline{\mathcal B_2}}
d_\tau(\overline z)\overline{\frak G_4(x,\tau)}dx\nonumber\\+2((B_1-B_2)d_\tau e^{\mathcal B_1+\tau\overline\Phi},
-e^{-\tau\overline\Phi}{\mathcal
R}_{\tau,-{\overline B_2}}\left\{\frac{\partial(e_1g_4)}{\partial
z}\right\})_{L^2(\Omega)}+o(\frac 1\tau)
                                           \nonumber =\\
-2\int_{\partial\Omega}(\nu_1-i\nu_2)
e^{\mathcal B_1+\overline{\mathcal B_2}}d_\tau(\overline z)
\overline{\frak G_4(x,\tau)}d\sigma\nonumber\\
+\int_\Omega\left (\int_\Omega\frac1\pi \frac{(B_1-B_2)d_\tau e^{\mathcal B_1+\overline{\mathcal B_2}}}{\overline \zeta- \overline z}dx\right )
\overline{\left\{\frac{\partial(e_1g_4)}{\partial
z}\right\}}e^{-\overline{\mathcal B_2}+\tau(\overline \Phi-\Phi)} d\xi_1d\xi_2+ o(\frac 1\tau)=\nonumber\\
\frac{ e^{-2i\tau\psi(\widetilde x)}}{\tau\vert\mbox{det}\,\psi^{''}(\widetilde x)\vert^\frac 12}\int_\Omega \frac{(B_1-B_2)d_\tau e^{\mathcal B_1+\overline{\mathcal B_2}}}{\overline {\widetilde z}- \overline z}dx
\overline{\left\{\frac{\partial g_4(\widetilde x)}{\partial
\zeta}\right\}}e^{-\overline{\mathcal B_2}(\widetilde x)}\nonumber\\
-2\int_{\partial\Omega}(\nu_1-i\nu_2)
e^{\mathcal B_1+\overline{B_2}}d_\tau(\overline z)
\overline{\frak G_4(x,\tau)}d\sigma+o(\frac 1\tau)
\quad\mbox{as}\,\,\vert\tau\vert\rightarrow +\infty.
\end{eqnarray}
Hence
\begin{eqnarray}
\frak L_3&=&(2 \frac{\partial }{\partial \overline z}(B_1- B_2)
d_\tau e^{{\mathcal B_1}},\widetilde{\mathcal R}_{-\tau,-\overline A_2}
\{e_1g_3\})_{L^2(\Omega)}                      \nonumber\\
-&(&({B_1}-{ B_2}){ A_1}a_\tau e^{\mathcal A_1},
{\mathcal R}_{\tau,-\overline B_2}\{e_1g_4\})_{L^2(\Omega)}\nonumber\\
+&(&({B_1}-{B_2})d_\tau e^{{\mathcal B_1}},-e_1g_3
+\overline A_2\widetilde{\mathcal {R}}_{-\tau,-\overline A_2}\{e_1g_3\})
_{L^2(\Omega)}\nonumber\\
+&(& (\nu_1+i\nu_2)({B_1}-{ B_2})d_\tau
e^{\mathcal B_1+\tau\Phi},V_1+V_2)_{L^2(\partial\Omega)}\nonumber\\
&-&\frac{ e^{-2i\tau\psi(\widetilde x)}}{\tau\vert\mbox{det}\,\psi^{''}(\widetilde x)\vert^\frac 12}
\overline{\left\{\frac{\partial g_4(\widetilde x)}{\partial
z}\right\}}e^{-\overline{\mathcal B_2} (\widetilde x)}\int_\Omega \frac{(B_1-B_2)d e^{\mathcal B_1+\overline{\mathcal B_2}}}{\overline {\widetilde z}- \overline z}dx\nonumber\\
&+&2\int_{\partial\Omega}(\nu_1-i\nu_2)e^{\mathcal B_1
+\overline{B_2}}d_\tau(\overline z)\overline{\frak G_4(x,\tau)}d\sigma
+o(\frac 1\tau) \quad\mbox{as}\,\,\vert\tau\vert\rightarrow +\infty.
\end{eqnarray}

By (\ref{victory}), (\ref{victory1}) and the stationary phase argument
$$
( (\nu_1+i\nu_2)({B_1}-{ B_2})d_\tau e^{\mathcal B_1+\tau\Phi},V_1+V_2)
_{L^2(\partial\Omega)}=o(\frac 1\tau) \quad\mbox{as}
\,\,\vert\tau\vert\rightarrow
+\infty.
$$
Therefore, applying Proposition \ref{Proposition 3.22},
we finally conclude that
\begin{eqnarray}\label{vikana3}
\frak L_3=&-&2( \frac{\partial }{\partial \overline z}(B_1-B_2)d_\tau
e^{{\mathcal B_1}},\frac{e_1g_3}{2\tau{\partial_z\Phi}})_{L^2(\Omega)}
                                     \nonumber\\
-&(&(B_1-{B_2}){ A_1}a_\tau e^{{\mathcal A_1}},\frac{e_1 g_4}
{2\tau\overline{\partial_z\Phi}})_{L^2(\Omega)}\nonumber\\
+&(&({B_1}-{B_2})d_\tau e^{{\mathcal B_1}},-e_1g_3-\frac{\overline A_2 e_1g_3}
{2\tau{\partial_z\Phi}})_{L^2(\Omega)}\nonumber\\
&-&\frac{ e^{-2i\tau\psi(\widetilde x)}}{\tau\vert\mbox{det}\,\psi^{''}(\widetilde x)\vert^\frac 12}
\overline{\left\{\frac{\partial g_4(\widetilde x)}{\partial
z}\right\}}e^{-\overline{\mathcal B_2} (\widetilde x)}\int_\Omega \frac{(B_1-B_2)d e^{\mathcal B_1+\overline{\mathcal B_2}}}{\overline {\widetilde z}- \overline z}dx\nonumber\\
&+&2\int_{\partial\Omega}(\nu_1-i\nu_2)e^{\mathcal B_1+\overline{B_2}}
d_\tau(\overline z)\overline{\frak G_4(x,\tau)}d\sigma
+ o(\frac 1\tau)\quad\mbox{as}\,\,\vert\tau\vert\rightarrow +\infty.
\end{eqnarray}
The sum $\sum_{k=0}^3 \frak L_k$ is equal to the left hand side of
(\ref{nika1}). Observe that
\begin{eqnarray}\label{losad}
&-&\frac{ e^{-2i\tau\psi(\widetilde x)}}{\tau\vert\mbox{det}\,\psi^{''}(\widetilde x)\vert^\frac 12}\overline{\left\{\frac{\partial g_4(\widetilde x)}{\partial
z}\right\}}e^{-\overline{\mathcal B_2} (\widetilde x)}\int_\Omega \frac{(B_1-B_2)d e^{\mathcal B_1+\overline{\mathcal B_2}}}{\overline {\widetilde z}- \overline z}dx
\nonumber\\
&-&\frac{ e^{2i\tau\psi(\widetilde x)}}{\tau\vert\mbox{det}\,\psi^{''}(\widetilde x)\vert^{\frac 12}}\overline{\frac{\partial g_3(\widetilde x)}{\partial \overline
z}}e^{-\overline{\mathcal A_2}(\widetilde x)}\int_\Omega\frac{(A_1-A_2)a e^{\mathcal A_1+\overline{\mathcal A_2}}}{ \widetilde z-z}dx\nonumber\\
&+&\frac{e^{2i\tau\psi(\widetilde x)}}{\tau\vert \mbox{det}\,\psi^{''}(\widetilde x)\vert^\frac 12} \frac{\partial g_2(\widetilde x)}{\partial \overline z}e^{-\mathcal B_1(\widetilde x)}\int_\Omega\frac{(B_1- B_2)\overline b e^{\mathcal{B}_1+\overline{\mathcal B_2} }}{\overline{\widetilde z}-\overline z}dx\nonumber\\
&+&\frac{ e^{-2i\tau\psi(\widetilde x)}}{\tau\vert\mbox {det}\, \psi^{''}(\widetilde x)\vert^\frac 12}\frac{\partial g_1(\widetilde x)}{\partial z}e^{-\mathcal A_1(\widetilde x)}\int_\Omega \frac{(A_1-A_2)\overline{c} e^{\mathcal A_1+\overline{{\mathcal A}_2}}}{\widetilde z-z}d x=\nonumber\\
& &\frac{ e^{-2i\tau\psi(\widetilde x)}}{\tau\vert\mbox{det}\,\psi^{''}(\widetilde x)\vert^\frac 12}\overline{\left\{\frac{\partial g_4(\widetilde x)}{\partial
z}\right\}}e^{-\overline{\mathcal B_2} (\widetilde x)}\int_{\partial\Omega} \frac{(\nu_1-i\nu_2)d e^{\mathcal B_1+\overline{\mathcal B_2}}}{\overline {\widetilde z}- \overline z}d\sigma \nonumber\\
&-&\frac{e^{-2i\tau\psi(\widetilde x)}}{\tau\vert\mbox{det}\,\psi^{''}(\widetilde x)\vert^\frac 12}\frac{\partial g_1(\widetilde x)}{\partial z}e^{-\mathcal A_1(\widetilde x)}\int_{\partial\Omega} \frac{(\nu_1+i\nu_2)\overline{c} e^{\mathcal A_1+\overline{{\mathcal A}_2}}}{\widetilde z-z}d \sigma\nonumber\\
&+&\frac{ e^{2i\tau\psi(\widetilde x)}}{\tau\vert\mbox{det}\,\psi^{''}(\widetilde x)\vert^{\frac 12}}\overline{\frac{\partial g_3(\widetilde x)}{\partial \overline
z}}e^{-\overline{\mathcal A_2}(\widetilde x)}\int_{\partial\Omega}\frac{(\nu_1+i\nu_2)a e^{\mathcal A_1+\overline{\mathcal A_2}}}{ \widetilde z-z}d\sigma\nonumber\\
&-&\frac{e^{2i\tau\psi(\widetilde x)}}{\tau\vert \mbox{det}\,\psi^{''}(\widetilde x)\vert^\frac 12} \frac{\partial g_2(\widetilde x)}{\partial \overline z}e^{-\mathcal B_1(\widetilde x)}\int_{\partial\Omega}\frac{(\nu_1-i\nu_2)\overline b e^{\mathcal{B}_1+\overline{\mathcal B_2} }}{\overline{\widetilde z}-\overline z}d\sigma\nonumber\\
&-& \frac{2\pi (\mathcal Q_+a\overline be^{\mathcal A_1+\overline{\mathcal B_2}+2i\tau\psi}+\mathcal Q_-\overline c de^{\mathcal B_1+\overline{\mathcal A_2}-2i\tau\psi})(\widetilde x)}{\tau\vert\mbox{det}\,\psi^{''}(\widetilde x)\vert^{\frac 12}}
\end{eqnarray}
By (\ref{vikana}), (\ref{vikana1UP}), (\ref{vikana2}) and
(\ref{vikana3}), (\ref{losad}) there exist  numbers $\kappa,\kappa_0$ such that
the asymptotic formula (\ref{nika1}) holds true. $\square$

{\bf Acknowledgements.}\\
The first named author was partly supported by NSF grant DMS
0808130, and the second named author was partly supported by NSF,
a Walker Family Endowed Professorship, a Chancellor Professorship at
UC Berkeley and a Senior Clay Award. The visits of O.Imanuvilov to
Seattle were supported by NSF and the Department of Mathematics of
the University of Washington. O.Imanuvilov also thanks the Global
COE Program ``The Research and Training Center for New Development
in Mathematics" for support of a visit to the University of Tokyo.
The authors thank Professor J. Sawon for discussions on the proof of
Corollary \ref{simplyconnected}.

\end{document}